\documentclass[11pt,reqno]{amsart}

\usepackage[a4paper,margin=30mm]{geometry}
\usepackage[T1]{fontenc}
\usepackage[utf8]{inputenc}
\usepackage{lmodern}
\usepackage{amsmath,amssymb,amsthm,mathtools,mathrsfs,bm}
\usepackage{esint}
\usepackage{microtype}
\microtypesetup{protrusion=true,expansion=true,tracking=true}
\usepackage{enumitem}
\usepackage{booktabs,array,longtable,tabularx}
\usepackage{xspace}
\usepackage{graphicx}
\usepackage{float}
\usepackage{placeins}
\usepackage{needspace}
\usepackage[dvipsnames,svgnames,table]{xcolor}
\usepackage{tikz}
\usepackage{pgfplots}
\usepgfplotslibrary{fillbetween,groupplots}
\pgfplotsset{compat=1.18}

\usetikzlibrary{matrix,patterns,arrows,arrows.meta,calc,positioning,fit,backgrounds,decorations.pathreplacing,decorations.pathmorphing,shapes.geometric}
\usepackage[numbers,sort&compress]{natbib}
\usepackage[hidelinks]{hyperref}
\usepackage{aliascnt}
\usepackage[nameinlink,noabbrev]{cleveref}
\input{glyphtounicode}
\providecommand{\doi}[1]{\href{https://doi.org/#1}{doi: #1}}
\hypersetup{%
  pdftitle={Schwartz-class Gabor windows and the Balian--Low classification},%
  pdfauthor={Andrei Caragea and Goetz Pfander},%
  pdfsubject={Schwartz-class Gabor windows and the Balian--Low classification},%
  pdfkeywords={Schwartz Gabor window, Balian--Low theorem, modulation space, Zak transform, Zibulski--Zeevi matrix, VMO, symplectically rational lattice}%
}

\setlist{itemsep=.2em,topsep=.4em}
\newcommand{\thmpart}[1]{\textup{(}\textit{#1}\textup{)}}
\allowdisplaybreaks
\numberwithin{equation}{section}

\makeatletter
\AtBeginDocument{\let\uppercasenonmath\@gobble}
\def\@settitle{\begin{center}%
  \normalfont\fontsize{14.5}{18}\selectfont\bfseries
  \@title\par
  \end{center}%
}

\def\@captionheadfont{\normalfont\bfseries}
\renewcommand\section{\@startsection{section}{1}%
  \z@{.7\linespacing\@plus\linespacing}{.5\linespacing}%
  {\normalfont\bfseries\centering}}
\renewcommand\part{\@startsection{part}{0}{\z@}%
  {1.65\baselineskip \@plus 0.35\baselineskip \@minus 0.15\baselineskip}%
  {0.72\baselineskip}%
  {\centering\normalfont\large\bfseries}}
\renewcommand\subsection{\@startsection{subsection}{2}{\z@}%
  {0.85\baselineskip \@plus 0.25\baselineskip \@minus 0.10\baselineskip}%
  {0.42\baselineskip}%
  {\normalfont\bfseries}}
\renewcommand\subsubsection{\@startsection{subsubsection}{3}{\z@}%
  {0.70\baselineskip \@plus 0.20\baselineskip \@minus 0.10\baselineskip}%
  {0.30\baselineskip}%
  {\normalfont\bfseries}}
\makeatother

\newtheorem{theorem}{Theorem}[section]
\newaliascnt{proposition}{theorem}
\newtheorem{proposition}[proposition]{Proposition}
\aliascntresetthe{proposition}
\newaliascnt{lemma}{theorem}
\newtheorem{lemma}[lemma]{Lemma}
\aliascntresetthe{lemma}
\newaliascnt{corollary}{theorem}
\newtheorem{corollary}[corollary]{Corollary}
\aliascntresetthe{corollary}
\newaliascnt{consequence}{theorem}

\aliascntresetthe{consequence}
\newaliascnt{claim}{theorem}

\aliascntresetthe{claim}
\newaliascnt{conjecture}{theorem}

\aliascntresetthe{conjecture}
\theoremstyle{definition}
\newaliascnt{definition}{theorem}

\aliascntresetthe{definition}
\newaliascnt{question}{theorem}
\newtheorem{question}[question]{Question}
\aliascntresetthe{question}
\newaliascnt{example}{theorem}

\aliascntresetthe{example}
\newaliascnt{examples}{theorem}
\newtheorem{examples}[examples]{Examples}
\aliascntresetthe{examples}
\theoremstyle{remark}
\newaliascnt{remark}{theorem}
\newtheorem{remark}[remark]{Remark}
\aliascntresetthe{remark}

\crefname{theorem}{Theorem}{Theorems}
\Crefname{theorem}{Theorem}{Theorems}
\crefname{proposition}{Proposition}{Propositions}
\Crefname{proposition}{Proposition}{Propositions}
\crefname{lemma}{Lemma}{Lemmas}
\Crefname{lemma}{Lemma}{Lemmas}
\crefname{corollary}{Corollary}{Corollaries}
\Crefname{corollary}{Corollary}{Corollaries}
\crefname{consequence}{Consequence}{Consequences}
\Crefname{consequence}{Consequence}{Consequences}
\crefname{claim}{Claim}{Claims}
\Crefname{claim}{Claim}{Claims}
\crefname{conjecture}{Conjecture}{Conjectures}
\Crefname{conjecture}{Conjecture}{Conjectures}
\crefname{definition}{Definition}{Definitions}
\Crefname{definition}{Definition}{Definitions}
\crefname{example}{Example}{Examples}
\Crefname{example}{Example}{Examples}
\crefname{examples}{Examples}{Examples}
\Crefname{examples}{Examples}{Examples}
\crefname{remark}{Remark}{Remarks}
\Crefname{remark}{Remark}{Remarks}
\crefname{question}{Question}{Questions}
\Crefname{question}{Question}{Questions}
\crefname{figure}{Figure}{Figures}
\Crefname{figure}{Figure}{Figures}
\crefname{table}{Table}{Tables}
\Crefname{table}{Table}{Tables}
\crefname{section}{Section}{Sections}
\Crefname{section}{Section}{Sections}
\crefname{subsection}{Subsection}{Subsections}
\Crefname{subsection}{Subsection}{Subsections}
\crefname{equation}{Equation}{Equations}
\Crefname{equation}{Equation}{Equations}

\newcommand{\R}{\mathbb R}
\newcommand{\Z}{\mathbb Z}
\newcommand{\N}{\mathbb N}
\newcommand{\Q}{\mathbb Q}
\newcommand{\C}{\mathbb C}

\newcommand{\eps}{\varepsilon}
\newcommand{\covol}{\operatorname{covol}}

\newcommand{\dist}{\operatorname{dist}}
\newcommand{\supp}{\operatorname{supp}}

\newcommand{\rank}{\operatorname{rank}}
\newcommand{\Pf}{\operatorname{Pf}}

\newcommand{\G}{\mathcal G}
\newcommand{\Sclass}{\mathcal S}

\newcommand{\ip}[2]{\langle #1,#2\rangle}

\newcommand{\norm}[1]{\lVert #1\rVert}

\newcommand{\SG}{\ensuremath{(\mathrm{SG})}\xspace}
\newcommand{\mSG}{\mathrm{mSG}}
\newcommand{\Sp}{\operatorname{Sp}}
\newcommand{\Id}{\mathrm I}

\newcommand{\Szero}{S_0}
\newcommand{\Ana}{\operatorname{Ana}}

\newcommand{\Fell}[2]{\mathcal F\ell_{#2}^{#1}}
\newcommand{\FL}[2]{\mathcal FL_{#2}^{#1}}
\newcommand{\BMO}{\operatorname{BMO}}
\newcommand{\VMO}{\operatorname{VMO}}
\newcommand{\VMOloc}{\operatorname{VMO}_{\mathrm{loc}}}

\newcommand{\diag}{\operatorname{diag}}
\newcommand{\T}{\mathbb T}
\newcommand{\angles}[1]{\langle #1\rangle}

\newcommand{\tr}{\operatorname{tr}}

\definecolor{gapzero}{RGB}{35,145,160}
\definecolor{gapone}{RGB}{76,132,205}
\definecolor{gaptwo}{RGB}{220,125,35}
\definecolor{gapthree}{RGB}{190,155,25}
\definecolor{gapfour}{RGB}{45,135,75}
\definecolor{gapfive}{RGB}{38,90,160}
\definecolor{gapsix}{RGB}{145,70,165}
\definecolor{BLblue}{RGB}{46,102,164}
\definecolor{BLgreen}{RGB}{55,145,105}
\definecolor{BLgold}{RGB}{218,154,54}
\definecolor{BLred}{RGB}{190,58,64}

\definecolor{resultgreen}{RGB}{220,243,217}
\definecolor{resultgreenborder}{RGB}{54,120,60}
\definecolor{resultred}{RGB}{211,67,67}
\definecolor{resultredborder}{RGB}{145,24,24}
\pgfplotsset{BLpanel/.style={
  width=.305\textwidth,
  height=.255\textwidth,
  xmin=0,xmax=10.4,
  ymin=0,ymax=1,
  xtick={0,2,4,6,8,10},
  ytick={0,0.5,1},
  yticklabels={$0$,$\tfrac12$,$1$},
  xlabel={$s$},ylabel={$1/p$},
  xlabel style={font=\scriptsize},ylabel style={font=\scriptsize},
  tick label style={font=\scriptsize},title style={font=\small},
  grid=both,major grid style={draw=black!10},
  axis line style={black!50},tick style={black!50},clip=false}}

\newcommand{\BLRedPoint}[2]{%
  \addplot[only marks,mark=*,mark size=1.35pt,draw=resultredborder,fill=resultredborder]
    coordinates {(#1,#2)};}

\newcommand{\BLGapPoint}[3]{%
  \addplot[only marks,mark=*,mark size=1.35pt,
           draw=#1!85!black,fill=#1!85!black]
    coordinates {(#2,#3)};}

\title[Schwartz Gabor windows and the Balian--Low classification]{Schwartz-class Gabor windows\\and the Balian--Low classification}

\author[A. Caragea]{Andrei Caragea}
\author[G. Pfander]{G\"otz Pfander}
\address{Mathematical Institute for Machine Learning and Data Science, Katholische Universit\"at Eichst\"att-Ingolstadt, Germany}

\email{andrei.caragea@gmail.com, pfander@ku.de}

\subjclass[2020]{42C15, 42C40, 42B35}
\date{12 October 2026; BalianLowWork-V29}
\keywords{Gabor frame, Schwartz window, Feichtinger algebra, Balian--Low theorem, modulation space, Zak transform, Zibulski--Zeevi transform, VMO, symplectically rational lattice, symplectic index gap, rectangular rank gap}

\newcommand{\Retr}{\mathscr R}
\definecolor{Sblue}{RGB}{43,99,161}
\definecolor{Steal}{RGB}{27,132,117}
\definecolor{Sorange}{RGB}{201,115,40}
\hypersetup{pdftitle={BalianLowWork-V29: Schwartz-class Gabor windows and the Balian--Low classification}}

\begin{document}
\raggedbottom

\begin{abstract}
A lattice Gabor frame consists of a window function and its time--frequency
shifts along a lattice. We classify the lattices that admit Schwartz-class
and, equivalently, Feichtinger frame windows.
Except for symplectically irrational lattices
of critical covolume~$1$, we determine the lattice-dependent achievable window
smoothness and decay using the scale of modulation spaces. For
symplectically rational lattices, including the rational lattices typically
used in digital applications, the classification is governed by an integer
lattice parameter introduced herein, the symplectic index gap. Our results extend the classical and
amalgam Balian--Low theorems.

The companion paper \emph{Tilings, Packings, and Smooth and Compactly
Supported Gabor Windows} develops the associated tiling--packing theorem
and compact-support constructions.
\end{abstract}

\maketitle
\setcounter{tocdepth}{1}
\tableofcontents

\part{Introduction and mathematical framework}\label{part:introduction-framework}

\section{Background and main results}\label{sec:introduction}

\subsection{Motivation}\label{subsec:intro-motivation}

In 1946, Dennis Gabor proposed representing communication signals by integer
time and frequency shifts of a Gaussian elementary signal, thereby combining optimal time and frequency localization of carrier signals
\cite{Gabor1946}.  It was subsequently recognized that the integer lattice
Gaussian system does not provide a stable expansion of any square
integrable signal, and the Balian--Low theorem revealed a general obstruction to simultaneous time and frequency
localization at critical density one.  For communications engineers,
this remains a relevant and widely studied problem within the area of
multicarrier pulse shaping.  In channels with both delay and Doppler spread, suitable
time--frequency localization of the transmit and receive pulses can reduce
intersymbol and intercarrier interference
\cite{KozekMolisch1998,MatzSchafhuberGrochenigHartmannHlawatsch2007}.
The 5G New Radio physical-layer specification uses
cyclic-prefix orthogonal frequency division multiplexing waveforms,
which are finite implementations of
Gabor-type time--frequency signaling.

\subsection{Gabor frames}
\label{subsec:intro-frames}

In mathematics and many applications one seeks simple atomic decompositions of the form
\[
 f=\sum_n\ip{f}{\phi_n}\phi_n,
\]
as for an orthonormal basis $(\phi_n)$ in a Hilbert space, using atoms with
the localization or regularity required by the application. Such
decompositions are supplied by Parseval frames and are studied within
frame theory; a prominent example is given by Gabor frames
\cite{ChristensenFrames2016,HeilBasisTheory2011,Grochenig2001}.

A \emph{full-rank lattice} in $\R^d$ is a subgroup of the form
$L=M\Z^d$ with $M\in GL_d(\R)$.  Its covolume is
$\covol(L)=|\det M|$ and its density is $\covol(L)^{-1}$.
For $x,\xi,t\in\R^d$, the \emph{time shift},
\emph{frequency shift}, and \emph{time--frequency shift} are
\[
 (T_xf)(t)=f(t-x),
 \qquad
 (M_\xi f)(t)=e^{2\pi i\ip{\xi}{t}}f(t),
 \qquad
 \pi(x,\xi)=M_\xi T_x.
\]
A family $(f_j)_{j\in J}\subset L^2(\R^d)$ is a \emph{frame} if there are
constants $0<A\le B<\infty$ such that
\[
 A\norm{f}_2^2
 \le \sum_{j\in J}|\ip{f}{f_j}|^2
 \le B\norm{f}_2^2,
 \qquad f\in L^2(\R^d).
\]
A frame is \emph{tight} if $A=B$ and \emph{Parseval} if $A=B=1$.  A
full-rank lattice $\Lambda\subset\R^{2d}$ indexes the \emph{Gabor system}
\[
 \G(g,\Lambda)=\{\pi(\lambda)g:\lambda\in\Lambda\};
\]
when this family is a frame, it is called a \emph{Gabor frame}.

A Parseval frame provides an orthonormal-basis-type reconstruction without
requiring orthogonality and allowing for redundancy in atomic decompositions.  In particular, in the case of a Parseval Gabor
frame, for every $f\in L^2(\R^d)$,
\[
 f=\sum_{\lambda\in\Lambda}
     \ip{f}{\pi(\lambda)g}\,\pi(\lambda)g,
\]
with unconditional convergence in $L^2(\R^d)$.

For signals with decay in time and/or frequency, respectively smoothness in
time and/or frequency, the choice of window controls the regularity of the
analysis and synthesis procedures.  For such signals with additional decay,
the localization of both $g$ and $\widehat g$ also controls the decay of the
Gabor coefficient sequence
$\{\ip{f}{\pi(\lambda)g}\}_{\lambda\in\Lambda}$ and hence the convergence of
natural finite truncations; compare \cite{DaubechiesGrossmannMeyer1986} and
\cite[Chapters~5--7]{Grochenig2001}.  Without imposing any regularity on the
window, Bekka proved that every full-rank phase-space lattice $\Lambda$
admits some $g\in L^2(\R^d)$ for which $\G(g,\Lambda)$ is a Gabor frame if
and only if $\covol(\Lambda)\le1$; see \cite[Theorem~4]{Bekka2004}.  The
classical Balian--Low theorem shows, though, that an orthonormal Gabor basis
cannot be generated by a window that is well localized in both time and
frequency.  This motivates the use of redundant frames in time--frequency
analysis.  Such a frame may subsequently be converted into a Parseval frame;
see \Cref{subsec:frames-duality-density}.

\subsection{The classical Balian--Low theorem}
\label{subsec:intro-classical-balian-low}

One natural way to impose regularity is to require that the window has a
finite uncertainty product. After division by $\|g\|_2^4$, this product is
minimized over the nonzero unit ball $0<\|g\|_2\le1$ in $L^2(\R^d)$ by
Gabor's centered Gaussian window. At critical density, a finite uncertainty
product is impossible:

\begin{theorem}[Classical Balian--Low theorem]
\label{thm:full-multivariate-balian-low}
Let $\Lambda\subset\R^{2d}$ be a full-rank phase-space lattice.  The following
are equivalent.
\begin{enumerate}[label=\thmpart{\alph*}]
\item There exists $g\in L^2(\R^d)$ such that $\G(g,\Lambda)$ is a Gabor
frame and
\[
 \left(\int_{\R^d}|x|^2|g(x)|^2\,dx\right)
 \left(\int_{\R^d}|\xi|^2|\widehat g(\xi)|^2\,d\xi\right)<\infty.
\]
\item The strict density condition $\covol(\Lambda)<1$ holds.
\end{enumerate}
\end{theorem}

For general full-rank phase-space lattices $\Lambda\subset\R^{2d}$,
combining the weak arbitrary-lattice Balian--Low theorem of
Gr\"ochenig--Han--Heil--Kutyniok
\cite[Theorem~8]{GrochenigHanHeilKutyniok2002} with the canonical-dual
regularity theorem of Lee--Philipp--Voigtlaender
\cite[Theorem~1.1]{LeePhilippVoigtlaender2023} yields the necessity
implication \textup{(a)}$\Rightarrow$\textup{(b)} in
the classical Balian--Low Theorem~\ref{thm:full-multivariate-balian-low}.  To our knowledge, this
arbitrary-lattice consequence has not previously been stated in this form.  The sufficiency of strict density follows, in the symplectically rational
case, from the finite-dimensional construction used to prove
\Cref{thm:modulation-classification}, and, in the symplectically irrational
case, from the general existence theorem of Enstad--Thiel--Vilalta
\cite[Theorem~C]{EnstadThielVilalta2025}.

\subsection{The symplectic integrality index and the symplectic index gap}
\label{subsec:intro-gabor}
\label{subsec:intro-symplectic-index-gap}

On $\R^{2d}=\R^d\times\R^d$ we define the symplectic form
$\sigma((x,\xi),(y,\eta))=\ip{\eta}{x}-\ip{\xi}{y}$.  The adjoint lattice
of a full-rank phase-space lattice $\Lambda$ is
\[
 \Lambda^\circ
 =\{z\in\R^{2d}:\sigma(z,\lambda)\in\Z
        \text{ for every }\lambda\in\Lambda\}.
\]
Its covolume satisfies
\[
 \covol(\Lambda^\circ)=\covol(\Lambda)^{-1}.
\]
If $\Lambda=\Gamma\times\Phi$ is a product of full-rank Euclidean lattices
in $\R^d$, then $\Lambda^\circ=\Phi^*\times\Gamma^*$.  Here the Euclidean
dual of a full-rank lattice $L\subset\R^d$ is
\[
 L^*=\{\xi\in\R^d:\ip{\xi}{\ell}\in\Z
       \text{ for every }\ell\in L\}.
\]
If $L=M\Z^d$, then $L^*=M^{-T}\Z^d$ and
$\covol(L^*)=\covol(L)^{-1}$.

Following Enstad--Thiel--Vilalta
\cite[Definition~B]{EnstadThielVilalta2025}, the integral symplectic subgroup is
\[
 \Lambda_{\rm int}=\Lambda\cap\Lambda^\circ
 =\{\mu\in\Lambda:\sigma(\mu,\lambda)\in\Z
                 \text{ for every }\lambda\in\Lambda\}.
\]
We call $\Lambda$ \emph{symplectically rational} if
$[\Lambda:\Lambda_{\rm int}]<\infty$, and \emph{symplectically irrational}
otherwise. In the rational case, the index is a perfect square;
see \Cref{cor:square-index}.  We
define the extended \emph{symplectic integrality index} by
\[
 \nu(\Lambda)=
 \begin{cases}
 [\Lambda:\Lambda_{\rm int}]^{1/2}\in\N,
    &\Lambda\text{ symplectically rational},\\
 +\infty,&\Lambda\text{ symplectically irrational}.
 \end{cases}
\]
Thus $\nu(\Lambda)^2$ counts the cosets of the subgroup of lattice vectors
having integer symplectic pairing with every lattice vector; in particular,
$\nu(\Lambda)=1$ exactly when $\Lambda\subset\Lambda^\circ$, that is,
when $\sigma(\mu,\lambda)\in\Z$ for every $\mu,\lambda\in\Lambda$.

\phantomsection\label{def:symplectic-index-gap}
If $\Lambda$ is symplectically rational, then so is $\Lambda^\circ$:
$\Lambda\cap\Lambda^\circ$ has finite index in both lattices, and
$(\Lambda^\circ)^\circ=\Lambda$.
For a symplectically rational lattice, we define the \emph{symplectic index
gap} by
\[
 \Delta(\Lambda)=\nu(\Lambda)-\nu(\Lambda^\circ)\in\Z.
\]
The symplectic index gap is defined only for symplectically rational
lattices.  Since $\Lambda_{\rm int}=\Lambda\cap\Lambda^\circ$, the
finite-index covolume identities give
\[
 \nu(\Lambda)^2\covol(\Lambda)
 =\covol(\Lambda_{\rm int})
 =\nu(\Lambda^\circ)^2\covol(\Lambda^\circ)
 =\frac{\nu(\Lambda^\circ)^2}{\covol(\Lambda)}.
\]
Taking positive square roots gives
\begin{equation}\label{eq:index-covolume-ratio}
 \covol(\Lambda)=\frac{\nu(\Lambda^\circ)}{\nu(\Lambda)}.
\end{equation}
Consequently, for every symplectically rational full-rank phase-space lattice,
\[
 \Delta(\Lambda)=0
 \quad\Longleftrightarrow\quad
 \covol(\Lambda)=1,
 \qquad
 \Delta(\Lambda)>0
 \quad\Longleftrightarrow\quad
 \covol(\Lambda)<1.
\]

On a diagonal rational representative
\[
 \Lambda_D=D\mathbb Z^d\times\mathbb Z^d,
 \qquad D=\operatorname{diag}(b_1/a_1,\ldots,b_d/a_d),
\]
with coordinatewise coprime numerators and denominators, one has
\[
 \nu(\Lambda_D)=N:=\prod_i a_i,
 \qquad
 \nu(\Lambda_D^\circ)=R:=\prod_i b_i,
 \qquad
 \Delta(\Lambda_D)=N-R,
\]
see \Cref{sec:zak-method} for the detailed diagonal calculation.  The integer
$N-R$ is the symplectic index gap of the diagonal representative.

\begin{examples}\label{ex:intro-SG}
Use the coordinate order $(x_1,x_2,\omega_1,\omega_2)$ and write
$\Lambda_j=M_j\Z^4$.
\begin{enumerate}[label=(\roman*)]
\item For $M_1=\operatorname{diag}(1,1,1,1/\sqrt2)$, the lattice is
symplectically irrational and $\covol(\Lambda_1)=1/\sqrt2<1$.  Here
$\nu(\Lambda_1)=\nu(\Lambda_1^\circ)=+\infty$.
\item For $M_2=\operatorname{diag}(1,1,1,1/2)$, one has
$\covol(\Lambda_2)=1/2$, $\nu(\Lambda_2)=2$,
$\nu(\Lambda_2^\circ)=1$, and $\Delta(\Lambda_2)=1<d=2$.
\item For $M_3=\operatorname{diag}(7,1,1/3,1/3)$, one has
$\covol(\Lambda_3)=7/9$, $\nu(\Lambda_3)=9$,
$\nu(\Lambda_3^\circ)=7$, and $\Delta(\Lambda_3)=2=d$.
\item Let $K=\begin{psmallmatrix}2&1\\0&2\end{psmallmatrix}$ and
$M_4=I_2\times K^{-T}$, where
$K^{-T}=\begin{psmallmatrix}1/2&0\\-1/4&1/2\end{psmallmatrix}$.  Then
$\covol(\Lambda_4)=1/4$, $\nu(\Lambda_4)=4$,
$\nu(\Lambda_4^\circ)=1$, and $\Delta(\Lambda_4)=3>d=2$.
\end{enumerate}
\end{examples}

\subsection{The amalgam Balian--Low theorem and the existence of Schwartz-class windows}
\label{subsec:intro-amalgam-schwartz}

The Schwartz class $\Sclass(\R^d)$ consists of all
$f\in C^\infty(\R^d)$ such that
\[
 \sup_{x\in\R^d}|x^\alpha\partial^\beta f(x)|<\infty
\]
for every pair of multiindices $\alpha,\beta$.  Its continuous dual is the
space of tempered distributions $\Sclass'(\R^d)$.  With the Gaussian
$\phi(t)=e^{-\pi|t|^2}$, $t\in\R^d$, define the short-time Fourier
transform \cite[Chapter~3]{Grochenig2001} for $f\in\Sclass'(\R^d)$ by
\[
 V_\phi f(x,\omega)=\ip{f}{M_\omega T_x\phi}.
\]
The \emph{Feichtinger algebra} \cite{Feichtinger1980,Feichtinger1981}
(see also \cite[Theorem~11.3.7 and Proposition~11.3.2]{Grochenig2001}) is
\[
 \Szero(\R^d)=M^1(\R^d)
 =\left\{f\in\Sclass'(\R^d):
   \int_{\R^{2d}}|V_\phi f(x,\omega)|\,dx\,d\omega<\infty
  \right\}.
\]
Every element of $\Szero$ is a continuous $L^2$ function.
Moreover, $\Sclass(\R^d)\subseteq\Szero(\R^d)$.

\begin{theorem}[Amalgam Balian--Low theorem]
\label{thm:main-gabor}
Let $\Lambda\subset\R^{2d}$ be a full-rank phase-space lattice.  The following
are equivalent.
\begin{enumerate}[label=\thmpart{\alph*}]
\item There exists $g\in\Sclass(\R^d)$ such that $\G(g,\Lambda)$ is a Gabor
frame.
\item There exists $g\in\Szero(\R^d)$ such that $\G(g,\Lambda)$ is a Gabor
frame.
\item $\covol(\Lambda)<1$ and, if $\Lambda$ is symplectically rational,
$\Delta(\Lambda)\ge d$.
\begingroup
\renewcommand{\theequation}{\ensuremath{\mathrm{SG}}}
\refstepcounter{equation}\label{eq:SG}\hfill\textup{(SG)}
\endgroup
\end{enumerate}
The frames in \thmpart{a} and \thmpart{b} may be chosen Parseval.
\end{theorem}

\begin{remark}\label{rem:gelfand-shilov-windows}
The equivalent conditions in the amalgam Balian--Low
Theorem~\ref{thm:main-gabor} also characterize the existence of a frame
window in the Gelfand--Shilov space $S^{1/2,1/2}(\R^d)$, in the Roumieu
convention. Indeed, finite linear combinations of time--frequency shifts
of a Gaussian lie in this space and are dense in $M^1$; see
Gr\"ochenig~\cite[Section~5.1]{Grochenig2007Weights} and
\cite[Proposition~11.4.2]{Grochenig2001}. A sufficiently close $M^1$
approximation of a frame window remains a frame window on the same lattice.

These conditions also characterize the existence of a Parseval window
in the Roumieu class $S^{1,1}(\R^d)$, with exponential decay in time
and frequency; see Jaffard~\cite[Proposition~2 and p.~470]{Jaffard1990}
and, in dimension one,
B\"olcskei--Janssen~\cite[Theorem~5]{BolcskeiJanssen2000}.
The exponential localization argument is given in the proof of this
remark in \Cref{sec:completion-all-lattices}.

\end{remark}

Condition~\eqref{eq:SG} reduces to $\covol(\Lambda)<1$ in the
symplectically irrational case and to
$\nu(\Lambda)\ge\nu(\Lambda^\circ)+d$ in the symplectically rational case;
the latter inequality already implies strict density.
The implication \textup{(c)}$\Rightarrow$\textup{(a)} follows from
Enstad--Thiel--Vilalta~\cite[Theorem~C]{EnstadThielVilalta2025}.  In the
symplectically irrational case, Jakobsen--Luef
\cite[Theorem~5.4]{JakobsenLuef2020} proved the corresponding existence of a
tight Gabor-frame window in $\Szero(\R^d)$; the Schwartz-class conclusion
used here is furnished directly by Enstad--Thiel--Vilalta.  In the rational
case,
\Cref{sec:zak-method,sec:construction-strategy,sec:terminal-construction,sec:elementary-perturbation,sec:lower-polar,sec:lower-completion}
give the framework and construction in the diagonal case, and
\Cref{sec:lattice-theory} transfers the construction metaplectically to all
symplectically rational lattices.  This gives an independent constructive
proof of Schwartz-class existence, that is, of
\textup{(c)}$\Rightarrow$\textup{(a)} in the rational case.  The necessity implication \textup{(b)}$\Rightarrow$\textup{(c)} builds on
our extension of the continuous common-zero theorem of de Dios
Pont--Liehr--Taylor \cite{deDiosLiehrTaylor2026}.  We provide a lengthy but elementary independent
proof of their fundamental common-zero theorem in
Sections~\ref{app:pfaffian-common-zero}--\ref{app:boundary-pfaffian}.

In the classical Balian--Low Theorem~\ref{thm:full-multivariate-balian-low} and the amalgam Balian--Low Theorem~\ref{thm:main-gabor}, the term ``Gabor
frame'' may be replaced by ``tight Gabor frame'' or ``Parseval Gabor frame,''
as discussed in \Cref{subsec:frames-duality-density}.

In the preceding examples, condition~\eqref{eq:SG} holds for
$\Lambda_1,\Lambda_3,$ and $\Lambda_4$, and fails for $\Lambda_2$.

The one-window statement has the following sharp multiwindow extension.  For
$\mathbf g=(g_1,\ldots,g_Q)$, write
\[
 \G(\mathbf g,\Lambda)
 =\{\pi(\lambda)g_j:\lambda\in\Lambda,\ 1\le j\le Q\}.
\]

\begin{theorem}
\label{thm:main-multiwindow}
Let $Q\ge1$ be an integer and let $\Lambda\subset\R^{2d}$ be a full-rank phase-space
lattice.  For $\mathbf g=(g_1,\ldots,g_Q)$, the following are equivalent.
\begin{enumerate}[label=\thmpart{\alph*}]
\item There exist $g_1,\ldots,g_Q\in\Sclass(\R^d)$ such that
$\G(\mathbf g,\Lambda)$ is a Gabor frame.
\item There exist $g_1,\ldots,g_Q\in\Szero(\R^d)$ such that
$\G(\mathbf g,\Lambda)$ is a Gabor frame.
\item The lattice satisfies the condition $(\mathrm{mSG}_Q)$, namely,
\begin{align*}
 \covol(\Lambda)<Q
 \quad\text{and, if $\Lambda$ is symplectically rational,}\quad
 Q\,\nu(\Lambda)\ge\nu(\Lambda^\circ)+d.
 \tag{\ensuremath{\mathrm{mSG}_Q}}\label{eq:mSG}
\end{align*}
\stepcounter{equation}
\end{enumerate}
\end{theorem}

In parts~\textup{(a)} and~\textup{(b)}, the term ``Gabor frame'' may be
replaced by ``tight Gabor frame'' or ``Parseval Gabor frame.''

In general, the minimum number of Schwartz windows, equivalently of
Feichtinger-algebra windows, is
$q_{\min}^{\Sclass}(\Lambda)=
\left\lceil(\nu(\Lambda^\circ)+d)/\nu(\Lambda)\right\rceil$ if
$\Lambda$ is symplectically rational, and
$q_{\min}^{\Sclass}(\Lambda)=\lfloor\covol(\Lambda)\rfloor+1$
otherwise.  For symplectically rational lattices, the one-window
criterion depends only on the threshold $\Delta(\Lambda)\ge d$, whereas the
multiwindow criterion uses the two symplectic integrality indices separately,
as illustrated in
\Cref{ex:intro-multiwindow-redundancy}.

\begin{examples}
\label{ex:intro-multiwindow-redundancy}
\begin{enumerate}[label=\textup{(\alph*)},leftmargin=2.5em]
\item Let
\[
 \Lambda=\bigl((\tfrac12\mathbb Z)\times\mathbb Z^9\bigr)
          \times\mathbb Z^{10}\subset\mathbb R^{20}.
\]
This separable time--frequency lattice has covolume $1/2$ and density $2$.
Thus any one-window Gabor frame on it is already overcomplete by a factor
of $2$; it has the same total density as the union of two
orthonormal bases.  Its symplectic integrality indices are
\[
 \nu(\Lambda)=2,
 \qquad
 \nu(\Lambda^\circ)=1.
\]
If all windows are required to lie in the Schwartz class, or even only in the
Feichtinger algebra, \Cref{thm:main-multiwindow} requires
\[
 2Q=Q\,\nu(\Lambda)
 \ge d+\nu(\Lambda^\circ)=10+1=11.
\]
Consequently, we require at least $6$ Schwartz-class windows for this
lattice, and the total density, or total redundancy, of such a regular
multiwindow Gabor frame is at least
\[
 \frac{Q}{\covol(\Lambda)}=2Q\ge12.
\]

\item Consider instead
\[
 \Lambda=\bigl((\tfrac{91}{101}\mathbb Z)\times\mathbb Z^9\bigr)
          \times\mathbb Z^{10}\subset\mathbb R^{20}.
\]
Here $d=10$, $\nu(\Lambda)=101$, $\nu(\Lambda^\circ)=91$, and hence
$\Delta(\Lambda)=10=d$.  Therefore one Schwartz-class window already
suffices.  The covolume is $91/101$, so the density, and thus the
one-window redundancy measure, is
\[
 \covol(\Lambda)^{-1}=\frac{101}{91}\approx1.11.
\]
This contrasts with part~\textup{(a)}: two lattices may have comparable
ordinary density while their integrality data impose very different minimum
numbers of regular windows.

\item Let $d=5$ and put
\[
 \Lambda=\bigl(\tfrac14\Z\times\Z^4\bigr)\times\Z^5,\qquad
 \Gamma=\bigl(\tfrac12\Z\times\Z^4\bigr)\times\Z^5.
\]
For $\Lambda$, the integrality indices are
\[
 N_\Lambda=\nu(\Lambda)=4,\qquad
 R_\Lambda=\nu(\Lambda^\circ)=1.
\]
By \Cref{thm:main-multiwindow}, a Schwartz $Q$-window Gabor frame
on $\Lambda$ exists exactly when
\[
 4Q=QN_\Lambda\ge d+R_\Lambda=5+1=6.
\]
Thus the minimum is $Q=2$.  Choose $g_1,g_2\in\Sclass(\R^5)$ such
that $\G(g_1,\Lambda)\cup\G(g_2,\Lambda)$ is a Gabor frame.

The sublattice $\Gamma$ has index two in $\Lambda$.  More explicitly,
with $a=(1/4,0,0,0,0)\in\R^5$,
\[
 \Lambda=\Gamma\,\dot\cup\,\bigl((a,0)+\Gamma\bigr).
\]
Consequently, the classical coset construction gives
\[
 \G(g_1,\Lambda)\cup\G(g_2,\Lambda)
 =\G(g_1,\Gamma)\cup\G(g_2,\Gamma)
  \cup\G(T_ag_1,\Gamma)\cup\G(T_ag_2,\Gamma).
\]
The equality retains the indices and their multiplicities.  Hence the
right-hand side is a four-window Schwartz Gabor frame on $\Gamma$,
with the same frame bounds as the original two-window frame.

Constructing a frame directly on $\Gamma$ requires fewer windows.
Here
\[
 N_\Gamma=\nu(\Gamma)=2,\qquad
 R_\Gamma=\nu(\Gamma^\circ)=1,
\]
so the sharp condition is
\[
 2Q=QN_\Gamma\ge d+R_\Gamma=5+1=6.
\]
Thus three Schwartz windows are necessary and sufficient on $\Gamma$,
whereas the classical coset construction supplies four.  Passing to a
sublattice by reindexing an existing frame need not produce the minimum
number of Schwartz windows.
\end{enumerate}
\end{examples}

\subsection{The modulation-space classification}
\label{subsec:intro-modulation-classification}

The scale of modulation spaces gives a quantitative measure of simultaneous
time--frequency localization
\cite{FeichtingerGrochenig1997,Grochenig2001}.  With the Gaussian
$\phi$ and the short-time Fourier transform $V_\phi$, we use modulation
spaces here only for the $L^2$ functions that occur as Gabor-frame windows.
For $g\in L^2(\mathbb R^d)$, we write $g\in M_s^p(\mathbb R^d)$ when
\[
 \langle z\rangle^sV_\phi g(z)\in L^p(\mathbb R^{2d}),
 \qquad 1\le p\le\infty,
\]
where $\langle z\rangle=(1+|z|^2)^{1/2}$; see
\cite[Chapter~11]{Grochenig2001}.

The following sharp classification contains the rational forms of both
the classical and amalgam Balian--Low obstructions.

\begin{theorem}
\label{thm:symplectic-index-balian-low}
Let $\Lambda\subset\R^{2d}$ be symplectically rational, let
$1\le p\le\infty$, and let $s\in\R$. The existence of
$g\in L^2(\R^d)\cap M_s^p(\R^d)$ such that $\G(g,\Lambda)$ is a Gabor
frame is classified as follows.
\begin{enumerate}[label=\thmpart{\roman*},leftmargin=2.5em]
\item If $\Delta(\Lambda)<0$, no such window exists.
\item If $0\le\Delta(\Lambda)\le d-1$ and $p<\infty$, the condition is
\begin{align}
 s&<2\bigl(\Delta(\Lambda)+1\bigr)(1-1/p),
 &&\text{for }1\le p\le2,\label{eq:symplectic-index-gap-necessary-low-p}\\
 s&<d+\Delta(\Lambda)+1-2d/p,
 &&\text{for }2<p<\infty.\label{eq:symplectic-index-gap-necessary-high-p}
\end{align}
\item If $0\le\Delta(\Lambda)\le d-1$ and $p=\infty$, the condition is
\begin{equation}\label{eq:symplectic-index-gap-necessary-infty}
 s\le d+\Delta(\Lambda)+1.
\end{equation}
\item If $\Delta(\Lambda)\ge d$, every $p,s$ is admissible, and a
Schwartz-class window may be chosen.
\end{enumerate}
Whenever a window exists in this statement, it may be chosen Parseval.
\end{theorem}

\begin{figure}[!t]
\centering
% Hatching angles use the physical axis aspect ratio, not just data slopes.
% The perpendicular spacing is (2/3)*(5/sqrt(2)) pt.
\pgfdeclarepatternformonly{BLphasehatchV11}
  {\pgfqpoint{-0.50000000pt}{-0.50000000pt}}{\pgfqpoint{4.18849759pt}{3.56429031pt}}{\pgfqpoint{3.68849759pt}{3.06429031pt}}{%
  \pgfsetlinewidth{0.35pt}
  \pgfpathmoveto{\pgfqpoint{0.00000000pt}{3.06429031pt}}
  \pgfpathlineto{\pgfqpoint{3.68849759pt}{0.00000000pt}}
  \pgfusepath{stroke}}
\pgfdeclarepatternformonly{BLflathatchV11gap0}
  {\pgfqpoint{-0.50000000pt}{-0.50000000pt}}{\pgfqpoint{10.57042917pt}{2.92436258pt}}{\pgfqpoint{10.07042917pt}{2.42436258pt}}{%
  \pgfsetlinewidth{0.35pt}
  \pgfpathmoveto{\pgfqpoint{0.00000000pt}{0.00000000pt}}
  \pgfpathlineto{\pgfqpoint{10.07042917pt}{2.42436258pt}}
  \pgfusepath{stroke}}
\pgfdeclarepatternformonly{BLflathatchV11gap1}
  {\pgfqpoint{-0.50000000pt}{-0.50000000pt}}{\pgfqpoint{5.93323592pt}{3.11600248pt}}{\pgfqpoint{5.43323592pt}{2.61600248pt}}{%
  \pgfsetlinewidth{0.35pt}
  \pgfpathmoveto{\pgfqpoint{0.00000000pt}{0.00000000pt}}
  \pgfpathlineto{\pgfqpoint{5.43323592pt}{2.61600248pt}}
  \pgfusepath{stroke}}
\pgfdeclarepatternformonly{BLflathatchV11gap2}
  {\pgfqpoint{-0.50000000pt}{-0.50000000pt}}{\pgfqpoint{4.52572268pt}{3.40746638pt}}{\pgfqpoint{4.02572268pt}{2.90746638pt}}{%
  \pgfsetlinewidth{0.35pt}
  \pgfpathmoveto{\pgfqpoint{0.00000000pt}{0.00000000pt}}
  \pgfpathlineto{\pgfqpoint{4.02572268pt}{2.90746638pt}}
  \pgfusepath{stroke}}
\pgfdeclarepatternformonly{BLflathatchV11gap3}
  {\pgfqpoint{-0.50000000pt}{-0.50000000pt}}{\pgfqpoint{3.89804058pt}{3.77218723pt}}{\pgfqpoint{3.39804058pt}{3.27218723pt}}{%
  \pgfsetlinewidth{0.35pt}
  \pgfpathmoveto{\pgfqpoint{0.00000000pt}{0.00000000pt}}
  \pgfpathlineto{\pgfqpoint{3.39804058pt}{3.27218723pt}}
  \pgfusepath{stroke}}
\pgfdeclarepatternformonly{BLflathatchV11gap4}
  {\pgfqpoint{-0.50000000pt}{-0.50000000pt}}{\pgfqpoint{3.56429031pt}{4.18849759pt}}{\pgfqpoint{3.06429031pt}{3.68849759pt}}{%
  \pgfsetlinewidth{0.35pt}
  \pgfpathmoveto{\pgfqpoint{0.00000000pt}{0.00000000pt}}
  \pgfpathlineto{\pgfqpoint{3.06429031pt}{3.68849759pt}}
  \pgfusepath{stroke}}
\setlength{\tabcolsep}{2pt}
\pgfplotsset{BLpanel/.append style={scale only axis,
  width=.26\textwidth,height=.18\textwidth,
  xmin=-1,xmax=11,xtick={0,2,4,6,8,10}}}
\renewcommand{\arraystretch}{1.0}
\begin{tabular}{ccc}
\begin{tikzpicture}
\begin{axis}[BLpanel,xlabel={},title={$\Delta(\Lambda)=0$}]
  \path[fill=gapzero!58,draw=none]
    (axis cs:-1,0) -- (axis cs:-1,1) -- (axis cs:0,1) --
    (axis cs:1,0.5) -- (axis cs:6,0) -- cycle;
  \path[fill=resultred,draw=none]
    (axis cs:0,1) -- (axis cs:11,1) -- (axis cs:11,0) --
    (axis cs:6,0) -- (axis cs:1,0.5) -- cycle;
  \path[pattern=BLphasehatchV11,pattern color=gapzero!80!black,draw=none]
    (axis cs:-1,0) -- (axis cs:6,0) -- (axis cs:-1,0.7) -- cycle;
  \path[pattern=BLflathatchV11gap0,pattern color=gapzero!80!black,draw=none]
    (axis cs:-1,0) -- (axis cs:2,0) -- (axis cs:0,1) -- (axis cs:-1,1) -- cycle;
  \addplot[resultredborder,very thick]
    coordinates {(0,1) (1,0.5) (6,0)};
  \addplot[gapzero!85!black,densely dotted,line width=.8pt]
    coordinates {(-1,0.7) (1,0.5)};
  \addplot[gapzero!85!black,densely dotted,line width=.8pt]
    coordinates {(1,0.5) (2,0)};
  \BLGapPoint{gapzero}{6}{0}
  \BLGapPoint{gapzero}{0}{0.5}
  \BLRedPoint{1}{0.5}
  \BLRedPoint{0}{1}
\end{axis}
\end{tikzpicture} &
\begin{tikzpicture}
\begin{axis}[BLpanel,xlabel={},ylabel={},title={$\Delta(\Lambda)=1$}]
  \path[fill=gapone!58,draw=none]
    (axis cs:-1,0) -- (axis cs:-1,1) -- (axis cs:0,1) --
    (axis cs:2,0.5) -- (axis cs:7,0) -- cycle;
  \path[fill=resultred,draw=none]
    (axis cs:0,1) -- (axis cs:11,1) -- (axis cs:11,0) --
    (axis cs:7,0) -- (axis cs:2,0.5) -- cycle;
  \path[pattern=BLphasehatchV11,pattern color=gapone!80!black,draw=none]
    (axis cs:-1,0) -- (axis cs:7,0) -- (axis cs:-1,0.8) -- cycle;
  \path[pattern=BLflathatchV11gap1,pattern color=gapone!80!black,draw=none]
    (axis cs:-1,0) -- (axis cs:4,0) -- (axis cs:0,1) -- (axis cs:-1,1) -- cycle;
  \addplot[resultredborder,very thick]
    coordinates {(0,1) (2,0.5) (7,0)};
  \addplot[gapone!85!black,densely dotted,line width=.8pt]
    coordinates {(-1,0.8) (2,0.5)};
  \addplot[gapone!85!black,densely dotted,line width=.8pt]
    coordinates {(2,0.5) (4,0)};
  \BLGapPoint{gapone}{7}{0}
  \BLGapPoint{gapone}{0}{0.5}
  \BLGapPoint{gapone}{1}{0.5}
  \BLRedPoint{0}{1}
\end{axis}
\end{tikzpicture} &
\begin{tikzpicture}
\begin{axis}[BLpanel,xlabel={},ylabel={},title={$\Delta(\Lambda)=2$}]
  \path[fill=gaptwo!58,draw=none]
    (axis cs:-1,0) -- (axis cs:-1,1) -- (axis cs:0,1) --
    (axis cs:3,0.5) -- (axis cs:8,0) -- cycle;
  \path[fill=resultred,draw=none]
    (axis cs:0,1) -- (axis cs:11,1) -- (axis cs:11,0) --
    (axis cs:8,0) -- (axis cs:3,0.5) -- cycle;
  \path[pattern=BLphasehatchV11,pattern color=gaptwo!80!black,draw=none]
    (axis cs:-1,0) -- (axis cs:8,0) -- (axis cs:-1,0.9) -- cycle;
  \path[pattern=BLflathatchV11gap2,pattern color=gaptwo!80!black,draw=none]
    (axis cs:-1,0) -- (axis cs:6,0) -- (axis cs:0,1) -- (axis cs:-1,1) -- cycle;
  \addplot[resultredborder,very thick]
    coordinates {(0,1) (3,0.5) (8,0)};
  \addplot[gaptwo!85!black,densely dotted,line width=.8pt]
    coordinates {(-1,0.9) (3,0.5)};
  \addplot[gaptwo!85!black,densely dotted,line width=.8pt]
    coordinates {(3,0.5) (6,0)};
  \BLGapPoint{gaptwo}{8}{0}
  \BLGapPoint{gaptwo}{0}{0.5}
  \BLGapPoint{gaptwo}{1}{0.5}
  \BLRedPoint{0}{1}
\end{axis}
\end{tikzpicture}\\[.6em]
\begin{tikzpicture}
\begin{axis}[BLpanel,title={$\Delta(\Lambda)=3$}]
  \path[fill=gapthree!58,draw=none]
    (axis cs:-1,0) -- (axis cs:-1,1) -- (axis cs:0,1) --
    (axis cs:4,0.5) -- (axis cs:9,0) -- cycle;
  \path[fill=resultred,draw=none]
    (axis cs:0,1) -- (axis cs:11,1) -- (axis cs:11,0) --
    (axis cs:9,0) -- (axis cs:4,0.5) -- cycle;
  \path[pattern=BLphasehatchV11,pattern color=gapthree!80!black,draw=none]
    (axis cs:-1,0) -- (axis cs:9,0) -- (axis cs:-1,1) -- cycle;
  \path[pattern=BLflathatchV11gap3,pattern color=gapthree!80!black,draw=none]
    (axis cs:-1,0) -- (axis cs:8,0) -- (axis cs:0,1) -- (axis cs:-1,1) -- cycle;
  \addplot[resultredborder,very thick]
    coordinates {(0,1) (4,0.5) (9,0)};
  \addplot[gapthree!85!black,densely dotted,line width=.8pt]
    coordinates {(-1,1) (4,0.5)};
  \addplot[gapthree!85!black,densely dotted,line width=.8pt]
    coordinates {(4,0.5) (8,0)};
  \BLGapPoint{gapthree}{9}{0}
  \BLGapPoint{gapthree}{0}{0.5}
  \BLGapPoint{gapthree}{1}{0.5}
  \BLRedPoint{0}{1}
\end{axis}
\end{tikzpicture} &
\begin{tikzpicture}
\begin{axis}[BLpanel,ylabel={},title={$\Delta(\Lambda)=4=d-1$}]
  \path[fill=gapfour!58,draw=none]
    (axis cs:-1,0) -- (axis cs:-1,1) -- (axis cs:0,1) --
    (axis cs:5,0.5) -- (axis cs:10,0) -- cycle;
  \path[fill=resultred,draw=none]
    (axis cs:0,1) -- (axis cs:11,1) -- (axis cs:11,0) --
    (axis cs:10,0) -- (axis cs:5,0.5) -- cycle;
  \path[pattern=BLphasehatchV11,pattern color=gapfour!80!black,draw=none]
    (axis cs:-1,0) -- (axis cs:10,0) -- (axis cs:0,1) -- (axis cs:-1,1) -- cycle;
  \path[pattern=BLflathatchV11gap4,pattern color=gapfour!80!black,draw=none]
    (axis cs:-1,0) -- (axis cs:10,0) -- (axis cs:0,1) -- (axis cs:-1,1) -- cycle;
  \addplot[resultredborder,very thick]
    coordinates {(0,1) (5,0.5) (10,0)};
  \BLGapPoint{gapfour}{10}{0}
  \BLGapPoint{gapfour}{0}{0.5}
  \BLGapPoint{gapfour}{1}{0.5}
  \BLRedPoint{0}{1}
\end{axis}
\end{tikzpicture} &
\begin{tikzpicture}
\begin{axis}[BLpanel,ylabel={},title={$\Delta(\Lambda)=5=d$}]
  \path[fill=gapfive!58,draw=none]
    (axis cs:-1,0) -- (axis cs:-1,1) -- (axis cs:11,1) --
    (axis cs:11,0) -- cycle;
  \node[font=\scriptsize,fill=white,fill opacity=.88,text opacity=1,
        rounded corners,inner sep=1.4pt]
    at (axis cs:5,0.83) {Schwartz window};
  \BLGapPoint{gapfive}{0}{0.5}
  \BLGapPoint{gapfive}{1}{0.5}
  \BLGapPoint{gapfive}{0}{1}
\end{axis}
\end{tikzpicture}
\end{tabular}

\par\smallskip
{\footnotesize
\tikz[baseline=-.5ex]{\path[fill=gapone!58] (0,0) rectangle (.38,.18);
  \path[pattern=BLphasehatchV11,pattern color=gapone!80!black] (0,0) rectangle (.38,.18);}
\,$g_{\rm phase}$\qquad
\tikz[baseline=-.5ex]{\path[fill=gapone!58] (0,0) rectangle (.38,.18);
  \path[pattern=BLflathatchV11gap1,pattern color=gapone!80!black] (0,0) rectangle (.38,.18);}
\,$g_{\rm flat}$\qquad
\tikz[baseline=-.5ex]{\path[fill=gapone!58] (0,0) rectangle (.38,.18);
  \path[pattern=BLphasehatchV11,pattern color=gapone!80!black] (0,0) rectangle (.38,.18);
  \path[pattern=BLflathatchV11gap1,pattern color=gapone!80!black] (0,0) rectangle (.38,.18);}
\,both bounds}

\caption{Proved existence and nonexistence regions in
Theorem~\ref{thm:symplectic-index-balian-low} and \Cref{thm:modulation-classification} for
symplectically rational lattices in dimension $d=5$, displaying $-1\le s\le11$.
For these lattices, $\Delta(\Lambda)=0$ is equivalent to
$\operatorname{covol}(\Lambda)=1$.  The red regions and their finite-$p$ boundary
curves give nonexistence, the differently colored regions give existence for
the respective symplectic index gap. For $0\le\Delta(\Lambda)\le3$, hatching parallel to the lower boundary
segment shows the guaranteed range of $g_{\rm phase}$, while hatching
perpendicular to the upper boundary segment shows that of $g_{\rm flat}$;
crosshatching shows their overlap. Dotted lines extend the separating
boundary segments into the existence region and mark the bounds of the
individual guaranteed ranges there. The stated bounds are strict for finite
$p$; each range includes its own endpoint at $p=\infty$.
For $\Delta(\Lambda)=4=d-1$, the two bounds coincide and $g_{d-1}$ attains
the common region; the Schwartz region in the last panel is left unhatched.
For $0\le\Delta(\Lambda)\le4$, every point on $1/p=0$ with
$-1\le s\le6+\Delta(\Lambda)$ is attained.
For $\Delta(\Lambda)=5=d$, a Schwartz Parseval window exists.  The marked spaces are Bekka's
$L^2=M_0^2$ at $(s,1/p)=(0,1/2)$, the classical threshold $M_1^2$ at
$(1,1/2)$, and the amalgam threshold $M_0^1=S_0$ at $(0,1)$.
Solid gap-colored and red markers indicate existence and nonexistence,
respectively.}
\label{fig:current-results-d5}
\end{figure}
\FloatBarrier

For a symplectically irrational lattice with
$\covol(\Lambda)<1$, a Schwartz-class Gabor-frame window exists.  At
critical density $\covol(\Lambda)=1$, the classical Balian--Low Theorem~\ref{thm:full-multivariate-balian-low}
excludes a Gabor-frame window in $M_1^2(\mathbb R^d)$.  Beyond this
exclusion, we do not pursue a full modulation-space classification in the
symplectically irrational critical-density case.

In the symplectically rational case, at $\Delta(\Lambda)=0$, the point
$(p,s)=(2,1)$ recovers the classical
Balian--Low obstruction.  At $(p,s)=(1,0)$, the theorem excludes a
Feichtinger-algebra window whenever $\Delta(\Lambda)\le d-1$, which is the
rational necessity statement in the sharp amalgam theorem.

Figure~\ref{fig:current-results-d5} illustrates the classification in
Theorem~\ref{thm:symplectic-index-balian-low}. The following constructive
statement specifies the stronger simultaneous-window quantifiers behind
the pointwise existence classification.

\Needspace{5\baselineskip}
\begin{theorem}
\label{thm:modulation-classification}
Let $d\ge1$ and let $\Lambda\subset\mathbb R^{2d}$ be symplectically
rational.
\begin{enumerate}[label=\thmpart{\roman*},leftmargin=2.5em]
\item Suppose $0\le\Delta(\Lambda)\le d-2$.  There are Parseval
Gabor-frame windows $g_{\rm flat}$ and $g_{\rm phase}$ such that
\[
 g_{\rm flat}\in M_s^p(\mathbb R^d),
 \quad\text{for every }1\le p<\infty,
 \qquad s<2\bigl(\Delta(\Lambda)+1\bigr)(1-1/p),
\]
and
\[
 g_{\rm flat}\in M_s^\infty(\mathbb R^d),
 \qquad\text{for every }s\le2\bigl(\Delta(\Lambda)+1\bigr).
\]
The phase-adjusted window satisfies
\[
 g_{\rm phase}\in M_s^p(\mathbb R^d),
 \quad\text{for}\quad
 \begin{cases}
  1\le p<\infty,
   &s<d+\Delta(\Lambda)+1-2d/p,\\[1mm]
  p=\infty,
   &s\le d+\Delta(\Lambda)+1.
 \end{cases}
\]
Each of these two windows works throughout its stated ranges.

\item Suppose $\Delta(\Lambda)=d-1$.  There is a Parseval Gabor-frame
window $g_{d-1}$ satisfying
\[
 g_{d-1}\in M_s^p(\mathbb R^d),
 \quad\text{for}\quad
 1\le p<\infty,
 \qquad
 s<2d\left(1-1/p\right),
\]
and $g_{d-1}\in M_s^\infty(\mathbb R^d)$ for every $s\le2d$.

\item If $\Delta(\Lambda)\ge d$, there is a Schwartz Parseval
Gabor-frame window; in particular it belongs to $M_s^p$ for every
$1\le p\le\infty$ and every $s\in\mathbb R$.
\end{enumerate}
\end{theorem}

In part~\thmpart{i}, a single window cannot have both stated regularity
ranges.  Indeed, put $q=\Delta(\Lambda)+1<d$.  Such a window would belong
to $M_{q-d}^1\cap M_{q+d}^\infty$, since $q-d<0$.  H\"older's inequality
for the same Gaussian analysis window gives
\begin{align*}
 \|g\|_{M_q^2}^{2}
 &=\int_{\R^{2d}}
       \bigl(\langle z\rangle^{q-d}|V_\phi g(z)|\bigr)
       \bigl(\langle z\rangle^{q+d}|V_\phi g(z)|\bigr)\,dz\\
 &\le\|g\|_{M_{q-d}^1}\,\|g\|_{M_{q+d}^\infty}<\infty.
\end{align*}
This contradicts \Cref{thm:symplectic-index-balian-low} at $p=2$ and
$s=q$.

In the separable rational model, $g_{\rm flat}$ is constructed via the
Zak transform, and $g_{\rm phase}$ is obtained by multiplying
$Zg_{\rm flat}$ by a periodic scalar function of modulus one. Thus only
the phase of the Zak transform changes, motivating the names. The
windows for general symplectically rational lattices are then obtained
by metaplectic transport.

For symplectically rational lattices, and for symplectically irrational
lattices with covolume different from~$1$, the preceding results give a
complete modulation-space classification; the remaining question is the
symplectically irrational critical-density case.

\begin{question}\label{qu:irrational-critical-density}
Let $\Lambda\subset\R^{2d}$ be a symplectically irrational lattice with
$\covol(\Lambda)=1$. For which $1\le p\le\infty$ and $s\in\R$ does
there exist $g\in L^2(\R^d)\cap M_s^p(\R^d)$ such that
$\mathcal G(g,\Lambda)$ is a frame?

\smallskip\noindent\emph{A concrete subquestion.}
For
\[
 \Lambda=\bigl(\sqrt2\,\Z\times\tfrac1{\sqrt2}\Z\bigr)\times\Z^2,
\]
what is the breaking point $b\in[0,1]$ such that a Gabor-frame window
$g\in L^2(\R^2)\cap M_s^2(\R^2)$ exists exactly for $s<b$, or exactly
for $s\le b$?  Determine both $b$ and whether the endpoint $s=b$ is
attained.  Equivalently,
\[
 b=\sup\bigl\{s\in\R:\ \exists\,g\in L^2(\R^2)\cap M_s^2(\R^2)
          \text{ with }\mathcal G(g,\Lambda)\text{ a frame}\bigr\}.
\]
If $b=1$, the endpoint is excluded by
\Cref{thm:full-multivariate-balian-low}.
\end{question}

This symplectically irrational critical-density case can occur only for
$d\ge2$: in dimension $d=1$, every covolume-one lattice is
symplectically integral.

Bekka's theorem~\cite[Theorem~4]{Bekka2004} gives an $L^2$ frame window,
which by the embeddings in Proposition~\ref{prop:modulation-embeddings}
belongs to $M_s^p$ for $2\le p\le\infty$, $s\le0$, and for
$1\le p<2$, $s<d-2d/p$.
Figure~\ref{fig:irrational-results-d5} illustrates these existence ranges
and the exclusions supplied by the classical and amalgam Balian--Low
theorems through the same embeddings.

\begin{figure}[htbp]
\centering
{\small $\Lambda$ symplectically irrational}
\par\medskip
\setlength{\tabcolsep}{2pt}
\pgfplotsset{BLirrpanel/.style={BLpanel,scale only axis,
  width=.38\textwidth,height=.25\textwidth,
  xmin=-6,xmax=7,xtick={-5,0,5},
  title style={font=\small,align=center},grid=none}}
\begin{tabular}{cc}
\begin{tikzpicture}
\begin{axis}[BLirrpanel,
  title={$\covol(\Lambda)=1$}]
  % Bekka's L^2 window and the embeddings of Proposition 3.3.
  \path[fill=gapzero!58,draw=none]
    (axis cs:-6,0) -- (axis cs:-6,1) -- (axis cs:-5,1) --
    (axis cs:0,.5) -- (axis cs:0,0) -- cycle;
  % Union of the classical and amalgam obstructions transported by embeddings.
  \path[fill=resultred,draw=none]
    (axis cs:0,1) -- (axis cs:7,1) -- (axis cs:7,0) --
    (axis cs:6,0) -- (axis cs:1,.5) -- (axis cs:1,.9) -- cycle;
  \addplot[gapzero!85!black,very thick] coordinates {(0,0) (0,.5)};
  \addplot[resultredborder,very thick] coordinates {(1,.9) (1,.5)};
  % Black dashed reference: the sharp rational critical-density boundary.
  \addplot[black,dash pattern=on 3pt off 2pt,line width=.9pt]
    coordinates {(0,1) (1,.5) (6,0)};
  \BLGapPoint{gapzero}{0}{.5}
  \BLRedPoint{0}{1}
  \BLRedPoint{1}{.9}
  \BLRedPoint{1}{.5}
  \addplot[only marks,mark=*,mark size=1.7pt,thick,
    draw=gapzero!85!black,fill=white] coordinates {(-5,1)};
  \addplot[only marks,mark=*,mark size=1.7pt,thick,
    draw=resultredborder,fill=white] coordinates {(6,0)};
  \node[font=\scriptsize,align=center,fill=white,fill opacity=.88,
    text opacity=1,rounded corners,inner sep=1.4pt]
    at (axis cs:-3.5,.27) {Bekka};
  \node[font=\scriptsize,fill=white,fill opacity=.88,
    text opacity=1,rounded corners,inner sep=1.4pt]
    at (axis cs:4.1,.77) {Excluded};
  \node[font=\scriptsize] at (axis cs:-1.6,.82) {Open};
\end{axis}
\end{tikzpicture} &
\begin{tikzpicture}
\begin{axis}[BLirrpanel,ylabel={},
  title={$\covol(\Lambda)<1$\\{}}]
  \path[fill=gapfive!58,draw=none]
    (axis cs:-6,0) -- (axis cs:-6,1) -- (axis cs:7,1) --
    (axis cs:7,0) -- cycle;
  \node[font=\scriptsize,fill=white,fill opacity=.88,text opacity=1,
    rounded corners,inner sep=1.4pt]
    at (axis cs:.5,.5) {Schwartz window};
\end{axis}
\end{tikzpicture}
\end{tabular}
\caption{Symplectically irrational lattices in dimension $d=5$, illustrating
\Cref{qu:irrational-critical-density}. Green shows existence, red
nonexistence, and white parameters not decided by the cited results.
Left: Bekka's theorem~\cite[Theorem~4]{Bekka2004} and embedding
\Cref{prop:modulation-embeddings} give the existence region at critical
covolume.  The black dashed line through $(0,1)$, $(1,1/2)$, and $(6,0)$
shows the sharp boundary for symplectically rational lattices with
$\Delta(\Lambda)=0$, for comparison.  The classical Balian--Low
Theorem~\ref{thm:full-multivariate-balian-low} and amalgam Balian--Low
Theorem~\ref{thm:main-gabor}, together with the same embeddings, give
the red region to the right of the broken line through $(0,1)$,
$(1,9/10)$, $(1,1/2)$, and $(6,0)$ in $(s,1/p)$ coordinates.
Right: a Schwartz window exists for every symplectically irrational
lattice of covolume less than one
\cite[Theorem~C]{EnstadThielVilalta2025}, so every $p,s$ is admissible.}
\label{fig:irrational-results-d5}
\end{figure}

\subsection{The symplectic index gap and Zak transforms}
\label{subsec:intro-zak-gap}

The symplectic index gap in condition~\eqref{eq:SG} is reflected in a
dimension restriction on quasiperiodic maps.  This restriction is a key necessity mechanism in the proof
of Theorem~\ref{thm:symplectic-index-balian-low}.  We record here the two
results that isolate it; their proofs are given in
\Cref{sec:rectangular-common-zero}.

For $g\in L^2(\mathbb R^q)$, the ordinary Zak transform is
\[
 Zg(u,v)=\sum_{k\in\mathbb Z^q}g(u-k)e^{2\pi i k\cdot v},
\]
initially interpreted in the standard $L^2$ sense.  It satisfies
\[
 Zg(u+n,v+m)=e^{2\pi i n\cdot v}Zg(u,v),
 \qquad n,m\in\mathbb Z^q.
\]
More generally, a measurable function
$F:\mathbb R^q\times\mathbb R^q\to\mathbb C^M$ is called
\emph{quasiperiodic} if
\begin{equation*}
 F(u+n,v+m)=e^{2\pi i n\cdot v}F(u,v),
 \qquad n,m\in\mathbb Z^q,
\end{equation*}
almost everywhere.

For a locally integrable scalar-, vector-, or matrix-valued function $H$ on a
cube $Q\subset\mathbb R^n$, put
\[
 H_Q=\fint_Q H(x)\,dx:=\frac1{|Q|}\int_Q H(x)\,dx, \quad \operatorname{osc}_Q(H)=\fint_Q\|H(x)-H_Q\|\,dx
\]
If $U\subset\mathbb R^n$ is open, then $H\in\VMOloc(U)$ means that for every
compact $K\Subset U$,
\[
 \sup_{\substack{Q\subset K\\ \ell(Q)\le r}}
 \operatorname{osc}_Q(H)\longrightarrow0
 \qquad\text{as }r\searrow0.
\]
Here $K\Subset U$ means that $K$ is compactly contained in $U$, equivalently
$K$ is compact and has positive distance from $\mathbb R^n\setminus U$. Also, $\ell(Q)$ denotes the side length of $Q$.
The definition is applied entrywise to vector- and matrix-valued functions.
Every continuous function and every VMO function is locally VMO.

Write
\[
 \mathcal Q_q^{\rm Zak}=[0,1]^{2q},
 \qquad
 (\mathcal Q_q^{\rm Zak})^*=[-2,3]^{2q}.
\]
We say that $F=(F_1,\ldots,F_M)$ has an \emph{essential common zero} if
\[
 \operatorname*{ess\,inf}_{z\in\mathcal Q_q^{\rm Zak}}|F(z)|=0.
\]
The continuous common-zero theorem of de Dios Pont, Liehr, and Taylor
\cite[Theorem~1.3]{deDiosLiehrTaylor2026} asserts that a continuous quasiperiodic map $\R^{2q}\to \C^M$ has a zero whenever $M\le q$.  We give an
elementary proof of this result in Sections~\ref{app:pfaffian-common-zero}--\ref{app:boundary-pfaffian},
using ordinary Zak identities, coordinate derivatives, and a Pfaffian
calculation.
The following locally VMO consequence follows from the matrix smoothing argument.

\begin{theorem}
\label{thm:vmo-essential-common-zero}
Let $F:\mathbb R^q\times\mathbb R^q\to\mathbb C^M$ be essentially bounded and quasiperiodic, and have locally VMO
components on a neighborhood of $(\mathcal Q_q^{\rm Zak})^*$.  If $M\le q$,
then $F$ has an essential common zero:
\[
 \operatorname*{ess\,inf}_{z\in\mathcal Q_q^{\rm Zak}}\|F(z)\|=0.
\]
Conversely, if $M>q$, there are smooth quasiperiodic maps with
no common zero.
\end{theorem}

The converse assertion in \Cref{thm:vmo-essential-common-zero} was established
by de Dios Pont--Liehr--Taylor
\cite[Theorem~1.3]{deDiosLiehrTaylor2026}.  An explicit smooth construction
is given in \Cref{lem:explicit-scalar-seeds}.

The corresponding matrix statement is the following sharp rectangular
rank condition.  Here $\sigma_{\min}$ denotes the smallest singular value.

\begin{theorem}
\label{thm:rectangular-common-zero-gap}
Let
\[
 S:\mathbb R^d\times\mathbb R^d\longrightarrow\mathbb C^{R\times N},
 \qquad N\ge R,
\]
be essentially bounded, be quasiperiodic entrywise, and have
entries in $\mathrm{VMO}_{\mathrm{loc}}(\mathbb R^{2d})$.  If
\begin{equation}
 \sigma_{\min}(S(u,v))\ge a>0,
 \quad\text{for almost every }(u,v),
 \label{eq:vmo-matrix-lower-bound}
\end{equation}
then
\[
 N\ge R+d.
\]
Conversely, whenever $N\ge R+d$, there is a smooth quasiperiodic map with values in $\mathbb C^{R\times N}$ with full row rank
everywhere.
\end{theorem}

The result above gives the most restrictive Gabor obstruction in the diagonal
case.

\begin{corollary}
\label{thm:intro-diagonal-zak-obstruction}
Let
\[
 D=\operatorname{diag}(b_1/a_1,\ldots,b_d/a_d),
 \qquad (a_i,b_i)=1,
\]
where the $a_i,b_i$ are positive integers, let
$\Lambda_D=D\mathbb Z^d\times\mathbb Z^d$, and put
\[
 N=\prod_{i=1}^d a_i,
 \qquad
 R=\prod_{i=1}^d b_i.
\]
If $N-R<d$, then there is no
$g\in L^2(\mathbb R^d)$ such that $Zg\in\VMOloc(\mathbb R^{2d})$ and
$\mathcal G(g,\Lambda_D)$ is a frame.  In particular, the conclusion applies
whenever $Zg$ is continuous.
\end{corollary}

For $Q$ windows, the lower fibre bound first forces $QN\ge R$.
The same argument excludes a frame whenever $QN-R<d$, assuming
$g_j\in L^2(\mathbb R^d)$ and
$Zg_j\in\VMOloc(\mathbb R^{2d})$ for $1\le j\le Q$.
Indeed, concatenate the $Q$ rational Zak matrices and apply the rectangular
VMO theorem with $QN$ columns.  For the $S_0$ windows in
\Cref{thm:main-multiwindow}, the continuous version suffices.

\subsection{A critical Fourier--Lebesgue--VMO embedding}
\label{subsec:intro-critical-FL-VMO}

For a periodic function $f$ on $\mathbb T^m$ with Fourier coefficients
$\widehat f(k)$, the Fourier--Lebesgue space $\Fell{p}{s}(\mathbb T^m)$ is
defined by the condition
\[
 \bigl(\langle k\rangle^s\widehat f(k)\bigr)_{k\in\mathbb Z^m}
 \in\ell^p(\mathbb Z^m).
\]
The parameter-free case of \Cref{lem:mixed-oscillation} gives the
following critical periodic embedding.

\begin{theorem}
\label{thm:critical-vmo}
Let $m\ge1$ and $1\le p<\infty$.  Then
\[
 \Fell{p}{m(1-1/p)}(\mathbb T^m)\subset\VMO(\mathbb T^m).
\]
More precisely, with $s=m(1-1/p)$,
\begin{equation}
 [f]_{\BMO(\mathbb T^m)}
 \le C_{m,p}\norm{f}_{\Fell{p}{s}(\mathbb T^m)}.
 \label{eq:FL-critical-BMO-estimate}
\end{equation}
\end{theorem}

The critical VMO embeddings used in earlier Balian--Low arguments appear in
the work of Gautam~\cite{Gautam2008} and Tinaztepe--Heil~\cite{TinaztepeHeil2012}.
The displayed periodic statement holds in every dimension and for
every $1\le p<\infty$.

\subsection{Consequences in algebraic topology}
\label{subsec:intro-topological-consequences}

The symplectic index gap of a lattice-generating matrix may also have
consequences in other areas of mathematics, including algebraic topology
and differential geometry.  A first example is the following.

Let $A_\Lambda=C_\pi^*(\Lambda)$ be the noncommutative torus associated
with $\Lambda$, and let $p_\Lambda$ be its associated projection
\cite[Notation~5.1]{EnstadThielVilalta2025}.  Enstad--Thiel--Vilalta show
that $p_\Lambda\precsim1_{A_\Lambda}^{\oplus Q}$ is equivalent to the
existence of a Schwartz-class $Q$-window Gabor frame
\cite[Proposition~5.2]{EnstadThielVilalta2025}; here $\precsim$ denotes
Murray--von Neumann subequivalence in matrix algebras over $A_\Lambda$,
with zero blocks adjoined as necessary.  Combining this with our sharp
multiwindow criterion in \Cref{thm:main-multiwindow} gives the following
projection-comparison criterion.  The necessity of the numerical condition
is new.

\begin{corollary}
\label{thm:heisenberg-projection-comparison}
Let $\Lambda\subset\mathbb R^{2d}$ be symplectically rational and let
$Q\ge1$ be an integer.  Then
\[
 p_\Lambda\precsim1_{A_\Lambda}^{\oplus Q}
 \quad\Longleftrightarrow\quad
 Q\,\nu(\Lambda)\ge\nu(\Lambda^\circ)+d.
\]
\end{corollary}

\subsection{The companion paper: tilings, packings, and compactly supported windows}
\label{subsec:intro-companion-paper}

The present paper and \emph{Tilings, Packings, and Smooth and Compactly
Supported Gabor Windows}~\cite{CompanionPaper1} grew out of the joint
preprint of Caragea, Kolountzakis, and Pfander
\cite{CarageaKolountzakisPfander2026v2}.  The companion paper proves the
tiling--packing criterion corresponding to \eqref{eq:SG}, including its
equivalence with the regular-window condition for separable lattices.
It constructs a bounded fundamental domain for the translation lattice
whose fixed neighborhood packs with the dual of the modulation lattice,
and deduces the following result.

\begin{theorem}[Caragea--Kolountzakis--Pfander~\protect\cite{CompanionPaper1}]
\label{thm:companion-separable-compact}
Let $\Lambda=\Gamma\times\Phi$ be a separable full-rank lattice
satisfying \eqref{eq:SG}.  There is a nonnegative
$g\in C_c^\infty(\R^d)$ such that $\mathcal G(g,\Gamma\times\Phi)$ is
Parseval and
\begin{align*}
 \supp g\cap(\supp g+\xi^*)&=\varnothing,
 &&\xi^*\in\Phi^*\setminus\{0\},\\
 \sum_{\gamma\in\Gamma}|g(x+\gamma)|^2&=\covol(\Phi),
 &&x\in\R^d.
\end{align*}
\end{theorem}

The geometric equivalence and the compact-support constructions are
proved in \cite{CompanionPaper1}.  Related measurable tiling and packing
results appear in
\cite{HanWang2001,Kolountzakis2000,KolountzakisPapageorgiou2022,
GrepstadKolountzakis2026,PfanderRashkovWang2012}.

\subsection{Plan of the paper}
\label{subsec:intro-plan}

The remainder of the paper is organized as follows.
\Cref{sec:literature} places the main results in context 
with respect to the existing literature
and identifies the earlier obstruction, frame, and regularity results on
which the later arguments build.
\Cref{sec:gabor-background} fixes the Fourier--Lebesgue, modulation-space,
Gabor-frame, BMO, and VMO tools used throughout both the nonexistence and
construction parts.
\Cref{sec:zak-method} develops the scalar and rational matrix Zak formalisms
and the coefficient characterizations that connect frame properties with
periodic Fourier regularity.

Part~\ref{part:zak-obstructions} proves the sharp obstructions.
\Cref{sec:rectangular-common-zero} turns the symplectic index gap into
continuous and VMO common-zero theorems, supplying the topological
obstruction used at critical regularity.
\Cref{sec:vmo} proves the critical Fourier--Lebesgue oscillation estimates
that convert modulation-space information into the VMO hypothesis required
by the preceding section.
\Cref{sec:critical-nonexistence} combines the rational Zak fibre criterion
with these two inputs to establish the critical finite-$p$ nonexistence
theorem.

Part~\ref{part:constructive-rank-geometry} gives the matching scalar
constructions.
\Cref{sec:construction-strategy} builds smooth quasiperiodic Zak seeds and
normalizes their finite-dimensional Zak matrices, globally when the index
gap is large and near a cubical skeleton otherwise.
\Cref{sec:terminal-construction} pulls the skeleton data across the full Zak
domain, producing the almost-everywhere Parseval flat Zak function and its
controlled exceptional set.
\Cref{sec:elementary-perturbation} converts the geometry and derivative
bounds of that exceptional set into Fourier decay and the modulation-space
range of the flat window.
\Cref{sec:lower-polar} constructs a single periodic Rudin--Shapiro phase
with exact dyadic refinement, designed to add cancellation without altering
the modulus or Parseval identities.
\Cref{sec:lower-completion} quantifies that cancellation, transfers the phase
to physical Zak coordinates, and completes the improved high-$p$ regularity
range for the phase-shaped window.

Part~\ref{part:global-classification} transfers and assembles the diagonal
results.
\Cref{sec:lattice-theory} records the symplectic, metaplectic,
adjoint-lattice, and modulation-space covariance needed to move between
lattice models.
\Cref{sec:assembly} combines the obstruction and construction results into
the complete classification for symplectically rational lattices.
\Cref{sec:nonrational-balian-low} states the external density and
Balian--Low inputs needed beyond the rational setting.
\Cref{sec:completion-all-lattices} joins those inputs with the rational
classification to prove the all-lattice one-window theorems.
\Cref{sec:multiwindow-extension} extends the argument to several windows and
then translates the sharp criterion into the projection-comparison
statement of \Cref{thm:heisenberg-projection-comparison}.
The appendix supplies a self-contained Pfaffian proof of the continuous
quasiperiodic common-zero theorem underlying
\Cref{sec:rectangular-common-zero}.

\section{Earlier results and relation to the literature}\label{sec:literature}

Balian in 1981 and Low in 1985 independently formulated the obstruction
for orthonormal Gabor bases of $L^2(\R)$ on rectangular lattices
$a\Z\times b\Z$ with $ab=1$: the window and its Fourier transform
cannot both have finite second moments
\cite{Balian1981,Low1985}.  Their Zak-transform arguments contained a
continuity gap.  Daubechies--Janssen credit Coifman and Semmes with the
first complete proof; that proof appears in Daubechies's 1990 article
\cite{DaubechiesJanssen1993,Daubechies1990}.  Battle's independently
published 1988 proof uses the canonical commutation relation and treats
orthonormal bases \cite{Battle1988}.  The Coifman--Semmes argument also
gives the one-dimensional Riesz-basis, or exact-frame, case; in 1993
Daubechies--Janssen extended the commutator approach to this setting
\cite{Daubechies1990,DaubechiesJanssen1993}.

In higher dimensions, Gr\"ochenig--Han--Heil--Kutyniok proved in 2002
the Riesz-basis obstruction first for the dual product lattices
$K\times K^*$ and then for every symplectic lattice
$S\Z^{2d}$, $S\in\Sp(2d,\R)$
\cite[Theorems~9 and~10]{GrochenigHanHeilKutyniok2002}.
This includes every covolume-one lattice in dimension one, but only a
subclass in higher dimensions.  Their Theorem~8 applies to arbitrary
lattices in every dimension and excludes simultaneous finite second
moments of a window and its canonical dual.  Their concluding discussion
leaves the corresponding window-only assertion open beyond the
symplectic class \cite[Section~3.4]{GrochenigHanHeilKutyniok2002}.
Lee--Philipp--Voigtlaender later proved preservation of the relevant
first-order time--frequency regularity under canonical duality for
arbitrary lattices \cite[Theorem~1.1]{LeePhilippVoigtlaender2023}.
Combined with the weak 2002 theorem, this gives the arbitrary-lattice
critical-density obstruction used below.

The amalgam Balian--Low theorem has a different regularity hypothesis.
Heil's 1990 thesis, and the published formulation of
Benedetto--Heil--Walnut in 1995, exclude both $g$ and $\widehat g$ from
the continuous Wiener amalgam space $W(C_0,\ell^1)$ when
$\mathcal G(g,a\Z\times b\Z)$ is a Riesz basis of $L^2(\R)$ and
$ab=1$ \cite{Heil1990Thesis,BenedettoHeilWalnut1995}.
The proof uses the zero of a continuous scalar Zak transform; this is
not a consequence of the second-moment theorem
\cite[Theorem~3.2 and Examples~3.3--3.4]{BenedettoHeilWalnut1995}.
For symplectic lattices in $\R^{2d}$, the metaplectically invariant
Feichtinger-algebra exclusion $g\notin S_0=M^1$ was obtained by
Feichtinger--Gr\"ochenig in 1997 and is recorded in
\cite[Theorem~12]{GrochenigHanHeilKutyniok2002}; see also
\cite{FeichtingerGrochenig1997}.  This $S_0$ statement must be distinguished
from the stronger Wiener-amalgam formulation on product lattices.
These critical-density exclusions alone do not distinguish rational
lattices of covolume less than one.  The symplectic index gap threshold in
\Cref{thm:main-gabor} and the modulation-space boundaries in
\Cref{thm:symplectic-index-balian-low} make that distinction.

Further developments of the classical theorem include endpoint,
quantitative, and exact-system results
\cite{BenedettoCzajaGadzinskiPowell2003,HeilPowell2006,
BenedettoCzajaPowellSterbenz2006,Gautam2008,NitzanOlsen2011,
NitzanOlsen2013,LeshenPowell2020,BenedettoCzajaPowell2006,JamingPowell2011}.
Heil's historical survey~\cite{Heil2007History} and the chapter of
Czaja--Powell~\cite{CzajaPowell2006} provide further background.

The 5G New Radio physical layer is surveyed by Parkvall et~al.
\cite{ParkvallDahlmanFuruskarFrenne2017}.  The relation between
pulse localization, Weyl--Heisenberg or Gabor structure, and multicarrier
transmission is studied, for example, by B\"olcskei
\cite{Bolcskei1999OFDM} and Jung--Wunder \cite{JungWunder2007}.

For higher-dimensional and symplectic lattice versions, see
\cite{GrochenigHanHeilKutyniok2002,BenedettoCzajaMaltsev2003,
Temur2020,NorthingtonPark2021}.  The weak arbitrary-lattice theorem of
Gr\"ochenig--Han--Heil--Kutyniok is one of the two critical-density inputs in
our proof of the sharp classical result; the second is the canonical-dual
regularity theorem of Lee--Philipp--Voigtlaender
\cite{LeePhilippVoigtlaender2023}.  Related formulations through Weyl
symbols and regularity of Gabor systems occur in
\cite{AscensiFeichtingerKaiblinger2014,
BenedettoCzajaGadzinskiPowell2003}.  Balian--Low phenomena for exact systems,
principal shift-invariant spaces, finite-dimensional systems, subspaces, and
Gabor Riesz sequences of arbitrary density are treated in
\cite{NitzanOlsen2011,AldroubiSunWang2011,NitzanOlsen2019,
CarageaLeePfanderPhilipp2019,CarageaLeePhilippVoigtlaender2021,
CarageaLeePhilippVoigtlaender2023,LiuMaZheng2026}.
In shift-invariant spaces, additional translation invariance yields
Balian--Low restrictions on the generators
\cite{AldroubiSunWang2011,HardinNorthingtonPowell2018}, with related
approximation constructions in \cite{AldroubiKrishtalTesseraWang2012}.

The modulation-space, BMO, VMO, and amalgam forms closest to the present
phase diagram appear in
\cite{Gautam2008,TinaztepeHeil2012,CabrelliMolterPfander2015,
CabrelliMolterPfander2016}.  We use the full modulation-space scale to
separate three regimes of the symplectic index gap.  The
critical Fourier--Lebesgue-to-VMO embedding and the anisotropic one-scale
argument provide the necessity boundaries, while interpolation of rational Zak-matrix entries, tightening,
and cubical constructions provide the complementary existence regions.
Quaternionic, Clifford, and other nonstandard variants are developed in
\cite{FuKahlerCerejeiras2012,FuKahlerCerejeiras2013,
FuKahlerCerejeirasChapter2013}.

The density theory for unrestricted Gabor generators goes back to the
classical density theorems and includes Bekka's representation-theoretic
result~\cite{Bekka2004} and Enstad's projective-representation density
theorem~\cite{Enstad2022Density}.  These results are used only for the
unrestricted existence statement or the weak density inequalities; the
regularity-sensitive exclusions at equality are taken instead from the
Heisenberg-module Balian--Low results cited below.  For noncommutative or flat tori, locally compact abelian
groups, homogeneous groups, and adelic Gabor systems, see
\cite{AbreuBalazsHolighausLuefSpeckbacher2024,Luef2018,Enstad2020,
GrochenigRomeroRottensteinerVanVelthoven2020,
EnstadJakobsenLuefOmland2022}.  Jakobsen--Luef
\cite{JakobsenLuef2020} identify the Feichtinger algebra with the regular
part of the relevant Heisenberg-module picture and obtain non-rational
regular-window existence.  The projective-module and multiwindow framework
was developed earlier by Luef~\cite{Luef2009Projections}.  Gerhold--Lamando--Luef
\cite{GerholdLamandoLuef2026} develop linear deformation theory for these
modules and prove the critical one-generator and multi-generator
regularity obstructions used in our non-rational endpoint arguments.

Enstad--Thiel--Vilalta~\cite[Theorems~C and~D]{EnstadThielVilalta2025} give the modern
Schwartz-window existence criterion over rational lattices and the strict
density Schwartz theorem in the symplectically irrational case.  Their
proof uses operator-algebraic and topological methods.  Our rational
sufficiency argument is different: it constructs a scalar Zak function
using a finite family that interpolates rational Zak-matrix entries, an elementary dimension count, scalar
tightening near a protected skeleton, and an explicit flat cubical map.  Deterministic scalar phases
provide the improved high-$p$ range, including the infinity endpoint.  The common-zero obstruction of de
Dios Pont--Liehr--Taylor~\cite{deDiosLiehrTaylor2026} supplies the continuous
model for the necessity argument; Sections~\ref{app:pfaffian-common-zero}--\ref{app:boundary-pfaffian}
give an independent coordinate/Pfaffian proof of the required continuous
statement.

A standard reference for frames, the Zak transform, modulation spaces,
duality, and density is Gr\"ochenig's monograph~\cite{Grochenig2001}.  The
finite rational Zak matrix is the classical Zibulski--Zeevi representation
\cite{ZibulskiZeevi1997}; explicit frame-operator factorizations are given by
B\"olcskei--Janssen~\cite{BolcskeiJanssen2000}, while multivariate and
locally compact abelian formulations appear in
Jakobsen--Lemvig~\cite{JakobsenLemvig2016} and
Gr\"ochenig--Koppensteiner~\cite{GrochenigKoppensteiner2019}.  What is
used here is the finite-dimensional rational Zak-matrix realization, together with the scalar
correlation and tightening formulation used in the construction.

A complementary problem fixes the window $g$ and asks which lattices
$\Lambda$, or more general discrete sets of time--frequency shifts,
make $\mathcal G(g,\Lambda)$ a frame.  The collection of admissible
lattices is the \emph{lattice frame set} of $g$.  For $d=1$, its
rectangular part is usually written
$\mathcal F(g)=\{(a,b)\in(0,\infty)^2:
\mathcal G(g,a\mathbb Z\times b\mathbb Z)\text{ is a frame}\}$.
The classical results of Lyubarskii and Seip--Wallst\'en
\cite{Lyubarskii1992,Seip1992,SeipWallsten1992} imply that the
Gaussian $g(x)=e^{-\pi x^2}$ generates a frame over a lattice
$\Lambda\subset\mathbb R^2$ exactly when $\covol(\Lambda)<1$;
see also \cite{Grochenig2001,Heil2007History}.

Total positivity gives further exact rectangular frame sets;
\cite{GrochenigRomeroStoeckler2018} treats Gaussian-type totally
positive windows.  Recent results of de Dios Pont, Gr\"ochenig, Liehr,
Shafkulovska, and Taylor show that
$\mathcal F(g)=\{(a,b)\in(0,\infty)^2:ab<1\}$ for every nonzero continuous
integrable totally positive function $g$
\cite[Theorem~1.1]{deDiosGrochenigLiehrShafkulovskaTaylor2026}.
Faulhuber and Petersen recently determined the rectangular
frame set of the first Hermite function, equivalently of
$g(x)=xe^{-\pi x^2}$: it consists precisely of the pairs with $ab<1$
and $ab\ne(q-1)/q$ for every integer $q\ge2$
\cite[Theorem~1.2]{FaulhuberPetersen2026}.
Another recent contribution establishes new rational-parameter frame
regions for the window $g(x)=x/((x^2+1)(x^2+4))$
\cite[Theorem~1.1]{Ghosh2026}.
These prescribed-window questions complement our classification,
which fixes the lattice and determines the achievable window regularity
and decay.

\section{Time-frequency analysis background}
\label{sec:gabor-background}

\subsection{Fourier--Lebesgue spaces and smooth localization}
\label{subsec:fourier-vmo-background}

Let $E$ denote either $\mathbb C$ or a finite-dimensional matrix space
$\mathbb C^{r\times s}$, equipped with one fixed norm $\|\cdot\|_E$.
For $F\in L^1(\mathbb R^n;E)$, define the Euclidean Fourier transform by
\begin{equation*}
 \widehat F(\xi)
 =\int_{\mathbb R^n}F(x)e^{-2\pi i x\cdot\xi}\,dx,
 \qquad \xi\in\mathbb R^n,
\end{equation*}
where the integral is taken entrywise in the matrix-valued case.  For
$F\in L^2(\mathbb R^n;E)$ we use the Plancherel extension.

For $1\le p\le\infty$ and $s\in\mathbb R$, the Euclidean
\emph{Fourier--Lebesgue space} $\FL{p}{s}(\mathbb R^n;E)$ consists of the
functions for which the following norm is finite.  When $1\le p<\infty$,
\begin{equation*}
 \|F\|_{\FL{p}{s}(\mathbb R^n)}
 =\left(
   \int_{\mathbb R^n}
     \langle\xi\rangle^{sp}\|\widehat F(\xi)\|_E^p\,d\xi
  \right)^{1/p},
 \qquad
 \langle\xi\rangle=(1+|\xi|^2)^{1/2},
\end{equation*}
and for $p=\infty$ one takes the essential supremum of
$\langle\xi\rangle^s\|\widehat F(\xi)\|_E$.  All norms on a fixed
finite-dimensional matrix space are equivalent, so the resulting space is
independent of the chosen pointwise norm.

For a periodic $E$-valued integrable function on
$\mathbb T^n=\mathbb R^n/\mathbb Z^n$, the Fourier integral is restricted to
one torus cell and evaluated at integer frequencies:
\begin{equation}
 \widehat F(k)
 =\int_{[0,1)^n}F(x)e^{-2\pi i k\cdot x}\,dx,
 \qquad k\in\mathbb Z^n.
 \label{eq:torus-fourier-coeff}
\end{equation}
The periodic Fourier--Lebesgue norm is
\begin{equation*}
 \|F\|_{\Fell{p}{s}(\mathbb T^n)}
 =\left\|
   \bigl(\langle k\rangle^s\|\widehat F(k)\|_E\bigr)_{k\in\mathbb Z^n}
  \right\|_{\ell^p},
\end{equation*}
with the usual supremum when $p=\infty$.

For the periodic Fourier coefficients, one fixed compact cutoff will
suffice.  Choose $\kappa_0\in C_c^\infty(\R^n)$ nonnegative and
strictly positive on $[0,1]^n$, and set
\begin{equation}\label{eq:compact-periodizing-cutoff}
 P(x)=\sum_{\nu\in\Z^n}\kappa_0(x+\nu),\qquad
 \chi(x)=\frac{\kappa_0(x)}{P(x)}.
\end{equation}
The sum defining $P$ is locally finite, so $P$ is smooth.  It is
positive and $\Z^n$-periodic; hence $\chi\in C_c^\infty(\R^n)$ and
\[
 \sum_{\nu\in\Z^n}\chi(x+\nu)
 =\frac{\sum_{\nu\in\Z^n}\kappa_0(x+\nu)}{P(x)}=1.
\]
If $G$ is $\mathbb Z^n$-periodic and integrable on $\T^n$, its Fourier coefficients
are therefore the Euclidean Fourier transform of $\chi G$ sampled
at integer frequencies.  For $k\in\Z^n$, the calculation is
\begin{align}
 \widehat G(k)
 &=\int_{[0,1)^n}G(x)e^{-2\pi i k\cdot x}
                    \sum_{\nu\in\Z^n}\chi(x+\nu)\,dx\notag\\
 &=\sum_{\nu\in\Z^n}\int_{\nu+[0,1)^n}
                  \chi(y)G(y)e^{-2\pi i k\cdot y}\,dy\notag\\
 &=\int_{\R^n}\chi(y)G(y)e^{-2\pi i k\cdot y}\,dy
 =\widehat{\chi G}(k).\label{eq:compact-cutoff-coefficients}
\end{align}
In the second line we used $y=x+\nu$, periodicity of $G$, and
$e^{2\pi i k\cdot\nu}=1$.  Only finitely many translates contribute
on the unit cell, which justifies interchanging the sum and integral.

\subsection{Modulation spaces, Gabor frames, and coefficient characterizations}
\label{subsec:frames-duality-density}

Let $\Gamma\subset\mathbb R^{2d}$ be a lattice and assume that
$\mathcal G(g,\Gamma)$ is a Gabor frame.  Its analysis, synthesis, and frame
operators are, respectively, 
\[
 \Ana_{g,\Gamma}f
 =\bigl(\langle f,\pi(\gamma)g\rangle\bigr)_{\gamma\in\Gamma},
 \qquad
 \operatorname{Syn}_{g,\Gamma}c
 =\sum_{\gamma\in\Gamma}c_\gamma\pi(\gamma)g,
 \qquad
 S_{g,\Gamma}=\operatorname{Syn}_{g,\Gamma}\circ\Ana_{g,\Gamma}.
\]
The function $\widetilde g=S_{g,\Gamma}^{-1}g$ is the canonical dual window,
and
\[
 f=\sum_{\gamma\in\Gamma}
      \langle f,\pi(\gamma)\widetilde g\rangle\pi(\gamma)g
  =\sum_{\gamma\in\Gamma}
      \langle f,\pi(\gamma)g\rangle\pi(\gamma)\widetilde g,
 \qquad f\in L^2(\mathbb R^d),
\]
with unconditional convergence in $L^2$.  The window
$S_{g,\Gamma}^{-1/2}g$ generates the associated Parseval Gabor frame.

As seen in \Cref{subsec:intro-modulation-classification}, modulation spaces
are defined through the short-time Fourier transform.  Fix a nonzero Schwartz
window $\phi$ and put
\[
 V_\phi f(x,\omega)
 =\int_{\mathbb R^d}f(t)\overline{\phi(t-x)}e^{-2\pi i\omega\cdot t}\,dt,
\]
and
\[
 \|f\|_{M_s^p}
 =\|\langle(x,\omega)\rangle^sV_\phi f(x,\omega)\|_{L^p(\mathbb R^{2d})}.
\]
Independence of the choice of nonzero Schwartz analysis window follows
from \cite[Proposition~11.3.2(c)]{Grochenig2001}.

For $s\in\mathbb R$, put $v_s(z)=\langle z\rangle^s$.  The polynomial weight
is moderate with respect to $v_{|s|}$,
\[
 v_s(z+y)\le C_s v_s(z)v_{|s|}(y);
\]
see \cite[Section~11.1]{Grochenig2001}.
For a full-rank lattice $\Gamma\subset\mathbb R^{2d}$ and a finite
family $\Psi=(\psi_1,\ldots,\psi_L)$, write
\[
 \Ana_\Psi(f)_{\ell,\gamma}
 :=\langle f,\pi(\gamma)\psi_\ell\rangle,
 \qquad \gamma\in\Gamma.
\]
For $1\le p<\infty$, set
\[
 \|\Ana_\Psi(f)\|_{\ell_s^p}
 =\left(
   \sum_{\ell=1}^L\sum_{\gamma\in\Gamma}
    \langle\gamma\rangle^{sp}
    |\Ana_\Psi(f)_{\ell,\gamma}|^p
  \right)^{1/p},
\]
and for $p=\infty$ set
\begin{equation}
 \|\Ana_\Psi(f)\|_{\ell_s^\infty}
 =\sup_{\ell,\gamma}\langle\gamma\rangle^s
   |\Ana_\Psi(f)_{\ell,\gamma}|.
 \label{eq:weighted-coefficient-supremum}
\end{equation}
On the integer lattice, we use the indexing
\[
 \Ana_\Psi(f)_{\ell,n,m}
 =\langle f,M_nT_{-m}\psi_\ell\rangle,
 \qquad n,m\in\mathbb Z^d.
\]

The following is the finite-multiwindow coefficient characterization;
see \cite[Propositions~12.2.3, 12.2.4, and~12.2.6]{Grochenig2001}.
\begin{theorem}
\label{thm:weighted-Gabor-coeff}
Let $\Psi=(\psi_1,\ldots,\psi_L)$ and
$\widetilde\Psi=(\widetilde\psi_1,\ldots,\widetilde\psi_L)$ be dual
finite multiwindow Gabor frames over a full-rank lattice $\Gamma$, with
all windows in $\mathcal S(\mathbb R^d)$.  Then, for every $f\in L^2(\mathbb R^d)$, every $1\le p\le\infty$, and every
$s\in\mathbb R$,
\[
 f\in M_s^p(\mathbb R^d)\quad\Longleftrightarrow\quad
 \Ana_\Psi(f)\in\ell_s^p,
 \qquad
 \|f\|_{M_s^p}\asymp\|\Ana_\Psi(f)\|_{\ell_s^p}.
\]
In particular, for $p=\infty$, membership in $M_s^\infty$ is exactly the
boundedness condition in \eqref{eq:weighted-coefficient-supremum}.
\end{theorem}

The cited results give the standard single-window statement.  The finite-
multiwindow formulation is a simple adjustment of the single-window proof;
compare also Feichtinger--Gr\"ochenig~\cite{FeichtingerGrochenig1997}.

We use the finite-multiwindow form of the following result; see also
\cite[Section~2]{Luef2009Projections} for arbitrary lattices.
\begin{theorem}[{Gr\"ochenig--Leinert~\protect\cite[Corollary~4.5]{GrochenigLeinert2004}}]
\label{prop:Schwartz-Parsevalization}
Let $\Lambda$ be a full-rank lattice and let
$g_1,\ldots,g_m\in\mathcal S(\mathbb R^d)$, $m\ge1$, generate a multiwindow
Gabor frame with frame operator $S$.  Then
$S^{-1/2}g_1,\ldots,S^{-1/2}g_m$ are Schwartz functions and generate a
Parseval multiwindow Gabor frame.
\end{theorem}

We shall also use the following fundamental embedding result.

See \cite[Theorem~12.2.2 and Propositions~12.2.3--12.2.6]{Grochenig2001}.
\begin{proposition}[Gr\"ochenig]
\label{prop:modulation-embeddings}
Let $n=2d$.  If $1\le p_1<p_0\le\infty$ and
$\displaystyle s_0-s_1>n(1/p_1-1/p_0)$, then
$\displaystyle M_{s_0}^{p_0}(\mathbb R^d)\hookrightarrow
M_{s_1}^{p_1}(\mathbb R^d)$; in particular,
$\displaystyle M_s^\infty\hookrightarrow M_q^2$ if
$\displaystyle s>q+d$.  If $1\le p_0\le p_1\le\infty$ and
$s_0\ge s_1$, then
$\displaystyle M_{s_0}^{p_0}(\mathbb R^d)\hookrightarrow
M_{s_1}^{p_1}(\mathbb R^d)$.
\end{proposition}

\subsection{BMO and VMO spaces}
\label{subsec:bmo-vmo-background}

Let $E$ be a fixed finite-dimensional normed space.  For a locally
integrable function $H:Q\to E$ on a cube $Q$, write
\[
 H_Q=\fint_Q H(x)\,dx,
 \qquad
 \operatorname{osc}_Q(H)=\fint_Q\|H(x)-H_Q\|_E\,dx.
\]
Throughout this paper,
\[
 [H]_{\operatorname{BMO}(\mathbb T^n)}
 :=\sup_Q\operatorname{osc}_Q(H)
\]
denotes the BMO seminorm; the supremum is over cubes in the periodic domain.
Writing $H_{\mathbb T^n}=\int_{\mathbb T^n}H$, we use
\[
 \|H\|_{\operatorname{BMO},*}
 :=\|H_{\mathbb T^n}\|_E+[H]_{\operatorname{BMO}(\mathbb T^n)}
\]
for a full norm.  The space $\VMO(\mathbb T^n;E)$ consists of the BMO functions for which
\[
 \sup_{\ell(Q)\le r}\operatorname{osc}_Q(H)\longrightarrow0
 \qquad\text{as }r\searrow0.
\]
Equivalently, on the compact periodic domain, it is the closure of the
continuous periodic $E$-valued functions in the BMO seminorm.  All choices of
norm on the target space $E$ give equivalent BMO and VMO conditions.  If
$T:E\to F$ is linear, then
\[
 \operatorname{osc}_Q(TH)\le\|T\|\operatorname{osc}_Q(H).
\]
Consequently, finite block amplification, vectorization, transposition,
duplication of components, and fixed finite-dimensional linear combinations
preserve local VMO.

Recall from \Cref{subsec:intro-zak-gap} the local version used on ordinary
Zak cells.  Let $K$ be a compact subset of an open set
$U\subset\mathbb R^n$, and let $H\in L^1_{\rm loc}(U)$.  Define
\begin{equation*}
 \eta_{H,K}(r)
 =\sup\bigl\{\operatorname{osc}_Q(H):
       Q\subset K,\ \ell(Q)\le r\bigr\}.
\end{equation*}
We say that $H$ is \emph{locally VMO} on $U$ if
$\eta_{H,K}(r)\to0$ as $r\searrow0$ for every compact $K\Subset U$.
For a quasiperiodic function in $q$ dimensions, the arguments
below use only the bounded collection of Zak cells
\begin{equation*}
 \mathcal Q_q^{\rm Zak}=[0,1]^{2q},
 \qquad
 (\mathcal Q_q^{\rm Zak})^*=[-2,3]^{2q}.
\end{equation*}
Thus ``locally VMO on the ordinary Zak domain'' means local VMO on a
neighborhood of $(\mathcal Q_q^{\rm Zak})^*$.  Every continuous function and
every VMO function is locally VMO. Continuity alone does not imply global
Euclidean VMO; on a compact torus, every continuous periodic function
belongs to VMO.

We shall use the following smooth-multiplier estimate.  If $Q\subset\mathbb R^n$ is a cube of side at most $\varepsilon$,
$b\in C^1$ near $Q$, and $h\in L^\infty(Q)$, then
\begin{equation}\label{eq:smooth-multiplier-oscillation}
 \operatorname{osc}_Q(bh)
 \le2\|b\|_{L^\infty(Q)}\operatorname{osc}_Q(h)
 +C_n\varepsilon\|\nabla b\|_{L^\infty(Q)}\|h\|_{L^\infty(Q)}.
\end{equation}
Indeed, fixing $x_Q\in Q$ and comparing with $b(x_Q)h_Q$ gives
\begin{align*}
 \operatorname{osc}_Q(bh)
 &\le2\fint_Q|b(x)h(x)-b(x_Q)h_Q|\,dx\\
 &\le2\|b\|_{L^\infty(Q)}\fint_Q|h-h_Q|
       +2|h_Q|\sup_{x\in Q}|b(x)-b(x_Q)|\\
 &\le2\|b\|_{L^\infty(Q)}\operatorname{osc}_Q(h)
       +2\sqrt n\,\varepsilon\|\nabla b\|_{L^\infty(Q)}
                              \|h\|_{L^\infty(Q)}.
\end{align*}

\section{Zak transforms}
\label{sec:zak-method}
\label{subsec:diagonal-rational-model}

Throughout this section, Part~II, and the diagonal constructions in
Part~III, we use the lattice
\begin{equation}
 \Lambda_D=D\mathbb Z^d\times\mathbb Z^d,
 \qquad
 D=\diag(b_1/a_1,\ldots,b_d/a_d),
 \qquad (a_i,b_i)=1,
 \label{eq:diagonal-model}
\end{equation}
where $a_i,b_i$ are positive integers.  The reduction of general
symplectically rational lattices to this model and the metaplectic transfer
are given in \Cref{sec:lattice-theory}.
Put
\begin{equation*}
 \mathsf A=\diag(a_1,\ldots,a_d),\qquad
 \mathsf B=\diag(b_1,\ldots,b_d),\qquad
 D=\mathsf B\mathsf A^{-1},
\end{equation*}
and
\begin{equation}
 N=\det\mathsf A=\prod_{i=1}^d a_i,
 \qquad
 R=\det\mathsf B=\prod_{i=1}^d b_i.
 \label{eq:NR}
\end{equation}
With the symplectic form fixed in \Cref{subsec:intro-gabor}, coprimality gives
\begin{equation*}
 \Lambda_D^\circ=\mathbb Z^d\times D^{-1}\mathbb Z^d,
 \qquad
 \Lambda_D\cap\Lambda_D^\circ
 =\mathsf B\mathbb Z^d\times\mathsf A\mathbb Z^d.
\end{equation*}
Consequently,
\begin{equation*}
 [\Lambda_D:\Lambda_D\cap\Lambda_D^\circ]
 =(\det\mathsf A)^2=N^2.
\end{equation*}
It follows that $\nu(\Lambda_D)=N$ and, similarly,
$\nu(\Lambda_D^\circ)=R$.  Thus
\begin{equation*}
 \covol(\Lambda_D)=\det D=R/N,
 \qquad \Delta(\Lambda_D)=N-R.
\end{equation*}

\subsection{Scalar, vector, and matrix-valued Zak transforms}
\label{sec:rational-zak-matrix-model}

For $f\in\mathcal S(\mathbb R^d)$, the scalar Zak transform is
\begin{equation*}
 Zf(u,\eta)
 =\sum_{k\in\mathbb Z^d}f(u-k)e^{2\pi i k\cdot\eta},
 \qquad (u,\eta)\in\mathbb R^d\times\mathbb R^d.
\end{equation*}
Its first variable $u\in\mathbb R^d$ is the position variable and its
second variable $\eta\in\mathbb R^d$ is the frequency variable.
It satisfies the quasiperiodicity condition
\begin{equation}
 Zf(u+n,\eta+m)
 =e^{2\pi i n\cdot\eta}Zf(u,\eta),
 \qquad n,m\in\mathbb Z^d,
 \label{eq:zak-quasiperiodicity}
\end{equation}
and extends uniquely to a unitary map
\begin{equation*}
 Z:L^2(\mathbb R^d)\xrightarrow{\;\cong\;}L^2([0,1)^{2d}),
 \qquad
 \|f\|_2^2=\int_{[0,1)^{2d}}|Zf(u,\eta)|^2\,du\,d\eta.
\end{equation*}
For $L^2$ functions, \eqref{eq:zak-quasiperiodicity} holds almost everywhere.

For a measurable quasiperiodic function $F$ with
$F|_{[0,1)^{2d}}\in L^2$, the inverse is
\begin{equation}
 f(u-k)=\int_{[0,1)^d}F(u,\eta)e^{-2\pi i k\cdot\eta}\,d\eta,
 \qquad u\in[0,1)^d,\quad k\in\mathbb Z^d.
 \label{eq:inverse-scalar-zak}
\end{equation}
See \cite[Proposition~1.19]{Toft2021} for unitarity.  The defining series
and \eqref{eq:inverse-scalar-zak} also give the bijection
\begin{equation*}
 Z:\mathcal S(\mathbb R^d)
 \xrightarrow{\;\cong\;}
 \{F\in C^\infty(\mathbb R^{2d}):F\text{ is quasiperiodic}\};
\end{equation*}
see also \cite[Lemma~2.2]{deDiosLiehrTaylor2026}.

Let
\begin{equation*}
 \mathcal I_{\mathsf A}=\prod_{i=1}^d\{0,\ldots,a_i-1\},
 \qquad
 \mathcal I_{\mathsf B}=\prod_{i=1}^d\{0,\ldots,b_i-1\},
\end{equation*}
so that $|\mathcal I_{\mathsf A}|=N$ and $|\mathcal I_{\mathsf B}|=R$.
For $f\in L^2(\mathbb R^d)$, the vector-valued Zak transform is
\begin{equation*}
 \mathcal Zf:\mathbb R^{2d}\to\mathbb C^R,
 \qquad
 \mathcal Zf(u,\eta)_s=Zf(u,\eta+\mathsf B^{-1}s),
 \quad s\in\mathcal I_{\mathsf B}.
\end{equation*}
For a scalar quasiperiodic function $F$, define its rational Zak matrix by
\begin{equation}
 \mathcal M_F(u,\eta)_{s,t}=F(u-Dt,\eta+\mathsf B^{-1}s),
 \qquad s\in\mathcal I_{\mathsf B},\quad t\in\mathcal I_{\mathsf A}.
 \label{eq:fine-zak-matrix}
\end{equation}
The matrix-valued Zak transform is
\begin{equation}\label{eq:zz-matrix-explicit}
 \mathcal Z_Df=\mathcal M_{Zf}:\mathbb R^{2d}\to\mathbb C^{R\times N},
 \qquad
 (\mathcal Z_Df)(u,\eta)_{s,t}=\mathcal Zf(u-Dt,\eta)_s.
\end{equation}
In particular,
\begin{equation*}
 Zf(u,\eta)=\mathcal Zf(u,\eta)_0
 =(\mathcal Z_Df)(u,\eta)_{0,0}.
\end{equation*}

For $r\in\mathbb Z^d$, let $[s+r]_{\mathsf B}$ denote the representative
in $\mathcal I_{\mathsf B}$.  For $t_0\in\mathbb Z^d$, write
$t+t_0=t'+\mathsf A\ell(t,t_0)$ with $t'\in\mathcal I_{\mathsf A}$.
Quasiperiodicity gives
\begin{align}
 (\mathcal M_F(u,\eta+\mathsf B^{-1}r))_{s,t}
  &=(\mathcal M_F(u,\eta))_{[s+r]_{\mathsf B},t},\notag\\
 (\mathcal M_F(u-Dt_0,\eta))_{s,t}
  &=e^{-2\pi i(\mathsf B\ell(t,t_0))\cdot\eta}
          (\mathcal M_F(u,\eta))_{s,t'},\notag\\
 (\mathcal M_F(u+n,\eta+k))_{s,t}
  &=e^{2\pi i n\cdot(\eta+\mathsf B^{-1}s)}
          (\mathcal M_F(u,\eta))_{s,t},\qquad n,k\in\mathbb Z^d.
 \label{eq:sample-translation-identities}
\end{align}
In the second identity,
$e^{-2\pi i(\mathsf B\ell)\cdot\mathsf B^{-1}s}
=e^{-2\pi i\ell\cdot s}=1$.
These row or column permutations and diagonal unitary factors preserve
singular values.  Since $D\mathbb Z^d+\mathbb Z^d=\mathsf A^{-1}\mathbb Z^d$,
the singular values of $\mathcal M_F$ are periodic under
\begin{equation*}
 \Gamma_D=\mathsf A^{-1}\mathbb Z^d\times\mathsf B^{-1}\mathbb Z^d.
\end{equation*}

Define the coarse and fine rectangles by
\begin{equation*}
 \mathcal Q_{\mathsf B}=[0,1)^d\times\prod_{i=1}^d[0,b_i^{-1}),
 \qquad
 Q_D=\prod_{i=1}^d[0,a_i^{-1})\times\prod_{i=1}^d[0,b_i^{-1}),
 \qquad \mathcal Q_D=\overline{Q_D}.
\end{equation*}
Restriction gives unitary maps
\begin{equation*}
 \begin{aligned}
 \mathcal Z:L^2(\mathbb R^d)&\xrightarrow{\;\cong\;}
 L^2(\mathcal Q_{\mathsf B};\mathbb C^R),\\
 \mathcal Z_D:L^2(\mathbb R^d)&\xrightarrow{\;\cong\;}
 L^2(Q_D;\mathbb C^{R\times N}),
 \end{aligned}
\end{equation*}
with
\begin{equation}
 \|f\|_2^2
 =\int_{\mathcal Q_{\mathsf B}}\|\mathcal Zf(u,\eta)\|_2^2\,du\,d\eta
 =\int_{Q_D}\|\mathcal Z_Df(u,\eta)\|_{\mathrm F}^2\,du\,d\eta.
 \label{eq:vector-matrix-zak-unitarity}
\end{equation}

\begin{figure}[H]
\centering
\begingroup
\newcommand{\ZakDomainPanel}[2]{%
  \begin{tikzpicture}[x=2.32cm,y=2.32cm,
      every node/.style={font=\scriptsize}]
    \path[use as bounding box] (-.42,-.5) rectangle (1.45,1.5);
    \draw[black!18,line width=.3pt] (-.3,-.3) rectangle (1.3,1.3);
    \fill[RoyalBlue!14] (0,0) rectangle (#1,#2);
    \foreach \i in {-1,...,5}
      \draw[black!22,line width=.3pt] ({\i/4},-.3) -- ({\i/4},1.3);
    \foreach \j in {0,...,3}
      \draw[black!22,line width=.3pt] (-.3,{\j/3}) -- (1.3,{\j/3});
    \draw[black!38,line width=.4pt] (0,0) rectangle (1,1);
    \draw[RoyalBlue!70!black,line width=.65pt]
      (0,0) rectangle (#1,#2);
    \draw[-{Stealth[length=1.3mm]},black!65,line width=.4pt]
      (-.3,0) -- (1.3,0) node[right] {$u$};
    \draw[-{Stealth[length=1.3mm]},black!65,line width=.4pt]
      (0,-.3) -- (0,1.3) node[above] {$\eta$};
    \draw[black!65,line width=.35pt] (1,-.025) -- (1,.025);
    \node[below,inner sep=2pt] at (1,-.025) {$1$};
    \draw[black!65,line width=.35pt] (-.025,1) -- (.025,1);
    \node[left,inner sep=2pt] at (-.025,1) {$1$};
    \node[below left,inner sep=2pt] at (0,0) {$0$};
  \end{tikzpicture}%
}
\begin{minipage}[t]{.32\linewidth}
  \centering\small Scalar: $Zf$\par\smallskip
  \ZakDomainPanel{1}{1}\par
  \small $[0,1)^2$
\end{minipage}\hfill
\begin{minipage}[t]{.32\linewidth}
  \centering\small Vector: $\mathcal Zf$\par\smallskip
  \ZakDomainPanel{1}{1/3}\par
  \small $\mathcal Q_{\mathsf B}=[0,1)\times[0,\tfrac13)$
\end{minipage}\hfill
\begin{minipage}[t]{.32\linewidth}
  \centering\small Matrix: $\mathcal Z_Df$\par\smallskip
  \ZakDomainPanel{1/4}{1/3}\par
  \small $Q_D=[0,\tfrac14)\times[0,\tfrac13)$
\end{minipage}
\endgroup
\caption{Canonical Zak domains for $d=1$, $a=4$, $b=3$, and $D=3/4$,
shaded in $[-.3,1.3]^2$.  The fine grid extends over the whole region,
with three rows and four columns per unit square; half-open boundaries
are drawn schematically.}
\label{fig:zak-canonical-domains}
\end{figure}
Indeed, the rectangles
\begin{equation*}
 \mathcal Q_{\mathsf B}+(0,\mathsf B^{-1}s)
 \quad\text{and}\quad
 Q_D+(-Dt,\mathsf B^{-1}s)
\end{equation*}
partition $[0,1)^{2d}$, respectively exactly and modulo $\mathbb Z^{2d}$,
up to null boundaries.  The second partition uses that multiplication by
$-b_i$ permutes the residues modulo $a_i$.
Thus arbitrary $L^2$ vector or matrix entries prescribe a scalar Zak
function on the cell; integer reductions contribute only the unimodular
factors in \eqref{eq:zak-quasiperiodicity}.  Scalar inversion proves
surjectivity, and integration over the pieces proves
\eqref{eq:vector-matrix-zak-unitarity}.

The following is the rational Zibulski--Zeevi criterion in our
normalization; see \cite{ZibulskiZeevi1997},
\cite[Theorem~7.2]{GrochenigKoppensteiner2019}, and
\cite[Corollary~4.4]{JakobsenLemvig2016}.
For the tight-frame identity, see also
\cite[Section~2.3]{JanssenStrohmer2002}.

\begin{proposition}[Zibulski--Zeevi~\protect\cite{ZibulskiZeevi1997}]
\label{prop:zz-criterion-full}
For $f,g\in L^2(\mathbb R^d)$,
\begin{equation*}
 \sum_{n,k\in\mathbb Z^d}|\langle f,M_nT_{Dk}g\rangle|^2
 =\frac1R\int_{\mathcal Q_{\mathsf B}}
   \|\mathcal Z_Dg(u,\eta)^*\mathcal Zf(u,\eta)\|_{\mathbb C^N}^2
   \,du\,d\eta,
\end{equation*}
with $+\infty$ allowed.  Hence $\mathcal G(g,\Lambda_D)$ has frame bounds
$A_0,B_0$ if and only if
\begin{equation*}
 RA_0I_R\le\mathcal Z_Dg(u,\eta)\mathcal Z_Dg(u,\eta)^*\le RB_0I_R
 \quad\text{a.e. on }\mathcal Q_{\mathsf B}.
\end{equation*}
The upper bound implies $\mathcal Z_Dg\in L^\infty$; the frame is Parseval
if and only if
\begin{equation}\label{eq:zz-parseval-identity}
 \mathcal Z_Dg(u,\eta)\mathcal Z_Dg(u,\eta)^*=RI_R
 \quad\text{a.e.}
\end{equation}
\end{proposition}

\subsection{Zak characterization of modulation spaces}

\begin{lemma}\label{lem:finite-Zak-family}
There exist $\psi_1,\ldots,\psi_L\in\mathcal S(\mathbb R^d)$ such that
\begin{equation}\label{eq:Theta-bounds}
 0<c_\Psi\le\Theta_\Psi(u,\eta)
 :=\sum_{\ell=1}^L|Z\psi_\ell(u,\eta)|^2
 \le C_\Psi<\infty
 \qquad\text{on }\mathbb T^{2d}.
\end{equation}
\end{lemma}

\begin{proof}
Take the smooth quasiperiodic functions $s_0,\ldots,s_d$ from
\Cref{lem:explicit-scalar-seeds}.  Set
\begin{equation*}
 L=d+1,\qquad \psi_\ell=Z^{-1}s_{\ell-1}\in\mathcal S(\mathbb R^d),
 \qquad
 \Theta_\Psi=\sum_{j=0}^d|s_j|^2=\Theta.
\end{equation*}
The bounds follow from \Cref{lem:explicit-scalar-seeds}.
\end{proof}

For $f\in L^2(\mathbb R^d)$ and $\psi\in\mathcal S(\mathbb R^d)$, the
cross-Zak product
\begin{equation*}
 G_{\psi,f}(u,\eta)=Zf(u,\eta)\,\overline{Z\psi(u,\eta)}
\end{equation*}
is $\mathbb Z^{2d}$-periodic, since the quasiperiodic phases cancel.
For $f\in\mathcal S(\mathbb R^d)$, put $Q=[0,1)^d$.  Inserting the
Zak series and integrating in $\eta$ gives
\begin{align}
 \widehat{G_{\psi,f}}(n,m)
 &=\sum_{k,j\in\mathbb Z^d}\int_Q
   f(u-k)\overline{\psi(u-j)}e^{-2\pi i n\cdot u}
   \times\left(\int_Q e^{2\pi i(k-j-m)\cdot\eta}\,d\eta\right)du\notag\\
 &=\sum_{k\in\mathbb Z^d}\int_Q
   f(u-k)\overline{\psi(u-k+m)}e^{-2\pi i n\cdot u}\,du
   =\int_{\mathbb R^d}
   f(x)\overline{\psi(x+m)}e^{-2\pi i n\cdot x}\,dx\notag\\
 &=\langle f,M_nT_{-m}\psi\rangle,
 \qquad n,m\in\mathbb Z^d.
 \label{eq:cross-zak-coefficients}
\end{align}
The inner integral is $\delta_{j,k-m}$.  The substitution $x=u-k$
then joins the cubes $Q-k$ into $\mathbb R^d$, since
$e^{-2\pi i n\cdot k}=1$.  The computation extends to $f\in L^2$
by density, Zak unitarity, and boundedness of $Z\psi$.

The following is the isotropic polynomial-weight case of Janssen's
Zak characterization.
The norm equivalence is its coefficient form, using
\eqref{eq:cross-zak-coefficients} and \Cref{thm:weighted-Gabor-coeff}.
\begin{proposition}[{\protect\cite[Theorem~2.6]{TinaztepeHeil2012}}]
\label{prop:cross-Zak-characterization}
Let $\Psi=(\psi_1,\ldots,\psi_L)$ be as in \Cref{lem:finite-Zak-family}.
For $f\in L^2(\mathbb R^d)$, $1\le p\le\infty$, and $s\in\mathbb R$,
\begin{equation*}
 \|f\|_{M_s^p(\mathbb R^d)}
 \asymp
 \left\|\bigl(\angles{k}^s\widehat{G_{\psi_\ell,f}}(k)\bigr)_{\ell,k}
 \right\|_{\ell^p(\{1,\ldots,L\}\times\mathbb Z^{2d})},
\end{equation*}
with either side allowed to be infinite.
\end{proposition}

The finite family also gives the reconstruction formula
\begin{equation}
 Zf=\Theta_\Psi^{-1}\sum_{\ell=1}^L G_{\psi_\ell,f}Z\psi_\ell
 =\sum_{\ell=1}^L b_\ell G_{\psi_\ell,f},
 \qquad b_\ell:=\Theta_\Psi^{-1}Z\psi_\ell.
 \label{eq:cross-zak-reconstruction}
\end{equation}
Here $G_{\psi_\ell,f}$ and $\Theta_\Psi$ are periodic, while $b_\ell$ is
smooth and quasiperiodic.

\part{Zak obstructions and critical nonexistence}\label{part:zak-obstructions}
This part proves the Zak-transform obstructions and the critical nonexistence results for the rational modulation-space classification.

\section{Ordinary-Zak common zeros and the symplectic index gap}
\label{sec:rectangular-common-zero}

We prove the three Zak obstruction statements in
\Cref{subsec:intro-zak-gap}.  One quantitative smoothing estimate applies
to matrices and, in the one-row case, to vector-valued Zak functions.
The continuous rectangular gap then yields the locally VMO obstruction;
a diagonal row phase correction connects it to the entries of the rational Zak matrix.

Recall that, for a finite-dimensional complex vector space $V$, a measurable
function $F:\mathbb R^q\times\mathbb R^q\to V$ is quasiperiodic if
\begin{equation}
 F(u+n,v+m)=e^{2\pi i\ip{n}{v}}F(u,v),
 \qquad n,m\in\mathbb Z^q,
 \label{eq:ordinary-zak-rule-vmo}
\end{equation}
almost everywhere.  Values outside the ordinary Zak fundamental cube
$\mathcal Q_q^{\rm Zak}$ are recovered from values inside that cube by
\eqref{eq:ordinary-zak-rule-vmo}.

We use the following continuous theorem as the only topological input in the
VMO argument.  It is Theorem~1.3 of de Dios Pont--Liehr--Taylor,
arXiv:2606.26052v1, dated 24 June 2026
\cite{deDiosLiehrTaylor2026}.

\begin{theorem}[de Dios Pont--Liehr--Taylor~\protect\cite{deDiosLiehrTaylor2026}]
\label{thm:ddlt-common-zero}
Let $G:\mathbb R^q\times\mathbb R^q\to\mathbb C^M$ be continuous and satisfy
\eqref{eq:ordinary-zak-rule-vmo}.  If $M\le q$, then $G$ has a common zero.
If $M>q$, there are smooth such maps with no common zero.
\end{theorem}

The common-zero implication is proved from first principles in
Sections~\ref{app:pfaffian-common-zero}--\ref{app:boundary-pfaffian}.
The smooth construction for $M>q$ is given in
\Cref{lem:explicit-scalar-seeds}.

\subsection{Ordinary-Zak smoothing}

Choose $\kappa\in C_c^\infty(\mathbb R^q)$ with
\[
 \kappa\ge0,
 \qquad
 \operatorname{supp}\kappa\subset B(0,1/4),
 \qquad
 \int_{\mathbb R^q}\kappa=1,
\]
and put $\kappa_\varepsilon(x)=\varepsilon^{-q}\kappa(x/\varepsilon)$.
For $F\in L^\infty_{\rm loc}(\mathbb R^{2q};V)$, with $V$ a finite-dimensional complex vector space, define
\begin{equation}
 F_\varepsilon(u,v)
 =\iint
   \kappa_\varepsilon(x)\kappa_\varepsilon(y)
   e^{2\pi i u\cdot y}F(u-x,v-y)\,dx\,dy.
 \label{eq:ordinary-zak-compatible-smoothing}
\end{equation}
The factor $e^{2\pi i u\cdot y}$ is the ordinary Zak phase correction.  It
compensates for the spatial Zak phase when the function is averaged in the
frequency variable.

\begin{lemma}
\label{lem:ordinary-zak-smoothing-rule}
The function $F_\varepsilon$ is smooth.  If $F$ satisfies
\eqref{eq:ordinary-zak-rule-vmo}, then
\[
 F_\varepsilon(u+n,v+m)
 =e^{2\pi i n\cdot v}F_\varepsilon(u,v),
 \qquad n,m\in\mathbb Z^q.
\]
\end{lemma}

\begin{proof}
With $u'=u-x$ and $v'=v-y$, formula
\eqref{eq:ordinary-zak-compatible-smoothing} becomes
\begin{equation}
 F_\varepsilon(u,v)
 =\iint
   \kappa_\varepsilon(u-u')\kappa_\varepsilon(v-v')
   e^{2\pi i u\cdot(v-v')}F(u',v')\,du'\,dv'.
 \label{eq:ordinary-zak-kernel-form}
\end{equation}
Every derivative in $(u,v)$ falls on the smooth compactly supported kernel
or on the smooth phase.  Since $F$ is bounded on the region of integration,
differentiation under the integral is valid to every order.

For quasiperiodicity, use \eqref{eq:ordinary-zak-rule-vmo}:
\begin{align*}
 F_\varepsilon(u+n,v+m)
 &=\iint \kappa_\varepsilon(x)\kappa_\varepsilon(y)
   e^{2\pi i(u+n)\cdot y}
   F(u+n-x,v+m-y)\,dx\,dy\\
 &=\iint \kappa_\varepsilon(x)\kappa_\varepsilon(y)
   e^{2\pi i u\cdot y}e^{2\pi i n\cdot y}
   e^{2\pi i n\cdot(v-y)}F(u-x,v-y)\,dx\,dy\\
 &=e^{2\pi i n\cdot v}F_\varepsilon(u,v).
\end{align*}
\end{proof}

\subsection{Preservation of a matrix lower bound at one scale}

We now apply the same smoothing directly to an ordinary-Zak matrix.
The useful quantity is its smallest singular value.  \Cref{lem:direct-matrix-smoothing} shows why control over the mean oscillation is required for the preservation of a positive smallest singular value after averaging. 

\begin{lemma}\label{lem:smallest-singular-value}
For $X\in\C^{R\times N}$ define
\begin{equation}\label{eq:sigma-min-def}
 \sigma_{\min}(X)
 =\inf_{\norm u=1}\norm{X^*u}.
\end{equation}
Then $X$ has full row rank exactly when $\sigma_{\min}(X)>0$, and for all
$X,Y\in\C^{R\times N}$,
\begin{equation}\label{eq:sigma-min-Lipschitz}
 |\sigma_{\min}(X)-\sigma_{\min}(Y)|
 \leq\norm{X-Y}_{\mathrm{op}}.
\end{equation}
\end{lemma}

\begin{proof}
For every unit vector $u$,
\[
 \norm{X^*u}\leq\norm{Y^*u}+\norm{X-Y}_{\mathrm{op}}.
\]
Taking the infimum gives
$\sigma_{\min}(X)\leq\sigma_{\min}(Y)+\norm{X-Y}_{\mathrm{op}}$.
Interchanging $X$ and $Y$ proves \eqref{eq:sigma-min-Lipschitz}. The rank
statement follows because $X$ has full row rank precisely when $X^*$ is
injective and bounded below on the finite-dimensional unit sphere.
\end{proof}

\begin{lemma}
\label{lem:direct-matrix-smoothing}
Let $q\ge1$, $N\ge R$, and let
$S:\R^{2q}\to\C^{R\times N}$ be essentially bounded and quasiperiodic entrywise.  Suppose
\begin{equation}\label{eq:one-scale-S-lower-bound}
 \sigma_{\min}(S(w))\ge a>0\quad\text{for almost every }w.
\end{equation}
For $0<\varepsilon\le1$ and $z\in\mathcal Q_q^{\rm Zak}$ let
$Q_z=z+[-\varepsilon/2,\varepsilon/2]^{2q}$ and put
\[
 \Omega_\varepsilon(S)=\sup_{z\in\mathcal Q_q^{\rm Zak}}
       \fint_{Q_z}\|S(w)-S_{Q_z}\|_{\rm op}\,dw.
\]
Apply \eqref{eq:ordinary-zak-compatible-smoothing} entrywise to $S$,
and denote the result by $S_\varepsilon$.  Then $S_\varepsilon$ is smooth,
is quasiperiodic, and satisfies
\begin{equation}\label{eq:direct-matrix-smoothing-bound}
 \inf_{z\in\mathcal Q_q^{\rm Zak}}\sigma_{\min}(S_\varepsilon(z))
 \ge a-C\Omega_\varepsilon(S)-C\varepsilon\|S\|_\infty.
\end{equation}
The constant depends only on the dimensions and the fixed mollifier.
In particular it is independent of $S$, $a$, and $\varepsilon$.  The
same estimate holds for a family $S=S_j$ with those data fixed.
\end{lemma}
\begin{proof}
All averages in this proof use the actual Euclidean representative on
$\R^{2q}$.  A cube crossing a fundamental-cell face is not cut at the
face.  For $z\in\mathcal Q_q^{\rm Zak}$ and $\varepsilon\le1$, its
closure lies in the fixed buffered box $(\mathcal Q_q^{\rm Zak})^*$.

\medskip\noindent\emph{Step 1: covariance and the positive kernel.}
The proof of \Cref{lem:ordinary-zak-smoothing-rule} applies to each
matrix entry without alteration.  In particular, $S_\varepsilon$ is
smooth and quasiperiodic.  With $z=(u,v)$, $w=(u',v')$, write
\[
 K_\varepsilon(z,w)=\kappa_\varepsilon(u-u')
                    \kappa_\varepsilon(v-v'),\qquad
 c(z,w)=e^{2\pi i u\cdot(v-v')}.
\]
Then
\[
 S_\varepsilon(z)=\int K_\varepsilon(z,w)c(z,w)S(w)\,dw.
\]
The kernel is nonnegative, has integral one, and the support
of $K_\varepsilon(z,\cdot)$ is contained in $Q_z$,
because each of $u-u'$ and $v-v'$ has norm at most $\varepsilon/4$.
Furthermore
\[
 K_\varepsilon(z,w)\le\|\kappa\|_\infty^2\varepsilon^{-2q},
 \qquad |Q_z|=\varepsilon^{2q}.
\]

\medskip\noindent\emph{Step 2: concentration around the uncorrected average.}
Put
\[
 \overline S_{\varepsilon,z}=\int K_\varepsilon(z,w)S(w)\,dw.
\]
Since the kernel has mass one,
\[
 \|\overline S_{\varepsilon,z}-S_{Q_z}\|_{\rm op}
 \le\int K_\varepsilon(z,w)\|S(w)-S_{Q_z}\|_{\rm op}\,dw.
\]
Apply the triangle inequality, and integrate once more.  This gives
\begin{align}
 \int K_\varepsilon(z,w)
       \|S(w)-\overline S_{\varepsilon,z}\|_{\rm op}\,dw
 &\le2\int K_\varepsilon(z,w)\|S(w)-S_{Q_z}\|_{\rm op}\,dw\notag\\
 &\le2\|\kappa\|_\infty^2\fint_{Q_z}\|S(w)-S_{Q_z}\|_{\rm op}\,dw
 \le C\Omega_\varepsilon(S).
 \label{eq:direct-matrix-concentration}
\end{align}
The second inequality is scale independent: the factor
$\varepsilon^{-2q}$ in the kernel is cancelled by $|Q_z|$.

\medskip\noindent\emph{Step 3: the uncorrected average retains the lower bound.}
For almost every $w$, \Cref{lem:smallest-singular-value} gives
\[
 a\le\sigma_{\min}(S(w))
 \le\sigma_{\min}(\overline S_{\varepsilon,z})
       +\|S(w)-\overline S_{\varepsilon,z}\|_{\rm op}.
\]
Integration against the mass-one kernel and
\eqref{eq:direct-matrix-concentration} yield
\begin{equation}\label{eq:direct-uncorrected-rank-bound}
 \sigma_{\min}(\overline S_{\varepsilon,z})
 \ge a-C\Omega_\varepsilon(S).
\end{equation}

\medskip\noindent\emph{Step 4: the phase correction is a small error.}
On the fundamental cell $|u|\le\sqrt q$, and on the kernel support
$|v-v'|\le\varepsilon/4$.  The inequality $|e^{it}-1|\le|t|$ gives
\[
 |c(z,w)-1|\le C_q\varepsilon.
\]
Consequently
\[
 \|S_\varepsilon(z)-\overline S_{\varepsilon,z}\|_{\rm op}
 \le C_q\varepsilon\|S\|_\infty.
\]
Apply \Cref{lem:smallest-singular-value} a second time and use
\eqref{eq:direct-uncorrected-rank-bound}.  This proves
\eqref{eq:direct-matrix-smoothing-bound} uniformly in $z$.

\medskip\noindent\emph{Step 5: the scope of the conclusion.}
Ordinary Zak translations multiply $S_\varepsilon$ by a scalar of
modulus one, so a positive lower bound on the fundamental cell holds on
all of $\R^{2q}$.  Every constant above is fixed independently of $S$.
The proof therefore applies without change when $S$ depends on a slice
parameter or on the smoothing scale.  No convergence in the uniform
norm $S_\varepsilon\to S$ is required.
\end{proof}

\subsection{The continuous rectangular gap and its VMO extension}

We first prove the continuous obstruction.  The auxiliary ordinary-Zak
vector occurs only in this algebraic step.  The measurable version will
then follow by smoothing the original matrix in its own variables.

\begin{proposition}\label{prop:continuous-rectangular-gap}
Let $S:\R^q\times\R^q\to\C^{R\times N}$ be continuous, quasiperiodic entrywise, and have full row rank everywhere.  Then
$N\ge R+q$.  Conversely, this inequality is sufficient for the existence
of a smooth matrix with these properties.
\end{proposition}
\begin{proof}
\phantomsection\label{proof:rectangular-common-zero-gap}
If $R=1$, the single row is an $N$-component nowhere-zero ordinary Zak
map in $q$ coordinate pairs.  Theorem~\ref{thm:ddlt-common-zero} requires
$N>q$, which is $N\ge R+q$.

Suppose $R\ge2$, and put $r=R-1$.  Choose a smooth nowhere-zero
ordinary Zak vector
\[
 \rho:\R^r\times\R^r\longrightarrow\C^R
\]
from \Cref{lem:explicit-scalar-seeds}, applied in $r$ coordinate pairs.
The norm of $\rho$ is periodic and continuous, so its minimum on the
ordinary compact cell is positive.  The same is true of
$\sigma_{\min}(S)$, since scalar Zak phases preserve singular values.
Write these minima as $c_\rho>0$ and $a>0$, respectively.

Define
\[
 \mathcal F((u,x),(v,y))=S(u,v)^T\rho(x,y)\in\C^N.
\]
For $n,k\in\Z^q$ and $n',k'\in\Z^r$, quasiperiodicity of both factors gives
\begin{align*}
 \mathcal F((u+n,x+n'),(v+k,y+k'))
 &=e^{2\pi i n\cdot v}e^{2\pi i n'\cdot y}
                         S(u,v)^T\rho(x,y)\\
 &=e^{2\pi i(n,n')\cdot(v,y)}\mathcal F((u,x),(v,y)).
\end{align*}
Thus $\mathcal F$ is an ordinary Zak map in $q+r=q+R-1$ coordinate
pairs.  It is continuous.  Also, for every $\xi\in\C^R$,
\[
 |S^T\xi|=|\overline{S^*\overline\xi}|
          =|S^*\overline\xi|\ge\sigma_{\min}(S)|\xi|.
\]
In particular $|\mathcal F|\ge ac_\rho>0$ everywhere.  If
$N\le q+R-1$, the continuous common-zero theorem contradicts this
lower bound.  Hence $N\ge q+R$.

For the converse, take smooth quasiperiodic functions
$s_0,\ldots,s_q$ with no common zero from
\Cref{lem:explicit-scalar-seeds}, applied in $q$ coordinate pairs.  First suppose $N=R+q$.  Index the rows by
$0,\ldots,R-1$ and the columns by $0,\ldots,R+q-1$, and set
\[
 (S_0(u,v))_{k\ell}=
 \begin{cases}
 s_{\ell-k}(u,v),&0\le\ell-k\le q,\\
 0,&\text{otherwise}.
 \end{cases}
\]
Every entry is smooth and quasiperiodic.  To check full
row rank, suppose $S_0(u,v)^Tc=0$ and form the polynomials
\[
 a(z)=\sum_{j=0}^q s_j(u,v)z^j,\qquad
 c(z)=\sum_{k=0}^{R-1}c_kz^k.
\]
The entries of $S_0(u,v)^Tc$ are exactly the coefficients of $a(z)c(z)$.
The first polynomial is not zero because the $s_j$ have no common zero.
Since $\C[z]$ has no zero divisors, $a(z)c(z)=0$ forces $c=0$.
This proves full row rank.  If $N>R+q$, append $N-R-q$ zero columns.
\end{proof}

\begin{proof}[Proof of \Cref{thm:rectangular-common-zero-gap}]
Apply \Cref{lem:direct-matrix-smoothing} with $q=d$ to the matrix in
the theorem.  Its entries are locally VMO on a neighborhood of the fixed
buffered Zak cell.  Equivalence of norms in the finite-dimensional
matrix space gives
\[
 \Omega_\varepsilon(S)\le C_{R,N}\max_{k,i}
   \eta_{S_{ki},(\mathcal Q_d^{\rm Zak})^*}(\varepsilon)
 \longrightarrow0.
\]
The norm $\|S\|_\infty$ is finite, and the lower bound $a$ is fixed.
Choose $\varepsilon$ so small that
\[
 C\Omega_\varepsilon(S)+C\varepsilon\|S\|_\infty<a/2.
\]
The resulting $S_\varepsilon$ is smooth, quasiperiodic,
and has smallest singular value at least $a/2$ everywhere.  Proposition
\ref{prop:continuous-rectangular-gap} therefore yields $N\ge R+d$.
Its smooth construction gives the converse as well.
\end{proof}

\begin{proof}[Proof of \Cref{thm:vmo-essential-common-zero}]
Suppose that the stated essential infimum is positive.  Then there is
$c>0$ such that $|F(w)|\ge c$ almost everywhere on
$\mathcal Q_q^{\rm Zak}$.  Quasiperiodicity preserves the norm,
so the same lower bound holds on every integer translate and hence on
all of $\mathbb R^{2q}$.

Regard $F=(F_1,\ldots,F_M)$ as a one-row matrix $S\in\mathbb C^{1\times M}$.
For a scalar $z$ with $|z|=1$,
\[
 \|S(w)^*z\|^2=\sum_{j=1}^M|F_j(w)|^2=|F(w)|^2.
\]
Thus $\sigma_{\min}(S)=|F|\ge c$.  On the fixed buffered Zak box,
local VMO of the finitely many components gives
\[
 \Omega_\varepsilon(S)
 \le C_M\max_{1\le j\le M}
       \eta_{F_j,(\mathcal Q_q^{\rm Zak})^*}(\varepsilon)
 \longrightarrow0.
\]
Apply \Cref{lem:direct-matrix-smoothing} with one row.  For sufficiently
small $\varepsilon$, the resulting smooth ordinary Zak map satisfies
$|F_\varepsilon|\ge c/2$ on the fundamental cell, and then everywhere
by quasiperiodicity.  This contradicts \Cref{thm:ddlt-common-zero}
when $M\le q$.  Hence the essential infimum is zero.

If $M>q$, the $q+1$ smooth components in
\Cref{lem:explicit-scalar-seeds}, with zero components appended when
$M>q+1$, have no common zero.  Their squared norm is positive and
periodic, so they are bounded and locally VMO and give the converse.
\end{proof}

The following reduction connects the rational Zak matrix to quasiperiodicity and records the two regularity facts needed later.

\begin{proposition}
\label{prop:ordinary-zak-reduction}
Let
\[
 A_{s,t}(x,\eta)=F(x-Dt,\eta+\mathsf B^{-1}s),
 \qquad s\in\mathcal I_{\mathsf B},\quad t\in\mathcal I_{\mathsf A},
\]
be an $R\times N$ rational Zak matrix.  Choose $q$ coordinate pairs and fix
the unused coordinates at a parameter $\zeta$.  The assertions below hold
for every $\zeta$ for which the sliced entries have the stated boundedness
and regularity.  In \Cref{sec:critical-nonexistence}, the mixed-norm estimate of
\Cref{lem:mixed-oscillation} supplies scale-dependent slices for every
finite $p$.  The direct smoothing in
\Cref{lem:direct-matrix-smoothing} then preserves their common positive
lower bound.

Write $x',\eta'\in\R^q$ for the selected variables and $s'$ for the selected
components of $s$.  Define
\begin{equation*}
 D_{\mathsf B}^{\zeta}(x')
 =\diag_{s\in\mathcal I_{\mathsf B}}
    \left(e^{-2\pi i x'\cdot(\mathsf B^{-1})'s'}\right),
 \qquad
 \widetilde A^{\zeta}=D_{\mathsf B}^{\zeta}(x')A^{\zeta},
 \qquad
 S^{\zeta}=\widetilde A^{\zeta}.
\end{equation*}
Then, entrywise,
\begin{equation}
 S^{\zeta}(x'+n,\eta'+m)
 =e^{2\pi i n\cdot\eta'}S^{\zeta}(x',\eta'),
 \qquad n,m\in\Z^q.
 \label{eq:ordinary-zak-explicit-rule}
\end{equation}
The matrices $S^{\zeta}$ and $A^{\zeta}$ have the same nonzero singular
values.  Boundedness and local VMO pass from $A^{\zeta}$ to $S^{\zeta}$.
More quantitatively, if the localized entries of $A^{\zeta}$ have
oscillation at most $\delta$ on cubes of side at most $C\varepsilon$, then
the entries of $S^{\zeta}$ have oscillation at most
\begin{equation}
 C_1\delta+C_2\varepsilon\|A^{\zeta}\|_\infty.
 \label{eq:ordinary-zak-explicit-oscillation}
\end{equation}
\end{proposition}

\begin{proof}
The scalar Zak identity gives
\begin{equation*}
 A_{s,t}^{\zeta}(x'+n,\eta'+m)
 =e^{2\pi i n\cdot(\eta'+(\mathsf B^{-1})'s')}
  A_{s,t}^{\zeta}(x',\eta').
\end{equation*}
Multiplying by the row phase factor at $x'+n$ cancels the additional rational
phase, since
\begin{equation*}
 e^{-2\pi i(x'+n)\cdot(\mathsf B^{-1})'s'}
 e^{2\pi i n\cdot(\eta'+(\mathsf B^{-1})'s')}
 =e^{2\pi i n\cdot\eta'}
  e^{-2\pi i x'\cdot(\mathsf B^{-1})'s'}.
\end{equation*}
This proves \eqref{eq:ordinary-zak-explicit-rule}.

The diagonal phase multiplier is unitary and therefore preserves all singular values.
The row phase factors have uniformly bounded first derivatives on the
fixed buffered box.  Applying \eqref{eq:smooth-multiplier-oscillation}
to each entry proves \eqref{eq:ordinary-zak-explicit-oscillation}.
If the entries of $A^\zeta$ are locally VMO, the two terms on its
right-hand side tend to zero with the cube size.  This proves local
VMO preservation as well.
\end{proof}

\begin{proof}[Proof of \Cref{thm:intro-diagonal-zak-obstruction}]
Suppose that $g\in L^2(\mathbb R^d)$ generates a frame on $\Lambda_D$
and that $F=Zg$ is locally VMO.  By \Cref{prop:zz-criterion-full},
the rational Zak matrix $A=\mathcal Z_Dg$ is essentially bounded and
has a uniform positive lower singular-value bound.  If $N<R$, its
$R$ rows cannot be independent, which already contradicts that bound.
We may therefore assume $N\ge R$.

Every entry of $A$ is a fixed translate of $F$, so it is locally VMO.
Apply \Cref{prop:ordinary-zak-reduction} with all $d$ coordinate pairs.
The resulting ordinary Zak matrix has the same singular values, is
bounded, and has locally VMO entries.  Theorem~\ref{thm:rectangular-common-zero-gap}
therefore gives $N\ge R+d$, contrary to $N-R<d$.  A continuous $F$
is locally VMO, so the continuous case follows as well.
\end{proof}

\section{Critical Fourier--Lebesgue regularity and vanishing oscillation}
\label{sec:vmo}

We prove one scalar oscillation estimate that covers both sides of
$p=2$.  It gives the critical Fourier--Lebesgue--VMO embedding when
there are no remaining variables.  With remaining variables it gives
small mean oscillation on a scale-dependent slice, which is the input
needed by direct matrix smoothing in the necessity proof.

\subsection{A mixed-norm estimate at one selected-variable scale}
\label{subsec:one-scale-anisotropic}

In preparation for \Cref{sec:critical-nonexistence}, we are interested in quantifying mean oscillations on slices of a function. To this end, 
write variables as $(y,z)\in\T^t\times\T^m$, where $m\ge1$
and $t\ge0$.  The variables $z$ are those retained in the later
common-zero argument; $y$ denotes the remaining coordinates and can be thought of as a parameter.  When
$t=0$, the space $\T^0$ consists of one point of measure one, and
its $L^r$ norm is simply the absolute value.

For $0<h\le1$ let $\mathcal Q(h)$ be all translates, on the
$m$-torus, of a coordinate cube of side $h$.  For almost every $y$ put
\begin{equation}\label{eq:mixed-oscillation-definition}
 \omega_h^z f(y)=\sup_{Q\in\mathcal Q(h)}
        \fint_Q|f(y,z)-(f_y)_Q|\,dz,
 \qquad f_y(z)=f(y,z).
\end{equation}
A torus cube crossing a coordinate seam is interpreted by a translated
Euclidean lift of the periodic function, not by cutting the cube at the
seam. 

For a fixed $h$, this supremum defines a measurable function of $y$, at least if $f$ is square integrable on the torus.
Indeed, in such a case, the slices are also $L^1$ on the torus and for an $L^1$ slice the mean oscillation on a cube of side $h$
depends continuously on its center.  To see this, note that
\[
 \big|\operatorname{osc}_{Q+c}(g)
       -\operatorname{osc}_{Q+c'}(g)\big|
 \le2h^{-m}\|g(\,\cdot+c)-g(\,\cdot+c')\|_{L^1(\T^m)}
\]
and recall that translation is continuous in $L^1$.  Consequently the supremum in
\eqref{eq:mixed-oscillation-definition} is unchanged if the centers
are restricted to the countable dense set of rational points.  For each of those centers the sliced integral is measurable by Fubini and taking the supremum over a countable family preserves measurability. 

\begin{lemma}
\label{lem:mixed-oscillation}
Let $1\le p<\infty$, $m\ge1$, and $t\ge0$, and set
\begin{equation}\label{eq:mixed-critical-exponent}
 r_p=\min\{p,2\},\qquad
 s_p(m,t)=m(1-1/p)+t(1/2-1/p)_+,
\end{equation}
where $x_+=\max\{x,0\}$.  For
$f\in\Fell{p}{s_p(m,t)}(\T^t\times\T^m)$,
\begin{equation}\label{eq:mixed-oscillation-bound}
 \|\omega_h^z f\|_{L^{r_p}(\T^t)}
 \le C_{m,t,p}\|f\|_{\Fell{p}{s_p(m,t)}},\qquad0<h\le1,
\end{equation}
and
\begin{equation}\label{eq:mixed-oscillation-vanishing}
 \|\omega_h^z f\|_{L^{r_p}(\T^t)}\longrightarrow0
 \quad\text{as }h\searrow0.
\end{equation}
The same conclusions hold after summing the oscillation functions of
any fixed finite family, with the sum of their norms on the right.
\end{lemma}
\begin{proof}
Put $s=s_p(m,t)$ and $\sigma=m(1-1/p)$.  We first prove that $f$
has an $L^2$ representative.  We then prove the scale-uniform estimate
for trigonometric polynomials by separating, with respect to the oscillation scale $h$, low and high frequencies in $z$.
Finally, steps $7$ and $8$ pass the approximation estimate and its vanishing oscillation consequence
to an arbitrary function in $\mathcal F\ell_s^p$. All torus measures in the proof have mass one.

\medskip\noindent\emph{Step 1: the $L^2$ representative and a lattice sum.}
Write
\[
 a_{\mu,\nu}=\widehat f(\mu,\nu),\qquad
 \mu\in\Z^t,\quad\nu\in\Z^m.
\]
When $p\le2$, the assumptions $s\ge0$ and
$\ell^p\subset\ell^2$ give
\[
 \|(a_{\mu,\nu})\|_{\ell^2}
 \le\|(a_{\mu,\nu})\|_{\ell^p}
 \le\|f\|_{\Fell{p}{s}}.
\]
When $p>2$, define
\[
 v=\frac{2p}{p-2},\qquad\frac12=\frac1p+\frac1v.
\]
Then
\[
 vs=t+\frac{2m(p-1)}{p-2}>t+m.
\]
H\"older's inequality for the product
$(\langle(\mu,\nu)\rangle^s|a_{\mu,\nu}|)
 \langle(\mu,\nu)\rangle^{-s}$ gives
\[
 \|(a_{\mu,\nu})\|_{\ell^2}
 \le\|f\|_{\Fell{p}{s}}
       \left(\sum_{\mu,\nu}
                \langle(\mu,\nu)\rangle^{-vs}\right)^{1/v}
 \le C\|f\|_{\Fell{p}{s}}.
\]
The last sum converges: the annulus of radius comparable to $2^j$
in $\Z^{t+m}$ has at most $C2^{j(t+m)}$ points, while each summand
there is at most $C2^{-jvs}$, which the placeholder constant $C$ independent of $j$.  Thus the annular sums are of order $2^{j(t+m-vs)}$. Since $vs>t+m$, these form a convergent
geometric series.  Plancherel supplies $f\in L^2(\T^{t+m})$ in
both ranges.  Fubini therefore supplies its actual $L^2$ slices for
almost every $y$.

We will also use, for $L\ge1$ and $\beta>t$, the estimate
\begin{equation}\label{eq:lattice-integral-comparison}
 \sum_{\mu\in\Z^t}(L^2+|\mu|^2)^{-\beta/2}
 \le C_{t,\beta}L^{t-\beta}.
\end{equation}
On a unit cube centered at $\mu$, the quantities
$L^2+|x|^2$ and $L^2+|\mu|^2$ are comparable with constants depending
only on $t$, since $L\ge1$ and $|x-\mu|\le\sqrt t/2$.  Integrating
on these cubes and summing gives a constant times
\[
 \int_{\R^t}(L^2+|x|^2)^{-\beta/2}\,dx
 =L^{t-\beta}\int_{\R^t}(1+|u|^2)^{-\beta/2}\,du.
\]
The integral over $|u|\le1$ is finite by boundedness; over $|u|>1$
polar coordinates give the convergent integral of
$r^{t-1-\beta}$.  For $t=0$, the lattice has one point and the formula
reads $L^{-\beta}\le L^{-\beta}$.

\medskip\noindent\emph{Step 2: the coefficient functions in the remaining variables.}
For now let $f$ be a trigonometric polynomial, and write
\[
 f(y,z)=\sum_\nu b_\nu(y)e^{2\pi i\nu\cdot z},\qquad
 b_\nu(y)=\sum_\mu a_{\mu,\nu}e^{2\pi i\mu\cdot y}.
\]
All sums are finite at this stage.  Set
\[
 E_\nu=\left(\sum_\mu
       \langle(\mu,\nu)\rangle^{ps}|a_{\mu,\nu}|^p\right)^{1/p}.
\]
These quantities contain exactly the original weighted coefficients:
\begin{equation}\label{eq:mixed-coefficient-norm}
 \|(E_\nu)_\nu\|_{\ell^p}=\|f\|_{\Fell{p}{s}}.
\end{equation}
We claim that
\begin{equation}\label{eq:mixed-bnu-bound}
 \|b_\nu\|_{L^{r_p}(\T^t)}
 \le C\langle\nu\rangle^{-\sigma}E_\nu.
\end{equation}

For $p\le2$, the synthesis form of Hausdorff--Young gives
\[
 \|b_\nu\|_{L^{p'}(\T^t)}
 \le\|(a_{\mu,\nu})_\mu\|_{\ell^p}.
\]
This is obtained by interpolation between $\ell^1\to L^\infty$
and $\ell^2\to L^2$; see \cite[Theorem~1.3.4]{Grafakos2014}.
Since the torus has measure one and $p\le p'$, H\"older gives
$\|b_\nu\|_p\le\|b_\nu\|_{p'}$.  Here $s=\sigma$, and
$\langle(\mu,\nu)\rangle\ge\langle\nu\rangle$.  Hence
\[
 \|b_\nu\|_p\le\|(a_{\mu,\nu})_\mu\|_{\ell^p}
 \le\langle\nu\rangle^{-\sigma}E_\nu.
\]
This includes $p=1$, where $p'=\infty$ and $\sigma=0$.

For $p>2$, use Parseval in $y$ and weighted H\"older with the exponent
$v$ from Step~1:
\begin{align*}
 \|b_\nu\|_2
 &=\left(\sum_\mu|a_{\mu,\nu}|^2\right)^{1/2}\\
 &\le E_\nu\left(\sum_\mu
               \langle(\mu,\nu)\rangle^{-vs}\right)^{1/v}\\
 &\le C E_\nu\langle\nu\rangle^{t/v-s}
 =C E_\nu\langle\nu\rangle^{-\sigma}.
\end{align*}
The last inequality comes from the observation that $\langle(\mu,\nu)\rangle^2 = 1+|\mu|^2+|\nu|^2=\langle\nu\rangle^2 +|\mu|^2$ and applying 
\eqref{eq:lattice-integral-comparison} with $L=\langle \nu\rangle$ and $\beta=vs>t+m>t$. 
The last equality uses
$s=\sigma+t(1/2-1/p)=\sigma+t/v$.  This proves
\eqref{eq:mixed-bnu-bound} in both ranges.

\medskip\noindent\emph{Step 3: one low--high frequency split.}
Fix $h$ and put $K=h^{-1}\ge1$.  Split only in the selected frequency:
\[
 f_{<K}(y,z)=\sum_{|\nu|<K}b_\nu(y)e^{2\pi i\nu\cdot z},\qquad
 f_{\ge K}=f-f_{<K}.
\]
Linearity of averaging and the triangle inequality give
\begin{equation}\label{eq:mixed-low-high-split}
 \omega_h^z f\le\omega_h^z f_{<K}+\omega_h^z f_{\ge K}.
\end{equation}
For every integrable scalar $g$ on $Q$, we shall repeatedly use
\begin{equation}\label{eq:mean-oscillation-by-size}
 \fint_Q|g-g_Q|\le2\fint_Q|g|.
\end{equation}
Indeed, $|g_Q|\le\fint_Q|g|$.  We also have
\[
 \fint_Q|g-g_Q|\le\fint_Q\fint_Q|g(z)-g(z')|\,dz'\,dz.
\]
These inequalities hold for every translate of the cube, with the same
constants.

\medskip\noindent\emph{Step 4: low frequencies and the critical lattice count.}
For a Fourier character, the fundamental theorem of calculus on the
segment joining $z'$ to $z$ gives
\begin{align*}
 |e^{2\pi i\nu\cdot z}-e^{2\pi i\nu\cdot z'}|
 &=\left|\int_0^1 2\pi i\nu\cdot(z-z')
           e^{2\pi i\nu\cdot(z'+\tau(z-z'))}\,d\tau\right|\\
 &\le2\pi|\nu|\,|z-z'|.
\end{align*}
A Euclidean lift of a cube of side $h$ has diameter $\sqrt m\,h$.
Thus, for every $y$,
\[
 \omega_h^z f_{<K}(y)\le C_mh\sum_{0<|\nu|<K}|\nu|\,|b_\nu(y)|.
\]
Minkowski and then \eqref{eq:mixed-bnu-bound} imply
\begin{equation}\label{eq:mixed-low-before-holder}
 \|\omega_h^z f_{<K}\|_{r_p} \le Ch\sum_{0<|\nu|<K}|\nu|\|b_\nu\|_{r_p}
 \le Ch\sum_{0<|\nu|<K}|\nu|\langle\nu\rangle^{-\sigma}E_\nu.
\end{equation}
For $p=1$, the last sum is bounded by $K\|E\|_{\ell^1}$.
For $1<p<\infty$, H\"older gives the two factors
\[
 \|E\|_{\ell^p}
 \left(\sum_{0<|\nu|<K}
          |\nu|^{p'}\langle\nu\rangle^{-\sigma p'}\right)^{1/p'}.
\]
Here $\sigma p'=m$.  We spell out the estimate of the second factor,
because this is where the critical weight enters.

For $j\ge0$ let
$\mathcal A_j=\{\nu\in\Z^m:2^j\le|\nu|<2^{j+1}\}$.
It is easy to see that $\#\mathcal A_j\le C_m2^{jm}$, since, for example, $\mathcal A_j\subset [-2^{j+1},2^{j+1}]^m\cap\Z^m$.
On this annulus,
\[
 |\nu|^{p'}\langle\nu\rangle^{-m}
 \le C_p2^{j(p'-m)}.
\]
Consequently the sum on the annulus is at most $C_{m,p}2^{jp'}$.
The number of lattice points has cancelled the power $2^{-jm}$
from the critical weight.  If $L=\lceil\log_2K\rceil$, summing gives
\[
 \sum_{0<|\nu|<K}|\nu|^{p'}\langle\nu\rangle^{-m}
 \le C\sum_{j=0}^{L}2^{jp'}\le C2^{Lp'}\le C'K^{p'}.
\]
The zero frequency contributes no oscillation and was excluded.  Taking
the $p'$th root and inserting in
\eqref{eq:mixed-low-before-holder}, with $hK=1$, proves
\begin{equation}\label{eq:mixed-low-bound}
 \|\omega_h^z f_{<K}\|_{r_p}\le C\|E\|_{\ell^p}.
\end{equation}

\medskip\noindent\emph{Step 5: high frequencies when $1\le p\le2$.}
For fixed $y$, Hausdorff--Young in $z$ gives
\[
 \|f_{\ge K}(y,\cdot)\|_{p'}
 \le\left(\sum_{|\nu|\ge K}|b_\nu(y)|^p\right)^{1/p}.
\]
H\"older on a cube of volume $h^m$, followed by
\eqref{eq:mean-oscillation-by-size}, yields
\[
 \omega_h^z f_{\ge K}(y)
 \le2h^{-m/p'}\left(\sum_{|\nu|\ge K}|b_\nu(y)|^p\right)^{1/p}.
\]
Taking the $L^p_y$ norm, interchanging the order of integration and summation by Tonelli and using the inequality from Step~2 give
\begin{align*}
 \|\omega_h^z f_{\ge K}\|_p
 &\le2K^\sigma
       \left(\sum_{|\nu|\ge K}\|b_\nu\|_p^p\right)^{1/p}\\
 &\le C K^\sigma
       \left(\sum_{|\nu|\ge K}
                  \langle\nu\rangle^{-p\sigma}E_\nu^p\right)^{1/p}
 \le C\|E\|_{\ell^p}.
\end{align*}
Recall that $h^{-m/p'}=h^{-\sigma}=K^\sigma$. 
The last inequality uses $\langle\nu\rangle\ge K$ on the tail and
$\sigma\ge0$.  In particular, it remains valid at $p=1$.

\medskip\noindent\emph{Step 6: high frequencies when $2<p<\infty$.}
Cauchy--Schwarz on each cube gives
\[
 \omega_h^z f_{\ge K}(y)
 \le2h^{-m/2}\|f_{\ge K}(y,\cdot)\|_{L^2_z}.
\]
The right side is independent of the center of the cube, which is why
it bounds the supremum over all translates.  Taking $L^2_y$ gives
\[
 \|\omega_h^z f_{\ge K}\|_2
 \le2K^{m/2}\|f_{\ge K}\|_{L^2_{y,z}}.
\]
By Parseval in $z$, Tonelli, and \eqref{eq:mixed-bnu-bound},
\[
 \|f_{\ge K}\|_2
 =\left(\sum_{|\nu|\ge K}\|b_\nu\|_2^2\right)^{1/2}
 \le C\left(\sum_{|\nu|\ge K}
                     \langle\nu\rangle^{-2\sigma}E_\nu^2\right)^{1/2}.
\]
Weighted H\"older with $1/2=1/p+1/v$ bounds this by
\[
 C\|E\|_{\ell^p}
 \left(\sum_{|\nu|\ge K}\langle\nu\rangle^{-v\sigma}\right)^{1/v}.
\]
The annulus $2^jK\le|\nu|<2^{j+1}K$ contains at most
$C(2^jK)^m$ points.  Since $v\sigma>m$, its sum and the geometric
series over $j\ge0$ give
\[
 \sum_{|\nu|\ge K}\langle\nu\rangle^{-v\sigma}
 \le C K^{m-v\sigma}\sum_{j\ge0}2^{j(m-v\sigma)}
 \le C'K^{m-v\sigma}.
\]
Finally,
\[
 \frac mv-\sigma=m(1/2-1/p)-m(1-1/p)=-\frac m2.
\]
Thus $\|f_{\ge K}\|_2\le CK^{-m/2}\|E\|_{\ell^p}$, and
\[
 \|\omega_h^z f_{\ge K}\|_2\le C\|E\|_{\ell^p}.
\]
The spatial factor $K^{m/2}$ is cancelled exactly by the weighted
Fourier tail.  Steps~3--6 and \eqref{eq:mixed-coefficient-norm} prove
\eqref{eq:mixed-oscillation-bound} for polynomials.

\medskip\noindent\emph{Step 7: passage to the actual measurable function.}
Let $P_Mf$ be the rectangular Fourier truncation to
$|\mu|_\infty,|\nu|_\infty\le M$.  The weighted coefficient tails
tend to zero because $p<\infty$, so
\[
 \|P_Mf-f\|_{\Fell{p}{s}}\longrightarrow0.
\]
Step~1 also gives convergence in $L^2(\T^{t+m})$.  For two $L^2$
slices $g_y,\bar g_y$, and for any cube of side $h$, the triangle
inequality and Cauchy--Schwarz show that
\[
 |\operatorname{osc}_Q(g_y)-\operatorname{osc}_Q(\bar g_y)|
 \le2h^{-m/2}\|g_y-\bar g_y\|_{L^2_z}.
\]
Taking the supremum over centers preserves this bound.  Since $r_p\le2$
and the $y$-torus has measure one,
\begin{equation}\label{eq:oscillation-fixed-scale-continuity}
 \|\omega_h^z g-\omega_h^z \bar g\|_{L^{r_p}_y}
 \le2h^{-m/2}\|g-\bar g\|_{L^2_{y,z}}.
\end{equation}
For each fixed positive $h$ the oscillation functions of $P_Mf$
therefore converge to that of the actual $f$ in $L^{r_p}_y$.
Passing to the limit proves \eqref{eq:mixed-oscillation-bound} 
for an arbitrary $f\in \mathcal F\ell_s^p$. 

\medskip\noindent\emph{Step 8: vanishing by polynomial approximation.}
For a fixed trigonometric polynomial $P$, the mean-value estimate gives
\[
 \omega_h^zP(y)\le C_mh\|\nabla_zP\|_\infty,
\]
uniformly in $y$.  By subadditivity and the bound already established,
\[
 \|\omega_h^z f\|_{r_p}
 \le C\|f-P\|_{\Fell{p}{s}}+C_mh\|\nabla_zP\|_\infty.
\]
Given $\delta>0$, first choose $P$ so that the first term is less than
$\delta/2$.  Keep that polynomial fixed, and then choose $h$ so small
that the second term is less than $\delta/2$.  This proves
\eqref{eq:mixed-oscillation-vanishing}.  For a fixed finite family, sum
the displayed estimates and use Minkowski's inequality.  Finiteness is
important: no constant depends on the scale, but the number of functions
is fixed in advance.
\end{proof}

\subsection{The critical VMO embedding}

\begin{proof}[Proof of \Cref{thm:critical-vmo}]
Take $t=0$ in \Cref{lem:mixed-oscillation}.  Then
$s_p(m,0)=m(1-1/p)$, and the $L^{r_p}(\T^0)$ norm is just the
scalar oscillation supremum.  Uniformly for every $0<h\le1$,
\[
 \sup_{\ell(Q)=h}\fint_Q|f-f_Q|
 \le C_{m,p}\|f\|_{\Fell{p}{m(1-1/p)}}.
\]
Taking the supremum over $h$ proves the BMO seminorm bound
\eqref{eq:FL-critical-BMO-estimate}.  The mean of $f$ is its zero
Fourier coefficient, whose modulus is at most the same
Fourier--Lebesgue norm, so the full BMO norm is controlled as well.

Step~8 of the proof of \Cref{lem:mixed-oscillation} also shows
that given an aribtrarily small $\delta$ one can pick $h$ small enough for the maximal oscillation across length $h$ cubes to be bounded above by $\delta$, which is precisely 
the definition of $f$ being VMO. 
Equivalently, the proof above approximates $f$ by trigonometric
polynomials with respect to the BMO seminorm, in agreement with
\cite{Sarason1975}.

\end{proof}

\section{Critical nonexistence in the modulation-space scale}
\label{sec:critical-nonexistence}

The two finite-$p$ critical curves follow from one argument.  The scalar
cross-Zak products provide periodic functions to which the mixed-norm
oscillation lemma applies.  A fixed finite reconstruction recovers the
rational Zak matrix on each selected slice.  After the ordinary-Zak row
phase correction, direct matrix smoothing preserves its positive lower bound.
The continuous rectangular gap then gives the contradiction.  The two
numerical thresholds are recorded in one theorem.

\subsection{The finite-$p$ endpoint theorem}

\begin{theorem}
\label{thm:finite-p-endpoint}

Let the diagonal model \eqref{eq:diagonal-model} satisfy
$0\le N-R\le d-1$, let $1\le p<\infty$, and put $q=N-R+1$.
Define
\[
 s_*(p)=
 \begin{cases}
  2q(1-1/p),&1\le p\le2,\\[1mm]
  d+q-2d/p,&2<p<\infty.
 \end{cases}
\]
If $s\ge s_*(p)$, then no window in
$L^2(\R^d)\cap M_s^p(\R^d)$ generates a Gabor frame on $\Lambda_D$.
\end{theorem}

\begin{proof}
It suffices to exclude the critical weight, since
$M_s^p\hookrightarrow M_{s_*}^p$ for $s\ge s_*$.  Put
\[
 q=N-R+1,\qquad m=2q,\qquad t=2d-m,\qquad r_p=\min\{p,2\}.
\]
The gap assumption gives $1\le q\le d$.  The unified critical weight is
\begin{equation}\label{eq:unified-necessity-weight}
 s_*=s_p(m,t)=m(1-1/p)+t(1/2-1/p)_+.
\end{equation}
For $p\le2$ this equals $2q(1-1/p)$.  For $p>2$, direct expansion gives
\[
 m(1-1/p)+t(1/2-1/p)
 =2q-\frac{2q}{p}+d-q-\frac{2d-2q}{p}
 =d+q-\frac{2d}{p}.
\]
These are exactly the two cases in the definition of $s_*(p)$.

\medskip\noindent\emph{Step 1: select coordinate pairs and fix the frame bounds.}
Suppose that $g\in L^2(\R^d)\cap M_{s_*}^p(\R^d)$ generates a frame.
Retain $q$ position--frequency pairs and write them as
$z=(u',\eta')\in\R^{2q}$; write the remaining coordinates as
$y\in\R^t$.  This is a permutation of coordinates, so it preserves
the isotropic weight and torus measure.  Let $A=\mathcal Z_Dg$.
The Zibulski-Zeevi criterion of \Cref{prop:zz-criterion-full} and the unitarity of the shifts described in
\Cref{sec:rational-zak-matrix-model} supply
constants $a>0$ and $B<\infty$ such that
\begin{equation}\label{eq:section10-frame-lower-bound}
 \sigma_{\min}(A(y,z))\ge a,\qquad
 \|A(y,z)\|_{\rm op}\le B
 \quad\text{for almost every }(y,z).
\end{equation}
Both constants are derived from the frame bounds, so they are
independent of slicing and scale.  In particular,
Fubini will preserve the same $a$ and $B$, not slice-dependent bounds.

\medskip\noindent\emph{Step 2: periodic scalar functions and exact reconstruction.}
Use the fixed Schwartz family $\psi_1,\ldots,\psi_L$ from
\Cref{lem:finite-Zak-family}.  For $s\in\mathcal I_{\mathsf B}$ and
$v\in\mathcal I_{\mathsf A}$ put
\[
 \tau_{s,v}=(-Dv,\mathsf B^{-1}s),\qquad
 f_{\ell,s,v}(w)=G_{\psi_\ell,g}(w+\tau_{s,v}),\qquad w=(u,\eta).
\]
These $LRN$-many functions are ordinary periodic.  Translation multiplies
Fourier coefficients by unimodular factors, and
\Cref{prop:cross-Zak-characterization} therefore gives
\[
 f_{\ell,s,v}\in\Fell{p}{s_*}(\T^{t}\times\T^m)
\]
for every index.  The finite family is fixed independently of the scale.
It is also uniformly bounded: the value $Zg(w+\tau_{s,v})$ is an
entry of $A(w)$, while $|Z\psi_\ell|$ is bounded.

Set $b_\ell=\Theta_\Psi^{-1}Z\psi_\ell$.  The reconstruction
\eqref{eq:cross-zak-reconstruction} gives the exact identity
\begin{equation}\label{eq:necessity-scalar-reconstruction}
 A_{s,v}(w)=\sum_{\ell=1}^L
                b_\ell(w+\tau_{s,v})f_{\ell,s,v}(w)
\end{equation}
almost everywhere.  The $f_{\ell,s,v}$ are the periodic functions
whose Fourier coefficients are estimated; the factors $b_\ell$ need
not be periodic.  They are globally smooth scalar quasiperiodic
functions, since $\Theta_\Psi$ is smooth and strictly positive.
Only their bounded derivatives on fixed compact sets will be used.
Thus no chart-dependent Fourier series of $b_\ell$ is introduced.

\medskip\noindent\emph{Step 3: small oscillation for the whole finite family.}
Apply \Cref{lem:mixed-oscillation} to each $f_{\ell,s,v}$ and define
\[
 \Omega_h(y)=\sum_{\ell,s,v}\omega_h^z f_{\ell,s,v}(y).
\]
The finite sum satisfies
\begin{equation}\label{eq:necessity-oscillation-small}
 \|\Omega_h\|_{L^{r_p}(\T^t)}\longrightarrow0
 \quad\text{as }h\searrow0.
\end{equation}
The supremum defining each $\omega_h$ ranges over every cube of side
$h$, including cubes crossing ordinary torus seams.  Thus it controls
exactly the cubes that will contain the convolution supports.

\medskip\noindent\emph{Step 4: a common full-measure set and scale-dependent slices.}
Fubini applied to \eqref{eq:section10-frame-lower-bound} and
\eqref{eq:necessity-scalar-reconstruction} gives a full-measure set
$Y\subset\T^t$ on which all these assertions hold for almost every
$z$.  Intersect also with the full-measure sets for the restricted Zak
identities.  There are finitely many entries and a countable group of
translations, so this intersection still has full measure.  Applying
Fubini on a countable collection of bounded $z$-boxes ensures the
identities hold on $\R^m$, not only on one cell.

If $t>0$ and $\|\Omega_h\|_{r_p}>0$, Chebyshev's inequality gives
\[
 \left|\{y:\Omega_h(y)>2\|\Omega_h\|_{r_p}\}\right|
 \le2^{-r_p}<1.
\]
Hence there is $y_h\in Y$ such that
\[
 \Omega_h(y_h)\le2\|\Omega_h\|_{r_p}.
\]
If the norm is zero, choose $y_h\in Y$ where $\Omega_h=0$, which
is possible because it vanishes almost everywhere.  If $t=0$, there
is just one slice; the same estimate holds on the one-point probability
space without any selection.  This also treats $q=d$ directly.
For $p>2$ we make no assertion that a single fixed $y$ works at all
scales or has VMO regularity.

\medskip\noindent\emph{Step 5: uniform estimates on the buffered selected-variable box.}
Keep $y$ in $[0,1]^t$ and $z$ in $(\mathcal Q_q^{\rm Zak})^*$.
Their product, after undoing the coordinate permutation, is one fixed
compact set in $\R^{2d}$.  Its finitely many translates by
$\tau_{s,v}$ are still contained in a fixed compact set.  All
$b_\ell$ and their first derivatives are bounded there by constants
independent of $h$ and $y_h$.  The same is true of the smooth row
phase correction from \Cref{prop:ordinary-zak-reduction}.

Apply that phase correction to $A(y_h,\cdot)$ and call the resulting matrix $S_h$.
Explicitly its $(s,v)$ entry is
\[
 (S_h)_{s,v}(u',\eta')
 =e^{-2\pi i u'\cdot(\mathsf B^{-1})'s'}
                  A_{s,v}(y_h,u',\eta').
\]
The ordinary Zak calculation in \Cref{prop:ordinary-zak-reduction}
shows that $S_h$ is quasiperiodic entrywise.  The multiplier
is unitary, so its lower and upper bounds are $a$ and $B$.

For clarity, combine the row phase and $b_\ell(w+\tau_{s,v})$ into
one smooth coefficient $c_{\ell,s,v}^{y_h}(z)$.  Then
\[
 (S_h)_{s,v}(z)=\sum_\ell
       c_{\ell,s,v}^{y_h}(z)f_{\ell,s,v}(y_h,z).
\]
Their sup norms and $z$-derivative bounds on the buffered box are
uniform in $y_h$.  On any cube $Q$ of side $h$ centered in the
fundamental selected-variable cell,
\eqref{eq:smooth-multiplier-oscillation} gives
\begin{align*}
 \operatorname{osc}_Q((S_h)_{s,v})
 &\le\sum_\ell\left(
       2\|c_{\ell,s,v}^{y_h}\|_\infty
                \operatorname{osc}_Q(f_{\ell,s,v}(y_h,\cdot))
       +Ch\|\nabla_zc_{\ell,s,v}^{y_h}\|_\infty B\right)\\
 &\le C\Omega_h(y_h)+ChB.
\end{align*}
The bound for $f_{\ell,s,v}$ used in the first line is a fixed
constant times $B$.  Sum over the finitely many entries to pass to
operator-norm oscillation.  We obtain
\begin{equation}\label{eq:necessity-matrix-oscillation}
 \Omega_h(S_h)\le C\Omega_h(y_h)+ChB
 \le2C\|\Omega_h\|_{r_p}+ChB\longrightarrow0.
\end{equation}
The notation on the left is the matrix oscillation of
\Cref{lem:direct-matrix-smoothing}; on the right it is the scalar
finite-family function defined in Step~3.  All domains and constants
have been fixed before choosing the slices.

\medskip\noindent\emph{Step 6: smoothing and the continuous contradiction.}
Apply \Cref{lem:direct-matrix-smoothing} to $S_h$ with radius
$\varepsilon=h$.  The support of each averaging kernel lies inside
the corresponding cube of side $h$, so
\eqref{eq:necessity-matrix-oscillation} applies without changing the
scale.  The smoothed matrix satisfies
\[
 \inf_z\sigma_{\min}((S_h)_h(z))
 \ge a-C\Omega_h(S_h)-ChB.
\]
For all sufficiently small $h$ the right side is positive, for example
at least $a/2$.  The matrix $(S_h)_h$ is smooth and quasiperiodic.  Its full rank on the fundamental cell extends everywhere
by quasiperiodicity.  Proposition~\ref{prop:continuous-rectangular-gap}, in
$q$ coordinate pairs, now requires
\[
 N\ge R+q=R+(N-R+1)=N+1,
\]
which is impossible.  This excludes the critical exponent in both
cases.  Weight monotonicity gives every larger weight, completing
the proof.
\end{proof}

\part{Scalar constructions of regular Gabor windows}\label{part:constructive-rank-geometry}
This part obtains the Parseval windows in
\Cref{thm:modulation-classification} by constructing their scalar Zak
transforms.  Sections~\ref{sec:construction-strategy}
and~\ref{sec:terminal-construction} construct $g_{\rm flat}$, and
Section~\ref{sec:elementary-perturbation} then shows membership in the
postulated modulation spaces.

Section~\ref{sec:lower-polar} constructs a periodic phase factor whose
product with the Zak transform of $g_{\rm flat}$ defines the Zak transform
of $g_{\rm phase}$.  Section~\ref{sec:lower-completion} proves the
associated Fourier cancellation estimate and shows the improved high-$p$
regularity of $g_{\rm phase}$, including the infinity endpoint.

\section{Smooth normalized Zak functions and the symplectic index gap}
\label{sec:construction-strategy}\label{sec:rank-geometry}

Throughout Part~III use the diagonal model of \Cref{sec:zak-method},
assume $1\le R\le N$, and put
\[
 m_0=2(N-R+1),\qquad n=2d.
\]
We construct a finite-dimensional space of scalar quasiperiodic
functions containing one whose rational Zak matrix has full row rank
on the cubical $(m_0-1)$-skeleton when $m_0\le n$.
Normalization gives a smooth function near this skeleton; when
$N-R\ge d$, the same argument gives a globally smooth normalized
Zak function and hence a Schwartz window.

\subsection{Smooth scalar Zak seeds without common zeros}
\label{subsec:explicit-scalar-seeds}

We construct $d+1$ smooth scalar Zak functions with no simultaneous
zero, using one smooth step function that will also be used later.

Throughout the construction we use the flat function
\[
 \vartheta(t)=
 \begin{cases}
  e^{-1/t},&t>0,\\
  0,&t\le0,
 \end{cases}
\]
and the associated symmetric smooth step
\begin{equation}\label{eq:fixed-smooth-step}
 \varphi(t)=\frac{\vartheta(4t-1)}
                {\vartheta(4t-1)+\vartheta(3-4t)}.
\end{equation}
Figure~\ref{fig:smooth-step} displays $\varphi$.
\begin{figure}[H]
\centering
\includegraphics[width=.40\linewidth]{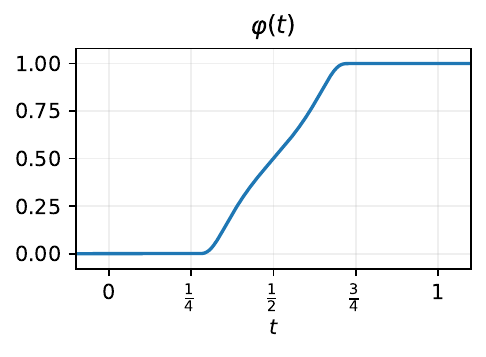}
\caption{The smooth nondecreasing step $\varphi$ from
\eqref{eq:fixed-smooth-step}.  It equals $0$ on $(-\infty,1/4]$ and $1$
on $[3/4,\infty)$, is strictly increasing between these intervals, and
satisfies $\varphi(1/2)=1/2$.}
\label{fig:smooth-step}
\end{figure}

The denominator defining $\varphi$ is positive on $\R$.  Thus $\varphi(t)=0$ for $t\le1/4$, $\varphi(t)=1$ for $t\ge3/4$, and
$\varphi(1-t)=1-\varphi(t)$.  For $1/4<t<3/4$,
\[
 \frac{\varphi(t)}{1-\varphi(t)}
 =\exp\left(\frac1{3-4t}-\frac1{4t-1}\right).
\]
The exponent has derivative
$4(3-4t)^{-2}+4(4t-1)^{-2}>0$ and is zero at $t=1/2$.
Hence $\varphi(t)=1/2$ exactly at $t=1/2$.

Lemma~\ref{lem:explicit-scalar-seeds}, including an explicit construction, is proved in
\cite[Theorem~1.3 and Section~3.2]{deDiosLiehrTaylor2026}.
The formulas below specify the seeds in our interpolation family and
the function $h$ used in \Cref{fig:scalar-flat-illustration}.

\begin{lemma}\label{lem:explicit-scalar-seeds}
For every integer $d\ge1$ there are $d+1$ smooth scalar quasiperiodic
functions $s_0,\ldots,s_d$ on $\R^d\times\R^d$ and constants
$0<c_\Theta\le C_\Theta<\infty$ such that
\[
 c_\Theta\le\sum_{j=0}^d|s_j(u,\eta)|^2\le C_\Theta,
 \qquad (u,\eta)\in\R^d\times\R^d.
\]
\end{lemma}
\begin{proof}
For $u=k+t$, $k\in\Z$, $0\le t<1$, define
\[
 h(u,v)=e^{2\pi i kv}\bigl(1-\varphi(t)+\varphi(t)e^{2\pi iv}\bigr).
\]
For $n\in\Z$,
\[
 \lim_{u\to n^-}h(u,v)
 =e^{2\pi i(n-1)v}e^{2\pi iv}
 =e^{2\pi inv}
 =h(n,v)=\lim_{u\to n^+}h(u,v).
\]
Since $\varphi=0$ near $0$ and $\varphi=1$ near $1$,
\[
 h(u,v)=e^{2\pi inv}\quad\text{if }|u-n|<\tfrac14.
\]
On this neighborhood, for integers $a\ge1$ and $b\ge0$,
\[
 \partial_u^a\partial_v^b h=0,
 \qquad \partial_v^b h=(2\pi in)^b e^{2\pi inv}.
\]
Thus $h$ is smooth.  For $\beta\in\R$, define
$h_\beta(u,v)=e^{2\pi i\beta u}h(u,v-\beta)$.  For $n,\ell\in\Z$,
direct substitution gives
\par\noindent
\mbox{$h_\beta(u+n,v+\ell)
=e^{2\pi i\beta(u+n)}e^{2\pi in(v-\beta)}h(u,v-\beta)
=e^{2\pi inv}h_\beta(u,v)$.}
\par\noindent Moreover,
\begin{align*}
 h_\beta(u,v)=0
 &\ \Longleftrightarrow\
 \varphi(t)=\tfrac12,\quad e^{2\pi i(v-\beta)}=-1\\
 &\ \Longleftrightarrow\
 (u,v)\in(\tfrac12,\tfrac12+\beta)+\Z^2.
\end{align*}
Choose $\beta_j=j/(d+1)$, $0\le j\le d$, and define the
quasiperiodic functions
\begin{equation}\label{eq:explicit-scalar-seeds}
 s_j(u,\eta)=\prod_{i=1}^d h_{\beta_j}(u_i,\eta_i),\qquad
 \Theta(u,\eta)=\sum_{j=0}^d|s_j(u,\eta)|^2.
\end{equation}
The $\beta_j$ are distinct modulo one, so for every $(u,\eta)\in\R^{2d}$,
\[
 \#\{j:s_j(u,\eta)=0\}
 \le\sum_{i=1}^d\#\{j:h_{\beta_j}(u_i,\eta_i)=0\}\le d.
\]
Hence $\Theta>0$.  Its continuity and integer periodicity give
\[
 0<c_\Theta:=\min_{[0,1]^{2d}}\Theta
 \le\Theta(u,\eta)\le
 \max_{[0,1]^{2d}}\Theta=:C_\Theta<\infty.
\]
\end{proof}

\subsection{Full row rank in a finite-dimensional Zak family}
\label{subsec:scalar-independence}

A \emph{coordinate rectangle} in $\R^{2d}$ is a product of compact
intervals; its dimension is the number of nondegenerate factors.

\begin{proposition}\label{prop:scalar-skeleton-avoidance}
Let $s_0,\ldots,s_d$ be the functions in
\Cref{lem:explicit-scalar-seeds}, and write $x=(u,\eta)\in\R^{2d}$.
For a vector $c=(c_{j,p,\ell})\in\C^{(d+1)NR}$, with
$0\le j\le d$, $p\in\mathcal I_{\mathsf A}$, and
$\ell\in\mathcal I_{\mathsf B}$, define
\begin{equation}\label{eq:finite-scalar-family}
 F_c(x)=\sum_{j=0}^d
       \sum_{p\in\mathcal I_{\mathsf A}}
       \sum_{\ell\in\mathcal I_{\mathsf B}}
 c_{j,p,\ell}s_j(x)e^{2\pi i\langle(p,\ell),x\rangle}.
\end{equation}
Let $K\subset\R^{2d}$ be a nonempty compact coordinate rectangle of
 dimension $k<m_0=2(N-R+1)$.
Then $F_c$ is smooth and scalar quasiperiodic for every $c$, and for
almost every $c\in\C^{(d+1)NR}$, the rational Zak matrix
$\mathcal M_{F_c}(x)$ from \eqref{eq:fine-zak-matrix} has full row rank
for every $x\in K$.
\end{proposition}
\begin{proof}
For $x\in\R^{2d}$ and $z\in\C^R$, expanding $F_c$ gives
$\mathcal M_{F_c}(x)^Tz=B(x,z)c$, where
$B(x,z)\in\C^{N\times(d+1)NR}$ is defined by
\begin{equation}\label{eq:scalar-coefficient-matrix}
 B(x,z)_{t,(j,p,\ell)}
 =\sum_{r\in\mathcal I_{\mathsf B}}
 z_r s_j\bigl(x+(-Dt,\mathsf B^{-1}r)\bigr)
 e^{2\pi i\langle(p,\ell),x+(-Dt,\mathsf B^{-1}r)\rangle}.
\end{equation}
Expanding the matrix product gives
\[
\begin{aligned}
 B(x,z)B(x,z)^*
 &=NR\,\diag_{t\in\mathcal I_{\mathsf A}}
       \left(\sum_{r\in\mathcal I_{\mathsf B}}
             |z_r|^2\Theta\bigl(x+(-Dt,\mathsf B^{-1}r)\bigr)\right)\\
 &\ge NRc_\Theta\|z\|_2^2 I_N,
\end{aligned}
\]
where we used coordinatewise coprimality and finite Fourier orthogonality,
\begin{equation}\label{eq:finite-scalar-character-orthogonality}
 \sum_{p\in\mathcal I_{\mathsf A}}
 \sum_{\ell\in\mathcal I_{\mathsf B}}
 e^{2\pi i[-p\cdot D(t-t')+\ell\cdot\mathsf B^{-1}(r-r')]}
 =NR\,\delta_{r,r'}\delta_{t,t'},
\end{equation}
together with $\Theta=\sum_{j=0}^d|s_j|^2\ge c_\Theta>0$ from
\Cref{lem:explicit-scalar-seeds}.
The matrix $B(x,z)B(x,z)^*$ is therefore invertible for $z\ne0$.
Hence $B(x,z)$ has rank $N$, so for every $x\in K$ and $z\ne0$,
at least one $N\times N$ column block is invertible.

For every $N$-element subset
$I\subset\{0,\ldots,d\}\times\mathcal I_{\mathsf A}\times\mathcal I_{\mathsf B}$,
let $B_I$ denote the corresponding column block and $I^c$ the complement
in the coefficient index set.  For $s\in\mathcal I_{\mathsf B}$ and an
integer $L\ge1$, define
\[
 A(s,I,L)=\left\{c\in\C^{(d+1)NR}:
 \begin{array}{l}
  \exists\,x\in K,\ z\in\C^R:\quad z_s=1,\quad\|z\|_\infty=1,\\[1mm]
  B(x,z)c=0,\quad |\det B_I(x,z)|\ge L^{-1},\\[1mm]
  \|c_{I^c}\|_\infty\le L
 \end{array}\right\}.
\]
The set of coefficient vectors giving rank loss on $K$ is contained in
\[
 \bigcup_{s\in\mathcal I_{\mathsf B}}
 \ \bigcup_{|I|=N}\ \bigcup_{L=1}^\infty A(s,I,L).
\]
Indeed, if $B(x,z)c=0$ for some $x\in K$ and $z\ne0$, divide $z$ by a
component of maximal modulus, choose an invertible column block of
$B(x,z)$, and take $L$ large enough.  Since the union is countable,
it suffices to show that every fixed $A(s,I,L)$ has measure zero.

Fix $s,I,L$.  We restrict $(x,z)$ to the domain specified in the definition
of $A(s,I,L)$, in particular to $|\det B_I(x,z)|\ge L^{-1}$; on this domain
$B_I(x,z)$ is invertible.  If no such pair exists, $A(s,I,L)$ is empty.
Ordering the columns indexed by $I$ and the corresponding
coefficients first, split
\[
 B(x,z)=\bigl(B_I(x,z)\ \ B_{I^c}(x,z)\bigr),\qquad
 c=\begin{pmatrix}c_I\\c_{I^c}\end{pmatrix}.
\]
For the parameters occurring in the definition of $A(s,I,L)$,
$B(x,z)c=0$ is equivalent to
\begin{equation}\label{eq:scalar-incidence-solved}
 c_I=-B_I(x,z)^{-1}B_{I^c}(x,z)c_{I^c}.
\end{equation}
All components of $z$ other than $z_s=1$ range over the closed unit disc,
and $\|c_{I^c}\|_\infty\le L$.  Hence the free data
\[
 x\in K,\qquad (z_r)_{r\ne s},\qquad
 c_{I^c}\in\C^{(d+1)NR-N}
\]
range over a bounded set.  Identifying the affine span of $K$ with $\R^k$
and each complex coordinate with two real coordinates, the total number
of free parameters is bounded above by
\begin{equation}\label{eq:scalar-incidence-dimension}
\begin{aligned}
 q&=k+2(R-1)+2\bigl((d+1)NR-N\bigr)\\
  &=2(d+1)NR+k-2(N-R+1)<2(d+1)NR.
\end{aligned}
\end{equation}

The reconstruction of $c$ from $c_{I^c}$ is Lipschitz in all the free
parameters.  Indeed, the entries of $B$
and their minors are bounded and Lipschitz for $x\in K$ and
$\|z\|_\infty\le1$.  The formula
$B_I^{-1}=\operatorname{adj}(B_I)/\det B_I$ and the estimate
\[
 |a^{-1}-b^{-1}|\le L^2|a-b|,\qquad |a|,|b|\ge L^{-1},
\]
show that $B_I^{-1}$ is bounded and Lipschitz on the restricted parameter
domain.  Together with $\|c_{I^c}\|_\infty\le L$, this proves the claim for
\eqref{eq:scalar-incidence-solved}, including the unchanged coordinates
$c_{I^c}$ in the original coefficient order.

Consequently, $A(s,I,L)$ is a Lipschitz image of a bounded subset of
$\R^q$, and hence a null set in $\C^{(d+1)NR}$.  Indeed,
$O(\varepsilon^{-q})$ cubes of side $\varepsilon$ cover the parameter
domain, and their nonempty pieces have images contained in balls of radius
$O(\varepsilon)$.  The total covering volume is
$O(\varepsilon^{2(d+1)NR-q})$, which tends to zero by
\eqref{eq:scalar-incidence-dimension}.

Therefore, for almost every $c$, the matrix $\mathcal M_{F_c}$ has full
row rank throughout $K$.
Smoothness and scalar quasiperiodicity of each $F_c$ follow from those of
the seeds and the smooth periodic characters in
\eqref{eq:finite-scalar-family}.
\end{proof}

\subsection{The matrix Parseval identity and local normalization}
\label{subsec:scalar-tightening}

We express Parsevality through the rational Zak matrix and normalize a
smooth scalar Zak function on any invariant open set where this matrix
has full row rank.

Let $F$ be measurable and scalar quasiperiodic, with
$F|_{[0,1]^{2d}}\in L^2$.  Put $x=(u,\eta)$ and define pointwise
\begin{equation}\label{eq:scalar-correlations}
 G_F(x)=\frac1R\mathcal M_F(x)\mathcal M_F(x)^*.
\end{equation}
By \eqref{eq:zz-parseval-identity}, $Z^{-1}F$ generates a Parseval
frame on $\Lambda_D$ exactly when
\begin{equation}\label{eq:scalar-parseval-identities}
 G_F(x)=I_R\qquad\text{almost everywhere};
\end{equation}
compare also \cite[Section~2.3]{JanssenStrohmer2002}.
In this case,
\[
 |F(x)|^2\le
 \bigl(\mathcal M_F(x)\mathcal M_F(x)^*\bigr)_{0,0}=R
 \qquad\text{almost everywhere}.
\]
Moreover, matrix Zak unitarity \eqref{eq:vector-matrix-zak-unitarity}
and $|Q_D|=1/(NR)$ give
\begin{equation}\label{eq:scalar-parseval-norm}
\begin{aligned}
 \|Z^{-1}F\|_2^2
 &=\int_{Q_D}\tr\bigl(\mathcal M_F(x)\mathcal M_F(x)^*\bigr)\,dx\\
 &=R^2|Q_D|=\frac RN.
\end{aligned}
\end{equation}

The same definition of $G_F$ applies when $F$ is given on a
$\Gamma_D$-invariant open set.  The matrix $G_F(x)$ is positive definite
exactly where $\mathcal M_F(x)$ has full row rank.

The next lemma records the standard pointwise canonical tight-window
normalization on an invariant open set; see
\cite[equation~(1.6) and pp.~260--261]{BolcskeiJanssen2000} and
\cite[Sections~1.1 and~2.3]{JanssenStrohmer2002}.  The proof checks that,
in our setup, the normalized matrix is the rational Zak
matrix of its $(0,0)$ entry.

\begin{lemma}\label{lem:scalar-tightening}
Let $V\subset\R^d\times\R^d$ be $\Gamma_D$-invariant and open, and let
$F$ be smooth and scalar quasiperiodic on $V$.  Suppose that
\begin{equation}\label{eq:scalar-full-rank-domain}
 \rank\mathcal M_F(x)=R,\qquad x\in V.
\end{equation}
Define
\begin{equation}\label{eq:scalar-canonical-normalization}
 (\mathcal N_DF)(x)
 =\bigl[G_F(x)^{-1/2}\mathcal M_F(x)\bigr]_{0,0}.
\end{equation}
Then $H=\mathcal N_DF$ is smooth and scalar quasiperiodic on $V$,
its derivatives are bounded on compact subsets of $V$, and
\begin{equation}\label{eq:scalar-tightened-correlations}
 G_H(x)=I_R,\qquad x\in V.
\end{equation}
\end{lemma}
\begin{proof}
By the rank hypothesis, $G_F(x)$ is positive definite for every
$x\in V$.  Put $B_F=G_F^{-1/2}\mathcal M_F$.  The inverse square root
is a smooth function on the open set of positive definite matrices, so
$B_F$ is smooth on $V$.
Under each translation in \eqref{eq:sample-translation-identities},
write $\mathcal M_F'=P\mathcal M_F Q$, where $P,Q$ are unitary.
Unitary covariance of the inverse square root gives
\[
 G_F'=P G_F P^*,\qquad
 B_F'=(P G_F P^*)^{-1/2}P\mathcal M_F Q=P B_F Q.
\]
Thus $B_F$ satisfies the same translation identities as
$\mathcal M_F$, and its $(0,0)$ entry $H$ is scalar quasiperiodic.
For $t\in\mathcal I_{\mathsf A}$ and $s\in\mathcal I_{\mathsf B}$,
the formulas \eqref{eq:sample-translation-identities} for the translations
$u\mapsto u-Dt$ and $\eta\mapsto\eta+\mathsf B^{-1}s$ give
\begin{align*}
 H(u-Dt,\eta+\mathsf B^{-1}s)
 &=(B_F(u-Dt,\eta+\mathsf B^{-1}s))_{0,0}\\
 &=(B_F(u,\eta+\mathsf B^{-1}s))_{0,t}
  =(B_F(u,\eta))_{s,t}.
\end{align*}
In the second equality, the column index changes from $0$ to $t$;
there is no phase factor because $t$ already lies in $\mathcal I_{\mathsf A}$.
In the third equality, the row index changes from $0$ to $s$.
All arguments stay in $V$ by its $\Gamma_D$-invariance.  Hence
$\mathcal M_H=B_F=G_F^{-1/2}\mathcal M_F$, and
\[
 G_H=\frac1R\mathcal M_H\mathcal M_H^*
 =G_F^{-1/2}\left(\frac1R\mathcal M_F\mathcal M_F^*\right)G_F^{-1/2}
 =I_R.
\]
The compact derivative bounds follow from smoothness.
\end{proof}

\subsection{\texorpdfstring{Global normalization when $N-R\ge d$}{Global normalization when N-R >= d}}
\label{subsec:global-smooth-normalization}

When the dimension threshold exceeds $2d$, the rational Zak matrix can be chosen full row rank
on an entire fundamental rectangle.  Global normalization then gives
a smooth scalar Zak transform and hence a Schwartz window.

\begin{corollary}\label{cor:Schwartz-construction}
Assume $N-R\ge d$.  There is $g_{\mathcal S}\in\mathcal S(\R^d)$ such
that $\mathcal G(g_{\mathcal S},\Lambda_D)$ is Parseval.  In particular,
$g_{\mathcal S}\in M_s^p(\R^d)$ for every $1\le p\le\infty$ and
$s\in\R$.
\end{corollary}
\begin{proof}
The hypothesis gives
\[
 m_0=2(N-R+1)\ge2d+2>2d.
\]
Thus the closed fine period rectangle $\mathcal Q_D$, considered as a
$2d$-dimensional coordinate rectangle, lies below the dimension threshold
in \Cref{prop:scalar-skeleton-avoidance}.  Choose a good coefficient
vector $c$ in the explicit scalar family.  Then $\mathcal M_{F_c}$
has full row rank on $\mathcal Q_D$.
By \eqref{eq:sample-translation-identities}, translations by $\Gamma_D$
change $\mathcal M_{F_c}$ only by unitary row and column factors and
therefore preserve rank.  Since the $\Gamma_D$-translates of
$\mathcal Q_D$ cover $\R^{2d}$, the matrix has full row rank everywhere.
Apply \Cref{lem:scalar-tightening} with
$V=\R^{2d}$ and define
\begin{equation}\label{eq:smooth-normalized-section}
 F_{\mathcal S}=\mathcal N_DF_c,\qquad g_{\mathcal S}=Z^{-1}F_{\mathcal S}.
\end{equation}
The scalar function $F_{\mathcal S}$ is globally smooth, quasiperiodic, and satisfies
\begin{equation}\label{eq:smooth-scalar-correlations}
 G_{F_{\mathcal S}}(x)=I_R,\qquad x\in\R^{2d}.
\end{equation}
It is bounded by $\sqrt R$, so the inverse unitary Zak transform is
defined.  Equation~\eqref{eq:scalar-parseval-identities} gives Parsevality.

The smooth Zak--Schwartz correspondence recorded in
\Cref{sec:rational-zak-matrix-model} gives
$g_{\mathcal S}\in\mathcal S(\R^d)$.  In particular,
\[
 g_{\mathcal S}\in M_s^p(\R^d),\qquad
 1\le p\le\infty,\quad s\in\R.\qedhere
\]
\end{proof}

\subsection{Normalization near the cubical skeleton}
\label{subsec:cubical-vocabulary}

For the remaining gaps, we use the lattice cubulation and normalize
near its $(m_0-1)$-skeleton.  This gives the smooth input $H$ for the
cubical pullback in \Cref{sec:terminal-construction}.

Set
\begin{equation}\label{eq:explicit-physical-cubulation}
 L=\diag(a_1^{-1},\ldots,a_d^{-1},
          b_1^{-1},\ldots,b_d^{-1}),
 \qquad \Gamma_D=L\Z^n.
\end{equation}
We use the centered rectangular cubulation whose physical cells are
\[
 L\bigl(k+[-1/2,1/2]^n\bigr),\qquad k\in\Z^n.
\]
The cubulation is periodic under $\Gamma_D$.

A $k$-dimensional face of the normalized centered cube is obtained by
fixing $n-k$ coordinates at their endpoints $\pm1/2$ and allowing
the other $k$ coordinates to vary in $[-1/2,1/2]$.  Its physical image
under $L$, followed by a lattice translation, is a rectangular face of
the same dimension.  The cubical $j$-skeleton is the union of the faces
of dimensions at most $j$.  Equivalently, a point in a normalized cube
belongs to this skeleton when at least $n-j$ of its coordinates lie
on cube faces.  Adjacent closed cells share their common faces, so
this definition does not depend on which adjacent cell is used.

Write $K_{\rm cub}^{(j)}$ for the normalized skeleton and
$L K_{\rm cub}^{(j)}$ for its physical image.  Both are closed, and
the physical skeleton is $\Gamma_D$-invariant.

Assume now $0\le N-R\le d-1$ and put $m=m_0\le n$.
For $\eps>0$, let $B_\eps(0)$ denote the $n$-dimensional open Euclidean
ball of radius $\eps$ centered at zero.

\begin{lemma}\label{tm:lem:skeleton}
Use the physical cubulation \eqref{eq:explicit-physical-cubulation},
and let $K_{\rm phys}^{(m-1)}=L K_{\rm cub}^{(m-1)}$ be its cubical
$(m-1)$-skeleton in $\R^n$.  For some $\eps>0$ there is a smooth
scalar quasiperiodic function $H$ on the $\Gamma_D$-invariant open
neighborhood $V=K_{\rm phys}^{(m-1)}+B_\eps(0)$ such that
\begin{equation}\label{tm:eq:skeleton-parseval}
 G_H(x)=I_R,\qquad x\in V.
\end{equation}
All derivatives of $H$ are bounded on the compact sets of representatives
of the skeleton needed below.
\end{lemma}
\begin{proof}
Let $\mathcal K$ be the union of the faces of dimensions at most $m-1$
of $L[-1/2,1/2]^n$.  Since every physical cell is a $\Gamma_D$-translate
of this one, $\mathcal K+\Gamma_D=K_{\rm phys}^{(m-1)}$.  Each face has
dimension $k\le m-1<m_0$.

For each of the finitely many faces $\sigma$ of $\mathcal K$,
Proposition~\ref{prop:scalar-skeleton-avoidance} gives an exceptional
set $E_\sigma\subset\C^{(d+1)NR}$ of measure zero.  Choose
\[
 c\in\C^{(d+1)NR}\setminus\bigcup_{\sigma\subset\mathcal K}E_\sigma,
\]
where the union runs over these faces.  Then $\mathcal M_{F_c}$ has
full row rank on every face and hence on $\mathcal K$.  Put
\[
 \Omega=\{x\in\R^n:\rank\mathcal M_{F_c}(x)=R\}.
\]
The set $\Omega$ is open: at each of its points, one maximal minor is
nonzero and remains nonzero on a neighborhood.  It is also
$\Gamma_D$-invariant, since the unitary row and column factors in
\eqref{eq:sample-translation-identities} preserve rank.

The compact set $\mathcal K$ lies in $\Omega$, so choose $\eps>0$ with
$\mathcal K+\overline{B_\eps(0)}\subset\Omega$.
Set $V=K_{\rm phys}^{(m-1)}+B_\eps(0)
=(\mathcal K+B_\eps(0))+\Gamma_D$.
This is a $\Gamma_D$-invariant open neighborhood of the skeleton, and
$V\subset\Omega$.  Thus $\mathcal M_{F_c}$ has full row rank on $V$.
Define $H=\mathcal N_DF_c$ on $V$ by the local canonical normalization \eqref{eq:scalar-canonical-normalization}.
Lemma~\ref{lem:scalar-tightening} gives smoothness, scalar
quasiperiodicity, and \eqref{tm:eq:skeleton-parseval}.
Each compact set of representatives of the skeleton has a compact closed
$\eps/2$-neighborhood in $V$, on which every derivative of $H$ is bounded.
\end{proof}

\section{The cubical pullback and the flat Zak function}
\label{sec:terminal-construction}
In \Cref{sec:construction-strategy}, we constructed a smooth quasiperiodic
function and obtained the required lower bound for its rational Zak matrix
near a cubical skeleton, but not on the full Zak transform domain.  We now
sacrifice smoothness on a designed null set to obtain a quasiperiodic
function whose rational Zak matrix has its smallest singular value
uniformly bounded away from zero almost everywhere.  The geometry of the exceptional set will
control the regularity of the window obtained by applying the inverse Zak
transform to the constructed function.

Throughout \Cref{sec:terminal-construction,sec:elementary-perturbation,sec:lower-polar,sec:lower-completion},
assume $0\le N-R\le d-1$ and put
\begin{equation}\label{eq:terminal-gap-assumption}
 n=2d,\qquad q=N-R+1,\qquad m=m_0=2q\le n.
\end{equation}
The function $H$ and its invariant domain $V$ are supplied by
\Cref{tm:lem:skeleton}.

\subsection{The exceptional set, cubical map, and Parsevality}
\label{subsec:cubical-pullback}
Rescale the physical coordinates so that the small rectangular
cells become unit cubes centered at $\Z^n$.  We construct the map in these
normalized coordinates, writing $K_{\rm cub}^{(m-1)}$ for the cubical
$(m-1)$-skeleton, and then undo the rescaling.

Recall from \Cref{subsec:cubical-vocabulary} that in the cube
$k+[-1/2,1/2]^n$, $k\in\Z^n$, a point belongs to the cubical
$(m-1)$-skeleton exactly when at least $n-m+1$ of its coordinates equal
$k_i-1/2$ or $k_i+1/2$, the two boundary values of the corresponding
coordinate interval.  At most $m-1$ coordinates may lie strictly between
these values.  In two dimensions this gives the vertices for $m=1$ and
the edges for $m=2$.

For $x\in\R^n$ and $i=1,\ldots,n$, define
\[
 d_i(x)=\dist(x_i,\Z)\in[0,1/2],\qquad
 f_i(x)=\sin^2(\pi x_i)\in[0,1].
\]
As exceptional set, we choose
\begin{equation}\label{tm:eq:S-flat}
 S=\bigcup_{\substack{I\subset\{1,\ldots,n\}\\|I|=m}}
       \{x:x_i\in\Z\text{ for every }i\in I\}.
\end{equation}
The points of $S$ are characterized by having at least $m$ integer
coordinates.  The set $S$ is a locally finite union of affine sets of
codimension $m$; in particular, it has measure zero.  Although the Gabor
application has $n=2d$ and even $m$, the definition of $S$ makes sense for
every pair of integers $1\le m\le n$.
Figure~\ref{tm:fig:S-three-dimensional} illustrates the exceptional set $S$
for $n=3$.  In one centered cube the only possible integer coordinate is
zero: for $m=1$ at least one coordinate vanishes, for $m=2$ at least two
vanish, and for $m=3$ all three vanish.
\input{figures/skeleton-figure.tex}

Equivalently, $S$ is the $(n-m)$-skeleton of the shifted cubulation with
cells $k+[0,1]^n$.  In contrast, the protected $(m-1)$-skeleton of the
centered cubulation is characterized by having at least $n-m+1$
half-integer coordinates.  Approaching $S$ forces at least $m$ coordinate
distances to integers to become small.  These same distances will control
the derivative bounds and the weighted coordinate integral used in
\Cref{sec:elementary-perturbation,sec:lower-completion}.

For each $x\in\R^n$, write
$d_{(1)}(x)\le\cdots\le d_{(n)}(x)$ for the increasing rearrangement of
$d_1(x),\ldots,d_n(x)$; thus $d_{(m)}(x)$ is their $m$th smallest value.

For any $x\in\R^n$, a nearest point of $S$ is obtained by moving the $m$
coordinates closest to integers to nearest integers and leaving the
remaining coordinates unchanged.  Every point of $S$ has at least $m$
integer coordinates, and choosing the $m$ smallest movements minimizes
the squared Euclidean distance.  Hence
\begin{equation}\label{eq:distance-to-exceptional-set}
 \dist(x,S)^2=\sum_{j=1}^m d_{(j)}(x)^2.
\end{equation}

For $I$ of cardinality $m$ put $f_I=\sum_{i\in I}f_i$.  On $\R^n\setminus S$
all these functions are positive.  Define the smooth periodic scale
\begin{equation}\label{tm:eq:rho}
 \rho(x)=\frac1{4\pi\sqrt m}
 \left(\sum_{|I|=m}\frac1{f_I(x)}\right)^{-1/2}.
\end{equation}
Figure~\ref{tm:fig:rho-two-dimensional} displays $\rho$ for $n=2$ and
$m=1,2$, with its continuous extension by zero on $S$.

\begin{lemma}\label{tm:lem:rho-comparable}
For every $x\in\R^n\setminus S$,
\begin{equation}\label{tm:eq:rho-comparability}
 \frac{d_{(m)}(x)}{2\pi\sqrt{m\binom nm}}
 \le\rho(x)\le\tfrac14d_{(m)}(x),
 \qquad \rho(x)\asymp\dist(x,S).
\end{equation}
Moreover,
\begin{equation}\label{eq:rho-uniform-upper-bound}
 \rho(x)\le\frac1{4\pi\sqrt{\binom nm}}\le\frac1{4\pi}<\frac18,
 \qquad x\notin S.
\end{equation}
\end{lemma}
\begin{proof}
Fix an arbitrary $x\in\R^n\setminus S$.  For every $a\in[0,1/2]$ we have
$2a\le\sin(\pi a)\le\pi a$.  Applying this with $a=d_i(x)$ and using
$f_i(x)=\sin^2(\pi d_i(x))$, for every $i=1,\ldots,n$, gives
\[
 4d_i(x)^2\le f_i(x)\le\pi^2d_i(x)^2.
\]
Let $I_0$ index the $m$ smallest coordinate distances, breaking ties
arbitrarily.  Since $\sin^2(\pi a)$ increases on $[0,1/2]$, these indices
also minimize the sum of $m$ values of $f_i$.  The distance identity
\eqref{eq:distance-to-exceptional-set} therefore gives
\[
 4\dist(x,S)^2
 \le \sum_{i\in I_0}f_i(x)=\min_{|I|=m}f_I(x)
 \le\pi^2\dist(x,S)^2.
\]
Indeed, both inequalities follow by summing the corresponding bounds for
$f_i(x)$ over $I_0$, and
$\sum_{i\in I_0}d_i(x)^2=\dist(x,S)^2$.

Since each $f_I(x)>0$,
\[
 \frac1{\min_{|I|=m}f_I(x)}
 \le\sum_{|I|=m}\frac1{f_I(x)}
 \le\frac{\binom nm}{\min_{|I|=m}f_I(x)}.
\]
Combining these estimates with the definition \eqref{tm:eq:rho} gives
\[
\begin{aligned}
 \rho(x)
 &=\frac1{4\pi\sqrt m}
       \left(\sum_{|I|=m}\frac1{f_I(x)}\right)^{-1/2}\\
 &\ge\frac{\sqrt{\min_{|I|=m}f_I(x)}}{4\pi\sqrt{m\binom nm}}
 \ge\frac{\dist(x,S)}{2\pi\sqrt{m\binom nm}},\\
 \rho(x)
 &\le\frac{\sqrt{\min_{|I|=m}f_I(x)}}{4\pi\sqrt m}
 \le\frac{\dist(x,S)}{4\sqrt m}
 \le\frac14d_{(m)}(x).
\end{aligned}
\]
The last inequality uses
$\dist(x,S)^2=\sum_{j=1}^m d_{(j)}(x)^2\le m\,d_{(m)}(x)^2$.
Since $\dist(x,S)\ge d_{(m)}(x)$, these estimates prove both the stated
bounds in terms of $d_{(m)}$ and $\rho(x)\asymp\dist(x,S)$.
For every $I$ of cardinality $m$, $f_I(x)\le m$.  Hence
\[
 \sum_{|I|=m}f_I(x)^{-1}\ge\frac{\binom nm}{m},
 \qquad
 \rho(x)\le\frac1{4\pi\sqrt m}
              \left(\frac{\binom nm}{m}\right)^{-1/2}
 =\frac1{4\pi\sqrt{\binom nm}}.
\]
This proves \eqref{eq:rho-uniform-upper-bound}.
\end{proof}

\begin{figure}[tbp]
\centering
\includegraphics[width=.848\linewidth]{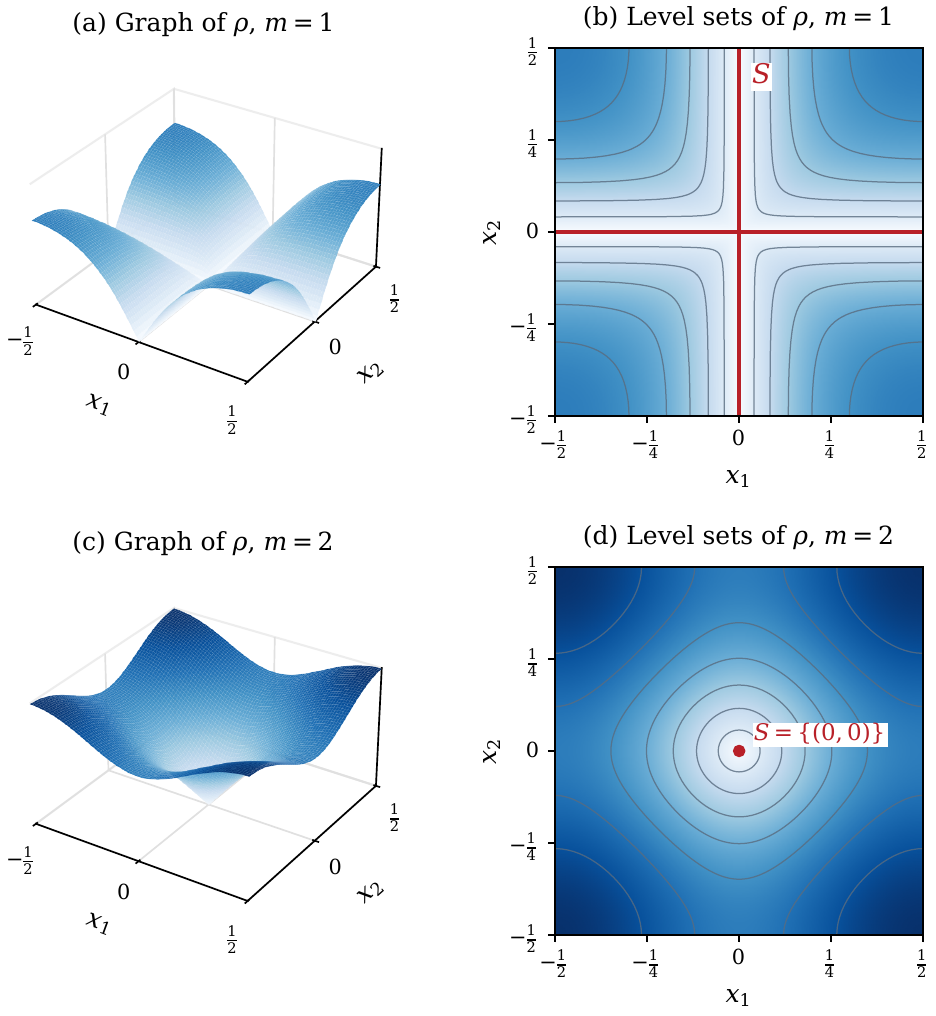}
\caption{The scale $\rho$ for $n=2$, extended by zero on $S$:
$m=1$ in the top row and $m=2$ in the bottom row. Graphs appear on the
left, level sets on the right, with the same height and color scales.
Red marks indicate the exceptional set: integer coordinate lines for
$m=1$, and integer lattice points for $m=2$, of which only the origin
is shown. In both cases $\rho$ is positive and smooth off $S$ and
comparable to the distance to $S$.  The maximum is $1/(4\pi\sqrt2)$
for $m=1$ and $1/(4\pi)$ for $m=2$, attained at the four corners of
the displayed square.}
\label{tm:fig:rho-two-dimensional}
\end{figure}

Recall from \eqref{eq:fixed-smooth-step} that $\varphi$ is smooth and
nondecreasing, strictly increasing on $(1/4,3/4)$, with
$\varphi(t)=0$ for $t\le1/4$, $\varphi(t)=1$ for $t\ge3/4$, and
$\varphi(1/2)=1/2$.  Its graph is shown in
Figure~\ref{fig:smooth-step}.

For $x\in\R^n\setminus S$ and $x_i\in[k_i-1/2,k_i+1/2]$, $k_i\in\Z$, define
\begin{equation}\label{tm:eq:cubical-map}
 r_i(x)=k_i-\frac12+
 \varphi\!\left(\frac12+\frac{x_i-k_i}{2\rho(x)}\right).
\end{equation}
It follows directly from the definition of $\varphi$ that
\[
 r_i(x)=k_i-\tfrac12\quad\text{if }x_i-k_i\le-\rho(x)/2,
 \qquad
 r_i(x)=k_i+\tfrac12\quad\text{if }x_i-k_i\ge\rho(x)/2.
\]
If $|x_i-k_i|<\rho(x)/2$, the coordinate $r_i(x)$ lies strictly between
$k_i-1/2$ and $k_i+1/2$.  It equals $k_i$ when $x_i=k_i$ and lies on
the same side of $k_i$ as $x_i$ otherwise.  Such a coordinate need not
reach an endpoint.  Nevertheless, as proved below, at least $n-m+1$
other coordinates do reach endpoints, placing the whole point $r(x)$
on the cubical $(m-1)$-skeleton.  Figure~\ref{tm:fig:zeta-ri} illustrates
the maps for $n=2$, with $m=1$ and $m=2$.
\begin{figure}[tbp]
\centering
\begin{minipage}[c]{.44\linewidth}
\centering
\includegraphics[width=\linewidth]{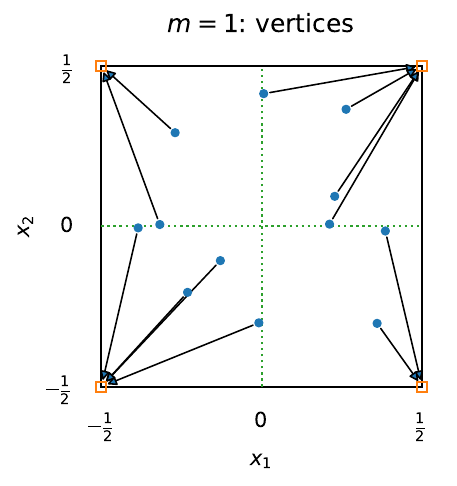}
\end{minipage}\hspace{.025\linewidth}%
\begin{minipage}[c]{.44\linewidth}
\centering
\includegraphics[width=\linewidth]{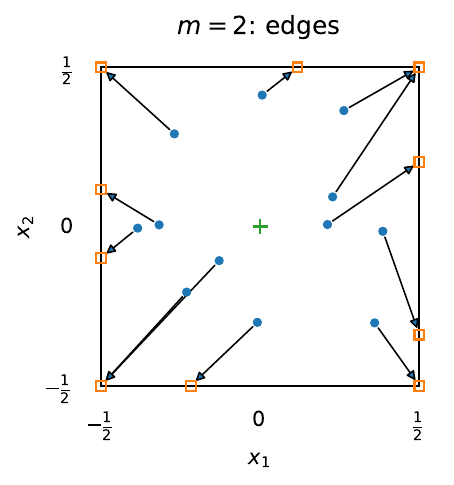}
\end{minipage}
\caption{The action of $r$ on the centered unit square for $n=2$, with
$m=1$ (left) and $m=2$ (right).  Filled circles mark $x$ and open squares
mark $r(x)$; arrows join the two.  The dotted coordinate lines for $m=1$
and the cross at the origin for $m=2$ mark the exceptional set $S$.
The two panels use the same twelve starting points, and each target is
computed from \eqref{tm:eq:rho} and \eqref{tm:eq:cubical-map}.
All twelve images are vertices when $m=1$.  When $m=2$, six images are
vertices and six lie in the interiors of edges.}
\label{tm:fig:zeta-ri}
\end{figure}

\begin{lemma}\label{tm:lem:cubical-map}
The formulas \eqref{tm:eq:cubical-map} define a smooth map
$r:\R^n\setminus S\to\R^n$, with
\[
 r(x)\in K_{\rm cub}^{(m-1)},\qquad x\in\R^n\setminus S.
\]
Moreover,
\[
 r(x+v)=r(x)+v,\qquad x\in\R^n\setminus S,\quad v\in\Z^n.
\]
\end{lemma}
\begin{proof}
Inside each open cube the formulas are smooth because $\varphi$ is smooth
and because $\rho$ is smooth and positive off $S$.  At $x_i=k_i+1/2$, the two possible centers for
this coordinate are $k_i$ and $k_i+1$.  On the strip
$|y_i-(k_i+1/2)|<1/4$, with $y\notin S$, we have
\[
 y_i-k_i>1/4>\rho(y)/2,\qquad
 y_i-(k_i+1)<-1/4<-\rho(y)/2,
\]
since $\rho(y)\le1/8$.  Both formulas are therefore constantly
$k_i+1/2$ throughout this strip.  They agree as functions on an open
neighborhood, so all partial derivatives computed from either formula
agree.  The other coordinate formulas are unchanged, since only the
choice of $k_i$ differs.  If several coordinates are half-integers,
apply this argument to each of them.  This proves smoothness off $S$.

For the translation rule, let $v\in\Z^n$.  The set $S$ is
$\Z^n$-invariant, and $\rho(x+v)=\rho(x)$ because every $f_i$ is periodic.
If $x$ lies in the cube centered at $k$, then $x+v$ lies in the cube
centered at $k+v$.  The displacement from the center is unchanged, so
\[
 r_i(x+v)=k_i+v_i-\tfrac12+
 \varphi\!\left(\tfrac12+\frac{x_i-k_i}{2\rho(x)}\right)=r_i(x)+v_i.
\]
The agreement of the formulas just proved makes this calculation valid
also when a coordinate lies at a cube endpoint.

Finally, fix $x\notin S$ in a cube centered at $k$.  Then
$|x_i-k_i|=d_i(x)$, and at least $n-m+1$ coordinates satisfy
\[
 |x_i-k_i|=d_i(x)\ge d_{(m)}(x)\ge4\rho(x)>0.
\]
Each such coordinate is sent exactly to one of the endpoints
$k_i\pm1/2$.  All coordinates of $r(x)$ lie in their respective closed
intervals, since $0\le\varphi\le1$.  Thus at most $m-1$ coordinates of
$r(x)$ remain in the interiors of these intervals, proving the skeleton
assertion and, in particular, that $r(x)$ lies on the boundary of its cube.
\end{proof}

With $L$ from \eqref{eq:explicit-physical-cubulation}, set
\[
 S_D=LS,\qquad
 \Retr(u,\eta)=Lr\bigl(L^{-1}(u,\eta)\bigr).
\]
Since $\Gamma_D=\mathsf A^{-1}\Z^d\times\mathsf B^{-1}\Z^d$,
\Cref{tm:lem:cubical-map} gives, for every $k,\ell\in\Z^d$,
\begin{equation}\label{tm:eq:physical-map}
 \begin{aligned}
 \Retr(u+\mathsf A^{-1}k,\eta+\mathsf B^{-1}\ell)
   &=\Retr(u,\eta)+(\mathsf A^{-1}k,\mathsf B^{-1}\ell),\\
 \Retr(u,\eta)&\in L K_{\rm cub}^{(m-1)}.
 \end{aligned}
\end{equation}
All scale estimates below are unchanged up to fixed constants by this
fixed diagonal rescaling.

Write $\Retr=(\Retr_1,\Retr_2)$, where
$\Retr_1,\Retr_2:\R^{2d}\setminus S_D\to\R^d$ are its position and
frequency components, respectively.

\begin{proposition}\label{prop:flat-scalar-zak}
The scalar function
\begin{equation}\label{tm:eq:scalar-pullback}
 F_{\rm flat}(u,\eta)=
 \begin{cases}
 e^{2\pi i u\cdot(\eta-\Retr_2(u,\eta))}
       H\bigl(\Retr(u,\eta)\bigr),&(u,\eta)\notin S_D,\\
 0,&(u,\eta)\in S_D,
 \end{cases}
\end{equation}
is bounded and measurable, is quasiperiodic, and is smooth
off $S_D$.  For $(u,\eta)\notin S_D$ it satisfies
\begin{equation}\label{tm:eq:scalar-pullback-correlations}
 G_{F_{\rm flat}}(u,\eta)
   =G_H\bigl(\Retr(u,\eta)\bigr)=I_R.
\end{equation}
Consequently $g_{\rm flat}=Z^{-1}F_{\rm flat}$ generates a Parseval
Gabor frame and has squared norm $R/N$.
\end{proposition}
\begin{proof}
\medskip\noindent\emph{Step 1: Quasiperiodicity.}
Let $k,\ell\in\Z^d$ and $(u,\eta)\notin S_D$.  Applying
\eqref{tm:eq:physical-map} with the integer pair $(\mathsf A k,\mathsf B\ell)$ gives
\[
 \Retr_1(u+k,\eta+\ell)=\Retr_1(u,\eta)+k,\qquad
 \Retr_2(u+k,\eta+\ell)=\Retr_2(u,\eta)+\ell.
\]
The defining phase contributes the factor
$e^{2\pi i k\cdot(\eta-\Retr_2(u,\eta))}$, and quasiperiodicity of $H$
contributes $e^{2\pi i k\cdot\Retr_2(u,\eta)}$.  Their product is
$e^{2\pi i k\cdot\eta}$, so
\[
 F_{\rm flat}(u+k,\eta+\ell)=e^{2\pi i k\cdot\eta}F_{\rm flat}(u,\eta).
\]

\medskip\noindent\emph{Step 2: the fine-lattice values and phase cancellation.}
Fix $(u,\eta)\notin S_D$, $t\in\mathcal I_{\mathsf A}$, and
$s\in\mathcal I_{\mathsf B}$.  Since $(-Dt,\mathsf B^{-1}s)\in\Gamma_D$
and $S_D$ is fine-lattice invariant, the translated point
$(u-Dt,\eta+\mathsf B^{-1}s)$ also lies outside $S_D$.
Apply \eqref{tm:eq:physical-map} with the integer pair $(-\mathsf B t,s)$.
Since $\mathsf A^{-1}\mathsf B=D$, this gives
\[
 \Retr(u-Dt,\eta+\mathsf B^{-1}s)=(\Retr_1(u,\eta)-Dt,\Retr_2(u,\eta)+\mathsf B^{-1}s).
\]
Substitute this identity in \eqref{tm:eq:scalar-pullback}:
\begin{equation}\label{tm:eq:scalar-pullback-samples}
\begin{aligned}
 F_{\rm flat}(u-Dt,\eta+\mathsf B^{-1}s)
 &=e^{2\pi i(u-Dt)\cdot(\eta-\Retr_2(u,\eta))}\\
 &\quad\cdot H\bigl(\Retr_1(u,\eta)-Dt,\Retr_2(u,\eta)+\mathsf B^{-1}s\bigr).
\end{aligned}
\end{equation}
The phase depends on the column $t$, but not on the row $s$.  Thus,
with the unitary diagonal matrix
\[
 Q(u,\eta)=\diag_{t\in\mathcal I_{\mathsf A}}
       \bigl(e^{2\pi i(u-Dt)\cdot(\eta-\Retr_2(u,\eta))}\bigr),
\]
formula \eqref{tm:eq:scalar-pullback-samples} reads
\[
 \mathcal M_{F_{\rm flat}}(u,\eta)
 =\mathcal M_H\bigl(\Retr(u,\eta)\bigr)Q(u,\eta).
\]
Since $QQ^*=I_N$, it follows that
\[
 G_{F_{\rm flat}}(u,\eta)
 =G_H\bigl(\Retr(u,\eta)\bigr)=I_R.
\]
Every argument of $H$ appearing above lies in $V$, by fine-lattice invariance.

\medskip\noindent\emph{Step 3: regularity, null-set values, and inversion.}
The map $\Retr$ is smooth off $S_D$ and takes values where $H$ is
smooth.  The exponential factor is smooth there as well.  Thus
$F_{\rm flat}$ is smooth off $S_D$.  Since $G_H=I_R$ on $V$,
we have $|H|^2\le(\mathcal M_H\mathcal M_H^*)_{0,0}=R$ there,
so $|F_{\rm flat}|\le\sqrt R$.
The zero extension is measurable and quasiperiodic
also on $S_D$, because that set is invariant under integer translations.
Its measure is zero by \eqref{tm:eq:S-flat}; its assigned value is
irrelevant to the almost-everywhere matrix Parseval identity.

The bounded scalar function on the ordinary cell has an inverse unitary
Zak transform $g_{\rm flat}\in L^2(\R^d)$.  Equations~\eqref{eq:scalar-parseval-identities} and
\eqref{tm:eq:scalar-pullback-correlations} prove Parsevality.
The norm identity \eqref{eq:scalar-parseval-norm} gives
$\|g_{\rm flat}\|_2^2=R/N$.
\end{proof}

\begin{figure}[!htbp]
\centering
\includegraphics[width=.43\linewidth]{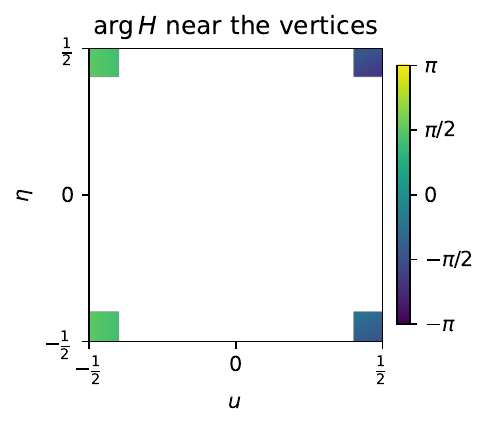}\hfill
\includegraphics[width=.43\linewidth]{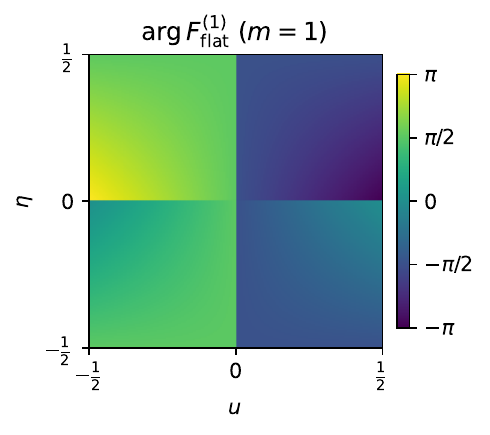}\\[3pt]
\includegraphics[width=.43\linewidth]{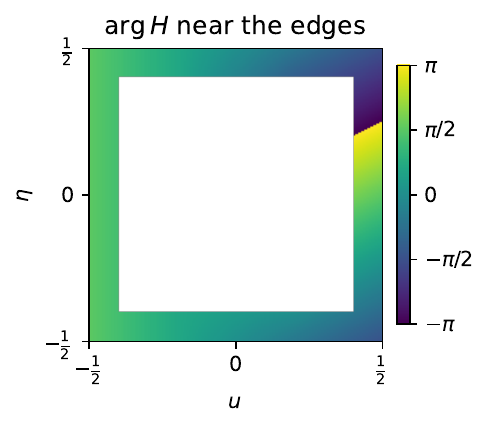}\hfill
\includegraphics[width=.43\linewidth]{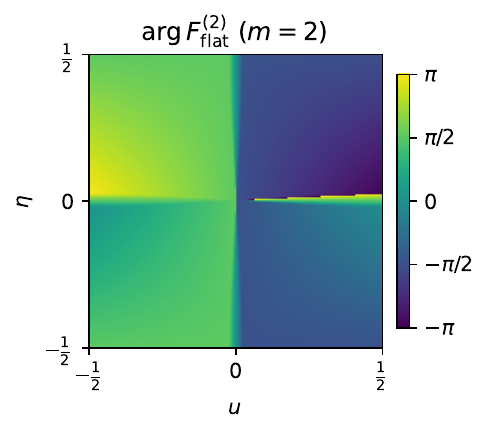}
\caption{A two-dimensional illustration of the corrected cubical pullback.
Left: a smooth normalized Zak function near the vertices (top) or edges
(bottom) of the centered unit square.  Right: the corresponding function
after pullback and phase correction.  In the displayed square, the
exceptional set is the coordinate axes in the top row and the origin
in the bottom row.  All displayed functions have modulus one; colors
represent their arguments, with $-\pi$ and $\pi$ representing the same
complex phase.}
\label{fig:scalar-flat-illustration}
\end{figure}

\FloatBarrier

\section{Fourier decay and regularity of the flat window}
\label{sec:elementary-perturbation}
Here, we estimate the localized cross-Zak Fourier coefficients of
$F_{\rm flat}$.  The coordinate geometry of the exceptional set provides
the derivative bounds and shell estimates that then lead to the full
modulation-space range of the respective window function.

\subsection{Derivative bounds inherited from the cubical map}
Throughout this section, fix $n=2d$ and $m=2(N-R+1)$ with $1\le m\le n$,
as in \eqref{eq:terminal-gap-assumption}; the case $m=n$ is included.
We use the normalized cubical coordinates and the functions $d_i$, $f_i$,
$r$, the exceptional set $S$, and the scale $\rho$ defined in
\Cref{subsec:cubical-pullback}, pp.~\pageref{tm:eq:S-flat}--\pageref{tm:eq:cubical-map}.
Thus $d_i(x)=\dist(x_i,\Z)$, $S$ consists of the points with at least $m$
integer coordinates, and $\rho$ is smooth and positive off $S$ and
comparable to $\dist(x,S)$.  Put
\[
 \lambda_i(x)=\max\{\rho(x),d_i(x)\}.
\]

\Needspace{18\baselineskip}
\begin{lemma}\label{tm:lem:anisotropic}
For every multi-index $\beta$,
\begin{equation}\label{tm:eq:rho-symbol}
 |\partial^\beta\rho(x)|
   \le C_\beta\rho(x)\prod_{i=1}^n \lambda_i(x)^{-\beta_i}.
\end{equation}
For every nonzero multi-index $\beta$, on any fixed finite union of the
centered unit cubes $k+[-1/2,1/2]^n$, $k\in\Z^n$, and away from $S$,
\begin{equation}\label{tm:eq:r-symbol}
 |\partial^\beta r_j(x)|
    \le C_\beta\prod_{i=1}^n \lambda_i(x)^{-\beta_i}.
\end{equation}
For every $a\in C_c^\infty(\R^{2d})$, the normalized function
$A(x)=a(Lx)F_{\rm flat}(Lx)$ is bounded, smooth off $S$, and satisfies
\begin{equation}\label{eq:normalized-flat-symbol}
 |\partial_x^\beta A(x)|
 \le C_{a,\beta,L}\prod_{i=1}^n\lambda_i(x)^{-\beta_i},
 \qquad x\notin S.
\end{equation}
This includes $\beta=0$.
\end{lemma}

\begin{proof}
Fix $I\subset\{1,\ldots,n\}$ with $|I|=m$, and choose $i\in I$.
For every $x\in\R^n\setminus S$,
\[
 f_I(x)=\sum_{\ell\in I}f_\ell(x)
 \ge c\bigl(\rho(x)^2+d_i(x)^2\bigr)\ge c'\lambda_i(x)^2.
\]
Here $f_I(x)\ge16\pi^2m\rho(x)^2$ follows directly from
\eqref{tm:eq:rho}, and $f_I(x)\ge f_i(x)\ge4d_i(x)^2$ because $i\in I$ and because $\sin(\pi a)\ge 2a$ for $0\le a\le 1/2$.
For $j=1$, periodicity gives
\[
 |\partial_i f_i(x)|=|\partial_i\sin^2(\pi x_i)|=\pi|\sin(2\pi x_i)|
 \le2\pi^2d_i(x)\le2\pi^2\lambda_i(x).
\]
For $j\ge2$,
\[
 |\partial_i^j f_i(x)|
 =\bigl|\partial_i^{j-1}\bigl(\pi\sin(2\pi x_i)\bigr)\bigr|
 \le\pi^j2^{j-1}.
\]
For $i\in I$, only the summand $f_i$ depends on $x_i$, so
$\partial_i^j f_I=\partial_i^j f_i$.  Dividing by
$f_I(x)\ge c'\lambda_i(x)^2$, the first derivative satisfies
\[
 \frac{|\partial_i f_I(x)|}{f_I(x)}
 =\frac{|\partial_i f_i(x)|}{f_I(x)}
 \le C\frac{\lambda_i(x)}{\lambda_i(x)^2}
 =C\lambda_i(x)^{-1}.
\]
For $j\ge2$, the constant bound above gives
\[
 \frac{|\partial_i^j f_I(x)|}{f_I(x)}
 =\frac{|\partial_i^j f_i(x)|}{f_I(x)}
 \le\frac{\pi^j2^{j-1}}{f_I(x)}
 \le C_j\lambda_i(x)^{-2}
 \le C_j\lambda_i(x)^{-j}.
\]
The last inequality uses $0<\lambda_i(x)\le1/2$ and $j\ge2$.
Combining the two cases gives
\[
 |\partial_i^j f_I(x)|\le C_j f_I(x)\lambda_i(x)^{-j},
 \qquad i\in I,\quad j\ge1.
\]
Derivatives of $f_I$ involving two different coordinates, or a coordinate
outside $I$, vanish.  Hence, for every nonzero multi-index $\gamma$,
\[
 |\partial^\gamma f_I(x)|
 \le C_\gamma f_I(x)\prod_{i=1}^n\lambda_i(x)^{-\gamma_i}.
\]
Repeated differentiation of $f_I f_I^{-1}=1$, followed by induction on
$|\beta|$, now gives, for $\beta\ne0$,
\[
\begin{aligned}
 |\partial^\beta(f_I^{-1})(x)|
 &\le f_I(x)^{-1}\sum_{0<\gamma\le\beta}\binom{\beta}{\gamma}
       |\partial^\gamma f_I(x)|\,
       |\partial^{\beta-\gamma}(f_I^{-1})(x)|\\
 &\le C_\beta f_I(x)^{-1}\sum_{0<\gamma\le\beta}
       \prod_{i=1}^n\lambda_i(x)^{-\gamma_i}
       \prod_{i=1}^n\lambda_i(x)^{-(\beta_i-\gamma_i)}\\
 &\le C_\beta f_I(x)^{-1}\prod_{i=1}^n\lambda_i(x)^{-\beta_i}.
\end{aligned}
\]
Here $\gamma\le\beta$ means coordinatewise inequality, and
$\binom{\beta}{\gamma}=\prod_i\binom{\beta_i}{\gamma_i}$.
The estimate is immediate for $\beta=0$.  Summing over the choices for $I$ yields
\[
 \Big|\partial^\beta\!\left(\sum_{|I|=m}f_I^{-1}\right)(x)\Big|
 \le\sum_{|I|=m}|\partial^\beta(f_I^{-1})(x)|
 \le C_\beta\!\left(\sum_{|I|=m}f_I(x)^{-1}\right)
             \prod_{i=1}^n\lambda_i(x)^{-\beta_i}.
\]
\Needspace{14\baselineskip}
For $\beta\ne0$, apply the chain rule to the inverse square root.
A term with $q$ differentiated factors contains the $q$th derivative of
$y\mapsto y^{-1/2}$, contributing a constant times $y^{-q-1/2}$.
The preceding estimate contributes $q$ factors of the positive sum,
while the derivative multi-indices add to $\beta$.  Thus, the Fa\`a di Bruno multivariate formula gives
\[
\begin{aligned}
 |\partial^\beta\rho(x)|
 &=\frac1{4\pi\sqrt m}
   \Big|\partial^\beta\!\left[\left(\sum_{|I|=m}f_I^{-1}\right)^{-1/2}\right](x)\Big|\\
 &\le\frac{C_\beta}{4\pi\sqrt m}\sum_{q=1}^{|\beta|}
   \left(\sum_{|I|=m}f_I(x)^{-1}\right)^{-q-1/2}
   \left(\sum_{|I|=m}f_I(x)^{-1}\right)^q
   \prod_{i=1}^n\lambda_i(x)^{-\beta_i}\\
 &\le C_\beta\rho(x)\prod_{i=1}^n\lambda_i(x)^{-\beta_i}.
\end{aligned}
\]
This proves \eqref{tm:eq:rho-symbol}; the case $\beta=0$ is immediate.

We next prove \eqref{tm:eq:r-symbol}.  First, note that retracing the steps above and replacing $f_I$ with $\rho$ gives 
\[
 |\partial^\gamma(\rho^{-1})(x)|
 \le C_\gamma\rho(x)^{-1}\prod_i\lambda_i(x)^{-\gamma_i}
\]
for any multi-index $\gamma$, with $\gamma=0$ holding trivially.

Fix an open cube centered at $k\in\Z^n$ and a component $r_j$.
Every positive-order derivative of $\varphi$ vanishes outside
$[1/4,3/4]$.  Thus a differentiated term of \eqref{tm:eq:cubical-map}
can be nonzero only where
\[
 |x_j-k_j|\le\rho(x)/2,\qquad
 d_j(x)=|x_j-k_j|\le\rho(x)/2,\qquad \lambda_j(x)=\rho(x).
\]
For $\beta\ne0$, the explicit product rule is
\[
\begin{aligned}
 \partial^\beta\!\left(\frac{x_j-k_j}{\rho(x)}\right)
 &=\sum_{\alpha\le\beta}\binom\beta\alpha
       \partial^\alpha(x_j-k_j)\,
       \partial^{\beta-\alpha}(\rho^{-1})(x)\\
 &=(x_j-k_j)\partial^\beta(\rho^{-1})(x)
       +\beta_j\partial^{\beta-e_j}(\rho^{-1})(x).
\end{aligned}
\]
Here $e_j$ is the $j$th coordinate vector, and the second term is absent
when $\beta_j=0$.  Only $\alpha=0$ and $\alpha=e_j$ contribute, because
$x_j-k_j$ is linear and depends only on $x_j$.  On the region just
specified, the reciprocal estimate implies
\[
\begin{aligned}
 \left|\partial^\beta\!\left(\frac{x_j-k_j}{\rho(x)}\right)\right|
 &\le C_\beta\left(\frac{|x_j-k_j|}{\rho(x)}
              +\beta_j\frac{\lambda_j(x)}{\rho(x)}\right)
                  \prod_i\lambda_i(x)^{-\beta_i}\\
 &\le C_\beta(\tfrac12+\beta_j)\prod_i\lambda_i(x)^{-\beta_i}.
\end{aligned}
\]
In the second product-rule term, one fewer derivative in the $j$th
coordinate contributes the factor $\lambda_j(x)$, which cancels
$\rho(x)$ in the denominator.

Each term in the higher-order chain rule for $\partial^\beta r_j(x)$
is a constant multiple of
\[
 \varphi^{(v)}\!\left(\frac12+\frac{x_j-k_j}{2\rho(x)}\right)
 \prod_{\ell=1}^v
       \partial^{\gamma_\ell}\!\left(\frac{x_j-k_j}{2\rho(x)}\right),
 \qquad
 \gamma_\ell\ne0,\quad \sum_{\ell=1}^v\gamma_\ell=\beta.
\]
Where the first factor is nonzero, the preceding estimate applies to
each factor in the product.  The derivatives of $\varphi$ are bounded,
and the coordinate weights multiply to
\[
 \prod_{\ell=1}^v\prod_i\lambda_i(x)^{-(\gamma_\ell)_i}
 =\prod_i\lambda_i(x)^{-\beta_i}.
\]
Summing the finitely many chain-rule terms proves \eqref{tm:eq:r-symbol}
in the open cube.  Since $r$ is smooth off $S$ by
\Cref{tm:lem:cubical-map} and the right-hand side of the bound is continuous
there, the estimate extends to cube boundaries by continuity.

Write $Lx=((Lx)_u,(Lx)_\eta)$.  Formula
\eqref{tm:eq:scalar-pullback} gives
\[
 F_{\rm flat}(Lx)
 =e^{2\pi i (Lx)_u\cdot((Lx)_\eta-(Lr(x))_\eta)}H(Lr(x)).
\]
The image $Lr(x)$ lies in the protected skeleton.  On the compact
set needed for $a(Lx)$, the derivatives of $H$ at these image points
are bounded.  Each term in the chain rule for $H(Lr(x))$ is a bounded
derivative of $H$ times derivatives of $Lr$.  If their multi-indices
are $\gamma_1,\ldots,\gamma_v$ with sum $\beta$, their product is bounded by
\[
 C\prod_{\nu=1}^v\prod_{i=1}^n\lambda_i(x)^{-(\gamma_\nu)_i}
 =C\prod_{i=1}^n\lambda_i(x)^{-\beta_i}.
\]
The exponential is a smooth function of $Lx$ and $Lr(x)$ on this compact
set, so the same calculation applies to it.  The derivatives of
$a(Lx)$ are bounded.  Leibniz' rule proves
\eqref{eq:normalized-flat-symbol}.
\end{proof}

\subsection{Dyadic shells and the weighted coordinate integral}
\label{subsec:flat-shell-geometry}

We first decompose the localized cross-Zak function into smooth shells
where $\rho(x)$ is comparable to $2^{-j}$.  We then express each shell
using coordinates measured from integer endpoints.  The first two lemmata
record the support and derivative bounds needed for the Fourier estimates
in the next subsection.  The final lemma gives the product-integral
estimate that will also be used for phase cancellation in
\Cref{sec:lower-completion}.

Fix one function $\psi_\ell$ from \Cref{lem:finite-Zak-family} and the
partition of unity generating function $\chi$ from
\eqref{eq:compact-periodizing-cutoff}.  Put
\[
 A_{\rm phys}(u,\eta)
 =\chi(u,\eta)F_{\rm flat}(u,\eta)\overline{Z\psi_\ell(u,\eta)},
 \qquad A(x)=A_{\rm phys}(Lx).
\]
Equation~\eqref{eq:compact-cutoff-coefficients} identifies the Fourier
coefficients of the ordinary periodic cross-Zak product with
$\widehat{A_{\rm phys}}(k)$.  The multiplier
$\chi\,\overline{Z\psi_\ell}$ is smooth and compactly supported, so
\eqref{eq:normalized-flat-symbol} applies to $A$.  Its support meets
finitely many centered unit cubes in the normalized coordinates.

For the dyadic shells use the fixed step from
\eqref{eq:fixed-smooth-step} and put
\[
 \zeta(t)=\varphi(2t+\tfrac12),\qquad
 \tau(x)=-\log_2\rho(x),\quad x\notin S.
\]
Thus $\zeta(t)=0$ for $t\le-1/8$ and $\zeta(t)=1$ for $t\ge1/8$.
For $\tau\in\R$ and $j\in\Z$ put
\begin{equation}\label{eq:aligned-shell-partition}
 \zeta_j(\tau)=\zeta(\tau-j)-\zeta(\tau-j-1).
\end{equation}
Figure~\ref{fig:shell-cutoffs} shows these weights and their composition
with $\tau$ in the one-dimensional case $n=m=1$.
\begin{figure}[tbp]
\centering
\includegraphics[width=.485\linewidth]{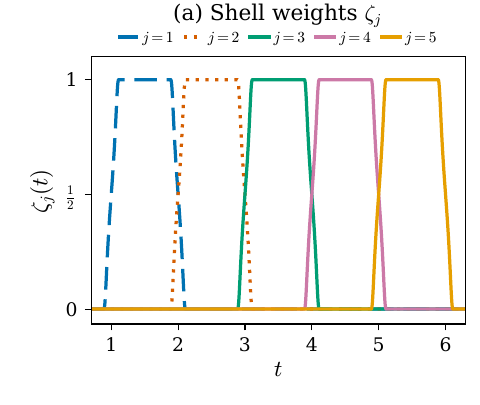}\hfill
\includegraphics[width=.485\linewidth]{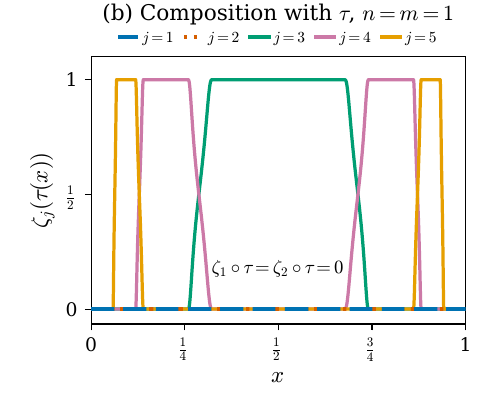}
\caption{The shell weights $\zeta_j$ from \eqref{eq:aligned-shell-partition}
for $j=1,\ldots,5$ (left), and the compositions $\zeta_j\circ\tau$ in matching
colors (right).  The right panel uses $n=m=1$,
$\rho(x)=|\sin(\pi x)|/(4\pi)$, and $\tau(x)=-\log_2\rho(x)$ on $(0,1)$.
Since $\tau(x)\ge\log_2(4\pi)>3+1/8$, the first two compositions are
identically zero.  Each fixed composition is extended continuously by zero
at $0$ and $1$; indices $j\ge6$, not displayed, are active closer to
these endpoints.}
\label{fig:shell-cutoffs}
\end{figure}

\Needspace{16\baselineskip}
\begin{lemma}[Dyadic shell decomposition]\label{lem:flat-shell-decomposition}
For $j\ge3$, define $A_j(x)=\zeta_j(\tau(x))A(x)$ off $S$ and set
$A_j=0$ on $S$.  Then $A_j\in C_c^\infty(\R^n)$ and
\begin{equation}\label{eq:aligned-flat-shells}
 A=\sum_{j\ge3}A_j\quad\text{almost everywhere and in }L^1(\R^n).
\end{equation}
Only boundedly many shells meet each point outside $S$, and
\begin{equation}\label{eq:aligned-shell-scale}
 2^{-j-9/8}\le\rho(x)\le2^{-j+1/8},\qquad x\in\supp A_j.
\end{equation}
For every multi-index $\gamma$,
\begin{equation}\label{eq:coordinate-shell-derivatives}
 |\partial^\gamma A_j(x)|
 \le C_\gamma\prod_{i=1}^n(2^{-j}+d_i(x))^{-\gamma_i},\qquad x\in\R^n.
\end{equation}
For some $C_0$ independent of $j$, at least $m$ coordinate distances are
at most $C_0 2^{-j}$ at every point of $\supp A_j$.  In particular,
\begin{equation}\label{tm:eq:flat-shell-basic}
 |\supp A_j|\le C2^{-jm},\qquad
 \|\partial^\gamma A_j\|_\infty\le C_\gamma2^{j|\gamma|}.
\end{equation}
The constants may depend on the fixed compact support and on $A$, but
not on $j$.  The conclusions hold for any bounded compactly supported
$A$ that is smooth off $S$ and satisfies \eqref{eq:normalized-flat-symbol}.
\end{lemma}
\begin{proof}
\medskip\noindent\emph{Step 1: partition and smoothness.}
The functions $\zeta_j$ are nonnegative and
$\supp\zeta_j\subset[j-1/8,j+1+1/8]$.  By
\eqref{eq:rho-uniform-upper-bound},
$\tau(x)\ge\log_2(4\pi)>3+1/8$ off $S$, so
$\zeta(\tau(x)-3)=1$.  The sum telescopes:
\[
 \sum_{j=3}^J\zeta_j(\tau)
 =\zeta(\tau-3)-\zeta(\tau-J-1)\longrightarrow1.
\]
This proves the almost-everywhere identity in
\eqref{eq:aligned-flat-shells}, since $S$ is null.  For fixed $j$, the
support restriction on $\tau$ is precisely
\eqref{eq:aligned-shell-scale}.  Since $\rho$ extends continuously by
zero on $S$, $\tau(x)\to\infty$ as $x$ approaches $S$ and thus $A_j$ vanishes on a neighborhood of $S$ and its zero
extension is smooth.  It has compact support contained in $\supp A$.
The intervals $[j-1/8,j+1+1/8]$ have uniformly bounded overlap.

\medskip\noindent\emph{Step 2: coordinatewise derivatives.}
For $|\gamma|>0$, every derivative of $\log\rho(x)$ is a finite sum
of terms
\[
 C\rho(x)^{-r}\prod_{\nu=1}^r\partial^{\gamma_\nu}\rho(x),
 \qquad |\gamma_\nu|>0,\qquad
 \gamma_1+\cdots+\gamma_r=\gamma.
\]
Lemma~\ref{tm:lem:anisotropic} bounds each factor in the product by
$C\rho(x)\prod_i\lambda_i(x)^{-(\gamma_\nu)_i}$.
The $r$ powers of $\rho(x)$ cancel $\rho(x)^{-r}$, giving,
since $\tau(x)=-\log_2(\rho(x))$,
\[
 |\partial^\gamma\tau(x)|
 \le C_\gamma\prod_i\lambda_i(x)^{-\gamma_i},\qquad x\notin S.
\]
Each positive-order derivative of $\zeta_j(\tau(x))$ is a sum of terms
\[
 C\zeta_j^{(v)}(\tau(x))
     \prod_{\ell=1}^v\partial^{\gamma_\ell}\tau(x),
 \qquad \gamma_\ell\ne0,\quad\sum_\ell\gamma_\ell=\gamma.
\]
The derivatives of $\zeta_j$ are bounded independently of $j$, since
$\zeta_j$ is a difference of two translates of the same smooth step.
Its composition therefore has the same coordinatewise bound; the
zeroth derivative is bounded as well. By definition of $\lambda_i$,
\[
 \tfrac12(\rho(x)+d_i(x))\le\lambda_i(x)\le\rho(x)+d_i(x),
\]
so on the shell
\[
 \qquad \lambda_i(x)\asymp2^{-j}+d_i(x).
\]
Leibniz' rule together with \eqref{eq:normalized-flat-symbol} gives
\[
 |\partial^\gamma A_j(x)|
 \le\sum_{\alpha\le\gamma}\binom\gamma\alpha
       |\partial^\alpha[\zeta_j(\tau(x))]|\,
       |\partial^{\gamma-\alpha}A(x)|
 \le C_\gamma\prod_i\lambda_i(x)^{-\gamma_i}.
\]
This proves \eqref{eq:coordinate-shell-derivatives}; outside the shell
all terms vanish.  The calculation also shows that
$|\partial^\gamma\tau(x)|\le C_\gamma
\prod_i(2^{-j}+d_i(x))^{-\gamma_i}$ on the shell for $|\gamma|>0$,
a bound used when the phase interpolation parameter is differentiated in \Cref{prop:deterministic-phase}.

\medskip\noindent\emph{Step 3: small coordinates and volume.}
By \Cref{tm:lem:rho-comparable},
$d_{(m)}(x)\le C\rho(x)\le C_0 2^{-j}$ on $\supp A_j$.
Thus at least $m$ of the $d_i(x)$ have this bound.
Cover the shell by the finitely many choices of these $m$ coordinates.
On the fixed compact support each such coordinate lies in a union of
finitely many intervals of total length $O(2^{-j})$; the other coordinates
range over fixed bounded intervals.  Each choice contributes
$O(2^{-jm})$, proving the volume bound in \eqref{tm:eq:flat-shell-basic}.
The derivative supremum bound follows from
$2^{-j}+d_i(x)\ge2^{-j}$ in \eqref{eq:coordinate-shell-derivatives}.
In particular,
\[
 \sum_{j\ge3}\|A_j\|_1
 \le C\sum_{j\ge3}2^{-jm}<\infty.
\]
Together with the pointwise identity, this proves $L^1$ convergence in
\eqref{eq:aligned-flat-shells}.
\end{proof}

The Fourier argument will use the fundamental theorem of calculus in
each coordinate, integrating toward the endpoint $1$.  We therefore
localize each unit cube so that the resulting functions vanish near that
endpoint and retain derivative bounds in terms of $2^{-j}+y_i$.
Set $\omega(t)=1-\varphi(t)$ for $t\ge0$.  Then $\omega=1$ on $[0,1/4]$,
$\omega=0$ on $[3/4,\infty)$, and
$\omega(t)+\omega(1-t)=1$ for $0\le t\le1$.
For $z\in\Z^n$ and $\epsilon\in\{0,1\}^n$, put
\begin{equation}\label{eq:orthant-coordinate-map}
 x_{z,\epsilon}(y)=z+\epsilon+\varepsilon\cdot y,
 \qquad \varepsilon_i=(-1)^{\epsilon_i},\qquad y\in[0,1]^n,
\end{equation}
where $\varepsilon\cdot y$ is coordinatewise multiplication.
The coordinate $y_i$ is measured from the integer endpoint $z_i$ when
$\epsilon_i=0$ and from $z_i+1$ when $\epsilon_i=1$.
In other words, $x_{z,\epsilon}$ is a reparametrization of the cube $z+[0,1]^n$ where one retains the freedom (encoded in $\epsilon$) to measure the coordinate entries of $y$ from either of the faces of the cube. 
For $0<h\le1$ put
\begin{equation}\label{eq:coordinate-shell-region}
 E_h=\{y\in[0,1]^n:\#\{i:y_i\le C_0h\}\ge m\},
\end{equation}
where $C_0$ is enlarged from \Cref{lem:flat-shell-decomposition} as needed.

\begin{lemma}[Fixed coordinate localization]\label{lem:flat-coordinate-localization}
For $j\ge3$, define
\[
 b_{j,z,\epsilon}(y)
 =A_j(x_{z,\epsilon}(y))\prod_{i=1}^n\omega(y_i),\qquad y\in[0,1]^n.
\]
Only finitely many $z$ occur, independently of $j$.  For every real
$\xi\in\R^n$,
\begin{equation}\label{eq:orthant-fourier-decomposition}
 \widehat A_j(\xi)
 =\sum_{z,\epsilon}e^{-2\pi i(z+\epsilon)\cdot\xi}
   \int_{[0,1]^n}b_{j,z,\epsilon}(y)
                   e^{-2\pi i\sum_i\varepsilon_i\xi_i y_i}\,dy.
\end{equation}
Each $b_{j,z,\epsilon}$ is smooth on the closed cube, supported in
$E_{2^{-j}}$, and zero when any $y_i\ge3/4$.  For every multi-index
$\gamma$,
\begin{equation}\label{eq:orthant-shell-derivatives}
 |\partial^\gamma b_{j,z,\epsilon}(y)|
 \le C_\gamma\prod_i(2^{-j}+y_i)^{-\gamma_i},\qquad y\in[0,1]^n,
\end{equation}
with constants independent of $j,z,\epsilon$.
\end{lemma}
\begin{proof}
For $x\in z+[0,1]^n$, the one-dimensional identity for $\omega$ gives
\[
\begin{aligned}
 1&=\prod_{i=1}^n[\omega(x_i-z_i)+\omega(z_i+1-x_i)]\\
  &=\sum_{\epsilon\in\{0,1\}^n}
       \prod_{i=1}^n\omega\bigl((-1)^{\epsilon_i}(x_i-z_i-\epsilon_i)\bigr).
\end{aligned}
\]
Thus, 
\[\int_{z+[0,1]^n} A_j(x)e^{-2\pi ix\cdot\xi} = \sum_{\epsilon\in\{0,1\}^n}\int_{z+[0,1]^n}A_j(x)\prod_{i=1}^n \omega\bigl((-1)^{\epsilon_i}(x_i-z_i-\epsilon_i)\bigr)e^{-2\pi i x\cdot\xi}\,dx.\]
In the term indexed by $\epsilon$, substitute
$y_i=(-1)^{\epsilon_i}(x_i-z_i-\epsilon_i)$.  Then
$x=x_{z,\epsilon}(y)$, the absolute Jacobian is one, and
\[
 e^{-2\pi i x\cdot\xi}
 =e^{-2\pi i(z+\epsilon)\cdot\xi}
    e^{-2\pi i\sum_i\varepsilon_i\xi_i y_i}.
\]
Putting this together, we get
\[\widehat A_j(\xi)=\sum_{z,\epsilon}e^{-2\pi i(z+\epsilon)\cdot\xi}\int_{[0,1]^n}\Big(A_j(x_{z,\epsilon}(y))\prod_{i=1}^n \omega(y_i)\Big)e^{-2\pi i \sum_i \varepsilon_i\xi_i y_i}\,dy\]
and \eqref{eq:orthant-fourier-decomposition} follows from the 
definition of $b_{j,z,\epsilon}$.
The fixed compact support of $A$, and therefore of each $A_j$, meets only finitely many unit cubes, justifying the claim that only finitely many $z$ appear, independently of the shell parameter $j$.

On $\supp\omega(y_i)$ we have $0\le y_i\le3/4$ and
\[
 d_i(x_{z,\epsilon}(y))=\min\{y_i,1-y_i\},\qquad
 y_i/3\le d_i(x_{z,\epsilon}(y))\le y_i.
\]
Thus $2^{-j}+d_i(x_{z,\epsilon}(y))\asymp2^{-j}+y_i$.
At least $m$ of the $d_i$ are at most $C_0 2^{-j}$ on the shell, so
at least $m$ of the $y_i$ are at most $3C_0 2^{-j}$.  Increasing $C_0$
in \eqref{eq:coordinate-shell-region} gives the asserted support.

The affine coordinate map $y\to x_{z,\epsilon}(y)$ has diagonal derivative entries $\varepsilon_i$,
so
\[
 \partial_y^\alpha(A_j\circ x_{z,\epsilon})(y)
 =\left(\prod_i\varepsilon_i^{\alpha_i}\right)
      (\partial_x^\alpha A_j)(x_{z,\epsilon}(y)).
\]
The prefactor has modulus one.  Apply
\eqref{eq:coordinate-shell-derivatives} and the coordinate comparison.
Derivatives of the fixed factors $\omega$ are bounded and supported
where $y_i\le3/4$; since $2^{-j}+y_i\le2$, a bounded derivative of order
$b$ is also bounded by $C_b(2^{-j}+y_i)^{-b}$.  Leibniz' rule gives
\eqref{eq:orthant-shell-derivatives}.  The zero values of $\omega$ for
$y_i\ge3/4$ give the vanishing of the $b_{j,z,\epsilon}$ in the same range. In particular, $b_{j,z,\epsilon}$ is zero in a neighborhood of the outer face $y_i=1$. 
\end{proof}

We record an elementary product-integral estimate that expresses the
presence of at least $m$ small coordinate distances.  It will be used later,
with $\theta=1-1/p$ for the flat function and with $\theta=1/2$ for
phase cancellation.

\begin{lemma}\label{lem:coordinate-weighted-integral}
For $0<h\le1$ and $\theta>0$,
\begin{equation}\label{eq:coordinate-weighted-integral}
 \int_{E_h}\prod_{i=1}^n(h+y_i)^{\theta-1}\,dy
 \le C_\theta h^{m\theta}.
\end{equation}
The constant may depend on $n,m,C_0,\theta$, but not on $h$.
\end{lemma}
\begin{proof}
For each $I\subset\{1,\ldots,n\}$ with $|I|=m$, let
$E_{h,I}=\{y\in[0,1]^n:y_i\le C_0h\text{ for }i\in I\}$.
Then $E_h=\bigcup_{|I|=m}E_{h,I}$.  A small coordinate contributes
\begin{align*}
 \int_0^{\min\{C_0h,1\}}(h+y)^{\theta-1}\,dy
 &\le\int_0^{C_0h}(h+y)^{\theta-1}\,dy\\
 &=h^\theta\int_0^{C_0}(1+t)^{\theta-1}\,dt
 \le C_\theta h^\theta,
\end{align*}
whereas any other coordinate contributes
\[
 \int_0^1(h+y)^{\theta-1}\,dy
 =\frac{(h+1)^\theta-h^\theta}{\theta}
 \le\frac{2^\theta}{\theta}.
\]
The integrand is nonnegative, so Tonelli and the finite union give
\[
 \int_{E_h}\prod_i(h+y_i)^{\theta-1}\,dy
 \le\sum_{|I|=m}\prod_{i=1}^n
       \int_{(E_{h,I})_i}(h+y_i)^{\theta-1}\,dy_i
 \le C_\theta h^{m\theta}.
\]
\end{proof}

\subsection{Fourier primitives and estimates for a single shell}
\label{subsec:flat-shell-fourier}

We prove two estimates for $A_j$.  The first controls its unweighted
Fourier coefficients in $\ell^p$ for $p>1$; the second gives decay at
all real frequencies, including the estimate needed for the infinity
endpoint.  The coefficient estimate uses the fixed coordinate localization
and the product-integral lemma.  We first isolate its one-dimensional
input, which also allows the fixed frequency spacings arising from $L$.

For a real frequency $\xi$ and $0\le y\le1$, define
\[
 P_\xi(y)=\int_0^y e^{-2\pi i\xi t}\,dt.
\]

\begin{lemma}[Fourier primitives]\label{lem:fourier-primitives}
For $\xi\ne0$ and $0\le y\le1$,
\begin{equation}\label{eq:fourier-primitive-size}
 |P_\xi(y)|\le\min\{y,(\pi|\xi|)^{-1}\},
\end{equation}
and $P_0(y)=y$.  For every fixed $\nu>0$ and $1<p\le\infty$,
\begin{equation}\label{eq:fourier-primitive-rectangular-lattice}
 \|(P_{\nu k}(y))_{k\in\Z}\|_{\ell^p}
 \le C_{p,\nu}y^{1-1/p},\qquad 0\le y\le1.
\end{equation}
The constant is independent of $y$, and $1/p=0$ when $p=\infty$.
\end{lemma}
\begin{proof}
The integral gives $|P_\xi(y)|\le y$.  For $\xi\ne0$, direct integration
also gives
\[
 P_\xi(y)=\frac{1-e^{-2\pi i\xi y}}{2\pi i\xi},\qquad
 |P_\xi(y)|\le\frac2{2\pi|\xi|}.
\]
This proves \eqref{eq:fourier-primitive-size}.
For $1<p<\infty$ and $y>0$, set $T=(\pi\nu y)^{-1}$.
The function $t\mapsto\min\{y,(\pi\nu t)^{-1}\}^p$ is nonincreasing
on $(0,\infty)$, so each of its values at an integer $k\ge1$ is at most
its integral over $[k-1,k]$.  Separating $k=0$ and splitting that integral
at $T$ gives
\[
\begin{aligned}
 \sum_{k\in\Z}|P_{\nu k}(y)|^p
 &\le y^p+2\sum_{k\ge1}\min\{y,(\pi\nu k)^{-1}\}^p\\
 &\le y^p+2\int_0^T y^p\,dt
             +2(\pi\nu)^{-p}\int_T^\infty t^{-p}\,dt\\
 &=y^p+\frac{2}{\pi\nu}y^{p-1}
            +\frac{2}{\pi\nu(p-1)}y^{p-1}
 \le C_{p,\nu}y^{p-1}.
\end{aligned}
\]
Here $y^p\le y^{p-1}$ because $y\le1$, and the last integral is finite
because $p>1$.  In particular, for integer frequencies,
\begin{equation}\label{eq:fourier-primitive-lp}
 \sum_{k\in\Z}|P_k(y)|^p\le C_p y^{p-1},\qquad 1<p<\infty.
\end{equation}
Taking the $p$th root proves
\eqref{eq:fourier-primitive-rectangular-lattice}.  At $y=0$ all terms
vanish.  For $p=\infty$, the bound follows from
$\sup_k|P_{\nu k}(y)|\le y$.
\end{proof}

We now apply the primitive estimate to each localized shell.  The
vanishing near $y_i=1$ in \Cref{lem:flat-coordinate-localization}
allows us to express the function through its mixed derivative; the
resulting Fourier primitives contribute the powers used in
\Cref{lem:coordinate-weighted-integral}.
For fixed positive spacings $\nu_1,\ldots,\nu_n$, write
$\nu k=(\nu_1k_1,\ldots,\nu_nk_n)$.

\begin{lemma}[Fourier estimates for one shell]\label{lem:flat-shell-estimates}
For the shells in \Cref{lem:flat-shell-decomposition}, every $j\ge3$
and $1<p\le\infty$ satisfy
\begin{equation}\label{eq:flat-direct-shell-lp}
 \|(\widehat A_j(\nu k))_{k\in\Z^n}\|_{\ell^p}
 \le C_{p,\nu}2^{-jm(1-1/p)}.
\end{equation}
For every integer $K\ge0$,
\begin{equation}\label{tm:eq:flat-shell-pointwise}
 |\widehat A_j(\xi)|
 \le C_K2^{-jm}(1+2^{-j}|\xi|)^{-K},\qquad \xi\in\R^n.
\end{equation}
The constants are independent of $j$ and of the frequency.  The first
bound includes the integer frequency lattice when all $\nu_i=1$.
\end{lemma}
\begin{proof}
\medskip\noindent\emph{Step 1: the mixed-derivative identity.}
Let $b=b_{j,z,\epsilon}$ and write
$\partial_{[n]}b=\partial_1\cdots\partial_n b$.
Since $b$ vanishes near $y_i=1$ for every $i$, first integrate in one
coordinate and then repeat in the others:
\[
 b(y)=-\int_{y_1}^1\partial_1b(v_1,y_2,\ldots,y_n)\,dv_1
     =(-1)^n\int_{y_1}^1\cdots\int_{y_n}^1
                        \partial_{[n]}b(v)\,dv_n\cdots dv_1.
\]
The same endpoint vanishing holds for derivatives in the other
coordinates, so the iteration is valid.  Substitute this identity into
the Fourier integral.  In each coordinate the integration region is
$0\le y_i\le v_i\le1$.  Reversing the integral order by Fubini gives,
for every $\xi\in\R^n$,
\begin{align}
 \int_{[0,1]^n}b(y)e^{-2\pi i\xi\cdot y}\,dy
 &=(-1)^n\int_{[0,1]^n}\partial_{[n]}b(v)
       \prod_{i=1}^n\left(\int_0^{v_i}e^{-2\pi i\xi_i t}\,dt\right)dv
       \notag\\
 &=(-1)^n\int_{[0,1]^n}\partial_{[n]}b(v)
                          \prod_{i=1}^nP_{\xi_i}(v_i)\,dv.
 \label{eq:fourier-primitive-identity}
\end{align}
All integrands are bounded on compact sets.  Since $P_{\xi_i}(0)=0$,
the contribution at each inner endpoint is zero.

\medskip\noindent\emph{Step 2: the coefficient norm.}
Recall that $\widehat A_j$ is a finite sum of Fourier integrals of terms $b_{j,z,\epsilon}$.
Fix the signs $\varepsilon_i=(-1)^{\epsilon_i}$ from the coordinate
localization, and evaluate \eqref{eq:fourier-primitive-identity} at
$\xi_i=\varepsilon_i\nu_i k_i$. For a given $v$, the sequence $k\mapsto \prod_{i=1}^n P_{\varepsilon_i \nu_i k_i}(v_i)$ is a product of independent one dimensional sequences. So, for $1<p<\infty$, we have
\[
 \sum_{k\in\Z^n}\Big|\prod_iP_{\varepsilon_i\nu_i k_i}(v_i)\Big|^p
 =\prod_{i=1}^n\left(\sum_{k_i\in\Z}
                         |P_{\varepsilon_i\nu_i k_i}(v_i)|^p\right).
\]
Taking the $p$th root then gives
\[\left\| \prod_{i=1}^n P_{\varepsilon_i\nu_i k_i}(v_i)\right\|_{\ell^p(\Z^n)} = \prod_{i=1}^n \left\| \left(P_{\varepsilon_i \nu_i k_i}(v_i)\right)_{k_i}\right\|_{\ell^p(\Z)}= \prod_{i=1}^n \left\| \left(P_{\nu_i k_i}(v_i)\right)_{k_i}\right\|_{\ell^p(\Z)}\]
with the last equality coming from absorbing the signs $\varepsilon_i$ by reordering the integer summation. 
Minknowski's inequality therefore gives 
\[\left\| \int \partial_{[n]}b(v)\prod_{i=1}^n P_{\varepsilon_i\nu_i k_i}(v_i)\,dv\right\|_{\ell^p(\Z^n)} \le \int|\partial_{[n]}b(v)|\prod_{i=1}^n \left\|\left( P_{\nu_i k_i}(v_i)\right)_{k_i}\right\|_{\ell^p(\Z)}\,dv.\]
Putting $\theta=1-1/p>0$, \Cref{lem:fourier-primitives}, and
\eqref{eq:orthant-shell-derivatives} give
\begin{align*}
 \left\|\left(\int b(y)e^{-2\pi i\sum_i\varepsilon_i\nu_i k_i y_i}\,dy
                  \right)_{k\in\Z^n}\right\|_{\ell^p}
 &\le\int_{E_{2^{-j}}}|\partial_{[n]}b(v)|
                 \prod_i\|(P_{\nu_i k_i}(v_i))_{k_i}\|_{\ell^p}\,dv\\
 &\le C_{p,\nu}\int_{E_{2^{-j}}}
                 \prod_i\frac{v_i^\theta}{2^{-j}+v_i}\,dv\\
 &\le C_{p,\nu}\int_{E_{2^{-j}}}
                 \prod_i(2^{-j}+v_i)^{\theta-1}\,dv\\
 &\le C_{p,\nu}2^{-jm\theta}.
\end{align*}
The mixed derivative is supported in $E_{2^{-j}}$.  The third line uses
$v_i^\theta\le(2^{-j}+v_i)^\theta$, and the last line is
\Cref{lem:coordinate-weighted-integral}. To avoid the explicit checking of the sequence of integrals in the first inequality being in $\ell^p$, Minkowski may first be applied
to finite rectangular sets of frequencies; monotone convergence then
gives the displayed bound for all of $\Z^n$.  The same argument holds for
$p=\infty$: the product of independent coordinate suprema factors,
and the supremum of an integral is at most the integral of the supremum.
Adding the finitely many localized pieces in
\eqref{eq:orthant-fourier-decomposition}, whose translation factors have
modulus one, proves \eqref{eq:flat-direct-shell-lp}.

\medskip\noindent\emph{Step 3: decay at real frequencies.}
The support and derivative bounds in \eqref{tm:eq:flat-shell-basic} give
\[
 \|\partial_i^K A_j\|_1
 \le |\supp A_j|\,\|\partial_i^K A_j\|_\infty
 \le C_K2^{-j(m-K)}.
\]
The $L^1$ bound with $K=0$ controls the regime $2^{-j}|\xi|\le1$ since $|\widehat A_j|\le \|A_j\|_1$.
Otherwise, if $2^{-j}|\xi|>1$, then $|\xi|>1$ and we can choose $i$ with $|\xi_i|\ge|\xi|/\sqrt n$. Integrate the $C_c^\infty$ function $A_j$ by
parts $K$ times with respect to the coordinate $x_i$:
\[
 |\widehat A_j(\xi)|
 \le (2\pi|\xi_i|)^{-K}\|\partial_i^K A_j\|_1
 \le C_K2^{-jm}(2^{-j}|\xi|)^{-K}.
\]
For $2^{-j}|\xi|\le1$ the quantity $1+2^{-j}|\xi|$ is at most two;
for $2^{-j}|\xi|>1$ it is at most $2\cdot2^{-j}|\xi|$.
Enlarging $C_K$ combines the two bounds into
\eqref{tm:eq:flat-shell-pointwise}.
\end{proof}

\subsection{Summing the shells: the flat-window regularity range}

We first sum the unweighted estimates and prove the pointwise infinity
endpoint.  Their combination then gives the polynomially weighted
finite-$p$ range.

\Needspace{12\baselineskip}
\begin{proposition}\label{tm:prop:flat-all-p}
Every compactly localized cross-Zak function $A$ as of the flat
window, as described in \Cref{prop:cross-Zak-characterization}, satisfies
\[
 |\widehat A(\xi)|\le C_A\langle\xi\rangle^{-m},\qquad \xi\in\R^n.
\]
The same flat window satisfies
\[
 \begin{gathered}
 g_{\rm flat}\in M_s^p(\R^d),\qquad
 1\le p<\infty,\quad s<m(1-1/p),\\
 g_{\rm flat}\in M_m^\infty(\R^d).
 \end{gathered}
\]
If $m=n=2d$, write $g_{d-1}=g_{\rm flat}$; then
$g_{d-1}\in M_s^p$ for every $1\le p<\infty$, $s<2d(1-1/p)$,
and $g_{d-1}\in M_{2d}^\infty$.
\end{proposition}
\begin{proof}
\medskip\noindent\emph{Step 1: unweighted coefficient estimates.}
Fix the spacings $\nu_i>0$ in \Cref{lem:flat-shell-estimates}.  For every
$1<p<\infty$, \eqref{eq:flat-direct-shell-lp} gives
\[
 \sum_{j\ge3}\|\bigl(\widehat A_j(\nu k)\bigr)_k\|_{\ell^p}
 \le C_{p,\nu}\sum_{j\ge3}2^{-jm(1-1/p)}<\infty.
\]
Recall from \Cref{lem:flat-shell-decomposition} that the shell series converges in $L^1$, so its Fourier transforms sum to
$\widehat A$ at every real frequency.  Hence
$\bigl(\widehat A(\nu k)\bigr)_k\in\ell^p$ for every $p>1$.

\medskip\noindent\emph{Step 2: the exact pointwise endpoint.}
We will use the following summation also for the phase-adjusted window.
For $\mu>0$, $K>\mu$, and $|\xi|\ge1$, splitting at the index $j$ where $2^{-j}$ and $|\xi|^{-1}$ are closest
gives
\begin{align}
 \sum_{j\ge3}2^{-j\mu}(1+2^{-j}|\xi|)^{-K}
 &\le\sum_{\ 2^{-j}\le|\xi|^{-1}}2^{-j\mu}(1+2^{-j}|\xi|)^{-K}
   +\sum_{\ 2^{-j}>|\xi|^{-1}}2^{-j\mu}(1+2^{-j}|\xi|)^{-K}
      \notag\\
 &\le\sum_{2^j\ge|\xi|}(2^{-\mu})^j
   +|\xi|^{-K}
      \sum_{2^j<|\xi|}2^{j(K-\mu)}
      \notag\\
 &\le C|\xi|^{-\mu}
      +C|\xi|^{-K}|\xi|^{K-\mu}
 \le C|\xi|^{-\mu}.
 \label{eq:dyadic-power-summation}
\end{align}
The first geometric sum is bounded by a constant times its largest
term, corresponding to the smallest admissible exponent $j$; the second is a finite sum bounded by a constant times its last and largest term. Thus the penultimate inequality just comes from the fact that $2^j$ is close to $|\xi|$. 
An empty sum contributes zero.  For $|\xi|<1$ the full sum is bounded
by $\sum_{j\ge3}2^{-j\mu}<\infty$.
Taking $\mu=m$ and using \eqref{tm:eq:flat-shell-pointwise} proves
$|\widehat A(\xi)|\le C\langle\xi\rangle^{-m}$.

\medskip\noindent\emph{Step 3: polynomial weights.}
Write $a_k=\widehat A(\nu k)$.  Since the spacings are fixed and positive,
$\langle\nu k\rangle\asymp\langle k\rangle$.  Steps~1--2 give
\[
 (a_k)_k\in\ell^{p_0}\quad\text{for every }p_0>1,
 \qquad \sup_k\langle k\rangle^m|a_k|<\infty.
\]
Let $1<p<\infty$ and $0\le s<m(1-1/p)$.  Choose
$1<p_0<p-sp/m$; this is possible because $p-sp/m>1$.
Then $p_0<p$ and $sp-m(p-p_0)<0$, so
\begin{align*}
 \sum_k\langle k\rangle^{sp}|a_k|^p
 &=\sum_k(\langle k\rangle^m|a_k|)^{p-p_0}
          \langle k\rangle^{sp-m(p-p_0)}|a_k|^{p_0}\\
 &\le\left(\sup_k\langle k\rangle^m|a_k|\right)^{p-p_0}
          \sum_k|a_k|^{p_0}<\infty.
\end{align*}
For $s<0$ and $p>1$, the result follows from the unweighted estimate
and $\langle k\rangle^s\le1$.

At $p=1$, write $s=-a$ with $a>0$.  Choose $p_0>1$ so close to one
that $a>n(1-1/p_0)$, equivalently $ap_0'>n$.
The unweighted conclusion and H\"older's inequality give
\begin{align*}
 \sum_{k\in\Z^n}\langle k\rangle^{-a}|a_k|
 &\le\left(\sum_k|a_k|^{p_0}\right)^{1/p_0}
       \left(\sum_k\langle k\rangle^{-ap_0'}\right)^{1/p_0'}
 <\infty.
\end{align*}
The second factor is finite since $ap_0'>n$.

\medskip\noindent\emph{Step 4: physical coordinates and Zak inversion.}
If $A(x)=A_{\rm phys}(Lx)$, then
\[
 \widehat{A_{\rm phys}}(k)
 =|\det L|\widehat A(L^Tk),\qquad k\in\Z^n.
\]
The diagonal entries of $L$ are fixed positive numbers, so the
rectangular-lattice estimates just proved apply to $L^T\Z^n$.
By \eqref{eq:compact-cutoff-coefficients}, these are exactly the Fourier
coefficients of $F_{\rm flat}\overline{Z\psi_\ell}$.
Apply this to the finite reference family and use
\Cref{prop:cross-Zak-characterization}.
This proves all the stated modulation-space memberships, including
$M_m^\infty$, for one and the same window.
\end{proof}

\section{\texorpdfstring{Phase shaping: obtaining $g_{\rm phase}$ from $g_{\rm flat}$}{Phase shaping: obtaining g phase from g flat}}
\label{sec:lower-polar}

Retain $1\le m\le n$, the normalized cubical coordinates, and the
periodic scale $\rho$ from \eqref{tm:eq:rho}.  We construct a periodic
unimodular multiplier $U$, smooth off $S$, whose physical-coordinate
version changes only the phase of $F_{\rm flat}$.  The product defines
the Zak transform of $g_{\rm phase}$ and preserves Parsevality.

Phase shaping redistributes coefficients concentrated in a few frequency
directions, reducing their largest values at the cost of spreading them
over more positions. This can be illustrated with a bivariate toy model.  For $k\in\Z$ and integers $\ell\ge1$, compare
\[
 a_{k,\ell}=\mathbf1_{\{k=0\}}\ell^{-1},\qquad
 b_{k,\ell}=\mathbf1_{\{0\le k<\ell\}}\ell^{-3/2}.
\]
The second sequence spreads each coefficient over $\ell$ positions while
preserving the energy of each row:
$\sum_k|a_{k,\ell}|^2=\ell^{-2}=\sum_k|b_{k,\ell}|^2$.
Since $\langle(k,\ell)\rangle\asymp\ell$ on both supports, for finite $p$,
\[
\begin{aligned}
 \|a\|_{\ell_s^p}^p&\asymp\sum_{\ell\ge1}\ell^{p(s-1)}<\infty
     &&\Longleftrightarrow\quad s<1-1/p,\\
 \|b\|_{\ell_s^p}^p&\asymp\sum_{\ell\ge1}\ell^{1+p(s-3/2)}<\infty
     &&\Longleftrightarrow\quad s<3/2-2/p.
\end{aligned}
\]
For $p>2$, the smaller coefficients outweigh their greater number.  Both
sequences are already in $\ell^2$; the gain is in the admissible weight.
This is just a toy model that illustrates the redistribution. The main point of this chapter is to explicitly construct a phase shaping unimodular multiplier.

In our setting, \Cref{sec:lower-completion} proves that every required
compactly localized cross-Zak function $A$ satisfies
\[
 |\widehat{UA}(\xi)|\le C_A\langle\xi\rangle^{-(n+m)/2},\qquad \xi\in\R^n.
\]
The resulting range $s<(n+m)/2-n/p$ improves the flat range
$s<m(1-1/p)$ by $(n-m)(1/2-1/p)$, which is positive for $m<n$ and $p>2$.

\subsection{Rudin--Shapiro signs and a smooth cell phase}
\label{subsec:phase-signs}

The interval bound below supplies square-root cancellation uniformly in
the Fourier frequency and the interval of cell indices.

Define the signs by
\begin{equation}\label{rs:eq:signs}
 r_0=1,\qquad r_{2k}=r_k,\qquad
 r_{2k+1}=(-1)^kr_k,\qquad k\ge0.
\end{equation}
Its first ten terms are
\[
 (r_0,\ldots,r_9)=(1,1,1,-1,1,1,-1,1,1,1).
\]
\begin{lemma}[Balister]\label{lem:RS-interval}
For every finite interval $I$ of nonnegative integers,
\begin{equation}\label{rs:eq:interval}
 \sup_{\omega\in\R}
 \Big|\sum_{k\in I}r_ke^{i\omega k}\Big|
 \le C\sqrt{|I|}.
\end{equation}
The same bound holds for shifted or reversed intervals and after
multiplying the terms by $(-1)^k$.
\end{lemma}
\begin{proof}
The displayed inequality is \cite[Theorem~3]{Balister2019}.
Shifting the interval changes a constant phase, reversal replaces
$\omega$ by $-\omega$, and multiplication by $(-1)^k$ replaces
$\omega$ by $\omega+\pi$.
\end{proof}

We interpolate consecutive signs by smooth unimodular functions whose
formulas agree at the next dyadic resolution.  This agreement makes
the phase smooth when the resolution changes with $\rho$.

Use the fixed smooth step $\varphi$ from
\eqref{eq:fixed-smooth-step}.  For a cell index $k\ge0$ define
\begin{equation}\label{eq:phase-cell-sign-data}
 b_k=\frac{1-(-1)^k}{2},\qquad
 \delta_k=\frac{(-1)^k(1-r_kr_{k+1})}{2}.
\end{equation}
Here $b_k\in\{0,1\}$ and $\delta_k\in\{-1,0,1\}$.
For $0\le t,a\le1$ put
\begin{equation}\label{eq:phase-two-profiles}
 \begin{split}
 T(t,a)&=(1-a)\varphi(t)+a\varphi(2t-1),\\
 E(t,a)&=a\bigl(\varphi(2t-1)-\varphi(2t)\bigr).
 \end{split}
\end{equation}
At resolution $j\ge0$, write $x=2^{-j}(k+t)$, where
$0\le k<2^j$ and $0\le t\le1$, and define
\begin{equation}\label{eq:phase-local-cell}
 V_{j,a}(x)
 =r_k\exp\!\left(i\pi(-1)^j
             [\delta_kT(t,a)+b_kE(t,a)]\right).
\end{equation}
The formulas at a shared endpoint agree, as the following calculation
shows.  We extend the resulting function periodically in $x$.

\begin{lemma}\label{lem:local-phase-refinement}
The function $(x,a)\mapsto V_{j,a}(x)$ is smooth on
$\R\times[0,1]$, periodic in $x$, and has modulus one.  Moreover,
\begin{equation}\label{eq:phase-exact-refinement}
 V_{j,1}(x)=V_{j+1,0}(x),\qquad x\in\R.
\end{equation}
\end{lemma}
\begin{proof}
Near $t=0$, the three values of $\varphi$ in \eqref{eq:phase-two-profiles}
are zero.  Thus $T=E=0$ and $V_{j,a}=r_k$ there.
Near $t=1$, they are all one, so $T=1$, $E=0$, and
\[
 V_{j,a}=r_ke^{i\pi(-1)^j\delta_k}=r_{k+1}.
\]
Indeed, if $r_kr_{k+1}=1$, then $\delta_k=0$; if their product is
$-1$, then $\delta_k=(-1)^k=\pm1$, whose exponential is $-1$.
These endpoint values are independent of $a$, and the formulas are
constant in $t$ on endpoint collars.  All derivatives in $x$ and $a$
therefore match across a cell boundary.  The recursion gives
$r_{2^j}=r_1=r_0=1$, so they also match at the ends of the unit
period.  Smoothness follows and unimodularity is immediate.

For refinement, consider a scale $j$ and fix a cell index $k$.
The sign recursion gives
\begin{align*}
 r_{2k}&=r_k,\qquad r_{2k+1}=(-1)^kr_k,\qquad r_{2k+2}=r_{k+1},\\
 \delta_{2k}
 &=\frac{1-r_{2k}r_{2k+1}}2
   =\frac{1-(-1)^k}2=b_k,\\
 \delta_{2k+1}
 &=-\frac{1-r_{2k+1}r_{2k+2}}2
   =-\frac{1-(-1)^kr_kr_{k+1}}2
   =-(\delta_k+b_k).
\end{align*}
The last equality uses
$\delta_k+b_k=[(-1)^k(1-r_kr_{k+1})+1-(-1)^k]/2
=(1-(-1)^kr_kr_{k+1})/2$.
At $a=1$, the exponent inside the brackets in
\eqref{eq:phase-local-cell} is
\[
 \delta_k\varphi(2t-1)+b_k\bigl(\varphi(2t-1)-\varphi(2t)\bigr).
\]
For $0\le t\le1/2$, $\varphi(2t-1)=0$, and hence
\begin{align*}
 V_{j,1}(2^{-j}(k+t))&=r_ke^{-(-1)^j\pi i b_k\varphi(2t)}
 =r_{2k}e^{(-1)^{j+1}\pi i\delta_{2k}\varphi(2t)}\\
 &=V_{j+1,0}(2^{-(j+1)}(2k+2t))=V_{j+1,0}(2^{-j}(k+t)).
\end{align*}
For $1/2\le t\le1$, $\varphi(2t)=1$.  Since
$e^{(-1)^{j+1}\pi i b_k}=(-1)^k$, the other half gives
\begin{align*}
 V_{j,1}(2^{-j}(k+t))
 &=r_ke^{(-1)^{j+1}\pi i b_k}
          e^{(-1)^j\pi i (\delta_k+b_k)\varphi(2t-1)}
 =r_{2k+1}e^{(-1)^{j+1}\pi i\delta_{2k+1}\varphi(2t-1)}\\
 &=V_{j+1,0}(2^{-(j+1)}(2k+1+2t-1))=V_{j+1,0}(2^{-j}(k+t)).
\end{align*}
These are the two fine cells, and the endpoint agreement already
proved completes \eqref{eq:phase-exact-refinement}.
\end{proof}

Figure~\ref{rs:fig:phases} illustrates adjacent resolutions and their
midpoint interpolation using continuous arguments of $V_{j,a}$.
\begin{figure}[tbp]
\centering
\includegraphics[width=.485\linewidth]{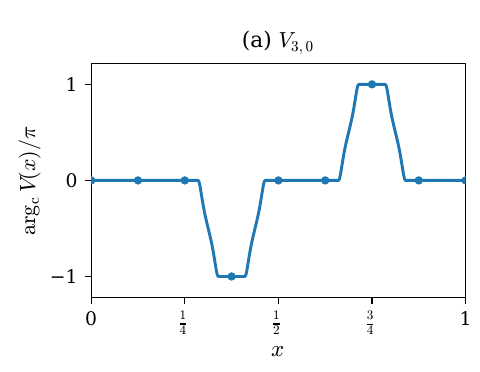}\hfill
\includegraphics[width=.485\linewidth]{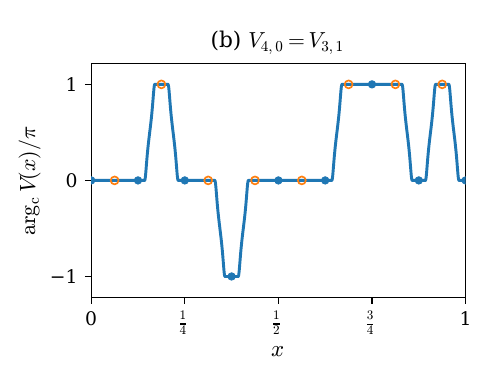}
\par\smallskip
\includegraphics[width=.485\linewidth]{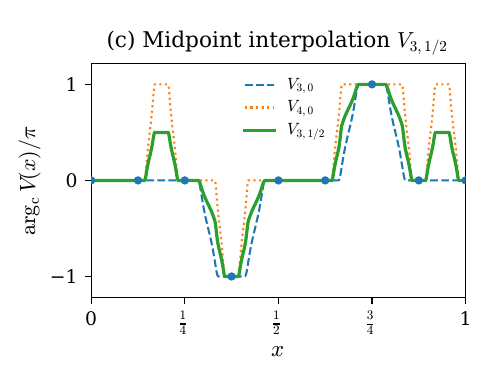}\hfill
\includegraphics[width=.485\linewidth]{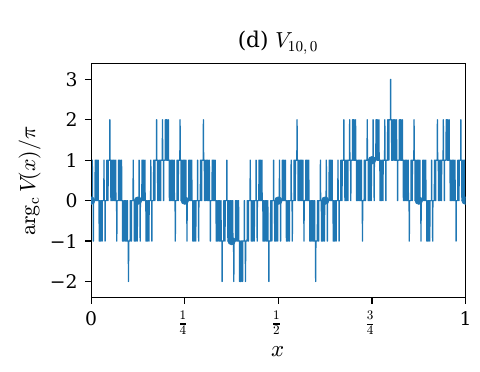}
\caption{Adjacent resolutions and their midpoint interpolation in
\eqref{eq:phase-local-cell}.  Here $\arg_{\mathrm c}V_{j,a}$ is the continuous
argument normalized to vanish at $x=0$, not the principal argument.
The upper panels show $V_{3,0}$ and $V_{4,0}=V_{3,1}$; the lower left also
shows $V_{3,1/2}$.  Filled dots mark the shared values at $x=k/8$;
open circles mark the additional nodes at $x=k/16$ for odd $k$.
The lower-right panel shows $V_{10,0}$, whose continuous argument ranges
from $-2\pi$ to $3\pi$.  All arguments are divided by $\pi$.}
\label{rs:fig:phases}
\end{figure}

We now let the resolution depend on the distance to $S$.  The following
lemma verifies that the changes of the integer index do not interrupt
smoothness.  Off $S$, define
\begin{equation}\label{rs:eq:global-phase}
 \begin{split}
 \tau(x)&=-\log_2\rho(x),\qquad j(x)=\lfloor\tau(x)\rfloor,\\
 a(x)&=\varphi(\tau(x)-j(x)),\qquad
 U(x)=\prod_{i=1}^n V_{j(x),a(x)}(x_i).
 \end{split}
\end{equation}
Set $U=1$ on $S$.

\Needspace{7\baselineskip}
\begin{lemma}[The global phase]\label{lem:global-phase-properties}
The function $U$ defined by \eqref{rs:eq:global-phase} is
$\Z^n$-periodic, has modulus one, and is smooth on $\R^n\setminus S$.
It depends only on $\rho$, the fixed step $\varphi$, and the sign sequence.
\end{lemma}
\begin{proof}
Since $\rho\le1/8$, the active index is at least three.  An integer
translation leaves $\rho$, and hence $\tau,j(x),a(x)$, unchanged;
each $V_{j,a}$ is periodic by \Cref{lem:local-phase-refinement}.
This proves periodicity off $S$, and the invariant assignment $U=1$ on
$S$ preserves it.  The factors have modulus one.
On each open region $j<\tau(x)<j+1$, the index is fixed and the
formula is a composition of smooth functions.  At a switching level,
if $j-1/4\le\tau(x)<j$, the active pair has parameter
$\varphi(\tau-j+1)=1$; if $j\le\tau(x)\le j+1/4$, its parameter is
$\varphi(\tau-j)=0$.  By \eqref{eq:phase-exact-refinement}, both formulas
therefore give
\[
 U(x)=\prod_{i=1}^nV_{j,0}(x_i),\qquad |\tau(x)-j|\le1/4.
\]
Thus the definitions agree on an open neighborhood of every switching
level, which proves smoothness throughout $\R^n\setminus S$.
\end{proof}

\begin{figure}[!tbp]
\centering
\includegraphics[width=.86\textwidth]{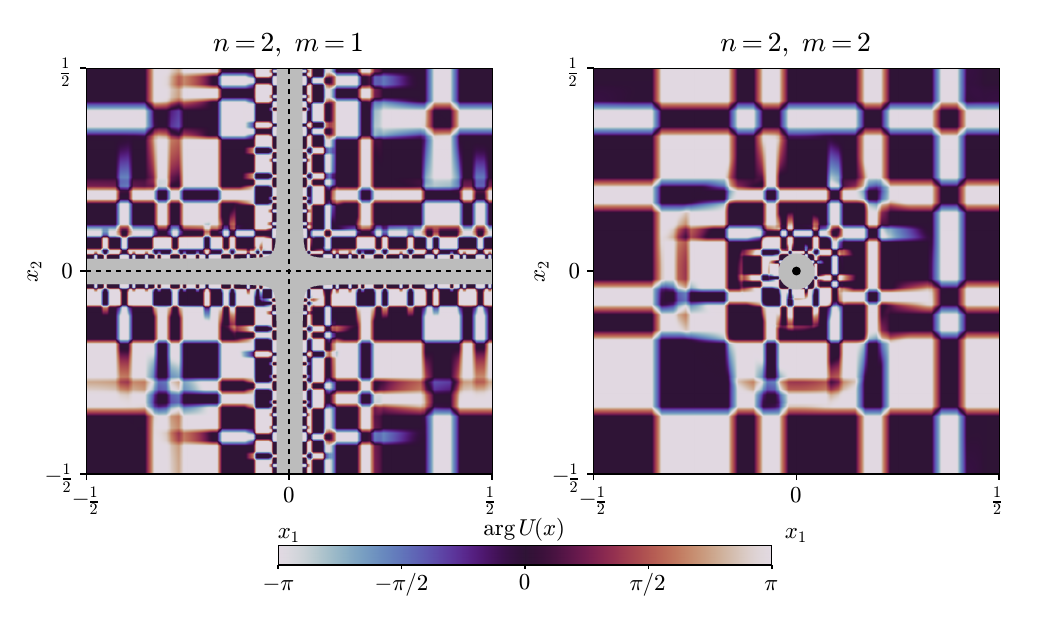}
\caption{The phase of $U$ for $n=2$, with $m=1$ (left) and $m=2$
(right). Black dashed lines or the black point mark $S$.
The cyclic color scale identifies $-\pi$ and $\pi$, which represent
the same value of $U$. The gray region $\rho<2^{-7}$ is omitted
from this finite-resolution display: progressively finer oscillations
occur on approaching $S$. No cutoff is applied to the function $U$
itself, and its assigned value $U=1$ on $S$ is not represented by the
black markers.}
\label{rs:fig:global-phase}
\end{figure}

\subsection{Fixed-resolution phase profiles on each shell}

The cancellation argument in \Cref{sec:lower-completion} needs a fixed
cell grid on each shell.  It also needs to distinguish the coefficients
that cancel from the smooth factors that vary between cells.  The next
lemma supplies both: one smooth profile expansion on the whole shell
and an interval bound for each of its coefficient sequences.

For $j\ge3$, put $a_j(x)=\varphi(\tau(x)-j)$, where
$\tau(x)=-\log_2\rho(x)$.  For $\epsilon,e\in\{0,1\}$ and
$\sigma\in\{-1,1\}$, define
\[
 \Phi_{\epsilon,e}^{\sigma}(t,a)
 =\exp\!\left(i\pi\sigma
       [\epsilon E(t,a)+e(-1)^\epsilon T(t,a)]\right),
 \qquad (t,a)\in[0,1]^2.
\]

\Needspace{16\baselineskip}
\begin{lemma}[Fixed-resolution profiles and interval cancellation]
\label{lem:fixed-shell-phase-profiles}
On a neighborhood of $\supp A_j$,
\begin{equation}\label{eq:phase-aligned-shell}
 U(x)=\prod_{i=1}^nV_{j,a_j(x)}(x_i),\qquad j\ge3.
\end{equation}
All derivatives of the finitely many profiles $\Phi$ are uniformly
bounded on $[0,1]^2$.  Write $x=z+2^{-j}(k+t)$, where $z\in\Z^n$,
$k\in\{0,\ldots,2^j-1\}^n$, and $t\in[0,1]^n$.
If $x\in\supp A_j$ and $\epsilon_i$ is the parity of $k_i$, then
\begin{equation}\label{eq:phase-tensor-cell-expansion}
 U(x)=\sum_{e\in\{0,1\}^n}
       \left(\prod_{i=1}^n\frac{r_{k_i}+(-1)^{e_i}r_{k_i+1}}2\right)
       \prod_{i=1}^n\Phi_{\epsilon_i,e_i}^{(-1)^j}(t_i,a_j(x)).
\end{equation}
For every finite interval $I$ of nonnegative integers, the coefficients
in this expansion satisfy
\begin{equation}\label{eq:phase-parity-interval}
 \sup_{\omega\in\R}
 \Big|\sum_{\substack{k\in I\\ k\equiv\epsilon\;({\rm mod}\,2)}}
       \frac{r_k+(-1)^e r_{k+1}}2 e^{i\omega k}\Big|
 \le C\sqrt{|I|},\qquad \epsilon,e\in\{0,1\}.
\end{equation}
The constant is independent of the interval, the frequency, and the
parity and profile choices.
\end{lemma}
\begin{proof}
\medskip\noindent\emph{Step 1: fixing the resolution on a shell.}
By \Cref{lem:flat-shell-decomposition}, on $\supp A_j$ we have
$j-1/8\le\tau\le j+1+1/8$.
For $j\le\tau\le j+1$, $j(x)$ is natively just $j$, so \eqref{eq:phase-aligned-shell} is already the defining
formula for $U$.  On the lower overlap $j-1/8\le\tau<j$, 
we have $j(x)=j-1$ and thus $\varphi(\tau-(j-1))=1$. By \eqref{eq:phase-exact-refinement}, we have $V_{j-1,1}=V_{j,0}$. Additionally, in this range $\varphi(\tau(x)-j)=0$, so $a_j=0$. 
This simply means that on the lower overlap, $U(x)=\prod_{i=1}^n V_{j,0}(x_i)$ holds throughout. 
On the upper overlap, $j+1<\tau\le j+1+1/8$, the argument is symmetric. $\varphi(\tau-(j+1))=0$ and $a_j=1$, so
$U(x)=\prod_{i=1}^n V_{j,1}(x_i)$ holds throughout. 
In fact, the shell overlap width $1/8$ is strictly smaller than the constant collar width $1/4$ of $\varphi$.  The same identities therefore hold on an open neighborhood of the shell support, so they
may also be used when differentiating the shell product.

\medskip\noindent\emph{Step 2: separating signs and profiles.}
Put $\sigma=(-1)^j$.  Definition \eqref{eq:phase-local-cell} can 
be rewritten as
\begin{equation}\label{eq:phase-two-case-expansion}
\begin{aligned}
 V_{j,a}(2^{-j}(k+t))
 &=\frac{r_k+r_{k+1}}2\,e^{i\pi\sigma b_kE(t,a)}\\
 &\quad+\frac{r_k-r_{k+1}}2\,
      e^{i\pi\sigma[b_kE(t,a)+(-1)^kT(t,a)]}.
\end{aligned}
\end{equation}
Indeed, if $r_{k+1}=r_k$, the second coefficient is zero and $\delta_k=0$, so only the first term survives. If $r_{k+1}=-r_k$, the
first coefficient is zero and $\delta_k=(-1)^k$, so the second term is
the defining formula.

On cells of fixed parity $\epsilon$, we have $(-1)^k=(-1)^\epsilon$ and
$b_k=\epsilon$.  Thus
\[
 V_{j,a}(2^{-j}(k+t))
 =\sum_{e=0}^1\frac{r_k+(-1)^e r_{k+1}}2
                    \Phi_{\epsilon,e}^{\sigma}(t,a).
\]
The functions $T,E$ are smooth and the indices range over finite sets,
so every derivative of $\Phi$ is bounded on the compact square.
Taking the product of the one-dimensional identities and using Step~1
and periodicity gives \eqref{eq:phase-tensor-cell-expansion}.
There are $2^n$ parity classes and $2^n$ profile terms on each class,
hence $4^n$ combinations, independently of $j$.

\medskip\noindent\emph{Step 3: cancellation on a fixed parity class.}
The sign recursion gives
\[
\begin{array}{c|cc}
 &r_k&r_{k+1}\\ \hline
 k=2\ell&r_\ell&(-1)^\ell r_\ell\\
 k=2\ell+1&(-1)^\ell r_\ell&r_{\ell+1}
\end{array}
\]
\Needspace{10\baselineskip}
For an interval $I$ of nonnegative integers, put
$J_\epsilon=\{\ell\ge0:2\ell+\epsilon\in I\}$.
Each $J_\epsilon$ is an interval, possibly empty, with
$|J_\epsilon|\le|I|$.  Substituting the table into the four parity sums gives
\[
\begin{aligned}
 \sum_{\substack{k\in I\\ k\equiv0\;({\rm mod}\,2)}}r_ke^{i\omega k}
   &=\sum_{\ell\in J_0}r_\ell e^{2i\omega\ell},\\
 \sum_{\substack{k\in I\\ k\equiv0\;({\rm mod}\,2)}}r_{k+1}e^{i\omega k}
   &=\sum_{\ell\in J_0}r_\ell e^{i(2\omega+\pi)\ell},\\
 \sum_{\substack{k\in I\\ k\equiv1\;({\rm mod}\,2)}}r_ke^{i\omega k}
   &=e^{i\omega}\sum_{\ell\in J_1}r_\ell e^{i(2\omega+\pi)\ell},\\
 \sum_{\substack{k\in I\\ k\equiv1\;({\rm mod}\,2)}}r_{k+1}e^{i\omega k}
   &=e^{i\omega}\sum_{\ell\in J_1}r_{\ell+1}e^{2i\omega\ell}.
\end{aligned}
\]
In the last line, replacing $\ell+1$ by $q$ produces the constant phase
$e^{-i\omega}$ times $\sum_{q\in J_1+1}r_qe^{2i\omega q}$.
Every right-hand side therefore has modulus at most $C\sqrt{|I|}$
by \eqref{rs:eq:interval}.  Taking the half-sum or half-difference
of the two expressions for each parity proves
\eqref{eq:phase-parity-interval}.  Empty intervals give zero.
\end{proof}

\section{Fourier cancellation and regularity of the phase-shaped window}
\label{sec:lower-completion}

We prove the Fourier cancellation estimate for the phase constructed in
\Cref{sec:lower-polar}.  Cumulative sums and the weighted coordinate
integral from \Cref{sec:elementary-perturbation} combine the exceptional-set
geometry with square-root cancellation.  We then pass to physical Zak
coordinates and obtain the improved high-$p$ regularity of $g_{\rm phase}$.

\subsection{Cumulative-sum estimates and phase-shaped Fourier decay}

The first lemma converts one-dimensional cumulative-sum bounds into a
bound for a cell sum with a varying smooth amplitude.  Its proof uses
the same mixed-derivative identity as
\eqref{eq:fourier-primitive-identity}.  We then verify its hypotheses
for the profiles from \Cref{lem:fixed-shell-phase-profiles}, multiply by
the cell Jacobian, and control scaled derivatives to obtain Fourier decay.

\Needspace{17\baselineskip}
\begin{lemma}[Cumulative sums with a smooth amplitude]
\label{lem:weighted-cumulative-sums}
Fix $j\ge3$ and put $h=2^{-j}$.  For $1\le i\le n$ and
$0\le k<2^j$, let $u_{i,k}\in\C$, and define on $[0,1]^2$ the function
\begin{equation}\label{eq:phase-cumulative-sums}
 \Sigma_i(v_i,t_i)
 =\sum_{\substack{0\le k<2^j\\h(k+t_i)\le v_i}}u_{i,k}.
\end{equation}
Suppose that, for $0\le v_i,t_i\le1$,
\begin{equation}\label{eq:phase-prefix-bound}
 |\Sigma_i(v_i,t_i)|\le C_u h^{-1/2}(h+v_i)^{1/2}.
\end{equation}
Let $a(y,t)$ be smooth for $y,t\in[0,1]^n$, satisfying $a(y,t)=0$ whenever $y_i\ge3/4$ for some $i$ or whenever $y\notin E_h$ (recall that $E_h$ was the set of points $y\in[0,1]^n$ for which at least $m$ coordinates were at most a fixed constant times $h$).
Write $\partial_{[n]}=\partial_{y_1}\cdots\partial_{y_n}$ and assume
\[
 |\partial_{[n]}a(y,t)|\le C_a\prod_i(h+y_i)^{-1}.
\]
Then, uniformly for $t\in[0,1]^n$,
\begin{equation}\label{eq:phase-whole-shell-sum}
 \Big|\sum_{k\in\{0,\ldots,2^j-1\}^n}
        a(h(k+t),t)\prod_i u_{i,k_i}\Big|
 \le C h^{-(n-m)/2}.
\end{equation}
The constant depends only on $C_a,C_u,n,m$, and the fixed constant in
$E_h$, not on $j,t$, or the individual coefficients $u_{i,k}$.
\end{lemma}
\begin{proof}
Fix $t$.  Since $a(y,t)$ vanishes near every endpoint $y_i=1$,
repeated application of the fundamental theorem of calculus gives
\[
 a(y,t)=(-1)^n\int_{y_1}^1\cdots\int_{y_n}^1
                     \partial_{[n]}a(v,t)\,dv_n\cdots dv_1.
\]
Insert $y=h(k+t)$ and interchange the finite sum with the integral.
\begin{align}
 &\sum_{k\in\{0,\ldots,2^j-1\}^n}
       a(h(k+t),t)\prod_i u_{i,k_i}\notag\\
 &\quad=(-1)^n\int_{[0,1]^n}\partial_{[n]}a(v,t)
       \sum_{k\in\{0,\ldots,2^j-1\}^n}
         \prod_i\bigl(u_{i,k_i}\mathbf1_{\{h(k_i+t_i)\le v_i\}}\bigr)\,dv\notag\\
 &\quad=(-1)^n\int_{[0,1]^n}\partial_{[n]}a(v,t)
                                  \prod_i\Sigma_i(v_i,t_i)\,dv,
 \label{eq:phase-cumulative-identity}
\end{align}
where the last equality is due to the fact that the restrictions on the $k_i$ sum indices are independent from each other, allowing for the swap of summation and product. 
Since $a$ is smooth on the compact cube, all integrals are finite.
By assumption, $\partial_{[n]}a$ is supported in $E_h$.  Apply
\eqref{eq:phase-prefix-bound} and the assumed mixed-derivative bound:
\begin{align*}
 &\Big|\sum_{k\in\{0,\ldots,2^j-1\}^n}
                  a(h(k+t),t)\prod_i u_{i,k_i}\Big|\\*
 &\quad\le C h^{-n/2}\int_{E_h}|\partial_{[n]}a(v,t)|
                                    \prod_i(h+v_i)^{1/2}\,dv\\
 &\quad\le C h^{-n/2}\int_{E_h}\prod_i(h+v_i)^{-1/2}\,dv
 \le C h^{-n/2}h^{m/2}=C h^{-(n-m)/2}.
\end{align*}
The last inequality is \Cref{lem:coordinate-weighted-integral} with
$\theta=1/2$.  The factorization in
\eqref{eq:phase-cumulative-identity} concerns the coefficient sequences
and their index restrictions; the function $a$ itself may well depend jointly on all coordinates.
\end{proof}

We now apply this lemma to the actual phase coefficients.  The next
proposition is the real-frequency estimate used for the physical-coordinate
transfer and the modulation-space conclusion.

\begin{proposition}\label{prop:deterministic-phase}
Let $A$ be a compactly localized cross-Zak function of the flat
construction, or any bounded compactly supported function smooth off
$S$ and satisfying \eqref{eq:normalized-flat-symbol} for every multi-index.
The single phase \eqref{rs:eq:global-phase} satisfies
\begin{equation}\label{rs:eq:main-estimate}
 |\widehat{UA}(\xi)|\le C_A\langle\xi\rangle^{-(n+m)/2},
 \qquad\xi\in\R^n.
\end{equation}
\end{proposition}
\begin{proof}
Lemma~\ref{lem:flat-shell-decomposition} applies to $A$.  Use its shell
decomposition, fix $j\ge3$, and put $h=2^{-j}$.
We prove, with constants independent of $j$ and $\xi$, that
\begin{equation}\label{rs:eq:shell-bound}
 |\widehat{UA_j}(\xi)|
 \le C_K h^{(n+m)/2}(1+h|\xi|)^{-K}
 \quad\text{for every integer }K\ge0.
\end{equation}

\medskip\noindent\emph{Step 1: physical-space localization and the smooth amplitudes.}
Recall the function $\omega(t)=1-\phi(t)$ introduced immediately before
\eqref{eq:orthant-coordinate-map}.  It satisfies
\[
 \omega(t)=1\quad(0\le t\le1/4),\qquad
 \omega(t)=0\quad(t\ge3/4),\qquad
 \omega(t)+\omega(1-t)=1\quad(0\le t\le1).
\]
Consequently, for $x\in z+[0,1]^n$,
\begin{equation}\label{eq:phase-coordinate-partition}
 1=\prod_{i=1}^n\bigl[\omega(x_i-z_i)+\omega(z_i+1-x_i)\bigr]
  =\sum_{\epsilon\in\{0,1\}^n}
    \prod_{i=1}^n
    \omega\bigl((-1)^{\epsilon_i}(x_i-z_i-\epsilon_i)\bigr).
\end{equation}
This is a pointwise partition of unity in the physical variable $x$.
Although \Cref{lem:flat-coordinate-localization} used the same partition
inside the Fourier integral for $A_j$, here we first apply it directly to
the physical-space product $UA_j$.

Fix one unit cube $z+[0,1]^n$ and one coordinate-localization choice
$\epsilon\in\{0,1\}^n$.  Put
\[
 \varepsilon_i=(-1)^{\epsilon_i},\qquad
 x=x_{z,\epsilon}(y)=z+\epsilon+\varepsilon\cdot y,
 \qquad y\in[0,1]^n.
\]
After this change of variables, the corresponding localized physical-space
piece is
\begin{equation}\label{eq:phase-localized-physical-piece}
 U(x_{z,\epsilon}(y))A_j(x_{z,\epsilon}(y))
                       \prod_{i=1}^n\omega(y_i).
\end{equation}

Subdivide the $y$-cube into dyadic cells of side $h$ by writing
\[
 y_i=h(k_i+t_i),\qquad 0\le k_i<2^j,\qquad 0\le t_i\le1.
\]
In a coordinate measured from the left endpoint, the phase cell has index
$k_i$ and local coordinate $t_i$.  In a reflected coordinate,
\[
 1-h(k_i+t_i)=h\bigl(2^j-1-k_i+(1-t_i)\bigr).
\]
Accordingly, write
\[
 \kappa_i(k)=
 \begin{cases}k,&\epsilon_i=0,\\2^j-1-k,&\epsilon_i=1,\end{cases}
 \qquad
 t_i'=\begin{cases}t_i,&\epsilon_i=0,\\1-t_i,&\epsilon_i=1.\end{cases}
\]
Here $\epsilon$ records the coordinate orientation, whereas the symbols
$\delta_i$ below record parities.  Formula
\eqref{eq:phase-tensor-cell-expansion}, applied with cell index
$\kappa_i(k_i)$ and local coordinate $t_i'$, gives
\begin{align}
 &U(x_{z,\epsilon}(h(k+t)))\notag\\
 &\quad=
 \sum_{\delta,e\in\{0,1\}^n}
 \left[
  \prod_{i=1}^n
  \mathbf1_{\{\kappa_i(k_i)\equiv\delta_i\;({\rm mod}\,2)\}}
  \frac{r_{\kappa_i(k_i)}+(-1)^{e_i}r_{\kappa_i(k_i)+1}}2
 \right]\notag\\
 &\qquad\qquad\times
 \prod_{i=1}^n
 \Phi_{\delta_i,e_i}^{(-1)^j}
 \bigl(t_i',a_j(x_{z,\epsilon}(h(k+t)))\bigr).
 \label{eq:phase-localized-profile-expansion}
\end{align}
For each fixed parity vector $\delta$ and profile vector $e$, multiply the
smooth profile factor in \eqref{eq:phase-localized-profile-expansion} by
the localized shell factor in
\eqref{eq:phase-localized-physical-piece}.  This defines
\begin{equation}\label{eq:phase-slow-fast-amplitude}
 a(y,t)=A_j(x_{z,\epsilon}(y))\prod_i\omega(y_i)
        \prod_i\Phi_{\delta_i,e_i}^{(-1)^j}
                    (t_i',a_j(x_{z,\epsilon}(y))).
\end{equation}
Thus $a$ is the smooth physical-space amplitude of one fixed parity and
profile term.  Note that at $y=h(k+t)$, the corresponding complete physical-space
summand that appears in $UA_j$ is
\begin{equation}\label{eq:phase-complete-local-summand}
 a(h(k+t),t)
 \prod_{i=1}^n
 \left[
  \mathbf1_{\{\kappa_i(k_i)\equiv\delta_i\;({\rm mod}\,2)\}}
  \frac{r_{\kappa_i(k_i)}+(-1)^{e_i}r_{\kappa_i(k_i)+1}}2
 \right].
\end{equation}
The discrete factors in brackets are deliberately kept outside $a$:
for $a$ we want to exploit smoothness while the coefficients will provide the Rudin--Shapiro cancellation.  While studying the 
differentiability bounds for $a$, we regard $y$ and
$t$ as independent variables; only afterward do we evaluate at $y=h(k+t)$.

The shell cutoff makes the amplitude smooth after zero extension near
$S$.  It vanishes when any $y_i\ge3/4$ and is supported, uniformly in
$t$, in $E_h$ from \eqref{eq:coordinate-shell-region}.
On the portion of the shell where $0\le y_i\le3/4$, the logarithmic
derivative estimate in the proof of \Cref{lem:flat-shell-decomposition}
and the comparison $h+d_i(x_{z,\epsilon}(y))\asymp h+y_i$ (recall 
that when $\omega(y_i)\neq 0$ we have $0\le y_i\le \frac{3}{4}$, 
so $\frac{y_i}{3}\le \min\{y_i,1-y_i\}\le y_i$) give
\[
 |\partial_y^\gamma a_j(x_{z,\epsilon}(y))|
 \le C_\gamma\prod_i(h+y_i)^{-\gamma_i},\qquad |\gamma|>0.
\]
All derivatives of the profiles $\Phi$ are uniformly bounded.
For example, a $y_i$ derivative of a profile differentiates its second
argument and contributes a derivative of $a_j\circ x_{z,\epsilon}$; a
$t_i$ derivative differentiates its first argument, with a possible minus
sign when $t_i'=1-t_i$.  In higher derivatives, a typical chain-rule term
from one profile has the form
\[
 C(\partial_1^{\nu_i}\partial_2^v\Phi_{\delta_i,e_i}^{(-1)^j})
     (t_i',a_j(x_{z,\epsilon}(y)))
   \prod_{\ell=1}^v
     \partial_y^{\gamma_\ell}(a_j\circ x_{z,\epsilon})(y),
\]
where the nonzero $\gamma_\ell$ add to the $y$-multi-index assigned to
that profile.  The profile derivative is bounded, and the remaining
factors contribute
\[
 \prod_\ell\prod_s(h+y_s)^{-(\gamma_\ell)_s}
 =\prod_s(h+y_s)^{-\sum_\ell(\gamma_\ell)_s}.
\]
The same multiplication of coordinate weights applies to
$A_j\circ x_{z,\epsilon}$ and the factors $\omega(y_i)$ under Leibniz'
rule.  Since all assigned $y$-multi-indices add to the total multi-index
$\gamma$, we obtain
\begin{equation}\label{eq:phase-weighted-amplitude-class}
 |\partial_y^\gamma\partial_t^\nu a(y,t)|
 \le C_{\gamma,\nu}\prod_i(h+y_i)^{-\gamma_i},
 \qquad y,t\in[0,1]^n.
\end{equation}
The constants do not depend on $j$ or on the finitely many localization,
parity, and profile choices.

\medskip\noindent\emph{Step 2: the coefficient sequences and their cumulative sums.}
The support, endpoint vanishing, and mixed-derivative assumptions of
\Cref{lem:weighted-cumulative-sums} follow from Step~1.  It remains to
identify the one-dimensional coefficient sequences and to verify their
cumulative-sum bound.

We first separate the several indices used below.  The number
$\epsilon_i\in\{0,1\}$ records the orientation of the coordinate
localization, and
\[
 \varepsilon_i=(-1)^{\epsilon_i}\in\{1,-1\}
\]
is its associated sign.  The actual phase-cell index is
\[
 \kappa_i(k)=
 \begin{cases}
  k,&\epsilon_i=0,\\
  2^j-1-k,&\epsilon_i=1.
 \end{cases}
\]
The number $\delta_i\in\{0,1\}$ is the fixed parity of this phase-cell
index, and $e_i\in\{0,1\}$ is the fixed profile choice.  Thus, before the
Fourier factor is inserted, the discrete coefficient in coordinate $i$
is
\[
 \mathbf1_{\{\kappa_i(k)\equiv\delta_i\;({\rm mod}\,2)\}}
 \frac{r_{\kappa_i(k)}+(-1)^{e_i}r_{\kappa_i(k)+1}}2.
\]
The indicator selects the parity class for which the profile
$\Phi_{\delta_i,e_i}^{(-1)^j}$ was fixed, while the half-sum or
half-difference selects the corresponding Rudin--Shapiro profile branch.

We now insert the Fourier factor.  For $y=h(k+t)$,
\begin{align*}
 e^{-2\pi i\xi\cdot x_{z,\epsilon}(y)}
 &=e^{-2\pi i(z+\epsilon)\cdot\xi}
   \prod_i e^{-2\pi i\varepsilon_i h\xi_i k_i}
   \prod_i e^{-2\pi i\varepsilon_i h\xi_i t_i}.
\end{align*}
The first factor is independent of both $k$ and $t$.  The second depends
on the cell indices $k_i$, while the third depends only on the local
coordinates $t_i$.  We combine the $k$-dependent factor with the discrete
coefficient above and define, for each coordinate $i$,
\[
 u_{i,k}=
 \mathbf1_{\{\kappa_i(k)\equiv\delta_i\;({\rm mod}\,2)\}}
 \frac{r_{\kappa_i(k)}+(-1)^{e_i}r_{\kappa_i(k)+1}}2
 e^{-2\pi i\varepsilon_i h\xi_i k}.
\]
Thus $u_{i,k}$ contains exactly the parity selector, the
Rudin--Shapiro coefficient, and the cell-index part of the Fourier
modulation.  The factor depending on $t_i$ remains outside the cell sum.

For fixed $v_i,t_i\in[0,1]$, the cumulative sum from
\eqref{eq:phase-cumulative-sums} is
\[
 \Sigma_i(v_i,t_i)
 =\sum_{k\in I_i(v_i,t_i)}u_{i,k},
 \qquad
 I_i(v_i,t_i)
 :=\{0\le k<2^j:h(k+t_i)\le v_i\}.
\]
The condition $h(k+t_i)\le v_i$ is equivalent to
$k\le v_i/h-t_i$.  Hence $I_i(v_i,t_i)$ is an initial interval of
integers, possibly empty, and
\[
 |I_i(v_i,t_i)|\le 1+\frac{v_i}{h}.
\]
We estimate this sum separately in the two coordinate orientations.

If $\epsilon_i=0$, then $\kappa_i(k)=k$ and $\varepsilon_i=1$, so
\[
 \Sigma_i(v_i,t_i)
 =\sum_{\substack{k\in I_i(v_i,t_i)\\
                   k\equiv\delta_i\;({\rm mod}\,2)}}
   \frac{r_k+(-1)^{e_i}r_{k+1}}2
   e^{-2\pi i h\xi_i k}.
\]
This is exactly a parity-restricted interval sum of the form controlled
by \eqref{eq:phase-parity-interval}, with modulation frequency
$\omega=-2\pi h\xi_i$.  Therefore
\[
 |\Sigma_i(v_i,t_i)|
 \le C\sqrt{|I_i(v_i,t_i)|}
 \le C\left(1+\frac{v_i}{h}\right)^{1/2}.
\]

If $\epsilon_i=1$, then $\kappa_i(k)=2^j-1-k$.  Set
\[
 q=2^j-1-k.
\]
As $k$ runs through the initial interval $I_i(v_i,t_i)$, the index $q$
runs, in reverse order, through a terminal interval $J_i$ of the same
cardinality.  Since the order of a finite sum is irrelevant, we may sum
over $J_i$ in increasing order.  The parity selector and the
Rudin--Shapiro coefficient become
\[
 \mathbf1_{\{q\equiv\delta_i\;({\rm mod}\,2)\}}
 \frac{r_q+(-1)^{e_i}r_{q+1}}2,
\]
while the Fourier factor satisfies
\[
 e^{-2\pi i\varepsilon_i h\xi_i k}
 =e^{-2\pi i\varepsilon_i h\xi_i(2^j-1)}
  e^{2\pi i\varepsilon_i h\xi_i q}.
\]
The first factor is constant in $q$ and has modulus one.  The second is
an ordinary modulation in $q$.  Since
\eqref{eq:phase-parity-interval} is uniform over every finite interval
and every modulation frequency, it gives
\[
 |\Sigma_i(v_i,t_i)|
 \le C\sqrt{|J_i|}
 =C\sqrt{|I_i(v_i,t_i)|}
 \le C\left(1+\frac{v_i}{h}\right)^{1/2}.
\]

Consequently, in either orientation,
\[
 |\Sigma_i(v_i,t_i)|
 \le C\left(1+\frac{v_i}{h}\right)^{1/2}
 =Ch^{-1/2}(h+v_i)^{1/2}.
\]
This includes the case of an empty interval and verifies
\eqref{eq:phase-prefix-bound} uniformly in $\xi,j,t_i$, and all parity,
profile, and localization choices.  Applying
Lemma~\ref{lem:weighted-cumulative-sums} to the amplitude function $a$ from Step~1
therefore yields
\begin{equation}\label{eq:phase-local-cell-sum-bound}
 \left|
 \sum_{k\in\{0,\ldots,2^j-1\}^n}
       a(h(k+t),t)\prod_i u_{i,k_i}
 \right|
 \le C h^{-(n-m)/2},
\end{equation}
uniformly in $t\in[0,1]^n$.

\Needspace{8\baselineskip}
\medskip\noindent\emph{Step 3: taking the Fourier transform of one shell.}
We now insert the physical-space decomposition above into
$\widehat{UA_j}(\xi)=\int UA_j(x)e^{-2\pi i\xi\cdot x}\,dx$.
For one fixed unit-cube localization, orientation, parity vector, and
profile vector, use $y=h(k+t)$ on every dyadic cell. The change of 
coordinates from $x=z+\epsilon+ \varepsilon h(k+t)$ to $t$ comes with a Jacobian $h\cdot\diag(\varepsilon_1,\ldots,\varepsilon_n)$ and the preceding factorization of the Fourier exponential shows
that, apart from the constant phase
$e^{-2\pi i(z+\epsilon)\cdot\xi}$, the contribution is
\[
 h^n\int_{[0,1]^n}e^{-2\pi i h\sum_i\varepsilon_i\xi_i t_i}
       \sum_{k\in\{0,\ldots,2^j-1\}^n}
                    a(h(k+t),t)\prod_i u_{i,k_i}\,dt.
\]
The outer exponential and the omitted constant phase have modulus one.
By \eqref{eq:phase-local-cell-sum-bound}, the absolute value of this
contribution is at most
\[
 Ch^n\int_{[0,1]^n}h^{-(n-m)/2}\,dt=Ch^{(n+m)/2}.
\]
Only finitely many unit cubes meet the fixed compact support, and the
numbers of coordinate localizations, parity choices, and profile choices
are also independent of $j$.  Adding all pieces gives
\begin{equation}\label{eq:phase-shell-unweighted}
 |\widehat{UA_j}(\xi)|\le Ch^{(n+m)/2},\qquad \xi\in\R^n.
\end{equation}

\medskip\noindent\emph{Step 4: scaled derivatives and the Fourier tail.}
The estimate \eqref{eq:phase-shell-unweighted} is uniform in the
frequency but gives no decay when $h|\xi|$ is large.  To obtain that
decay, we prove that the same $h^{(n+m)/2}$ estimate holds after applying
any scaled physical derivative
\[
 h^{|\alpha|}\partial_x^\alpha.
\]
Fourier integration by parts will then convert these derivative bounds
into powers of $(h|\xi|)^{-1}$.

We first explain how a physical derivative acts on a single localized
cell expression.  Remove temporarily the fixed endpoint cutoffs and
write
\[
 a_0(y,t)
 =A_j(x_{z,\epsilon}(y))
  \prod_i\Phi_{\delta_i,e_i}^{(-1)^j}
             (t_i',a_j(x_{z,\epsilon}(y))),
\]
so that the amplitude from Step~1 is
\[
 a(y,t)=a_0(y,t)\prod_i\omega(y_i).
\]
The same chain-rule calculation used in Step~1 gives, on
$[0,3/4]^n$ in the $y$ variables,
\[
 |\partial_y^\gamma\partial_t^\nu a_0(y,t)|
 \le C_{\gamma,\nu}\prod_i(h+y_i)^{-\gamma_i}.
\]

Inside the dyadic cell indexed by $k$, the local coordinates are
\[
 t_i=t_i(y)=h^{-1}y_i-k_i,
 \qquad \frac{\partial t_i}{\partial y_i}=h^{-1}.
\]
Also, from $x=x_{z,\epsilon}(y)$,
\[
 \partial_{x_i}=\varepsilon_i\partial_{y_i},
 \qquad \varepsilon_i=(-1)^{\epsilon_i}.
\]
The derivative with respect to $y_i$ must act both on the explicit
$y$-dependence of $a_0$ and on the substituted variable $t_i(y)$.  The
ordinary chain rule therefore gives
\begin{align*}
 \partial_{x_i}[a_0(y,t(y))]
 &=\varepsilon_i
   \left(\partial_{y_i}a_0(y,t)
         +h^{-1}\partial_{t_i}a_0(y,t)\right)_{t=t(y)}.
\end{align*}
After multiplying by $h$ this becomes
\begin{equation}\label{eq:phase-scaled-chain-rule}
 h\partial_{x_i}[a_0(y,t(y))]
 =\varepsilon_i(h\partial_{y_i}+\partial_{t_i})
                         a_0(y,t)\big|_{t=t(y)}.
\end{equation}
Thus a scaled physical derivative is represented, up to the harmless
sign $\varepsilon_i$, by the operator
\[
 D_i:=h\partial_{y_i}+\partial_{t_i}.
\]
The discrete Rudin--Shapiro coefficients are constant on each open cell,
so no derivative falls on them in this cell-interior calculation.

We next check that every finite product of the operators $D_i$ preserves
the weighted amplitude class.  The operators commute, and for a
multi-index $\alpha$ the binomial formula gives
\[
 \prod_{i=1}^nD_i^{\alpha_i}a_0
 =\sum_{\beta\le\alpha}\binom\alpha\beta
      h^{|\beta|}\partial_y^\beta
                        \partial_t^{\alpha-\beta}a_0.
\]
The product of reflection signs generated by the physical derivatives is
$\prod_i\varepsilon_i^{\alpha_i}$ and has modulus one.  For a fixed
summand and arbitrary additional multi-indices $\gamma,\nu$, the bound
for $a_0$ gives
\begin{align*}
 &\left|\partial_y^\gamma\partial_t^\nu
       \left(h^{|\beta|}\partial_y^\beta
                         \partial_t^{\alpha-\beta}a_0\right)(y,t)\right|\\
 &\qquad\le
 C_{\alpha,\gamma,\nu}
 h^{|\beta|}
 \prod_i(h+y_i)^{-(\gamma_i+\beta_i)}\\
 &\qquad=
 C_{\alpha,\gamma,\nu}
 \prod_i\left(\frac{h}{h+y_i}\right)^{\beta_i}
 \prod_i(h+y_i)^{-\gamma_i}\\
 &\qquad\le
 C_{\alpha,\gamma,\nu}
 \prod_i(h+y_i)^{-\gamma_i},
\end{align*}
because $0<h/(h+y_i)\le1$.  This calculation explains why the scaling
factor $h^{|\alpha|}$ is the correct one: every slow $y_i$ derivative
costs a factor $(h+y_i)^{-1}$, and the accompanying factor $h$ absorbs
that cost uniformly.

We now restore the endpoint cutoffs $\prod_i\omega(y_i)$.  When
Leibniz' rule differentiates a cutoff, its derivatives are uniformly
bounded and supported where $y_i\le3/4$.  Since $h+y_i\le2$ on
$[0,1]$, a bounded derivative of order $b$ also satisfies
\[
 |\omega^{(b)}(y_i)|\le C_b(h+y_i)^{-b}
\]
after enlarging the constant.  Hence distributing derivatives between
$a_0$ and the cutoffs preserves the same coordinatewise weighted bounds.
The cutoffs also restore support in $E_h$ and vanishing near every outer
face $y_i=1$.

It follows that, for each $\alpha$, the localized expression for
$h^{|\alpha|}\partial_x^\alpha(UA_j)$ is a finite sum of terms with the
same coefficient sequences as in Step~2 and with amplitudes satisfying
the hypotheses of Lemma~\ref{lem:weighted-cumulative-sums}.  Repeating
Steps~2--3 term by term gives
\begin{equation}\label{eq:phase-scaled-derivative-fourier}
 \left|\bigl(h^{|\alpha|}\partial_x^\alpha(UA_j)\bigr)^\wedge(\xi)\right|
 \le C_\alpha h^{(n+m)/2}.
\end{equation}

It is important to be precise here to clarify why the cell by cell
computations are sufficient for global estimates.
The function $UA_j$ is globally smooth and compactly supported:
$A_j$ vanishes on a neighborhood of $S$, while $U$ is smooth off $S$.
Moreover, \eqref{eq:phase-aligned-shell} holds on a neighborhood of the
shell support, including the switching collars.  We therefore first take
the classical derivative of the globally smooth function $UA_j$ and
then use the cell formulas on the open cell interiors, which cover all
but a null set.  No distributional derivative of an individual, piecewise
constant cell coefficient is being taken.

Because $UA_j\in C_c^\infty(\R^n)$, integration by parts has no boundary
terms and yields, for every coordinate $i$ and integer $K\ge0$,
\[
 \bigl(h^K\partial_{x_i}^K(UA_j)\bigr)^\wedge(\xi)
 =(2\pi i h\xi_i)^K\widehat{UA_j}(\xi).
\]
Suppose first that $h|\xi|>1$.  Choose an index $i$ with
$|\xi_i|\ge|\xi|/\sqrt n$.  Applying
\eqref{eq:phase-scaled-derivative-fourier} with $\alpha=Ke_i$ gives
\begin{align*}
 |\widehat{UA_j}(\xi)|
 &=(2\pi h|\xi_i|)^{-K}
   \left|\bigl(h^K\partial_{x_i}^K(UA_j)\bigr)^\wedge(\xi)\right|\\
 &\le C_K(h|\xi|)^{-K}h^{(n+m)/2},
\end{align*}
where the factor $n^{K/2}$ has been absorbed into $C_K$.

If $h|\xi|\le1$, the unweighted estimate
\eqref{eq:phase-shell-unweighted} gives
\[
 |\widehat{UA_j}(\xi)|\le Ch^{(n+m)/2}.
\]
Finally,
\[
 1+h|\xi|\le2
 \quad\text{when }h|\xi|\le1,
\]
whereas
\[
 1+h|\xi|\le2h|\xi|
 \quad\text{when }h|\xi|>1.
\]
After enlarging $C_K$, the two regimes combine into
\eqref{rs:eq:shell-bound} for every real frequency $\xi$.

\medskip\noindent\emph{Step 5: summing the shells.}
Unimodularity gives $\|UA_j\|_1=\|A_j\|_1\le C2^{-jm}$, so the
shell series converges in $L^1$ and its Fourier transform may be
summed term by term.  Apply \eqref{eq:dyadic-power-summation} with
$\mu=(n+m)/2$ and $K>\mu$ to \eqref{rs:eq:shell-bound}.  This gives
\[
 \sum_{j\ge3}|\widehat{UA_j}(\xi)|
 \le C\langle\xi\rangle^{-(n+m)/2},
\]
and proves \eqref{rs:eq:main-estimate} for all real frequencies.
\end{proof}

For the next two subsections, assume $0\le N-R\le d-1$ and retain $n=2d$ and $m=2(N-R+1)$.
Use the compactly supported cross-Zak products formed with the single
cutoff \eqref{eq:compact-periodizing-cutoff}.

\subsection{A fine-lattice-periodic phase in physical coordinates}

The following transfer supplies fine-lattice periodicity for the matrix
Parseval calculation and real-frequency decay for the final regularity
statement.

Let $U_0(x)$ be the phase of \Cref{prop:deterministic-phase} for the
fixed normalized scale $\rho(x)$, and set
\[
 U(u,\eta)=U_0\bigl(L^{-1}(u,\eta)\bigr).
\]
For each reference window $\psi_\alpha$ from
\Cref{lem:finite-Zak-family}, put
\[
 A_\alpha(u,\eta)=\chi(u,\eta)F_{\rm flat}(u,\eta)\overline{Z\psi_\alpha(u,\eta)}.
\]

\begin{lemma}[Physical-coordinate transfer]\label{lem:phase-physical-transfer}
The function $U$ is $\Gamma_D$-periodic, has modulus one, and is smooth
off $S_D$.  The same phase satisfies
\begin{equation}\label{eq:compact-phase-decay}
 |\widehat{UA_\alpha}(\xi)|
 \le C\langle\xi\rangle^{-(n+m)/2},\qquad\xi\in\R^n.
\end{equation}
for every member of the finite reference family, with a common constant.
\end{lemma}
\begin{proof}
For $(\gamma_u,\gamma_\eta)\in\Gamma_D=L\Z^n$, periodicity of $U_0$ gives
\begin{align*}
 U(u+\gamma_u,\eta+\gamma_\eta)
 &=U_0\bigl(L^{-1}(u,\eta)+L^{-1}(\gamma_u,\gamma_\eta)\bigr)\\*
 &=U_0\bigl(L^{-1}(u,\eta)\bigr)\\*
 &=U(u,\eta).
\end{align*}
Thus $U$ is fine-lattice periodic, in particular ordinary periodic.
It has modulus one and is smooth off $S_D$.

These functions have one fixed compact support.  By \eqref{eq:normalized-flat-symbol}, each $A_\alpha(Lx)$ satisfies the hypotheses of
\Cref{prop:deterministic-phase}, with the same $\rho$ and $U_0$.
Changing variables $(u,\eta)=Lx$ gives, for $\xi\in\R^n$,
\begin{align*}
 \widehat{UA_\alpha}(\xi)
 &=\int_{\R^n}U(u,\eta)A_\alpha(u,\eta)e^{-2\pi i\xi\cdot(u,\eta)}\,du\,d\eta\\
 &=|\det L|\int_{\R^n}U_0(x)A_\alpha(Lx)
                              e^{-2\pi i(L^T\xi)\cdot x}\,dx.
\end{align*}
If $a_L,b_L>0$ are the smallest and largest diagonal entries of $L$,
then $a_L|\xi|\le|L^T\xi|\le b_L|\xi|$; hence
$\langle L^T\xi\rangle\asymp\langle\xi\rangle$.
Applying \Cref{prop:deterministic-phase} to the last integral gives
\[
 |\widehat{UA_\alpha}(\xi)|
 \le |\det L|C_\alpha\langle L^T\xi\rangle^{-(n+m)/2}
 \le C_\alpha'\langle\xi\rangle^{-(n+m)/2}.
\]
Taking the maximum of these constants over the finite reference family
proves \eqref{eq:compact-phase-decay}.  The phase is the same for every
member of that family.
\end{proof}

\subsection{Parsevality and the full phase-shaped regularity range}

We verify the matrix Parseval identity for $F_{\rm phase}$ and evaluate the
Fourier estimate at integer frequencies to obtain its infinity endpoint.
The finite-$p$ conclusions then follow from the embedding already recorded
in \Cref{prop:modulation-embeddings}.

\begin{corollary}[The phase-shaped window]\label{cor:phase-shaped-window}
With the physical phase from \Cref{lem:phase-physical-transfer}, define
\begin{equation}\label{eq:F-deterministic}
 F_{\rm phase}(u,\eta)=U(u,\eta)F_{\rm flat}(u,\eta),\qquad
 g_{\rm phase}=Z^{-1}F_{\rm phase}.
\end{equation}
Then $\mathcal G(g_{\rm phase},\Lambda_D)$ is Parseval, and this same
window satisfies
\[
 \begin{gathered}
 g_{\rm phase}\in M_s^p,\qquad
 1\le p<\infty,\quad s<d+N-R+1-2d/p,\\
 g_{\rm phase}\in M_{d+N-R+1}^\infty.
 \end{gathered}
\]
\end{corollary}
\begin{proof}
\medskip\noindent\emph{Step 1: the matrix Parseval identity.}
The product is quasiperiodic because $U$ is ordinary periodic.
It is bounded, so its inverse Zak transform is in $L^2(\R^d)$.
For $x=(u,\eta)\notin S_D$, $s\in\mathcal I_{\mathsf B}$, and
$t\in\mathcal I_{\mathsf A}$, fine-lattice periodicity gives
\[
\begin{aligned}
 (\mathcal M_{F_{\rm phase}}(u,\eta))_{s,t}
 &=U(u-Dt,\eta+\mathsf B^{-1}s)
       F_{\rm flat}(u-Dt,\eta+\mathsf B^{-1}s)\\
 &=U(u,\eta)(\mathcal M_{F_{\rm flat}}(u,\eta))_{s,t}.
\end{aligned}
\]
Here $(-Dt,\mathsf B^{-1}s)\in\Gamma_D$, and the translated point
also lies outside $S_D$ by its fine-lattice invariance.  Thus
\[
 \mathcal M_{F_{\rm phase}}(x)=U(x)\mathcal M_{F_{\rm flat}}(x).
\]
Since $|U|=1$, off the invariant null set $S_D$ we obtain
\[
 G_{F_{\rm phase}}(x)=|U(x)|^2G_{F_{\rm flat}}(x)=I_R.
\]
The criterion \eqref{eq:scalar-parseval-identities} proves Parsevality.

\medskip\noindent\emph{Step 2: Fourier coefficients and regularity.}
The function
$G_{\psi_\alpha,g_{\rm phase}}
 =UF_{\rm flat}\overline{Z\psi_\alpha}$ is ordinary periodic.
Apply \eqref{eq:compact-cutoff-coefficients} to obtain, for $k\in\Z^n$,
\begin{align*}
 \widehat{G_{\psi_\alpha,g_{\rm phase}}}(k)
 &=\int_{\R^n}\chi(u,\eta)U(u,\eta)F_{\rm flat}(u,\eta)
                   \overline{Z\psi_\alpha(u,\eta)}e^{-2\pi i k\cdot(u,\eta)}\,du\,d\eta\\
 &=\widehat{UA_\alpha}(k).
\end{align*}
Combining this identity with \eqref{eq:compact-phase-decay} gives,
for every required cross-Zak product $G=G_{\psi_\alpha,g_{\rm phase}}$,
\begin{equation}\label{eq:high-pointwise}
 |\widehat G(k)|\le C\langle k\rangle^{-(n+m)/2},\qquad k\in\Z^n.
\end{equation}

At the infinity endpoint, \eqref{eq:high-pointwise} gives
\[
 \sup_{k\in\Z^n}\langle k\rangle^{(n+m)/2}|\widehat G(k)|\le C.
\]
Applying the cross-Zak characterization
\Cref{prop:cross-Zak-characterization} to the finite reference family
therefore yields
\[
 g_{\rm phase}\in M_{(n+m)/2}^\infty.
\]
For $1\le p<\infty$, apply \Cref{prop:modulation-embeddings} with
$p_0=\infty$, $s_0=(n+m)/2$, $p_1=p$, and $s_1=s$.
Its condition is
\[
 \frac{n+m}{2}-s>n\left(\frac1p-\frac1\infty\right)=\frac np,
\]
which is exactly $s<(n+m)/2-n/p$.
Since $(n+m)/2=d+N-R+1$ and $n=2d$, these are the stated regularity
bounds.  The phase was fixed before $p$ and $s$ were chosen, so the same
window has all these properties.  When $m=n$, \Cref{tm:prop:flat-all-p}
already gives the required conclusion with the single flat window.
\end{proof}

\part{Symplectic transfer and global classification}\label{part:global-classification}
This part transfers the diagonal constructions and obstructions to general lattices and assembles the global classifications.

\section{Symplectic and metaplectic extension machinery}
\label{sec:lattice-theory}

Let
\[
 J=\begin{pmatrix}0&I_d\\-I_d&0\end{pmatrix},
 \qquad
 \Sp(2d,\R)=\{S\in GL_{2d}(\R):S^TJS=J\}.
\]
Then $\sigma(z,w)=z^TJw$, and $S\in\Sp(2d,\R)$ preserves this form.
Two lattices are symplectically equivalent if $\Lambda'=S\Lambda$
for such a matrix $S$.  Since $|\det S|=1$, their covolumes agree.
The adjoint lattices satisfy
\[
 z\in(S\Lambda)^\circ
 \quad\Longleftrightarrow\quad
 \sigma(S^{-1}z,\lambda)\in\Z\quad(\lambda\in\Lambda)
 \quad\Longleftrightarrow\quad
 S^{-1}z\in\Lambda^\circ,
\]
so $(S\Lambda)^\circ=S\Lambda^\circ$.

We use the symplectic rationality and index conventions from the
introduction.  A real matrix is called \emph{entrywise rational} if all of
its entries lie in $\Q$.  Two full-rank Euclidean lattices are
\emph{commensurable} if their intersection has finite index in each of them.

\begin{lemma}
\label{lem:symplectic-rationality-characterization}
Let $\Lambda=M\Z^{2d}$ and put
$\Theta_\Lambda=M^TJM$.  Then $\Lambda$ is symplectically rational if and
only if $\Theta_\Lambda$ is entrywise rational.  In particular, let
\[
 \Lambda=\Gamma\times\Phi,
 \qquad
 \Gamma=A_0\Z^d,
 \qquad
 \Phi=B_0\Z^d,
\]
with $A_0,B_0\in GL_d(\R)$.  The following conditions are equivalent:
\begin{enumerate}[label=\thmpart{\roman*}]
\item $\Gamma\times\Phi$ is symplectically rational;
\item the pairing matrix $A_0^TB_0$ is entrywise rational;
\item the Euclidean lattices $\Gamma$ and $\Phi^*$ are commensurable.
\end{enumerate}
\end{lemma}

\begin{proof}
Suppose first that $\Theta_\Lambda$ is entrywise rational.  Choose
$q\in\N$ such that $q\Theta_\Lambda$ is integral.  If
$\lambda=Mk$ and $\mu=M\ell$, with $k,\ell\in\Z^{2d}$, then
\[
 \sigma(q\lambda,\mu)=qk^T\Theta_\Lambda\ell\in\Z.
\]
Thus $q\Lambda\subset\Lambda^\circ$ and, since clearly $q\Lambda\subset \Lambda$, we have $q\Lambda\subset\Lambda_{\rm int}$. Hence
$$q^{2d}=[\Lambda:q\Lambda]\geq [\Lambda:\Lambda_{\rm int}].$$

Conversely, suppose $[\Lambda:\Lambda_{\rm int}]=q<\infty$, then by Lagrange's theorem $q\Lambda\subset\Lambda_{\rm int}\subseteq \Lambda^\circ$.  Applying this to
the basis vectors $Me_j$ gives
\[
 q(\Theta_\Lambda)_{jk}
 =\sigma(qMe_j,Me_k)\in\Z,
 \qquad 1\le j,k\le2d.
\]
Hence $\Theta_\Lambda$ is entrywise rational.

For the separable lattice $\Gamma\times\Phi$, the symplectic Gram matrix of
the basis $\operatorname{diag}(A_0,B_0)$ is
\[
\begin{pmatrix}
 A_0^T&0\\
 0&B_0^T
 \end{pmatrix}\begin{pmatrix}
 0&I_d\\
 -I_d&0
 \end{pmatrix}\begin{pmatrix}
 A_0&0\\
 0&B_0
 \end{pmatrix}=
 \begin{pmatrix}
 0&A_0^TB_0\\
 -B_0^TA_0&0
 \end{pmatrix}.
\]
This proves the equivalence of (i) and (ii).  Under the common linear change
of Euclidean coordinates $A_0^{-1}$, the pair $(\Gamma,\Phi^*)$ becomes
\[
 \bigl(A_0^{-1}A_0\Z^d,A_0^{-1}B_0^{-T}\Z^d\bigr) = \bigl(\Z^d,(A_0^TB_0)^{-T}\Z^d\bigr).
\]
A lattice $C\Z^d$ is commensurable with $\Z^d$ exactly when
$C\in GL_d(\Q)$.  This proves the equivalence with (iii).
\end{proof}

We use the alternating Smith normal form
\cite[Chapter~IV]{Newman1972}.

\begin{theorem}[Smith--Newman]\label{thm:smith-newman}
Let $\Theta\in M_{2d}(\Z)$ be nonsingular and alternating.  There are
$W\in GL_{2d}(\Z)$ and positive integers $h_1,\ldots,h_d$ such that
\[
 W^T\Theta W=\begin{pmatrix}0&H\\-H&0\end{pmatrix},
 \qquad H=\operatorname{diag}(h_1,\ldots,h_d).
\]
\end{theorem}

We apply \Cref{thm:smith-newman} to an integral multiple of the
symplectic pairing matrix.

\begin{theorem}
\label{thm:diagonal-symplectic-normal-form}
Every symplectically rational full-rank phase-space lattice
$\Lambda\subset\R^{2d}$ is symplectically equivalent to
\begin{equation}\label{eq:normal-form-lattice}
 \Lambda_D=D\Z^d\times\Z^d,
 \qquad
 D=\operatorname{diag}\left(\frac{b_1}{a_1},\ldots,
                             \frac{b_d}{a_d}\right),
 \qquad\text{with }(a_i,b_i)=1,
\end{equation}
where all $a_i,b_i$ are positive integers.  Thus there exists
$S\in\Sp(2d,\R)$ such that $\Lambda=S\Lambda_D$.

If
\[
 N=\prod_{i=1}^d a_i,
 \qquad
 R=\prod_{i=1}^d b_i,
\]
then
\begin{equation}\label{eq:normal-form-invariants}
 \nu(\Lambda)=\nu(\Lambda_D)=N,
 \qquad
 \nu(\Lambda^\circ)=\nu(\Lambda_D^\circ)=R,
 \qquad
 \covol(\Lambda)=\covol(\Lambda_D)=\frac{R}{N}.
\end{equation}
For every integer $Q\ge1$, the symplectically rational conditions $\SG$ and $\mSG_Q$
become, respectively,
\[
                             N\ge R+d,
 \qquad
                             QN\ge R+d.
\]

\end{theorem}

\begin{proof}
Choose a basis matrix $M\in GL_{2d}(\R)$ with
$\Lambda=M\Z^{2d}$ and set $\Theta_\Lambda=M^TJM$.  By \Cref{lem:symplectic-rationality-characterization}, $\Theta_\Lambda$ is entrywise rational.  Choose $q\in\N$ such that
$q\Theta_\Lambda$ is integral.  \Cref{thm:smith-newman}, applied to the nonsingular alternating matrix
$q\Theta_\Lambda$, gives
$U\in GL_{2d}(\Z)$ and positive integers
$h_1,\ldots,h_d$ such that
\[
 U^T\Theta_\Lambda U
 =\begin{pmatrix}0&D_0\\-D_0&0\end{pmatrix},
 \qquad
 D_0=\operatorname{diag}\left(\frac{h_1}{q},\ldots,
                               \frac{h_d}{q}\right).
\]
Let
\[
 C_0=\begin{pmatrix}D_0&0\\0&I_d\end{pmatrix},
 \qquad
 S=MUC_0^{-1}.
\]
Since
\[
 C_0^TJC_0
 =\begin{pmatrix}0&D_0\\-D_0&0\end{pmatrix}
 =U^TM^TJMU,
\]
we obtain
\[
 S^TJS=C_0^{-T}U^TM^TJMU C_0^{-1}=J.
\]
Thus $S$ is symplectic and
\[
 \Lambda=M\Z^{2d}
 =MU\Z^{2d}
 =SC_0\Z^{2d}
 =S(D_0\Z^d\times\Z^d).
\]
Reducing the diagonal entries of $D_0$ gives the coprime pairs
$(a_i,b_i)$ in \eqref{eq:normal-form-lattice}.

For the diagonal lattice, the three formulas in
\eqref{eq:normal-form-invariants} follow directly from the diagonal data
\eqref{eq:diagonal-model}--\eqref{eq:NR}.  If $\Lambda=S\Lambda_D$, then
$(S\Lambda_D)^\circ=S\Lambda_D^\circ$,
$S(\Lambda_D\cap\Lambda_D^\circ)
 =\Lambda\cap\Lambda^\circ$, and $\det S=1$.  Therefore the two indices and
the covolume are unchanged, proving all three equalities in
\eqref{eq:normal-form-invariants}.

\end{proof}

\begin{corollary}
\label{cor:square-index}
If $\Lambda\subset\mathbb R^{2d}$ is symplectically rational, then
$[\Lambda:\Lambda_{\rm int}]$ is the square of an integer.
\end{corollary}

\begin{proof}
By Theorem~\ref{thm:diagonal-symplectic-normal-form}, there are
$S\in\operatorname{Sp}(2d,\mathbb R)$ and
\[
 D=\mathsf B\mathsf A^{-1},
 \qquad
 \mathsf A=\operatorname{diag}(a_1,\ldots,a_d),
 \qquad
 \mathsf B=\operatorname{diag}(b_1,\ldots,b_d),
\]
with $(a_i,b_i)=1$, such that
$\Lambda=S\Lambda_D$ for
$\Lambda_D=D\mathbb Z^d\times\mathbb Z^d$.  Its adjoint lattice is
\[
 \Lambda_D^\circ=\mathbb Z^d\times D^{-1}\mathbb Z^d,
\]
and coprimality gives
\[
 \Lambda_D\cap\Lambda_D^\circ
 =\mathsf B\mathbb Z^d\times\mathsf A\mathbb Z^d.
\]
Therefore, with $N=\det\mathsf A=\prod_i a_i$,
\[
 [\Lambda_D:\Lambda_D\cap\Lambda_D^\circ]
 =[D\mathbb Z^d:\mathsf B\mathbb Z^d]
  [\mathbb Z^d:\mathsf A\mathbb Z^d]
 =N\cdot N=N^2.
\]
Since symplectic maps carry adjoint lattices to adjoint lattices,
$S(\Lambda_D\cap\Lambda_D^\circ)=\Lambda\cap\Lambda^\circ$, and hence
\[
 [\Lambda:\Lambda_{\rm int}]
 =[\Lambda_D:\Lambda_D\cap\Lambda_D^\circ]
 =N^2.
\]
\end{proof}

For each $S\in\Sp(2d,\R)$, a metaplectic lift $\mu(S)$ is unitary on
$L^2(\R^d)$, maps $\mathcal S(\R^d)$ onto itself, and satisfies
\[
 \mu(S)\pi(z)\mu(S)^{-1}=c_S(z)\pi(Sz),\qquad |c_S(z)|=1.
\]
These facts, and the explicit operators associated with symplectic
dilations, shears, and the Fourier transform, are given in
\cite[Section~9.4]{Grochenig2001}; see also
\cite[Chapter~4]{Folland1989}.

\begin{proposition}
\label{prop:metaplectic-Msp}
Let $S\in\Sp(2d,\mathbb R)$ and let $\mu(S)$ be a metaplectic lift.  Then,
for every $1\le p\le\infty$ and $s\in\mathbb R$,
\begin{equation}
 \mu(S):M_s^p(\mathbb R^d)\longrightarrow M_s^p(\mathbb R^d)
 \label{eq:metaplectic-Msp}
\end{equation}
is an isomorphism.  Moreover,
$\mu(S)\pi(z)\mu(S)^{-1}=c_S(z)\pi(Sz)$ with $|c_S(z)|=1$, so frame bounds
are preserved when the lattice and window are transformed together.
\end{proposition}

\begin{proof}
Fix $0\ne\phi\in\mathcal S(\R^d)$ and put
$\psi=\mu(S)^{-1}\phi\in\mathcal S(\R^d)$.  Metaplectic covariance gives
\[
 |V_\phi(\mu(S)f)(z)|=|V_\psi f(S^{-1}z)|.
\]
For $p<\infty$, change variables $z=S\zeta$ and use $|\det S|=1$:
\begin{align*}
 \|\mu(S)f\|_{M_s^p,\phi}^{p}
 &=\int_{\R^{2d}}\langle z\rangle^{sp}
                   |V_\psi f(S^{-1}z)|^p\,dz\\
 &=\int_{\R^{2d}}\langle S\zeta\rangle^{sp}
                   |V_\psi f(\zeta)|^p\,d\zeta
 \asymp_{S,s}\|f\|_{M_s^p,\psi}^{p}
 \asymp\|f\|_{M_s^p,\phi}^{p}.
\end{align*}
The last equivalence is independence of the Schwartz analysis window.
For $p=\infty$, the same change of variables in the essential supremum
gives the result.  Unitarity and covariance preserve the frame bounds.
See also \cite[Theorems~3.2 and~4.6]{FuehrShafkulovska2024}.
\end{proof}

\section{Assembly of the rational modulation-space regions}
\label{sec:assembly}

\begin{proof}[Proof of \Cref{thm:modulation-classification}]
Choose $S\in\Sp(2d,\R)$ and a diagonal rational model $\Lambda_D$
such that $\Lambda=S\Lambda_D$, as in
\Cref{thm:diagonal-symplectic-normal-form}.  Then
\[
 \Delta(\Lambda)=N-R,\qquad q=N-R+1=\Delta(\Lambda)+1.
\]
We first construct the windows on $\Lambda_D$.

If $0\le\Delta(\Lambda)\le d-2$, then $q<d$ and $m=2q<2d=n$.
Proposition~\ref{tm:prop:flat-all-p} gives one Parseval window
$g_{\rm flat}$ with
\begin{align*}
 g_{\rm flat}&\in M_s^p,
 &&1\le p<\infty,\quad s<m(1-1/p)
                      =2(\Delta(\Lambda)+1)(1-1/p),\\
 g_{\rm flat}&\in M_s^\infty,
 &&s\le m=2(\Delta(\Lambda)+1).
\end{align*}
The phase construction in \Cref{sec:lower-completion} gives a second
Parseval window $g_{\rm phase}$ with
\begin{align*}
 g_{\rm phase}&\in M_s^p,
 &&1\le p<\infty,\quad
 s<\frac{n+m}{2}-\frac np=d+\Delta(\Lambda)+1-\frac{2d}{p},\\
 g_{\rm phase}&\in M_s^\infty,
 &&s\le\frac{n+m}{2}=d+\Delta(\Lambda)+1.
\end{align*}
Each construction chooses its window before $p$ and $s$, so each window
belongs simultaneously to all the spaces in its displayed range.

If $\Delta(\Lambda)=d-1$, then $q=d$ and $m=n=2d$.
Set $g_{d-1}=g_{\rm flat}$.  Proposition~\ref{tm:prop:flat-all-p}
then gives
\[
 g_{d-1}\in M_s^p\quad(1\le p<\infty,\ s<2d(1-1/p)),
 \qquad
 g_{d-1}\in M_s^\infty\quad(s\le2d).
\]
If $\Delta(\Lambda)\ge d$, Corollary~\ref{cor:Schwartz-construction}
gives a Schwartz Parseval window on $\Lambda_D$.

For each of these windows $g$, use $\mu(S)g$ on $\Lambda$.
Metaplectic covariance and unitarity give, for every $f\in L^2$,
\begin{align*}
 \sum_{\lambda\in\Lambda_D}
 |\langle f,\pi(S\lambda)\mu(S)g\rangle|^2
 &=\sum_{\lambda\in\Lambda_D}
 |\langle\mu(S)^{-1}f,\pi(\lambda)g\rangle|^2\\
 &=\|\mu(S)^{-1}f\|_2^2=\|f\|_2^2.
\end{align*}
Thus the transformed window is Parseval.  Proposition
\ref{prop:metaplectic-Msp} preserves all the displayed modulation-space
memberships, and metaplectic operators preserve the Schwartz class.
This proves all three cases.
\end{proof}

\begin{proof}[Proof of \Cref{thm:symplectic-index-balian-low}]
If $\Delta(\Lambda)<0$, then
$\covol(\Lambda)=\nu(\Lambda^\circ)/\nu(\Lambda)>1$, so the density
theorem excludes every frame window.  Suppose $\Delta(\Lambda)\ge0$
and reduce to $\Lambda_D$ by the normal form and metaplectic invariance.

For $0\le\Delta(\Lambda)\le d-1$, put $q=\Delta(\Lambda)+1$.
For $1\le p<\infty$, Theorem~\ref{thm:finite-p-endpoint} excludes
all weights
\[
 s\ge
 \begin{cases}
  2q(1-1/p),&1\le p\le2,\\
  d+q-2d/p,&2<p<\infty.
 \end{cases}
\]
For $p=\infty$ and $s>d+q$, the embedding condition in
\Cref{prop:modulation-embeddings} is
\[
 s-q>d=2d(1/2-0).
\]
It follows that $M_s^\infty\hookrightarrow M_q^2$.  The $p=2$
case of the same nonexistence theorem excludes a frame window in
$M_q^2$, since $2q(1-1/2)=q$.  Hence $s>d+q$ is impossible at
$p=\infty$.

For the converse, compare the two bounds in
\Cref{thm:modulation-classification}:
\begin{align*}
 \left(d+q-\frac{2d}{p}\right)-2q\left(1-\frac1p\right)
 &=d-q-\frac{2d-2q}{p}\\
 &=2(d-q)\left(\frac12-\frac1p\right).
\end{align*}
For $1\le p\le2$, the last quantity is nonpositive, and the flat
window attains every $s<2q(1-1/p)$.  For $2<p<\infty$, it is
nonnegative, and the phase window attains every $s<d+q-2d/p$.
When $q=d$, the two bounds agree and the flat window suffices.
At $p=\infty$, the phase window attains $s=d+q$ when $q<d$, and
the flat window attains $s=2d=d+q$ when $q=d$.  Weight monotonicity
gives every smaller $s$.

Finally, if $\Delta(\Lambda)\ge d$, the Schwartz Parseval window
belongs to every $M_s^p$.  All constructed windows are Parseval, as
asserted.
\end{proof}
For $\covol(\Lambda)=1$, \eqref{eq:index-covolume-ratio} gives
$\Delta(\Lambda)=0$, so the critical-density case is the $\Delta=0$
specialization of these statements.

\section{External inputs for general lattices}
\label{sec:nonrational-balian-low}

The following is the specialization of
Lee--Philipp--Voigtlaender~\cite[Theorem~1.1]{LeePhilippVoigtlaender2023}
used below.
\begin{proposition}
\label{prop:external-canonical-dual-regularity}
Let $\Lambda\subset\mathbb R^{2d}$ be a lattice and let
$\mathcal G(g,\Lambda)$ be a Gabor frame with frame operator $S$.  If
$g\in\mathbb H^1(\mathbb R^d)=M_1^2(\mathbb R^d)$, equivalently if $g$ has
finite second moments in both time and frequency, then the canonical dual
$S^{-1}g$ and canonical Parseval window $S^{-1/2}g$ also belong to
$M_1^2(\mathbb R^d)$.
\end{proposition}

The weak arbitrary-lattice Balian--Low input is the following
specialization of Gr\"ochenig--Han--Heil--Kutyniok
\cite[Theorem~8]{GrochenigHanHeilKutyniok2002}.
\begin{proposition}
\label{prop:external-weak-arbitrary-lattice-BL}
Let $\Lambda\subset\mathbb R^{2d}$ have covolume one, and suppose that
$\mathcal G(g,\Lambda)$ is a Gabor Riesz basis.  Let $h$ be its unique
biorthogonal Gabor dual, so that $\mathcal G(h,\Lambda)$ is biorthogonal to
$\mathcal G(g,\Lambda)$.  Then $g$ and $h$ cannot both have finite second
moments in time and frequency.  More explicitly, it is impossible that each
of the four integrals
\[
 \int |x|^2|g(x)|^2\,dx,\quad
 \int |\xi|^2|\widehat g(\xi)|^2\,d\xi,\quad
 \int |x|^2|h(x)|^2\,dx,\quad
 \int |\xi|^2|\widehat h(\xi)|^2\,d\xi
\]
is finite, where the Fourier-transform normalization is the one fixed in
\Cref{subsec:fourier-vmo-background}.  In particular, $g$ and $h$ cannot
both belong to $M_1^2(\mathbb R^d)$.
\end{proposition}

Enstad--Thiel--Vilalta's multiwindow existence theorem
\cite[Theorem~D]{EnstadThielVilalta2025} gives the following.
\begin{proposition}\label{prop:external-Enstad-existence}
Let $\Lambda\subset\R^{2d}$ be a full-rank lattice, with
$(d-1)/\infty=0$.  There is a Schwartz multiwindow Gabor frame with
\[
 Q_0=\left\lfloor\covol(\Lambda)+\frac{d-1}{\nu(\Lambda)}\right\rfloor+1
\]
windows.  In particular, the condition
\[
 \covol(\Lambda)+\frac{d-1}{\nu(\Lambda)}<1
\]
implies the existence of one Schwartz frame window
\cite[Theorem~C]{EnstadThielVilalta2025}.  For
$\nu(\Lambda)=\infty$, this condition is $\covol(\Lambda)<1$.
\end{proposition}

\section{Completion of the all-lattice one-window theorems}
\label{sec:completion-all-lattices}

\begin{proof}[Proof of \Cref{thm:full-multivariate-balian-low}]
The $M_1^2$ norm is equivalent to the sum of the $L^2$ norms of $g$,
$x_jg$, and $\partial_jg$
\cite[Chapter~11]{Grochenig2001}.
Suppose a window with finite time and frequency second moments
generates a frame.  The density theorem gives $\covol(\Lambda)\le1$.
At equality the frame is a Riesz basis
\cite[Theorem~10(c)]{Heil2007History}.  Its canonical dual
$S^{-1}g$ belongs to $M_1^2$ by
\Cref{prop:external-canonical-dual-regularity}.  Since the canonical
dual is the biorthogonal generator, this contradicts
\Cref{prop:external-weak-arbitrary-lattice-BL}.  Thus
$\covol(\Lambda)<1$.

Conversely, suppose $\covol(\Lambda)<1$.  In the rational case,
\eqref{eq:index-covolume-ratio} gives $\Delta\ge1$.
If $1\le\Delta\le d-1$, the flat estimate gives
\[
 g_{\rm flat}\in M_1^2,
 \qquad 1<2(\Delta+1)(1-1/2)=\Delta+1.
\]
If $\Delta\ge d$, use the smooth branch.  Metaplectic invariance
transfers the resulting Parseval window to $\Lambda$.
For an irrational lattice, $\nu(\Lambda)=\infty$ and
$0<\covol(\Lambda)<1$, so the number in
\Cref{prop:external-Enstad-existence} is
\[
 Q_0=\lfloor\covol(\Lambda)+(d-1)/\infty\rfloor+1=1.
\]
Thus there is a Schwartz window, and
\Cref{prop:Schwartz-Parsevalization} makes it Parseval.
The corresponding $S_0$ existence result was established earlier by
Jakobsen--Luef~\cite[Theorem~5.4]{JakobsenLuef2020}.
\end{proof}

\begin{proof}[Proof of \Cref{thm:main-gabor}]
The inclusion $\mathcal S\subset S_0$ proves
\thmpart{a}$\Rightarrow$\thmpart{b}.
Suppose $g\in S_0$ generates a frame.  In the rational case, reduce
to $\Lambda_D$ using metaplectic invariance.  Its Zak transform is
continuous~\cite{JanssenS0}.  Corollary
\ref{thm:intro-diagonal-zak-obstruction} gives
$\Delta=N-R\ge d$, which also implies $R/N<1$.
In the irrational case the density theorem gives
$\covol(\Lambda)\le1$; equality is excluded by
Gerhold--Lamando--Luef~\cite[Lemma~3.11 and Corollary~3.30]{GerholdLamandoLuef2026}.
This proves \thmpart{b}$\Rightarrow$\thmpart{c}.

Under \thmpart{c}, the rational smooth branch gives a Schwartz
Parseval window, and metaplectic invariance transfers it to the
original lattice.  In the irrational case,
\Cref{prop:external-Enstad-existence,prop:Schwartz-Parsevalization}
give a Schwartz Parseval window.  This proves
\thmpart{c}$\Rightarrow$\thmpart{a} and the Parseval assertion.
\end{proof}

\begin{proof}[Proof of \Cref{rem:gelfand-shilov-windows}]
Suppose that $\Lambda$ satisfies \eqref{eq:SG}.  By
\Cref{thm:main-gabor}, choose a frame window $g\in M^1(\R^d)$.
Finite linear combinations of time--frequency shifts of
$\phi(x)=e^{-\pi|x|^2}$ are dense in $M^1$; see
\cite[Section~5.1]{Grochenig2007Weights} and
\cite[Proposition~11.4.2]{Grochenig2001}.  Choose such a combination
$h$ sufficiently close to $g$.  The coefficient estimate for $M^1$
windows gives
\[
 \|\Ana_{h,\Lambda}-\Ana_{g,\Lambda}\|_{L^2\to\ell^2}
 \le C_\Lambda\|h-g\|_{M^1}.
\]
If $a>0$ is a lower frame bound for $g$ and the right side is less than
$\sqrt a$, the reverse triangle inequality yields
\[
 \|\Ana_{h,\Lambda}f\|_2
 \ge\|\Ana_{g,\Lambda}f\|_2
       -\|(\Ana_{h,\Lambda}-\Ana_{g,\Lambda})f\|_2
 \ge(\sqrt a-C_\Lambda\|h-g\|_{M^1})\|f\|_2.
\]
Thus $h$ is a frame window.  Each shifted Gaussian belongs to the
Roumieu class $S^{1/2,1/2}$, and so does its finite sum $h$.

For the Parseval assertion, write $D=\operatorname{Syn}_{h,\Lambda}$,
$S=DD^*$, and $G=D^*D$ on $\ell^2(\Lambda)$.  Let $0<a_0\le b_0$
be frame bounds for $h$.  Then
\[
 \sigma(S)\subset[a_0,b_0],\qquad
 \sigma(G)\subset\{0\}\cup[a_0,b_0].
\]
The Gaussian inner-product formula and the finite number of shifts in
$h$ give constants $C,c>0$ such that
\[
 |G_{\lambda\mu}|
 =|\langle\pi(\mu)h,\pi(\lambda)h\rangle|
 \le C e^{-c|\lambda-\mu|},\qquad\lambda,\mu\in\Lambda.
\]
We give the exponential resolvent estimate, as in the localization
argument of Jaffard~\cite[Proposition~2]{Jaffard1990}.
Choose a positively oriented contour $\mathcal C$ in the right
half-plane enclosing $[a_0,b_0]$ and excluding zero, and put
$\delta=\operatorname{dist}(\mathcal C,\{0\}\cup[a_0,b_0])>0$.
For $\mu\in\Lambda$ and an integer $L\ge1$, let $W_{\varepsilon,L}$
be the bounded invertible diagonal operator with entries
\[
 (W_{\varepsilon,L})_{\lambda\lambda}
   =\exp\bigl(\varepsilon\min\{|\lambda-\mu|,L\}\bigr).
\]
The function $\lambda\mapsto\min\{|\lambda-\mu|,L\}$ is
$1$-Lipschitz.  Hence the entries of
$E_{\varepsilon,L}=W_{\varepsilon,L}GW_{\varepsilon,L}^{-1}-G$
satisfy
\[
 |(E_{\varepsilon,L})_{\lambda\nu}|
 \le C e^{-c|\lambda-\nu|}
             (e^{\varepsilon|\lambda-\nu|}-1).
\]
The Schur test therefore gives, uniformly in $L$ and $\mu$,
\[
 \|E_{\varepsilon,L}\|
 \le C\sum_{\gamma\in\Lambda}e^{-c|\gamma|}
                     (e^{\varepsilon|\gamma|}-1)
 \longrightarrow0\qquad(\varepsilon\downarrow0).
\]
For $0<\varepsilon<c/2$, the summands are bounded by
$C e^{-c|\gamma|/2}$, a summable lattice sequence.  Choose
$\varepsilon$ so that $\|E_{\varepsilon,L}\|<\delta/2$ for every $L$.
Since $G$ is self-adjoint,
$\|(zI-G)^{-1}\|\le\delta^{-1}$ on $\mathcal C$.  The Neumann
series gives
\[
 \|(zI-G-E_{\varepsilon,L})^{-1}\|\le2\delta^{-1}.
\]
Conjugating the resolvent and taking its $(\lambda,\mu)$ entry gives,
for every $L\ge1$,
\[
 |((zI-G)^{-1})_{\lambda\mu}|
 \le2\delta^{-1}e^{-\varepsilon\min\{|\lambda-\mu|,L\}}.
\]
Letting $L\to\infty$ yields
\[
 |((zI-G)^{-1})_{\lambda\mu}|
 \le2\delta^{-1}e^{-\varepsilon|\lambda-\mu|},
 \qquad z\in\mathcal C.
\]

Define the spectral multiplier
\[
 B=\frac1{2\pi i}\int_{\mathcal C}z^{-1/2}(zI-G)^{-1}\,dz.
\]
It equals $G^{-1/2}$ on the positive spectral subspace and is zero on
$\ker G$.  The preceding bound implies
\[
 |B_{\lambda\mu}|\le C_1e^{-\varepsilon|\lambda-\mu|}.
\]
Since $DG=SD$, the resolvent identity gives
$D(zI-G)^{-1}=(zI-S)^{-1}D$ on $\mathcal C$.  Integrating and using
$De_0=h$, we obtain
\begin{align*}
 h^\circ=S^{-1/2}h
 &=S^{-1/2}De_0=DB e_0\\
 &=\sum_{\lambda\in\Lambda}B_{\lambda0}\pi(\lambda)h,
 \qquad |B_{\lambda0}|\le C_1e^{-\varepsilon|\lambda|}.
\end{align*}
This is the canonical Parseval window.

To verify its $S^{1,1}$ regularity directly, use that the finite Gaussian
combination $h$ extends to an entire function.  On every fixed strip
$|\operatorname{Im}z|\le r$, its formula gives
\[
 |h(z)|\le C_r e^{-c_r|\operatorname{Re}z|^2}.
\]
Write $\lambda=(t,\omega)$ and $z=x+iy$.  Then
\[
 |\pi(\lambda)h(z)|
 =|e^{2\pi i\omega\cdot z}h(z-t)|
 \le C_r e^{2\pi r|\omega|}e^{-c_r|x-t|^2}.
\]
Choose $r>0$ and $\kappa>0$ with $2\pi r+\kappa<\varepsilon$.
Using $|x|\le|x-t|+|t|$ and
$e^{\kappa|x-t|}e^{-c_r|x-t|^2}\le C_{r,\kappa}$ gives
\begin{align*}
 e^{\kappa|x|}|h^\circ(x+iy)|
 &\le C\sum_{(t,\omega)\in\Lambda}
       e^{-\varepsilon|(t,\omega)|}
       e^{2\pi r|\omega|+\kappa|t|}\\
 &\le C\sum_{\lambda\in\Lambda}
       e^{-(\varepsilon-2\pi r-\kappa)|\lambda|}<\infty,
 \qquad |y|\le r.
\end{align*}
The series converges locally uniformly on the strip, so its sum is
holomorphic there.  Choose a fixed polydisc radius $r_0>0$ with
$\sqrt d\,r_0<r$.  Cauchy's derivative estimate yields
\[
 |\partial^\alpha h^\circ(x)|
 \le C\alpha!r_0^{-|\alpha|}e^{-\kappa|x|}.
\]
Finally, since $|x|\ge d^{-1/2}\sum_i|x_i|$,
\begin{align*}
 |x^\beta\partial^\alpha h^\circ(x)|
 &\le C\alpha!r_0^{-|\alpha|}
       \prod_{i=1}^d|x_i|^{\beta_i}e^{-\kappa|x_i|/\sqrt d}\\
 &\le C\alpha!\beta!r_0^{-|\alpha|}
                         (\sqrt d/\kappa)^{|\beta|}.
\end{align*}
These are the Roumieu $S^{1,1}$ estimates.  Conversely, either of the
stated Gelfand--Shilov classes is contained in $\mathcal S\subset S_0$,
so the existence of such a frame window implies \eqref{eq:SG} by
\Cref{thm:main-gabor}.
\end{proof}

\section{The multiple-window extension}
\label{sec:multiwindow-extension}

\begin{proof}[Proof of Theorem~\ref{thm:main-multiwindow}]
In the rational branch, it suffices to prove necessity under
condition~\thmpart{b}, since
$\Sclass(\mathbb R^d)\subset\Szero(\mathbb R^d)$.  Pass to the diagonal
representative, preserving $S_0$ membership and frame bounds.
For $1\le j\le Q$, put $A_j=\mathcal Z_Dg_j$ and concatenate them:
\[
 A=(A_1\ \cdots\ A_Q)\in\C^{R\times QN},\qquad
 AA^*=\sum_{j=1}^Q A_jA_j^*.
\]
Summing the coefficient identity of \Cref{prop:zz-criterion-full}
over the windows gives
\begin{align*}
 \sum_{j=1}^Q\sum_{n,k\in\Z^d}
       |\langle f,M_nT_{Dk}g_j\rangle|^2
 &=\frac1R\int_{\mathcal Q_{\mathsf B}}
       \sum_{j=1}^Q\|A_j(u,\eta)^*\mathcal Zf(u,\eta)\|_2^2\,du\,d\eta\\
 &=\frac1R\int_{\mathcal Q_{\mathsf B}}
       \|A(u,\eta)^*\mathcal Zf(u,\eta)\|_2^2\,du\,d\eta.
\end{align*}
If $a_0,b_0$ are the multiwindow frame bounds, vector-Zak surjectivity
therefore yields
\[
 Ra_0I_R\le AA^*\le Rb_0I_R\quad\text{almost everywhere}.
\]
In particular, $\sigma_{\min}(A)\ge\sqrt{Ra_0}>0$, so $QN\ge R$.
Since each
$g_j\in\Szero(\mathbb R^d)$ has a continuous and bounded Zak transform
\cite{JanssenS0}, \Cref{prop:ordinary-zak-reduction} may be applied with all
$d$ coordinate pairs to every block.  The same smooth unitary row phase multiplier acts
on all blocks; hence it acts by left multiplication on their horizontal
concatenation and preserves its singular values.  The resulting $R\times QN$ matrix is continuous, quasiperiodic entrywise, and has the same positive lower
singular-value bound.  By continuity, the bound holds at every point:
a violation would persist on an open set of positive measure, contradicting
the almost-everywhere inequality.  Proposition~\ref{prop:continuous-rectangular-gap}
gives
\[
 QN\ge R+d.
\]
This directly implies strict density:
\[
 \covol(\Lambda)=\frac RN\le Q-\frac dN<Q.
\]

For a symplectically irrational lattice, the multiwindow density theorem
gives only the weak inequality $\covol(\Lambda)\le Q$.  Suppose equality
held and write $\Lambda=L_0\mathbb Z^{2d}$, so $|\det L_0|=Q$.  Since every
$g_j\in S_0(\mathbb R^d)=M^1(\mathbb R^d)$ lies in the regularity class
covered by Gerhold--Lamando--Luef, their
\cite[Lemma~3.11 and Corollary~3.39]{GerholdLamandoLuef2026} exclude a
Gabor frame generated by $Q$ such windows.  Therefore equality is impossible
and $\covol(\Lambda)<Q$.

Conversely, \Cref{prop:external-Enstad-existence} supplies a
Schwartz multiwindow Gabor frame with
\[
 \begin{aligned}
 Q_0
 &=\left\lfloor\frac{R+d-1}{N}\right\rfloor+1
   =\left\lceil\frac{R+d}{N}\right\rceil\\
 &=\left\lceil
       \frac{\nu(\Lambda^\circ)+d}{\nu(\Lambda)}
      \right\rceil
 \end{aligned}
\]
windows in the symplectically rational branch and with
$Q_0=\lfloor\covol(\Lambda)\rfloor+1$ windows in the symplectically
irrational branch. Condition~\eqref{eq:mSG} gives $QN\ge R+d$ in the
rational case and $\covol(\Lambda)<Q$ in the irrational case; since $Q$ is
an integer, either inequality implies $Q_0\le Q$.
The necessity established above shows that these values are minimal.
If $Q_0<Q$, choose one nonzero window $h$, put
$m=Q-Q_0+1$, and replace $h$ by $m$ copies $m^{-1/2}h$.  Their
frame-operator contribution is
\begin{equation*}
 \sum_{j=1}^m S_{m^{-1/2}h}
 =m\,m^{-1}S_h=S_h.
\end{equation*}
Thus the total frame operator and bounds are unchanged and the family has
exactly $Q$ windows.  Theorem~\ref{prop:Schwartz-Parsevalization} then yields a Parseval family without leaving the Schwartz class.
\end{proof}

\subsection{Projection comparison}\label{subsec:projection-comparison}
\begin{proof}[Proof of \Cref{thm:heisenberg-projection-comparison}]
In the stable matrix algebras over $A_\Lambda$, the relation
$p_\Lambda\precsim1_{A_\Lambda}^{\oplus Q}$ means that there is a
partial isometry $v$ with $v^*v=p_\Lambda$ and
$vv^*\le1_{A_\Lambda}^{\oplus Q}$, after adjoining zero blocks.
By \cite[Proposition~5.2]{EnstadThielVilalta2025}, this is equivalent
to the existence of a $Q$-window Schwartz Gabor frame.  The rational
case of \Cref{thm:main-multiwindow} therefore gives
\[
 p_\Lambda\precsim1_{A_\Lambda}^{\oplus Q}
 \quad\Longleftrightarrow\quad
 Q\nu(\Lambda)\ge\nu(\Lambda^\circ)+d.
\]
\end{proof}

\part{Back matter}\label{part:back-matter}
This part contains the acknowledgements, the first-principles appendix proof, and the bibliography.

\section*{Acknowledgements}
The authors made extensive use of ChatGPT as a research and writing
assistant in developing proof strategies, exploring constructions, checking
mathematical arguments, searching the literature, and drafting and revising
the manuscript. The authors evaluated and verified the resulting material
and take full responsibility for the paper's mathematical content and
attribution of prior work.

The authors thank John Benedetto, Hans Feichtinger, Karlheinz Gr\"ochenig,
Mihail N. Kolountzakis, Dominik St\"oger, David Walnut, and Yang Wang for
helpful discussions.

A. Caragea and G. Pfander are supported by the German Research Foundation
(DFG) Grant CA 3683/1-1.

\section*{Appendix: A first-principles Pfaffian proof of the common-zero theorem}
\setcounter{section}{0}
\renewcommand{\thesection}{A\arabic{section}}
\renewcommand{\theHsection}{A\arabic{section}}

The smooth nonvanishing construction in
\Cref{lem:explicit-scalar-seeds}, with zero components appended if needed,
proves the existence assertion of Theorem~\ref{thm:ddlt-common-zero}.
The following sections prove its common-zero implication from first
principles.  The argument is entirely coordinate based.  Its
ingredients are a
covariant smoothing operator, ordinary partial derivatives, a real
skew-symmetric matrix, the Pfaffian written as an explicit finite sum,
finite-dimensional linear algebra, and repeated applications of the
fundamental theorem of calculus on a cube.

\section{The common-zero implication}
\label{app:pfaffian-common-zero}

A continuous function
\[
 s:\R^d\times\R^d\longrightarrow\C
\]
is called \emph{quasiperiodic} if
\begin{equation*}
 s(u+n,v+m)=e^{2\pi i n\cdot v}s(u,v),
 \qquad u,v\in\R^d,\quad n,m\in\Z^d.
\end{equation*}
A vector-valued function
$S=(s_1,\ldots,s_N):\R^{2d}\to\C^N$ is quasiperiodic when
\begin{equation*}
 S(u+n,v+m)=e^{2\pi i n\cdot v}S(u,v).
\end{equation*}

\begin{theorem}\label{pf-thm:common-zero}
Let $1\le N\le d$, and let
$s_1,\ldots,s_N:\R^d\times\R^d\to\C$ be continuous
quasiperiodic functions.  Then the functions have a common zero.
Equivalently, the map
\[
 S=(s_1,\ldots,s_N):\R^{2d}\longrightarrow\C^N
\]
cannot be everywhere nonzero.
\end{theorem}

The proof has four steps.
\begin{enumerate}[label=(\arabic*)]
\item A continuous nowhere-zero quasiperiodic map can be smoothed without
      changing its quasiperiodicity or creating a zero.
\item After normalizing a smooth map to have length one, we form scalar
      coefficients $a_r$ from its first derivatives and a skew-symmetric
      matrix $F=(\partial_r a_s-\partial_s a_r)$.
\item A pointwise rank calculation forces $\Pf F=0$ everywhere.
\item The boundary increments imposed by quasiperiodicity force
      $\int_{[0,1]^{2N}}\Pf F=1$.
\end{enumerate}
The last two conclusions contradict one another.

It suffices first to prove the result when $d=N$.  Indeed, if $d>N$,
we fix the final $d-N$ coordinate pairs and apply the $d=N$ case to the
remaining variables.  We therefore work from Section~\ref{pf-sec:normalize}
onward with $N$ coordinate pairs.

\section{Covariant smoothing of continuous quasiperiodic maps}

We record the smoothing step because the Pfaffian calculation uses ordinary
partial derivatives.

Choose a nonnegative function
$\rho\in C_c^\infty(\R^{2N})$ with
\[
 \int_{\R^{2N}}\rho(x,y)\,dx\,dy=1,
\]
and put
\[
 \rho_\varepsilon(x,y)=\varepsilon^{-2N}
 \rho(x/\varepsilon,y/\varepsilon).
\]
For a continuous quasiperiodic map
$S:\R^{2N}\to\C^N$, define
\begin{equation}\label{pf-eq:smoothing}
 (\mathcal M_\varepsilon S)(u,v)
 =\int_{\R^{2N}}
 \rho_\varepsilon(x,y)\,
 e^{2\pi i u\cdot y}
 S(u-x,v-y)\,dx\,dy.
\end{equation}

\begin{lemma}\label{pf-lem:smoothing}
The function $\mathcal M_\varepsilon S$ is smooth and satisfies the same
Zak boundary rule as $S$.  Moreover,
\[
 \mathcal M_\varepsilon S\longrightarrow S
\]
uniformly on $[0,1]^{2N}$ as $\varepsilon\to0$.  Consequently, if $S$ is
nowhere zero, then $\mathcal M_\varepsilon S$ is nowhere zero for all
sufficiently small $\varepsilon>0$.
\end{lemma}

\begin{proof}
Let $n,m\in\Z^N$.  Using quasiperiodicity of $S$,
\begin{align*}
 &(\mathcal M_\varepsilon S)(u+n,v+m)\\
 &\quad=\int \rho_\varepsilon(x,y)
 e^{2\pi i(u+n)\cdot y}
 S(u+n-x,v+m-y)\,dx\,dy\\
 &\quad=\int \rho_\varepsilon(x,y)
 e^{2\pi i u\cdot y}e^{2\pi i n\cdot y}
 e^{2\pi i n\cdot(v-y)}S(u-x,v-y)\,dx\,dy\\
 &\quad=e^{2\pi i n\cdot v}
 (\mathcal M_\varepsilon S)(u,v).
\end{align*}
Thus the boundary rule is preserved exactly.

To see smoothness without differentiating the original continuous function,
make the change of variables
\[
 p=u-x,\qquad q=v-y.
\]
Then
\begin{equation*}
 (\mathcal M_\varepsilon S)(u,v)
 =\int_{\R^{2N}}
 \rho_\varepsilon(u-p,v-q)
 e^{2\pi i u\cdot(v-q)}S(p,q)\,dp\,dq.
\end{equation*}
For $(u,v)$ in a compact set, the factor
$\rho_\varepsilon(u-p,v-q)$ restricts $(p,q)$ to another compact set.
Every derivative in $u$ and $v$ falls on the smooth kernel
\[
 \rho_\varepsilon(u-p,v-q)e^{2\pi i u\cdot(v-q)},
\]
so differentiation under the integral is legitimate to every order.

For uniform convergence, rewrite \eqref{pf-eq:smoothing} as
\begin{align*}
 (\mathcal M_\varepsilon S)(u,v)-S(u,v)
  =\int \rho_\varepsilon(x,y)
  \bigl(e^{2\pi i u\cdot y}S(u-x,v-y)-S(u,v)\bigr)\,dx\,dy.
\end{align*}
On a slightly enlarged compact fundamental cube, $S$ is uniformly
continuous and bounded.  The expression in parentheses tends uniformly to
zero as $(x,y)\to(0,0)$.  The usual approximate-identity estimate therefore
gives uniform convergence on $[0,1]^{2N}$.

Finally, $\norm{S(u,v)}_2$ is $\Z^{2N}$-periodic.  If $S$ has no zero, then
\[
 \delta=\min_{[0,1]^{2N}}\norm{S(u,v)}_2>0.
\]
For sufficiently small $\varepsilon$ the uniform approximation error is
less than $\delta/2$, and then
$\norm{\mathcal M_\varepsilon S}_2\ge\delta/2$ everywhere.
\end{proof}

\section{Normalization and the derivative matrix}\label{pf-sec:normalize}

Assume for contradiction that a smooth quasiperiodic map
\[
 S:\R^N\times\R^N\longrightarrow\C^N
\]
is nowhere zero.  Normalize it by
\begin{equation}\label{pf-eq:normalize}
 U(x,y)=\frac{S(x,y)}{\norm{S(x,y)}_2}.
\end{equation}
Then
\begin{equation}\label{pf-eq:unit}
 U(x,y)^*U(x,y)=1
\end{equation}
and
\begin{equation*}
 U(x+n,y+m)=e^{2\pi i n\cdot y}U(x,y).
\end{equation*}

We order the real coordinates as
\begin{equation}\label{pf-eq:coordinate-order}
 t_1=x_1,\quad t_2=y_1,\quad t_3=x_2,\quad t_4=y_2,
 \quad\ldots,\quad t_{2N-1}=x_N,\quad t_{2N}=y_N,
\end{equation}
and write $\partial_r=\partial/\partial t_r$.

For $r=1,\ldots,2N$, define
\begin{equation}\label{pf-eq:a-r}
 a_r(x,y)=\frac{1}{2\pi i}\,U(x,y)^*\partial_rU(x,y).
\end{equation}
These coefficients are real-valued.  Indeed, differentiating
\eqref{pf-eq:unit} gives
\[
 (\partial_rU)^*U+U^*\partial_rU=0,
\]
so
\[
 \overline{U^*\partial_rU}
 = (\partial_rU)^*U
 =-U^*\partial_rU.
\]
Thus $U^*\partial_rU$ is purely imaginary.

For clarity, write
\[
 a_{x_k}=\frac{1}{2\pi i}U^*\partial_{x_k}U,
 \qquad
 a_{y_k}=\frac{1}{2\pi i}U^*\partial_{y_k}U.
\]

\subsubsection*{Boundary relations for the coefficients}
From
\[
 U(x+e_j,y)=e^{2\pi i y_j}U(x,y)
\]
we obtain
\begin{equation}\label{pf-eq:ax-boundary}
 a_{x_k}(x+e_j,y)=a_{x_k}(x,y).
\end{equation}
For the $y_k$ derivative,
\begin{align*}
 \partial_{y_k}U(x+e_j,y)
 &=\partial_{y_k}\bigl(e^{2\pi i y_j}U(x,y)\bigr)\\
 &=e^{2\pi i y_j}
 \bigl(2\pi i\delta_{jk}U(x,y)+\partial_{y_k}U(x,y)\bigr).
\end{align*}
Multiplying on the left by $U(x+e_j,y)^*=e^{-2\pi i y_j}U(x,y)^*$ gives
\begin{equation}\label{pf-eq:ay-boundary}
 a_{y_k}(x+e_j,y)=a_{y_k}(x,y)+\delta_{jk}.
\end{equation}
Since $U$ is periodic in every $y_j$ variable,
\begin{equation}\label{pf-eq:y-periodicity-a}
 a_r(x,y+e_j)=a_r(x,y).
\end{equation}

Define the real skew-symmetric matrix
\begin{equation}\label{pf-eq:F-def}
 F=(F_{rs})_{r,s=1}^{2N},
 \qquad
 F_{rs}=\partial_ra_s-\partial_sa_r.
\end{equation}
The additive constants in \eqref{pf-eq:ay-boundary} disappear after
differentiation, so every entry $F_{rs}$ is $\Z^{2N}$-periodic.

\section{The Pfaffian and pointwise degeneracy}

For a real skew-symmetric $2N\times2N$ matrix $B=(B_{rs})$, define
\begin{equation*}
 \Pf B
 =\frac{1}{2^NN!}
 \sum_{i_1,\ldots,i_{2N}=1}^{2N}
 \varepsilon_{i_1\ldots i_{2N}}
 B_{i_1i_2}B_{i_3i_4}\cdots B_{i_{2N-1}i_{2N}},
\end{equation*}
where $\varepsilon_{i_1\ldots i_{2N}}$ is the alternating symbol.
We use the elementary algebraic identity
\begin{equation}\label{pf-eq:det-pf}
 \det B=(\Pf B)^2.
\end{equation}
It follows directly by expanding both sides as alternating polynomials in the
entries of $B$; it is also proved by induction using the expansion of the
Pfaffian along its first row.

\begin{proposition}\label{pf-prop:pointwise-zero}
For the matrix $F$ in \eqref{pf-eq:F-def},
\[
 \Pf F(x,y)=0
\]
at every point $(x,y)$.
\end{proposition}

\begin{proof}
Differentiate \eqref{pf-eq:a-r}.  The product rule gives
\[
 \partial_r(U^*\partial_sU)
 = (\partial_rU)^*\partial_sU+U^*\partial_r\partial_sU,
\]
and the same formula with $r$ and $s$ interchanged.  Since $U$ is smooth,
$\partial_r\partial_sU=\partial_s\partial_rU$, and the second-derivative
terms cancel.  Hence
\begin{align}
 F_{rs}
 &=\frac{1}{2\pi i}
 \left[
 \partial_r(U^*\partial_sU)-\partial_s(U^*\partial_rU)
 \right]\notag\\
 &=\frac{1}{2\pi i}
 \left[
 (\partial_rU)^*\partial_sU-(\partial_sU)^*\partial_rU
 \right].
 \label{pf-eq:F-derivative}
\end{align}

Fix a point $(x,y)$ for the remainder of the argument and abbreviate
$U=U(x,y)$.  Put
\begin{equation*}
 Q=\Id_N-UU^*.
\end{equation*}
We verify explicitly that $Q$ is the orthogonal projection onto the complex
orthogonal complement of $U$.  Since $U^*U=1$,
\[
 Q^*=\Id_N-(UU^*)^*=Q
\]
and
\[
 Q^2=(\Id_N-UU^*)^2
 =\Id_N-2UU^*+U(U^*U)U^*
 =\Id_N-UU^*=Q.
\]
Moreover,
\[
 QU=U-U(U^*U)=0,
 \qquad
 U^*Q=U^*-(U^*U)U^*=0.
\]
Thus $\operatorname{ran}Q\subseteq U^\perp$.  Conversely, if $z\in U^\perp$,
then $U^*z=0$ and therefore $Qz=z$.  It follows that
\[
 \operatorname{ran}Q=U^\perp
 =\{z\in\C^N:U^*z=0\}.
\]
Because $U\ne0$, this space has complex dimension $N-1$ and real dimension
$2N-2$.

For each $r=1,\ldots,2N$, set
\begin{equation*}
 w_r=Q\partial_rU,
 \qquad
 \alpha_r=U^*\partial_rU.
\end{equation*}
The vector $w_r$ belongs to $U^\perp$, because $U^*Q=0$.  The identity
$\Id_N=UU^*+Q$ gives the orthogonal decomposition
\[
 \partial_rU
 =(UU^*+Q)\partial_rU
 =U(U^*\partial_rU)+Q\partial_rU
 =U\alpha_r+w_r.
\]
Differentiating $U^*U=1$ in the $t_r$ direction gives
\[
 (\partial_rU)^*U+U^*\partial_rU=0,
\]
so $\overline{\alpha_r}=-\alpha_r$.  Thus there is a real number
$\beta_r$ with $\alpha_r=i\beta_r$.  In particular,
\[
 \overline{\alpha_r}\alpha_s
 =(-i\beta_r)(i\beta_s)=\beta_r\beta_s
 =\overline{\alpha_s}\alpha_r.
\]
Also $U^*w_s=0$ and $w_r^*U=0$.  Consequently,
\begin{align*}
 (\partial_rU)^*\partial_sU
 &=(\overline{\alpha_r}U^*+w_r^*)
   (U\alpha_s+w_s)
   =\overline{\alpha_r}\alpha_s+w_r^*w_s,\\
 (\partial_sU)^*\partial_rU
 &=\overline{\alpha_s}\alpha_r+w_s^*w_r.
\end{align*}
The scalar terms are equal and cancel in \eqref{pf-eq:F-derivative}.  We
therefore obtain
\begin{equation}\label{pf-eq:F-w}
 F_{rs}
 =\frac{1}{2\pi i}\bigl(w_r^*w_s-w_s^*w_r\bigr)
 =\frac{1}{\pi}\operatorname{Im}(w_r^*w_s).
\end{equation}
The second equality follows from
$w_s^*w_r=\overline{w_r^*w_s}$.

We now turn \eqref{pf-eq:F-w} into an explicit real matrix factorization.
Choose a complex orthonormal basis
$e_1,\ldots,e_{N-1}$ of $U^\perp$.  Write
\[
 w_r=\sum_{\ell=1}^{N-1}(p_{\ell r}+iq_{\ell r})e_\ell,
 \qquad p_{\ell r},q_{\ell r}\in\R.
\]
Let $W$ be the real $(2N-2)\times2N$ matrix whose $r$th column is
\[
 W_{\cdot r}
 =(p_{1r},q_{1r},p_{2r},q_{2r},\ldots,
   p_{N-1,r},q_{N-1,r})^T,
\]
and let
\[
 \Omega=\frac1\pi\operatorname{diag}
 \left(
 \begin{pmatrix}0&1\\-1&0\end{pmatrix},\ldots,
 \begin{pmatrix}0&1\\-1&0\end{pmatrix}
 \right)
 \in M_{2N-2}(\R).
\]
A direct calculation gives
\[
 (W_{\cdot r})^T\Omega W_{\cdot s}
 =\frac1\pi\sum_{\ell=1}^{N-1}
   (p_{\ell r}q_{\ell s}-q_{\ell r}p_{\ell s})
 =\frac1\pi\operatorname{Im}(w_r^*w_s)
 =F_{rs}.
\]
Thus, entry by entry,
\begin{equation*}
 F=W^T\Omega W.
\end{equation*}
When $N=1$, the space $U^\perp$ is zero-dimensional, $W$ has no rows, and
this formula simply says $F=0$.

Finally,
\[
 \rank F=\rank(W^T\Omega W)
 \le \rank W\le2N-2<2N.
\]
Hence the $2N\times2N$ matrix $F$ is singular, and so $\det F=0$.  The
identity \eqref{pf-eq:det-pf} now yields $(\Pf F)^2=0$, and therefore
$\Pf F=0$ at the chosen point.  Since the point was arbitrary, the
conclusion holds everywhere.
\end{proof}

\section{The boundary-Pfaffian identity and the common-zero contradiction}
\label{app:boundary-pfaffian}

We now isolate the boundary identity and complete the contradiction argument.

\subsection{An abstract boundary-Pfaffian identity}

The next lemma is independent of the map $U$.  It uses only smooth real
functions with the boundary increments found in
\eqref{pf-eq:ax-boundary}--\eqref{pf-eq:y-periodicity-a}.  We keep the
coordinate order
\[
 x_1,y_1,x_2,y_2,\ldots,x_N,y_N.
\]
Thus, inside an alternating symbol, the notation $x_j$ means the index
$2j-1$ and the notation $y_j$ means the index $2j$.

\begin{lemma}\label{pf-lem:boundary-pfaffian}
For $j=1,\ldots,N$, let $a_{x_j},a_{y_j}$ be smooth real functions on
$\R^{2N}$ satisfying, for every $j,k=1,\ldots,N$,
\begin{align}
 a_{x_k}(x+e_j,y)&=a_{x_k}(x,y),\notag\\
 a_{y_k}(x+e_j,y)&=a_{y_k}(x,y)+\delta_{jk},
 \label{pf-eq:abstract-bc-2}\\
 a_r(x,y+e_j)&=a_r(x,y).
 \label{pf-eq:abstract-bc-3}
\end{align}
Order the coordinates as in \eqref{pf-eq:coordinate-order}, and define
$F_{rs}=\partial_ra_s-\partial_sa_r$.  Then
\begin{equation*}
 \int_{[0,1]^{2N}}\Pf F\,dt_1\cdots dt_{2N}=1.
\end{equation*}
\end{lemma}

Before proving the lemma, note that every entry of $F$ is periodic in all
$2N$ variables.  Indeed, across an $x_j$-face each coefficient $a_s$
changes by a constant, namely either $0$ or $1$.  Differentiating the
boundary relation therefore gives
\[
 \partial_ra_s(x+e_j,y)=\partial_ra_s(x,y).
\]
Across a $y_j$-face the coefficients themselves are periodic.  Hence
$F_{rs}$ has identical values on every pair of opposite faces of the unit
cube.

The proof of the lemma is a repeated boundary calculation.  We first derive
an explicit divergence identity.
For $r=1,\ldots,2N$, define
\begin{align}
 V_r
 &=\frac{1}{2^{N-1}N!}
 \sum_{i_2,\ldots,i_{2N}=1}^{2N}
 \varepsilon_{r i_2\ldots i_{2N}}\,
 a_{i_2}
 F_{i_3i_4}F_{i_5i_6}\cdots F_{i_{2N-1}i_{2N}}.
 \label{pf-eq:V-r}
\end{align}
Here and below a term is automatically zero when two indices coincide,
because the alternating symbol then vanishes.

\begin{proposition}\label{pf-prop:divergence}
The functions in \eqref{pf-eq:V-r} satisfy
\begin{equation}\label{pf-eq:divergence}
 \sum_{r=1}^{2N}\partial_rV_r=\Pf F.
\end{equation}
\end{proposition}

\begin{proof}
Replace the outer index $r$ by $i_1$ and apply the product rule.  There are
two kinds of terms: the derivative may fall on $a_{i_2}$, or it may fall
on one of the $N-1$ factors
$F_{i_3i_4},\ldots,F_{i_{2N-1}i_{2N}}$.  Thus
\[
 \sum_{r=1}^{2N}\partial_rV_r=A+\sum_{p=2}^N B_p,
\]
where
\begin{align}
 A
 &=\frac{1}{2^{N-1}N!}
 \sum_{i_1,\ldots,i_{2N}}
 \varepsilon_{i_1i_2\ldots i_{2N}}
 (\partial_{i_1}a_{i_2})
 F_{i_3i_4}\cdots F_{i_{2N-1}i_{2N}},
 \label{pf-eq:derivative-a}
\end{align}
and, for $p=2,\ldots,N$,
\begin{align*}
 B_p
 &=\frac{1}{2^{N-1}N!}
 \sum_{i_1,\ldots,i_{2N}}
 \varepsilon_{i_1i_2\ldots i_{2N}}a_{i_2}
 (\partial_{i_1}F_{i_{2p-1}i_{2p}})\\
 &\hspace{28mm}\times
 \prod_{\substack{q=2\\q\ne p}}^N
 F_{i_{2q-1}i_{2q}}.
\end{align*}

We first compute $A$.  Let
\[
 \Sigma=
 \sum_{i_1,\ldots,i_{2N}}
 \varepsilon_{i_1i_2\ldots i_{2N}}
 (\partial_{i_1}a_{i_2})
 F_{i_3i_4}\cdots F_{i_{2N-1}i_{2N}}.
\]
Interchange the two dummy indices $i_1$ and $i_2$.  The alternating symbol
changes sign, while the remaining factors are unchanged.  Hence the same
sum also equals
\[
 -\sum_{i_1,\ldots,i_{2N}}
 \varepsilon_{i_1i_2\ldots i_{2N}}
 (\partial_{i_2}a_{i_1})
 F_{i_3i_4}\cdots F_{i_{2N-1}i_{2N}}.
\]
Adding these two representations of $\Sigma$ gives
\begin{align*}
 2\Sigma
 &=\sum_{i_1,\ldots,i_{2N}}
 \varepsilon_{i_1i_2\ldots i_{2N}}
 (\partial_{i_1}a_{i_2}-\partial_{i_2}a_{i_1})
 F_{i_3i_4}\cdots F_{i_{2N-1}i_{2N}}\\
 &=\sum_{i_1,\ldots,i_{2N}}
 \varepsilon_{i_1i_2\ldots i_{2N}}
 F_{i_1i_2}F_{i_3i_4}\cdots F_{i_{2N-1}i_{2N}}.
\end{align*}
Therefore the coefficient in \eqref{pf-eq:derivative-a} gives
\[
 A=\frac{1}{2^NN!}
 \sum_{i_1,\ldots,i_{2N}}
 \varepsilon_{i_1\ldots i_{2N}}
 F_{i_1i_2}\cdots F_{i_{2N-1}i_{2N}}
 =\Pf F.
\]

It remains to prove $B_p=0$ for each $p$.  The elementary identity used for
this cancellation is
\begin{equation}\label{pf-eq:bianchi}
 \partial_rF_{ab}+\partial_aF_{br}+\partial_bF_{ra}=0.
\end{equation}
To verify it from first principles, substitute
$F_{ab}=\partial_aa_b-\partial_ba_a$:
\begin{align*}
 &\partial_rF_{ab}+\partial_aF_{br}+\partial_bF_{ra}\\
 &\quad=
 \partial_r\partial_aa_b-\partial_r\partial_ba_a
 +\partial_a\partial_ba_r-\partial_a\partial_ra_b
 +\partial_b\partial_ra_a-\partial_b\partial_aa_r=0.
\end{align*}
The first and fourth terms cancel, the second and fifth terms cancel, and
the third and sixth terms cancel, because mixed partial derivatives
commute.

Fix $p$ and abbreviate
$a=i_{2p-1}$ and $b=i_{2p}$.  In the sum defining $B_p$, cyclically rename
the three dummy indices
\[
 (i_1,a,b)\longmapsto(a,b,i_1).
\]
A three-cycle is an even permutation, so the alternating symbol is
unchanged.  The value of the sum is therefore unchanged if
$\partial_{i_1}F_{ab}$ is replaced by $\partial_aF_{bi_1}$, and it is also
unchanged if it is replaced by $\partial_bF_{i_1a}$.  Averaging these three
equal representations yields
\begin{align*}
 B_p
 &=\frac{1}{3\cdot2^{N-1}N!}
 \sum_{i_1,\ldots,i_{2N}}
 \varepsilon_{i_1i_2\ldots i_{2N}}a_{i_2}\\
 &\quad\times
 \bigl(
 \partial_{i_1}F_{ab}+\partial_aF_{bi_1}
 +\partial_bF_{i_1a}
 \bigr)
 \prod_{\substack{q=2\\q\ne p}}^N
 F_{i_{2q-1}i_{2q}}=0
\end{align*}
by \eqref{pf-eq:bianchi}.  Thus every $B_p$ vanishes, while $A=\Pf F$,
which proves \eqref{pf-eq:divergence}.
\end{proof}

We now verify the boundary identity by expanding the divergence formula.

\begin{proof}[Proof of Lemma~\ref{pf-lem:boundary-pfaffian}]
We argue by induction on $N$.  Throughout, an integral over a face means
integration with respect to all coordinates that remain free on that face.

For $N=1$, the coordinate order is $(x_1,y_1)$ and
\[
 \Pf F=F_{x_1y_1}
 =\partial_{x_1}a_{y_1}-\partial_{y_1}a_{x_1}.
\]
Applying the one-variable fundamental theorem of calculus first in $x_1$
and then in $y_1$ gives
\begin{align*}
 \int_0^1\int_0^1\Pf F\,dx_1\,dy_1
 &=\int_0^1
 \bigl(a_{y_1}(1,y_1)-a_{y_1}(0,y_1)\bigr)\,dy_1\\
 &\quad-
 \int_0^1
 \bigl(a_{x_1}(x_1,1)-a_{x_1}(x_1,0)\bigr)\,dx_1.
\end{align*}
The first bracket equals $1$ by
\eqref{pf-eq:abstract-bc-2}, and the second bracket equals $0$ by
\eqref{pf-eq:abstract-bc-3}.  Hence the integral equals $1$.

Assume now that the assertion has been proved for $N-1$ coordinate pairs.
Integrate \eqref{pf-eq:divergence} over the unit cube.  For a coordinate
$t_r$, write $dt_{\widehat r}$ for integration in all variables except
$t_r$.  The fundamental theorem of calculus gives
\begin{equation}\label{pf-eq:boundary-sum}
 \int_{[0,1]^{2N}}\Pf F
 =\sum_{r=1}^{2N}
 \int_{[0,1]^{2N-1}}
 \left(V_r\big|_{t_r=1}-V_r\big|_{t_r=0}\right)
 \,dt_{\widehat r}.
\end{equation}

We first examine a pair of $y_j$-faces.  By
\eqref{pf-eq:abstract-bc-3}, every coefficient $a_s$ has the same value at
$y_j=0$ and $y_j=1$.  As observed before the divergence calculation, every
entry of $F$ is also periodic.  Since $V_{y_j}$ is a polynomial expression
in the $a_s$ and $F_{rs}$, it follows pointwise that
\[
 V_{y_j}\big|_{y_j=1}=V_{y_j}\big|_{y_j=0}.
\]
Thus every pair of $y_j$-faces contributes zero to
\eqref{pf-eq:boundary-sum}.

Now fix $j$ and consider the two $x_j$-faces.  In the notation of
\eqref{pf-eq:V-r}, take $r=x_j$.  All $F$-factors agree on the two faces.
For the remaining factor, the boundary relations say
\[
 a_{i_2}(x+e_j,y)-a_{i_2}(x,y)
 =\begin{cases}
 1,&i_2=y_j,\\
 0,&i_2\ne y_j.
 \end{cases}
\]
Therefore only the terms with $i_2=y_j$ survive in the difference, and
\begin{align}
 &V_{x_j}\big|_{x_j=1}-V_{x_j}\big|_{x_j=0}\notag\\
 &\quad=\frac{1}{2^{N-1}N!}
 \sum_{i_3,\ldots,i_{2N}}
 \varepsilon_{x_jy_ji_3\ldots i_{2N}}
 F_{i_3i_4}\cdots F_{i_{2N-1}i_{2N}}.
 \label{pf-eq:x-face-jump}
\end{align}
A nonzero term in this sum must use each of the remaining $2N-2$ indices
exactly once.

Let
\[
 I_j=\{x_1,y_1,\ldots,x_N,y_N\}\setminus\{x_j,y_j\},
\]
and let $F^{(j)}$ be the submatrix of $F$ with rows and columns indexed by
$I_j$, evaluated on the face $x_j=0$.  The same submatrix is obtained on
$x_j=1$ because $F$ is periodic.  To compare alternating symbols, move the
ordered pair $(x_j,y_j)$ from its positions $(2j-1,2j)$ to the first two
positions.  Moving $x_j$ requires $2j-2$ transpositions, and after that
moving $y_j$ requires another $2j-2$ transpositions.  The total number
$4j-4$ is even, so no sign is introduced.  Hence the alternating symbol in
\eqref{pf-eq:x-face-jump} is exactly the alternating symbol associated with
the inherited order on $I_j$.

The Pfaffian of the $(2N-2)\times(2N-2)$ matrix $F^{(j)}$ is
\[
 \Pf F^{(j)}
 =\frac{1}{2^{N-1}(N-1)!}
 \sum_{i_3,\ldots,i_{2N}\in I_j}
 \varepsilon^{(j)}_{i_3\ldots i_{2N}}
 F_{i_3i_4}\cdots F_{i_{2N-1}i_{2N}},
\]
where $\varepsilon^{(j)}$ is the alternating symbol for the inherited
coordinate order.  Since $N!=N(N-1)!$, comparison with
\eqref{pf-eq:x-face-jump} gives
\begin{equation}\label{pf-eq:x-face-pf-minor}
 V_{x_j}\big|_{x_j=1}-V_{x_j}\big|_{x_j=0}
 =\frac1N\Pf F^{(j)}.
\end{equation}

It remains to integrate this minor.  Fix $y_j\in[0,1]$ and set $x_j=0$.
For every $k\ne j$, restrict the functions $a_{x_k}$ and $a_{y_k}$ to this
slice and regard them as functions of the remaining $2N-2$ variables.
For shifts in a remaining $x_\ell$ variable, $\ell\ne j$, they satisfy
\[
 a_{x_k}(x+e_\ell,y)=a_{x_k}(x,y),
 \qquad
 a_{y_k}(x+e_\ell,y)=a_{y_k}(x,y)+\delta_{\ell k},
\]
and they are periodic in every remaining $y_\ell$ variable.  Thus these
restricted functions satisfy exactly the hypotheses of the lemma with
$N-1$ coordinate pairs.  Their skew derivative matrix, taken only with
respect to the remaining variables, is precisely $F^{(j)}$.  The induction
hypothesis therefore gives
\begin{equation*}
 \int_{[0,1]^{2N-2}}\Pf F^{(j)}=1
\end{equation*}
for every fixed value of $y_j$.

Using \eqref{pf-eq:x-face-pf-minor}, the total contribution of the two
$x_j$-faces to \eqref{pf-eq:boundary-sum} is therefore
\[
 \frac1N\int_0^1
 \left(\int_{[0,1]^{2N-2}}\Pf F^{(j)}\right)dy_j
 =\frac1N\int_0^1 1\,dy_j
 =\frac1N.
\]
There are $N$ choices of $j$.  All $y_j$-face contributions vanish, and
each pair of $x_j$-faces contributes $1/N$.  Hence
\[
 \int_{[0,1]^{2N}}\Pf F
 =\sum_{j=1}^N\frac1N=1,
\]
which completes the induction.

For example, when $N=2$ and the coordinates are
$(x_1,y_1,x_2,y_2)$,
\[
 \Pf F
 =F_{x_1y_1}F_{x_2y_2}
  -F_{x_1x_2}F_{y_1y_2}
  +F_{x_1y_2}F_{y_1x_2}.
\]
The argument above says that the $y_1$- and $y_2$-faces cancel, while the
$x_1$-faces and the $x_2$-faces each contribute $1/2$.
\end{proof}

\subsection{The contradiction and the continuous theorem}

The remaining step rules out a smooth nowhere-zero quasiperiodic map and then returns
to the original continuous function by approximation.

\begin{proposition}\label{pf-prop:smooth-no-nonzero}
There is no smooth nowhere-zero quasiperiodic map
\[
 S:\R^N\times\R^N\longrightarrow\C^N.
\]
\end{proposition}

\begin{proof}
Assume that such an $S$ exists.  Normalize it to $U$ as in
\eqref{pf-eq:normalize}, and define the coefficients $a_r$ and matrix $F$ as in
\eqref{pf-eq:a-r} and \eqref{pf-eq:F-def}.  Proposition~\ref{pf-prop:pointwise-zero}
gives
\[
 \Pf F(x,y)=0
\]
for every $(x,y)$.  Therefore
\[
 \int_{[0,1]^{2N}}\Pf F=0.
\]
On the other hand, the boundary relations
\eqref{pf-eq:ax-boundary}--\eqref{pf-eq:y-periodicity-a} satisfy the hypotheses of
Lemma~\ref{pf-lem:boundary-pfaffian}, which gives
\[
 \int_{[0,1]^{2N}}\Pf F=1.
\]
This is impossible.
\end{proof}

\begin{proof}[Proof of Theorem~\ref{pf-thm:common-zero}]
First suppose $d=N$ and assume that the continuous map
$S=(s_1,\ldots,s_N)$ is nowhere zero.  Lemma~\ref{pf-lem:smoothing} produces,
for sufficiently small $\varepsilon>0$, a smooth nowhere-zero
quasiperiodic map $\mathcal M_\varepsilon S$.  This contradicts
Proposition~\ref{pf-prop:smooth-no-nonzero}.

Now let $d>N$.  Fix the final $d-N$ coordinate pairs, for example at zero:
\[
 u_{N+1}=\cdots=u_d=0,
 \qquad
 v_{N+1}=\cdots=v_d=0.
\]
The restricted functions retain the Zak boundary rule in the first $N$
coordinate pairs.  The already proved case $d=N$ gives a common zero of the
restricted functions, and therefore a common zero of the original
functions.
\end{proof}

\subsection{The first two dimensions}

The formulas become especially transparent for small $N$.

For $N=1$,
\[
 F=
 \begin{pmatrix}
 0&F_{x_1y_1}\\
 -F_{x_1y_1}&0
 \end{pmatrix},
 \qquad
 \Pf F=F_{x_1y_1}.
\]
The boundary calculation is simply
\[
 \int_{[0,1]^2}
 (\partial_{x_1}a_{y_1}-\partial_{y_1}a_{x_1})=1,
\]
which is the ordinary winding-number increment.

For $N=2$, in the coordinate order
$(x_1,y_1,x_2,y_2)$,
\begin{align*}
 \Pf F
 &=F_{x_1y_1}F_{x_2y_2}
   -F_{x_1x_2}F_{y_1y_2}
   +F_{x_1y_2}F_{y_1x_2}.
\end{align*}
The pointwise rank estimate gives $\rank F\le2$, so this expression is zero
at every point.  The boundary calculation nevertheless gives
\[
 \int_{[0,1]^4}\Pf F=1.
\]
Thus the higher-dimensional contradiction is an iterated version of the
one-dimensional winding computation.

\subsection{Concluding comment}

The proof is entirely coordinate based.  Nonvanishing allows us to
normalize the tuple.  The normalized tuple produces a skew derivative
matrix $F$.  Its columns are generated by vectors in a real space of
dimension $2N-2$, so its top Pfaffian vanishes pointwise.  The Zak boundary
phase, however, contributes one unit of winding in each coordinate pair,
and the explicit divergence calculation packages those increments into
\[
 \int_{[0,1]^{2N}}\Pf F=1.
\]
The contradiction proves the common-zero theorem from ordinary calculus
and finite-dimensional linear algebra.

\sloppy
% Compact reference type leaves room for the shared affiliation and emails.
\renewcommand{\bibfont}{\small}

\end{document}